\documentclass[11pt,a4paper,twoside,
]{article}

\usepackage[english,french]{babel}
\usepackage{amsfonts}
\usepackage{amsmath,amssymb,amscd}
\usepackage{mathrsfs}  
\usepackage{amsthm}
\usepackage{a4wide}   
\usepackage[T1]{fontenc} 
\usepackage{newlfont} 
\usepackage{color}

\usepackage[
backref=page,bookmarksopen=true]{hyperref} 

\usepackage[all]{xy}  

\usepackage[applemac]{inputenc}       

\def\choixcompteur{subsection}
\newtheorem{theo}[\choixcompteur]{Th\'eor\`eme}

\newtheorem{prop}[\choixcompteur]{Proposition}
\newtheorem{lemm}[\choixcompteur]{Lemme}
\newtheorem{coro}[\choixcompteur]{Corollaire}

\theoremstyle{definition}

\newtheorem{defi}[\choixcompteur]{D\'efinition}

\newtheorem{exem}[\choixcompteur]{Exemple}
\newtheorem{exems}[\choixcompteur]{Exemples}

\newtheorem{rema}[\choixcompteur]{Remarque}
\newtheorem{remas}[\choixcompteur]{Remarques}

\newtheorem*{exem*}{Exemple}
\newtheorem*{exems*}{Exemples}
\newtheorem*{exam*}{Exemple}
\newtheorem*{exams*}{Exemples}
\newtheorem*{rema*}{Remarque}
\newtheorem*{remas*}{Remarques}
\newtheorem*{NB}{N.B}

\theoremstyle{definition}
\newtheorem*{defi*}{D\'efinition}
\newtheorem*{defiprop*}{D\'efinition-Proposition}

\theoremstyle{plain}
\newtheorem*{prop*}{Proposition}
\newtheorem*{lemm*}{Lemme}
\newtheorem*{coro*}{Corollaire}
\newtheorem*{theo*}{Th\'eor\`eme}

 \def\cdr@enoncedef{%
 \newenvironment{enonce*}[2][plain]%
 {\let\cdrenonce\relax \theoremstyle{##1}%
 \newtheorem*{cdrenonce}{##2}%
 \begin{cdrenonce}}%
 {\end{cdrenonce}}   }%

 \AtBeginDocument{%
    \cdr@enoncedef}

\def\cf{{\it cf.\/}\ }
\def\ie{{\it i.e.\/}\ }
\def\eg{{\it e.g.\/}\ }

\def\lc{{\it l.c.\/}\ }

\newcommand{\set}[1]{\{ #1\} } 
\def\resp{{resp.,\/}\ }

\def\parni{\par\noindent}

\def\ed{ \'editeur}
\def\eds{\'editeurs}

\def\Aff{\mathrm{Aff}}
\def\End{\mathrm{End}}
\def\germ{\mathrm{germ}}
\def\conv{\mathrm{conv}}

\def\bc{\buildrel\circ\over}

\def\N{{\mathbb N}}    
\def\Z{{\mathbb Z}}
\def\Q{{\mathbb Q}}
\def\R{{\mathbb R}}
\def\C{{\mathbb C}}

\def\F{{\mathbb F}}
\def\A{{\mathbb A}}

\newcommand{\g}[1]{\mathfrak{#1}} 

\def\qa{\alpha}     
\def\qb{\beta}
\def\qd{\delta}
\def\qe{\varepsilon}
\def\qf{\varphi}
\def\qg{\gamma}

 \def\ql{\lambda}
\def\qm{\mu}
\def\qn{\nu}
\def\qo{\omega}
\def\qp{\pi}
\def\qpp{\varpi}
\def\qr{\rho}
\def\qs {\sigma}
\def\qt{\tau}

\def\QD{\Delta}
\def\QF{\Phi}

\def\QL{\Lambda}
\def\QO{\Omega}

\def\shb{{\mathcal B}}

\def\shf{{\mathcal F}}

\def\shi{{\mathscr I}}

\def\shm{{\mathcal M}}

\def\sho{{\mathcal O}}
\def\shp{{\mathcal P}}

\def\shs{{\mathcal S}}
\def\sht{{\mathcal T}}
\def\shu{{\mathcal U}}
\def\shv{{\mathcal V}}

\def\SHI{{\mathscr I}}

\def\SHS{{\mathscr S}}

\begin{document}


\title{Groupes de Kac-Moody d\'eploy\'es \\sur un corps local,\goodbreak II.  Masures ordonn\'ees}
\author{Guy Rousseau}

\date{12 Mai 2015}

\maketitle

%






\begin{abstract}
Pour un groupe de Kac-Moody d\'eploy\'e (au sens de J. Tits) sur un corps r\'eellement valu\'e quelconque, on construit une masure affine ordonn\'ee sur laquelle ce groupe agit.
Cette construction g\'en\'eralise celle d\'ej\`a effectu\'ee par S. Gaussent et l'auteur quand le corps r\'esiduel contient le corps des complexes 
et celle de F. Bruhat et J. Tits quand le groupe est r\'eductif. On montre que cette masure v\'erifie bien toutes les propri\'et\'es des masures affines ordonn\'ees comme d\'efinies pr\'ec\'edemment par l'auteur.
On utilise le groupe de Kac-Moody maximal au sens d'O. Mathieu et on montre quelques r\'esultats pour celui-ci sur un corps quelconque; en particulier on prouve, dans certains cas, un r\'esultat de simplicit\'e pour ce groupe maximal.

Un erratum \`a la fin corrige une erreur dans le contre-exemple de \ref{4.12} 3 (c).
\end{abstract}
\footnotetext[1]{MSC (2010): 20G44; 20G25; 20E42; 51E24; 20E32}
\selectlanguage{english}
\begin{abstract}
For a split Kac-Moody group (in J. Tits' definition) over a field endowed with a real valuation, we build an ordered affine hovel on which the group acts. This construction generalizes the one already done by S. Gaussent and the author when the residue field contains the complex field 
and the one by F. Bruhat and J. Tits when the group is reductive.
We prove that this hovel has all the properties of ordered affine hovels (masures affines ordonn\'ees) as defined previously by the author.
We use the maximal Kac-Moody group as defined by O. Mathieu and we prove a few new results about it over any field; in particular we prove, in some cases, a simplicity result for this group.

At the end an erratum corrects a mistake in the counter-example of \ref{4.12} 3 (c).
\end{abstract}
\selectlanguage{french}

\setcounter{tocdepth}{1}    
\tableofcontents

\section*{Introduction}
\label{seIntro}

\bigskip
\par L'\'etude des groupes de Kac-Moody sur un corps local a \'et\'e initi\'ee par Howard Garland \cite{Gd-95} pour certains groupes de lacets. Dans \cite{Ru-06} on a construit un immeuble "microaffine" pour tous les groupes de Kac-Moody (minimaux au sens de J. Tits) sur un corps muni d'une valuation r\'eelle. C'est un immeuble (en g\'en\'eral non discret) avec les bonnes propri\'et\'es habituelles des immeubles. Cependant cet immeuble microaffine n'est pas l'analogue des immeubles de F. Bruhat et J. Tits pour les groupes r\'eductifs. Il correspond plut\^ot \`a leur fronti\`ere dans la compactification de Satake ou compactification poly\'edrique. De plus il ne traduit pas les d\'ecompositions de Cartan de \cite{Gd-95}.

\par Une autre construction est envisageable pour un groupe de Kac-Moody sur un corps r\'eellement valu\'e. Elle est la g\'en\'eralisation directe de celle de Bruhat-Tits (\cite{BtT-72} et \cite{BtT-84a}) et traduit les d\'ecompositions de Cartan de  \cite{Gd-95} dans le cas des groupes de lacets (voir \ref{5.14}.1).
Cependant, comme ces d\'ecompositions ne sont v\'erifi\'ees que pour un sous-semi-groupe,  
  l'espace $\shi$ ainsi construit \`a la Bruhat-Tits est tel que deux points quelconques ne sont pas toujours dans un m\^eme appartement. Il ne m\'erite donc pas le nom d'immeuble. Il s'est av\'er\'e cependant utile et a donc \'et\'e consid\'er\'e sous le nom de masure (hovel).

\par La construction de cette masure a \'et\'e effectu\'ee dans \cite{GR-08} et on a pu l'utiliser pour des r\'esultats en th\'eorie des repr\'esentations. Le corps valu\'e int\'eressant dans ce cadre est le corps des s\'eries de Laurent complexes $\C(\!(t)\!)$. On s'est donc plac\'e dans le cas d'un corps $K$ muni d'une valuation discr\`ete avec un corps r\'esiduel contenant $\C$. Cette situation d'\'egale caract\'eristique $0$ simplifie les raisonnements et permet, en particulier, d'utiliser les r\'esultats du livre de S. Kumar \cite{Kr-02}. On a cependant d\^u faire quelques hypoth\`eses restrictives sur le groupe de Kac-Moody (en particulier la sym\'etrisabilit\'e).

\par Par ailleurs dans \cite{Ru-10} on a \'elabor\'e une d\'efinition abstraite de masure affine (inspir\'ee de la d\'efinition abstraite des immeubles affines de \cite{T-86a}) et on a montr\'e qu'elle est satisfaite par la plupart des masures d\'efinies pr\'ec\'edemment. De cette d\'efinition d\'ecoulent des propri\'et\'es int\'eressantes: les r\'esidus en chaque point sont des immeubles jumel\'es, \`a l'infini on trouve des immeubles jumel\'es et deux immeubles microaffines, il existe un pr\'eordre invariant sur la masure.

\par Le but du pr\'esent article est de construire la masure affine d'un groupe de Kac-Moody d\'eploy\'e sur un corps muni d'une valuation r\'eelle non triviale et de montrer qu'elle satisfait aux axiomes abstraits de masure affine ordonn\'ee de \cite{Ru-10}. Ceci est r\'ealis\'e sans aucune restriction sur le groupe de Kac-Moody ni sur le corps valu\'e. Pour cela on a essentiellement remplac\'e dans \cite{GR-08} les groupes de Kac-Moody \`a la Kumar par ceux d'Olivier Mathieu \cite{M-89} et la repr\'esentation adjointe dans l'alg\`ebre de Lie par celle dans l'alg\`ebre enveloppante enti\`ere, puisqu'on va consid\'erer aussi de la caract\'eristique (r\'esiduelle) positive.

\par Ces complications techniques sont r\'ecompens\'ees par la possibilit\'e (nouvelle) de travailler sur un vrai corps local $K$, \ie complet pour une valuation discr\`ete avec corps r\'esiduel fini.
Alors la masure associ\'ee $\SHI$ est semi-discr\`ete et d'\'epaisseur finie (\ref{5.16}), ce qui constitue la bonne g\'en\'eralisation des immeubles affines localement finis.
On a ainsi le bon support g\'eom\'etrique pour d\'efinir (pour tout groupe de Kac-Moody) une alg\`ebre de Hecke sph\'erique (associ\'ee \`a des sommets sp\'eciaux de $\SHI$) ou d'Iwahori-Hecke  (associ\'ee \`a des chambres de $\SHI$); on peut alors, par exemple, prouver l'isomorphisme de Satake ou retrouver les alg\`ebres de Hecke affines doubles (DAHA) dans le cas affine.
 Ces applications sont d\'evelopp\'ees dans \cite{GR-12} et \cite{GR-13}. Jusqu'\`a pr\'esent  seul le cas affine \'etait connu, partiellement ou avec des moyens d\'etourn\'es de g\'eom\'etrie alg\'ebrique.

\par Il y a en fait beaucoup de choix possibles pour les groupes de Kac-Moody, \cf \cite{T-89b}. On consid\`ere ici les groupes d\'eploy\'es "minimaux" tels que d\'efinis par Jacques Tits \cite{T-87b}; leurs propri\'et\'es essentielles sont expliqu\'ees dans \cite{T-92a} et aussi \cite{Ry-02a}. On a r\'esum\'e celles-ci et prouv\'e quelques compl\'ements (essentiellement sur les morphismes) dans la premi\`ere partie de cet article.

\par La seconde partie est consacr\'ee \`a l'alg\`ebre enveloppante enti\`ere introduite par J. Tits pour construire ses groupes et \`a la repr\'esentation adjointe, d\'ej\`a largement utilis\'ee par B. R\'emy. On y montre un th\'eor\`eme de Poincar\'e-Birkhoff-Witt (avec des puissances divis\'ees tordues) et on y construit des exponentielles tordues. Ces r\'esultats tirent leur origine dans un travail de g\'en\'eralisation du th\'eor\`eme de simplicit\'e de R. Moody \cite{My-82}, qui est expliqu\'e dans l'appendice.

\par Les relations de commutation dans un groupe de Kac-Moody minimal $G$ sont compliqu\'ees (voire inexistantes). Pour mener \`a bien des calculs on va raisonner dans un groupe maximal qui ici sera celui ($G^{pma}$ ou $G^{nma}$) d\'efini par O. Mathieu. On a cependant besoin d'une connaissance plus concr\`ete de celui-ci, c'est le but de la troisi\`eme partie. Il est en partie accompli gr\^ace aux exponentielles tordues de la partie pr\'ec\'edente.

\par Dans ces trois premi\`eres parties on \'etudie les groupes de Kac-Moody de mani\`ere g\'en\'erale sur un anneau quelconque, ou sur un corps si on veut plus de r\'esultats de structure. On les consid\`ere ensuite sur un corps valu\'e (mais aussi sur son anneau des entiers). Dans la quatri\`eme partie on construit l'appartement t\'emoin et les sous-groupes parahoriques (ou assimil\'es) associ\'es aux sous-ensembles ou filtres de cet appartement t\'emoin. C'est la partie la plus technique. Comme dans \cite{GR-08}, ces sous-groupes parahoriques du groupe de Kac-Moody $G$ sont construits en plusieurs \'etapes en utilisant les groupes maximaux $G^{pma}$ et $G^{nma}$ contenant $G$. On a cependant d\^u apporter des changements substantiels aux raisonnements de \cite{GR-08}.

\par On r\'ecolte enfin dans la cinqui\`eme partie les fruits du travail pr\'ec\'edent. Par un proc\'ed\'e classique utilisant l'appartement t\'emoin et les sous-groupes parahoriques, on y construit la masure affine (ordonn\'ee) d'un groupe d\'eploy\'e sur un corps r\'eellement valu\'e. Et on montre les m\^emes r\'esultats qu'en \cite{GR-08} ou \cite{Ru-10}. Les raisonnements sont sans changement substantiel, mais, gr\^ace aux parties pr\'ec\'edentes, il n'y a plus d'hypoth\`ese technique superflue.

\par L'appendice rassemble des r\'esultats sur le groupe de Kac-Moody maximal \`a la Mathieu $G^{pma}$ (sur un corps $k$ quelconque) qui sont cons\'equences des r\'esultats des parties \ref{s2} et \ref{s3}. On y compare $G^{pma}$ avec d'autres groupes de Kac-Moody maximaux d\'efinis par M. Ronan et B. R\'emy \cite{RR-06} ou L. Carbone et H. Garland \cite{CG-03}.
On y g\'en\'eralise aussi le th\'eor\`eme de simplicit\'e d'un sous-quotient de $G^{pma}$ d\^u \`a R. Moody \cite{My-82} en caract\'eristique  $0$; on l'obtient ici en caract\'eristique $p$ assez grande, pour un corps non alg\'ebrique sur $\F_p$.
Ce th\'eor\`eme de simplicit\'e vient d'\^etre compl\'et\'e par T. Marquis \cite{Mar13} pour tous les corps finis, en s'appuyant \'egalement sur les parties \ref{s2} et \ref{s3} de cet article.

\par Je remercie Bertrand R\'emy pour m'avoir sugg\'er\'e d'\'etudier l'article \cite{My-82} de R. Moody et Olivier Mathieu pour m'avoir signal\'e son article \cite{M-96} et pour quelques \'eclaircissements sur la construction de son groupe de Kac-Moody ou la simplicit\'e des alg\`ebres de Kac-Moody.

\section{Le groupe de Kac-Moody minimal (\`a la Tits)}\label{s1}

\par On introduit ici le groupe de Kac-Moody objet principal d'\'etude de cet article et on y \'etudie en particulier sa fonctorialit\'e.

\subsection{Syst\`emes g\'en\'erateurs de racines}\label{1.1}

\par
\par 1) Une {\it matrice de Kac-Moody} (ou matrice de Cartan g\'en\'eralis\'ee)
 est une matrice carr\'ee $A = (a_{i,j})_{i,j\in I}$ , \`a coefficients entiers,
 index\'ee par un ensemble fini $I$ et qui v\'erifie:
\par (i) $a_{i,i} = 2 \quad \forall i \in I$ ,
\par (ii) $a_{i,j} \leq 0 \quad \forall i \not = j $ ,
\par (iii) $a_{i,j} = 0 \iff a_{j,i} = 0 $.

\medskip
\par 2) Un {\it syst\`eme g\'en\'erateur de racines } (en abr\'eg\'e SGR) \cite{By-96} est un quadruplet
${\mathcal S}=(A,Y,({\overline{\alpha_i}})_{i\in I},({\alpha}{_i^\vee})_{i\in I})$ form\'e d'une matrice
 de Kac-Moody $A$ index\'ee par $I$, d'un $\mathbb Z-$module libre $Y$ de rang fini $n$,  d'une famille $({\overline{\alpha_i}})_{i\in I}$ dans son dual $X=Y^*$ et
 d'une famille $({\alpha}{_i^\vee})_{i\in I}$ dans $Y$. Ces donn\'ees sont
soumises \`a la condition de compatibilit\'e suivante: $\;$
 $a_{i,j} = {\overline{\alpha_j}}({\alpha}{_i^\vee}) $.

 \par On dit que $\vert I\vert$ est le {\it rang} du SGR ${\mathcal S}$ et $n$ sa {\it dimension}.

 \par On dit que le SGR ${\mathcal S}$ est {\it libre} (ou {\it adjoint}) (resp. {\it colibre}  (ou {\it coadjoint})) si $({\overline{\alpha_i}})_{i\in I}$ (resp. $({\alpha}{_i^\vee})_{i\in I}$) est libre dans (ou engendre) $X$ (resp. $Y$).

 \par Par exemple le SGR {\it simplement connexe} ${\mathcal S}_A=(A,Y_A,({\overline{\alpha_i}})_{i\in I},({\alpha}{_i^\vee})_{i\in I})$ avec $Y_A$ de base les ${\alpha}{_i^\vee}$ est le seul SGR colibre et coadjoint associ\'e \`a $A$.

\medskip
\par 3) On introduit le $\Z-$module libre $Q=\bigoplus_{i\in I}\,\Z\alpha_i$. Il y a donc un homomorphisme de groupes $bar: Q\rightarrow X$ , $\alpha\mapsto {\overline\alpha}$ tel que $bar(\alpha_i)={\overline{\alpha_i}}$.
 On note $Q^+=\sum_{i\in I}\,\N\alpha_i\subset Q$ et $Q^-=-Q^+$. Quand ${\mathcal S}$ est libre on identifie $Q$ \`a un sous-module de $X$  et on ne fait pas de diff\'erence entre $\alpha$ et ${\overline\alpha}$.
 Le SGR {\it adjoint minimal} $\shs_{Am}=(A,Q^*,({{\alpha_i}})_{i\in I},({\alpha}{_i^\vee})_{i\in I})$ est libre et adjoint.

\par On note $Q^\vee=\sum_{i\in I}\Z\alpha_i{^\vee}$. On dit que le SGR ${\mathcal S}$ est {\it sans cotorsion} si $Q^\vee$ est facteur direct dans $Y$ \ie $Y/Q^\vee$ est sans torsion. C'est le cas de ${\mathcal S}_A$ pour lequel $Q^\vee=Y$.

\medskip
\par 4) On note $V = Y\otimes_\Z\mathbb R$ et $V_\C = Y\otimes_\Z\mathbb C$; pour $\alpha\in Q$ et $h\in V_\C$, on \'ecrit $\alpha(h)={\overline{\alpha}}(h)$. L'{\it alg\`ebre de Kac-Moody complexe} ${\mathfrak g}_{\mathcal S}$ associ\'ee \`a $\mathcal S$ est engendr\'ee par ${\mathfrak h}_{\mathcal S}=V_\C$ et des \'el\'ements $(e_i,f_i)_{i\in I}$ avec les relations suivantes (pour $h,h'\in{\mathfrak h}_{\mathcal S}$ et $i\neq j\in I$):
\medskip
\par (KMT1) \quad$[h,h']=0$\quad;\;$[h,e_i]=\alpha_i(h)e_i$\quad;\;$[h,f_i]=-\alpha_i(h)f_i$\quad;\;$[e_i,f_i]=-\alpha_i^{\vee}$.
\medskip
\par (KMT2) \quad$[e_i,f_j]=0$\quad;\quad$(ad\,e_i)^{1-a_{i,j}}(e_j)=(ad\,f_i)^{1-a_{i,j}}(f_j)=0$.
\medskip
\par On a alors une graduation de l'alg\`ebre de Lie $\g g_\shs$ par $Q$ : $\g g_\shs=\g h_\shs\oplus(\oplus_{\alpha\in\Delta}\;\g g_\alpha)$ o\`u $\Delta\subset Q\setminus\{0\}$ est le {\it syst\`eme de racines} de $\g g_\shs$.
On a $\g h_\shs=(\g g_\shs)_0$, $\g g_{\alpha_i}=\C.e_i$ et $\g g_{-\alpha_i}=\C.f_i$, de plus $\g g_\alpha\subset \g g_{\overline\alpha}=\{x\in\g g_\shs\;\mid\;[h,x]={\overline\alpha}(x)\;\forall h\in\g h\}$ est non nul pour $\alpha\in\Delta\cup\{0\}$.

\par On sait que $\Delta$ et les $\g g_\alpha$ ne d\'ependent que de $A$ et non de $\shs$. On note
$\g g_A=\g g_{\shs_A}$.

\par Les racines $\alpha_i$ sont dites {\it simples}. Les racines de $\Delta^+=\Delta\cap Q^+$ (resp. $\Delta^-=-\Delta^+$) sont dites {\it positives} (resp. {\it n\'egatives}). On a $\Delta=\Delta^-\bigsqcup\Delta^+$.

\par La sous-alg\`ebre {\it nilpotente} (resp. {\it de Borel}) {\it positive} est $\g n^+_\shs=\oplus_{\alpha\in\Delta^+}\;\g g_\alpha$ (ind\'ependante de $\shs$ et donc not\'ee aussi $\g n^+_A$) (resp. $\g b^+_\shs=\g h_\shs\oplus\g n^+_\shs$). On a de m\^eme des sous-alg\`ebres n\'egatives $\g n^-_\shs=\g n^-_A$ et $\g b^-_\shs=\g h_\shs\oplus\g n^-_\shs$ avec la d\'ecomposition triangulaire $\g g_\shs=\g n^-_\shs\oplus\g h_\shs\oplus\g n^+_\shs$.

N.B. Souvent on r\'eserve le terme "alg\`ebre de Kac-Moody" au cas o\`u $\shs$ est libre, colibre et de dimension $2\vert I\vert-rang(A)$, voir \eg \cite{K-90} ; ce n'est pas appropri\'e \`a notre situation.

\medskip
\par 5) Le {\it groupe de Weyl (vectoriel)} $W^v$ associ\'e \`a $A$ est un groupe de Coxeter de syst\`eme de g\'en\'erateurs l'ensemble $S=\{\,s_i\,\mid\,i\in I\}$ des automorphismes de $Q$ d\'efinis par $s_i(\alpha) = \alpha - {\alpha}(\alpha{_i^\vee}) \alpha_i$. Il stabilise $\Delta$ et agit aussi sur $X$ et $Y$ par des formules semblables.

\par On note $\Phi=\Delta_{re}$ l'ensemble des {\it racines r\'eelles} c'est-\`a-dire
 des \'el\'ements de $\Delta$ ou $Q$ de la forme $\alpha = w(\alpha_i)$ avec
$w \in W^v$ et $ i \in I$. Si $\alpha \in \Phi$, alors $s_{\alpha} =
w.s_i.w^{-1} $ est bien d\'etermin\'e par $\alpha$, ind\'ependamment du choix
de $w$ et de $i$ tels que $\alpha = w(\alpha_i)$. Pour $\beta\in Q$ on a
$s_{\alpha}(\beta) = \beta -{ \beta}({\alpha}{^\vee })\alpha $ pour un $\alpha{^\vee } \in Y $ avec
 ${\alpha}(\alpha{^\vee }) = 2 $.  Si $\Phi^+ = \Phi\cap\Delta^+$ et $\Phi^- = - \Phi^+$, on a
$\Phi = \Phi^+ \bigsqcup \Phi^-$.

 \par Les racines de $\Delta_{im}=\Delta\setminus\Phi$ sont dites {\it imaginaires}.

 Si on a deux parties $\Psi\subset\Psi'$ de $\Delta\cup\{0\}$, on dit que $\Psi$ est {\it close} (resp. un {\it id\'eal} de $\Psi'$) si : \quad $\alpha , \beta \in \Psi$, (resp.  $\alpha\in\Psi , \beta \in \Psi'$), $p,q\geq{}1, p\alpha +q \beta \in \Delta\cup\{0\} \Rightarrow p\alpha + q\beta \in \Psi$.
 La partie $\Psi$ est dite {\it pr\'enilpotente} s'il
 existe $w , w' \in W^v$ tels que $w\Psi \subset \Delta^+$ et $w'\Psi \subset
 \Delta^-$, alors $\Psi$ est finie et contenue dans la partie
$w^{-1}(\Phi^+) \cap (w')^{-1}(\Phi^-)$ de $\Phi$ qui est {\it nilpotente} (\ie
pr\'enilpotente et close).

\medskip
\par 6) Un {\it morphisme de SGR}, $\varphi\;:\;{\mathcal S}=(A,Y,({\overline{\alpha_i}})_{i\in I},({\alpha}{_i^\vee})_{i\in I})\rightarrow{\mathcal S}'=(A',Y',({\overline{\alpha'_i}})_{i\in I'},({\alpha'}{_i^\vee})_{i\in I'})$ est une application lin\'eaire $\varphi \;:\;Y\rightarrow Y'$  et une injection $I\rightarrow I'$ , $i\mapsto i$ telles que $A=A'_{\vert I\times I}$, $\varphi(\alpha_i^\vee)={\alpha'}{_i^\vee}$ et ${\overline{\alpha'_i}}\circ\varphi={\overline{\alpha_i}}$ $\forall i\in I$. On note $\varphi^*\;:\;X'={Y'}^*\rightarrow X=Y^*$ l'application duale


\par Cette d\'efinition de morphisme est duale de celle de \cite[4.1.1]{By-96}  qui est  valable dans un cadre plus g\'en\'eral. L'exemple le plus simple de morphisme est $\qf^{ad}=bar^*:\shs\to\shs_{Am}$.

\par Pour un morphisme de SGR $\varphi\;:\;{\mathcal S}\rightarrow{\mathcal S}'$, on d\'efinit de mani\`ere \'evidente un morphisme $\g g_\varphi\;:\;\g g_\shs\rightarrow \g g_{\shs'}$ entre les alg\`ebres de Lie correspondantes.

\medskip
\par 7) Soit  $\varphi\;:\;{\mathcal S}\rightarrow{\mathcal S}'$ un morphisme de SGR. On dit que:

\par $\shs$ est une {\it extension centrale torique} de $\shs'$ si $\varphi$ est surjective et $I=I'$; elle est dite de plus {\it scind\'ee} si  Ker$(\varphi)$ a un suppl\'ementaire contenant $Q^\vee$,

\par $\shs$ est une {\it extension centrale finie} de $\shs'$ si $\varphi$ est injective, $I=I'$ et les dimensions sont \'egales,

\par plus g\'en\'eralement $\shs$ est {\it extension centrale} de $\shs'$ si $\varphi^*$ est injective et $I=I'$,

\par $\shs$ est un {\it sous-SGR} de $\shs'$ si $\varphi$ est injective et $Y'/\varphi(Y)$ est sans torsion,

\par $\shs'$ est une {\it extension semi-directe} (torique) de $\shs$ si $\shs$ est un  sous-SGR de $\shs'$ et si $I=I'$. Elle est dite {\it directe} si, de plus, il existe un suppl\'ementaire de $\varphi(Y)$ dans $Y'$ contenu dans Ker$({\overline \alpha_i})$, $\forall i \in I$.

\par $\varphi$ est une {\it extension commutative} si $I=I'$ (alors $A=A'$ et $Q=Q'$).

\subsection{Tores associ\'es}\label{1.2}

\par \`A un SGR $\shs$, on associe un $\Z-$sch\'ema en groupes $\g T_\shs=\g T_Y$ (vu comme foncteur en groupes sur la cat\'egorie des anneaux) d\'efini par $\g T_Y(k)=Y\otimes_\Z\,k^*$ (pour tout anneau $k$) ou par $\g T_Y={\mathrm Spec}(\Z[X])$. C'est un tore (isomorphe \`a $(\g{Mult})^{dim(\shs)}$) de groupe de cocaract\`eres (resp. caract\`eres) $Y$ (resp. $X$), voir \cite{DG-70}.

\par Un morphisme $\varphi\;:\;{\mathcal S}\rightarrow{\mathcal S}'$ de SGR donne un homomorphisme $\g T_\varphi\,:\,\g T_\shs\rightarrow\g T_{\shs'}$ de tores. Le groupe de Weyl $W^v$ agit sur $\g T_\shs$ via ses actions sur $Y$ et $X$.

\par Si $\varphi$ est extension centrale, $\g T_\varphi$ est surjectif (au sens sch\'ematique) de noyau le groupe multiplicatif $\g Z$ de groupe des caract\`eres $X/\varphi^*(X')$.
Si $\varphi$ est extension centrale torique, $\g T_\varphi$ est surjectif (au sens fonctoriel) de noyau le tore $\g T_{{\mathrm Ker}(\varphi)}$.
Si $\varphi$ est extension centrale finie, alors le th\'eor\`eme des diviseurs \'el\'ementaires appliqu\'e \`a $\varphi^*(X')\subset X$ donne une description explicite de $\g Z(k)$ comme produit de groupes de racines de l'unit\'e et l'homomorphisme $\g T_\varphi(k)$ est en g\'en\'eral non surjectif.

\par Si $\shs$ est un {\it sous-SGR} de $\shs'$, $\g T_\varphi$ identifie $\g T_\shs$ \`a un sous-tore facteur direct de $\g T_{\shs'}$.

\begin{prop}
\label{1.3} Soit $\shs=(A,Y,({\overline{\alpha_i}})_{i\in I},({\alpha}{_i^\vee})_{i\in I})$ un SGR.

\par a) Il existe une extension centrale finie $\shs^s$ de $\shs$ qui est sans cotorsion.
 Le SGR $\shs$ est extension centrale de $\shs^{ad}=(A,\overline Q^*,({\overline{\alpha_i}})_{i\in I},({\alpha}{_i^\vee})_{i\in I})$ qui est adjoint.
Si $\shs$ est libre alors $\shs^s$ et $\shs^{ad}$ aussi, de plus $\shs^{ad}=\shs_{Am}$.

\par b) Le SGR $\shs^1=(A,Q^\vee,({\overline{\alpha_i}}_{\vert Q^\vee})_{i\in I},({\alpha}{_i^\vee})_{i\in I})$  n'est en g\'en\'eral pas libre. Le SGR $\shs^s$ (resp. $\shs_A$) en est extension semi-directe (resp. centrale torique).

\par c) Il existe une extension centrale torique $\shs^{sc}$ de $\shs$ qui est  sans cotorsion et colibre. \label{N15} Si $\shs$ est colibre et sans cotorsion, cette extension est scind\'ee. Si $\shs$ est libre ou si $Q$ est facteur direct dans $X$, alors c'est \'egalement vrai pour $\shs^{sc}$.

\par d) Il existe une extension semi-directe $\shs^{\ell}$ de $\shs$ qui est libre (et avec $Q^\ell$ facteur direct de $X^\ell$). Si $\shs$ est libre et $Q$ facteur direct de $X$, cette extension est directe. \label{N15} Si $\shs$ est colibre ou si $Q^\vee$ est facteur direct dans $Y$, alors c'est \'egalement vrai pour $\shs^{\ell}$.

\par e) Si $\shs$ est libre, colibre et sans cotorsion, alors $\shs$ est extension semi-directe d'un SGR $\shs^{mat}$ libre, colibre, sans cotorsion et de dimension $2r-s$ o\`u $r$ est le rang de $\shs$ et $s$ ($\leq{}r$) le rang de la matrice $A$.

\end{prop}

\begin{remas*}

\par 1) On a $\shs^{sc\ell}=\shs^{\ell sc}$ et les constructions de c) et d) sont fonctorielles en $\shs$.

\par 2) Dans le cas e) le SGR $\shs^{mat}$ v\'erifie donc les hypoth\`eses de \cite[p 16]{M-88a}. \label{N1} Mais $\shs$ n'est pas toujours extension directe de $\shs^{mat}$.

\par 3) Quand la matrice $A$ est inversible (par exemple dans le cas classique d'une matrice de Cartan) tout SGR est libre, colibre et extension semi-directe \label{N15}  (ou aussi centrale torique) d'un SGR de dimension \'egale au rang (\ie semi-simple dans le cas classique).

\par 4) L'extension commutative $\qf^{ad}$ se factorise par $\shs^\ell$, $\qf^{ad}:\shs\to\shs^\ell\to\shs^{lad}=\shs_{Am}$ avec le premier morphisme extension semi-directe et le second extension centrale.

\end{remas*}

\begin{proof}

\par On consid\`ere un suppl\'ementaire $Y_0$ de $(Q^\vee\otimes\Q)\cap Y$ dans $Y$ et on pose $Y^s=Q^\vee\oplus Y_0$. Les assertions a) et b) sont alors \'evidentes.

\par c) On d\'efinit $Y^{sc}=Y\oplus(\oplus_{i\in I}\,\Z u_i^\vee)$, son dual est donc $X^{sc}=X\oplus(\oplus_{i\in I}\,\Z u_i)$ o\`u les $u_i$ forment une base duale des $u_i^\vee$. On note $\overline \alpha_i^{sc}=\overline \alpha_i\in X$, $\alpha_i^{sc\vee}=\alpha_i^\vee+u_i^\vee$ et $\varphi$ est la projection de $Y^{sc}$ sur $Y$. \cf \cite[7.1.2]{Ry-02a}.

\par d) On d\'efinit $Y^{\ell}=Y\oplus(\oplus_{i\in I}\,\Z v_i^\vee)$, son dual est donc $X^{\ell}=X\oplus(\oplus_{i\in I}\,\Z v_i)$ o\`u les $v_i$ forment une base duale des $v_i^\vee$. On note $\overline \alpha_i^{\ell}=\overline \alpha_i+v_i$, $\alpha_i^{\ell\vee}=\alpha_i^\vee$ et $\varphi$ est l'injection canonique de $Y$ dans $Y^{\ell}$ .

\par e) Soit $\psi\,:\,Y\rightarrow\Z^r$, $y\mapsto(\alpha_i(y))_{i\in I}$. Comme $\shs$ est libre, $\psi(Y)$ est de rang $r$, tandis que $\psi(Q^\vee)$ est de rang $s$. Le module $\Q\psi(Q^\vee)\cap\psi(Y)$ a un suppl\'ementaire dans $\psi(Y)$ de rang $r-s$; on suppose que la base de ce suppl\'ementaire est $\psi(v_1),\cdots,\psi(v_{r-s})$ pour des $v_i\in Y$. Comme $\shs$ est colibre $Y^{mat}=Q^\vee\oplus(\oplus_{j=1}^{r-s}\,\Z v_j)$ est de rang $2r-s$. Comme le quotient de $Y$ par $Q^\vee$ est sans torsion , il en est de m\^eme pour $Y^{mat}$. On note $\alpha_i^{mat}=\alpha_{i}\vert_{ Y^{mat}}$ et $\alpha_i^{mat\vee}=\alpha_i^\vee$. L'injection de $Y^{mat}$ dans $Y$ fait bien de $\shs$ une extension semi-directe de $\shs^{mat}$ et $\shs^{mat}$ est colibre, sans cotorsion, de dimension $2r-s$. Comme $\psi(Y^{mat})=\psi(Q^\vee)\oplus(\oplus_{j=1}^{r-s}\,\Z \psi(v_j))$ est de rang $r$, les $\alpha_i^{mat}$ sont bien libres dans $X^{mat}=(Y^{mat})^*$.
\end{proof}

\subsection{Relations dans ${\mathrm Aut}(\g g_A)$ }\label{1.4} \cf \cite[3.3 et 3.6]{T-87b}

\par Pour $i\in I$, $ad(e_i)$ et $ad(f_i)$ sont  localement nilpotents dans $\g g_A$ ou $\g g_\shs$, $\exp(ad(e_i))$, $\exp(ad(f_i))$ sont des automorphismes de $\g g_A$ ou $\g g_\shs$ et l'on a :
 $$\exp(ad(e_i)).\exp(ad(f_i)).\exp(ad(e_i))=\exp(ad(f_i)).\exp(ad(e_i)).\exp(ad(f_i))$$
 on note cet \'el\'ement $s_i^*$ ou $s^*_{\alpha_i}$. L'application $s_i^*\mapsto s_i$ s'\'etend en un homomorphisme surjectif $w^*\mapsto w$ du sous-groupe $W^*$ de ${\mathrm Aut}(\g g_A)$ engendr\'e par les $s_i^*$ sur le groupe de Weyl $W^v$. Le groupe $W^*$ stabilise $\g h$ et permute les $\g g_\alpha$ pour $\alpha\in\Delta$; les actions induites sur $\g h$ ou $\Delta$ se d\'eduisent des actions de $W^v$ via l'homomorphisme pr\'ec\'edent. En particulier $W^*$ est le m\^eme qu'il soit d\'efini sur $\g g_A$ ou $\g g_\shs$.

\par Pour $\alpha\in\Phi$, l'espace $\g g_\alpha$ est de dimension $1$ et on peut choisir les \'el\'ements de base $(e_\alpha)_{\alpha\in\Phi}$ de fa\c{c}on que $e_{\alpha_i}=e_i$, $[e_\alpha,f_\alpha]=-\alpha^\vee$ et $w^*e_\alpha=\pm{}e_{w\alpha}$ pour $\alpha\in\Phi$, $i\in I$ et $w^*\in W^*$.

\par Si $\alpha,\beta\in\Phi$ forment une paire pr\'enilpotente de racines, l'ensemble $[\alpha,\beta]=\{\,p\alpha+q\beta\in\Phi\,\vert\,p,q\in\N\}$ est fini. Si $\alpha$ et $\beta$ sont non colin\'eaires, on note $]\alpha,\beta[=[\alpha,\beta]\setminus\{\alpha,\beta\}$, et on l'ordonne, par exemple de fa\c{c}on que $p/q$ soit croissant; on a alors la formule suivante pour les commutateurs (pour $r,r'\in\C$):

$$(\exp(ad(re_\alpha)),\exp(ad(r'e_\beta)))=\prod\;\exp(ad(C_{p,q}^{\alpha,\beta}r^pr'^qe_\gamma))$$

\par o\`u $\gamma=p\alpha+q\beta$ parcourt $]\alpha,\beta[$ et les $C_{p,q}^{\alpha,\beta}$ sont des entiers bien d\'efinis.

\subsection{Foncteur de Steinberg}\label{1.5}  \cf \cite[3.6]{T-87b}

\par Ce foncteur $\g{St}_A$ de la cat\'egorie des anneaux dans celle des groupes est engendr\'e par  $\vert\Phi\vert$ exemplaires du groupe additif $\g{Add}$. Plus pr\'ecis\'ement, pour $\alpha\in\Phi$, $k$ un anneau et $r\in k$, on introduit le symbole $\g x_\alpha(r)$ et le groupe $\g{St}_A(k)$ est engendr\'e par les \'el\'ements  $\g x_\alpha(r)$ pour $\alpha\in\Phi$, $r\in k$ soumis aux relations $\g x_\alpha(r+r')=\g x_\alpha(r).\g x_\alpha(r')$ et aux relations de commutation:

\medskip
\par (KMT3)\quad$(\g x_\alpha(r),\g x_\beta(r'))=\prod\;\g x_\gamma(C_{p,q}^{\alpha,\beta}r^pr'^q)$
\medskip
\par pour $r,r'\in k$, $\{\alpha,\beta\}$ pr\'enilpotente et $\gamma$, $C_{p,q}^{\alpha,\beta}$ comme ci-dessus en \ref{1.4}.

\medskip
\par Si $\Psi$ est une partie nilpotente de $\Phi$ et si on remplace ci-dessus $\Phi$ par $\Psi$, on d\'efinit un foncteur en groupes $\g U_\Psi$. C'est un groupe alg\'ebrique unipotent isomorphe comme sch\'ema \`a $(\g{Add})^{\vert\Psi\vert}$ et qui ne d\'epend que de $\Psi$ et $A$ (donc $\Phi$) mais pas de $\shs$. Si $\alpha\in\Phi$ et $\Psi=\{\alpha\}$, $\g U_{\alpha}=\g U_{\{\alpha\}}$ est un sous-foncteur en groupes de $\g{St}_A$ isomorphe \`a $\g{Add}$ par $\g x_\alpha$. En fait $\g{St}_A$ est l'amalgame des groupes  $\g U_{\alpha}$ et  $\g U_{[\alpha,\beta]}$ pour  $\alpha\in\Phi$ et $\{\alpha,\beta\}$ pr\'enilpotente.

\par Pour $\alpha\in\Phi$, $k$ un anneau et $r\in k^*$, on d\'efinit dans $\g{St}_A(k)$, $\tilde s_\alpha(r)=\g x_\alpha(r).\g x_{-\alpha}(r^{-1}).\g x_\alpha(r)$, $\tilde s_\alpha=\tilde s_\alpha(1)$ et $\alpha^*(r)=\tilde s_\alpha^{-1}.\tilde s_\alpha(r^{-1})$. On a $\tilde s_\alpha(-r)=\tilde s_\alpha(r)^{-1}=\tilde s_\alpha.\alpha^*(-r^{-1})$. En particulier $\tilde s_\alpha^{-1}=\tilde s_\alpha(-1)$.

\begin{rema*}
 Si jamais le quotient ${\overline{\g g}}_A$ de $\g g_A$ par son id\'eal gradu\'e maximal d'intersection triviale avec $\g h_A$ est diff\'erent de $\g g_A$, la diff\'erence n'affecte que les espaces radiciels imaginaires. Le remplacement de $\g g_A$ par ${\overline{\g g}}_A$ ne change donc pas $\g{St}_A$ (ni $\g G_\shs$ d\'efini ci-dessous); ce n'est par contre pas vrai pour le groupe $\g G_\shs^{pma}$ du {\S{}} 3.
\end{rema*}

\subsection{Le groupe de Kac-Moody minimal $\g G_\shs$ (\`a la Tits)}\label{1.6} \cf \cite[3.6]{T-87b}, \cite[8.3.3]{Ry-02a}

\par Pour un anneau $k$, on d\'efinit le groupe $\g G_\shs(k)$ comme le  produit libre $\g{St}_A(k)*\g T_\shs(k)$ quotient\'e par les quatre relations suivantes, pour $\alpha$ racine simple, $\beta\in\Phi$, $r\in k$ et $t\in\g T_\shs(k)$:
\medskip
\par(KMT4) \quad  $t.\g x_\alpha(r).t^{-1}=\g x_\alpha(\alpha(t)r)$,
\medskip
\par (KMT5) \quad $\tilde s_\alpha.t.\tilde s_\alpha^{-1}=s_\alpha(t)$,
\medskip
\par (KMT6) \quad $\tilde s_\alpha(r^{-1})=\tilde s_\alpha.\alpha^\vee(r)$,
\medskip
\par (KMT7) \quad $\tilde s_\alpha.\g x_\beta(r).\tilde s_\alpha^{-1}=\g x_\gamma(\epsilon r)$,
\par \qquad\qquad\qquad si $\gamma=s_\alpha(\beta)$ et $s^*_\alpha(e_\beta)=\epsilon e_\gamma$ (avec $\epsilon=\pm{}1$).
\medskip
\begin{rema*}

Les relations (KMT5) et (KMT7) permettent aussit\^ot de g\'en\'eraliser (KMT4) au cas o\`u $\alpha\in\Phi$ n'est pas simple.

\end{rema*}

\begin{enonce*}[definition]{Propri\'et\'es}

\par 1) D'apr\`es \cite[3.10.b]{T-87b} il existe des foncteurs en groupes $\g G_\#$ et $\g U^\pm_\#$  satisfaisant aux axiomes (KMG 1 \`a 9) de \lc. D'apr\`es le th\'eor\`eme 1 (i) p. 553 de \lc, il en r\'esulte que, pour tout anneau $k$, l'homomorphisme canonique de $\g T_\shs(k)$ dans $\g G_\shs(k)$ est injectif.

\medskip
\par 2) Pour tout $\alpha\in\Phi$, l'homomorphisme $\g x_{\alpha k}\,:\,k\rightarrow\g U_\alpha(k)\rightarrow \g G_\shs(k)$ est injectif, voir la d\'emonstration de \cite[9.6.1]{Ry-02a} si $Y/\Z\alpha_i^\vee$ est sans torsion ($\forall i\in I$) ou si $k$ n'a pas d'\'el\'ements nilpotents; le cas g\'en\'eral se traite comme dans la partie b) de la d\'emonstration du lemme \ref{4.11} ci-dessous. Plus g\'en\'eralement ces m\^emes raisonnements montrent que, pour toute partie nilpotente $\Psi$ de $\Phi$, l'homomorphisme canonique de $\g U_\Psi(k)$ dans $\g G_\shs(k)$ est injectif et aussi que, pour $\alpha\not=\beta$ et $r$ ou $r'$ non nul $\g x_\alpha(r)\not=\g x_\beta(r')$.

\medskip
\par 3) On identifie donc $\g T_\shs$, $\g U_\alpha$ et $\g U_\Psi$ \`a des sous-foncteurs en groupes de $\g G_\shs$; si $\Psi'$ est un id\'eal de $\Psi$, alors $\g U_{\Psi'}$ est distingu\'e dans $\g U_\Psi$. De m\^eme on note $\g U_\shs^{\pm}$ (resp. $\g B_\shs^{\pm}$) le sous-foncteur en groupes de $\g G_\shs$ tel que, pour un anneau $k$, $\g U_\shs^{\pm}(k)$ (resp. $\g B_\shs^{\pm}(k)$) est le sous-groupe engendr\'e par les $\g U_\alpha(k)$ pour $\alpha\in\Phi^{\pm}$ (resp. et par $\g T_\shs(k)$).

\medskip
\par 4) On note $\g N_\shs$ le sous-foncteur de $\g G_\shs$ tel que $\g N_\shs(k)$ soit le sous-groupe engendr\'e par $\g T_\shs(k)$ et les $\tilde s_\alpha$ pour $\alpha$ racine simple. On a un homomorphisme $\nu^v$ de $\g N_\shs$ sur $W^v$ trivial sur  $\g T_\shs$ tel que $\nu^v(\tilde s_\alpha)=s_\alpha$ et $\g N_\shs(k)$ agit sur $\g T_\shs$ et $\Phi$ via $\nu^v$, \cf (KMT4), (KMT5), (KMT7) et 2) ci-dessus. D'apr\`es \cite[3.7.d]{T-87b} les $\tilde s_{\alpha_i}$ satisfont aux relations de tresse, d'apr\`es (KMT6) $(\tilde s_{\alpha_i})^2\in\g T_\shs(k)$, le noyau de $\nu^v$ est donc \'egal \`a $\g T_\shs(k)$.

\par Si $k$ est un corps avec au moins quatre \'el\'ements, $\g N_\shs(k)$ est le normalisateur de $\g T_\shs(k)$ dans $\g G_\shs(k)$  \cite[8.4.1]{Ry-02a}. Si $k$ est un corps infini, on sait de plus que tous les sous-tores $k-$d\'eploy\'es maximaux de $\g G_\shs$ sont conjugu\'es \`a $\g T_\shs$ par $\g G_\shs(k)$ \cite[10.4.1]{Ry-02a}.

\medskip
\par 5) Les th\'eor\`emes 1  et 1' p. 553 de  \cite{T-87b} montrent que, sur les corps, $\g G_\shs$ v\'erifie les axiomes (KMG 1 \`a 9) de \lc: il existe un homomorphisme $\pi$ de foncteurs en groupes $\g G_\shs\rightarrow\g G_\#$ qui est un isomorphisme sur les corps et v\'erifie $\pi(\g U_\shs^\pm)\subset\g U_\#^\pm$. En particulier (KMG4) dit que, pour une extension de corps $k\rightarrow k'$, $\g G_\shs(k)$ s'injecte dans $\g G_\shs(k')$.

\par Si $k$ est un corps, $(\g G_\shs(k),(\g U_\alpha(k))_{\alpha\in\Phi},\g T_\shs(k))$ est une donn\'ee radicielle de type $\Phi$, au sens de \cite[1.4]{Ru-06}, ou donn\'ee radicielle jumel\'ee enti\`ere \cite[8.4.1]{Ry-02a} (c'est plus pr\'ecis que les donn\'ees radicielles jumel\'ees de \lc, \ie les RGD-systems de \cite[8.6.1]{AB-08}).
 En particulier, pour $u\in \g U_\alpha(k)\setminus\{1\}$, il existe $u',u''\in \g U_{-\alpha}(k)$ tels que $m(u):=u'uu''$ conjugue $\g U_{\beta}(k)$ en $\g U_{s_\alpha(\beta)}(k)$, $\forall\beta\in\Phi$, donc $m(u)\in\g N_\shs(k)$; on a $m(\g x_\alpha(r))=\widetilde s_{-\alpha}(r^{-1})$.
On montre que $\g B_\shs^{\pm}(k)$ est produit semi-direct de $\g T_\shs(k)$ et $\g U_\shs^{\pm}(k)$ \cite[1.5.4]{Ry-02a}. Pour  une partie nilpotente de la forme $\Psi=\Phi^+\cap w\Phi^-$, on a  $\g U_\Psi(k)=\g U_\shs^+(k)\cap w\g U_\shs^-(k)w^{-1}$ \cf \lc 3.5.4. Le centre de $\g G_\shs(k)$ est $\{t\in \g T_\shs(k)\mid\qa_i(t)=1,\forall i\in I\}$  \cf \lc 8.4.3.

\par Toujours si $k$ est un corps, la donn\'ee radicielle ci-dessus permet de construire des immeubles combinatoires jumel\'es $\shi^v_+(k)$ et $\shi^v_-(k)$, agit\'es par $\g G_\shs(k)$: l'immeuble $\shi^v_\epsilon(k)$ de $\g G_\shs$ sur $k$ est associ\'e au syst\`eme de Tits $(\g G_\shs(k),\g B^\epsilon_\shs(k),\g N_\shs(k))$.

\end{enonce*}

\subsection{Un quotient du foncteur de Steinberg}\label{1.7}

\par Pour un anneau $k$, on d\'efinit $\overline{\g{St}}_A(k)$ comme le quotient de ${\g{St}}_A(k)$ par les relations suivantes, pour $\alpha,\beta$ racines simples, $\gamma\in\Phi$ et $r,r'\in k$:

\medskip
\par ($\overline{\mathrm KMT}$5) \; $\tilde s_\alpha.\beta^*(r).\tilde s_\alpha^{-1}=\beta^*(r).\alpha^*(r^{-\alpha(\beta^\vee)})$,

\medskip
\par ($\overline{\mathrm KMT}$7) \; $\tilde s_\alpha(r).\g x_\gamma(r').\tilde s_\alpha(r)^{-1}=\g x_\delta(\epsilon.r^{-\gamma(\alpha^\vee)}.r')$,
\par\qquad\qquad\qquad si $\delta=s_\alpha(\gamma)$ et $s^*_\alpha(e_\gamma)=\epsilon.e_\delta$,

\medskip
\par ($\overline{\mathrm KMT}$8) \; a) $\alpha^*\,:\, k^*\rightarrow\overline{\g{St}}_A(k)$ est un homomorphisme de groupes,

\par\qquad\qquad b) $\alpha^*(r).\beta^*(r')=\beta^*(r').\alpha^*(r)$, on note ${\g{T}}_A(k)$ le groupe commutatif engendr\'e par tous ces \'el\'ements (pour $\alpha,\beta$ racines simples et $r,r'\in k$),

\par\qquad\qquad c) La formule $\gamma(\alpha^*(r))=r^{\gamma(\alpha^\vee)}$ d\'efinit un homomorphisme de groupes ${\g{T}}_A(k)\rightarrow k^*$, not\'e $\gamma$.

\par On note de la m\^eme mani\`ere les \'el\'ements et leurs images dans $\overline{\g{St}}_A(k)$.

\begin{remas*} 1) Dans le cas classique, Steinberg consid\`ere lui aussi un quotient de ${\g{St}}_A$ un peu analogue: comparer ($\overline{\mathrm KMT}$7) et ($\overline{\mathrm KMT}$8) respectivement aux conditions (B') et (C) de \cite{Sg-62} ou \cite{Sg-68}.

\medskip
\par 2) On n'a pas cherch\'e ici un syst\`eme minimal de relations pour $\overline{\g{St}}_A(k)$. Le lecteur trouvera des r\'eductions faciles par lui-m\^eme et d'autres moins \'evidentes dans \cite[3.7 ou 3.8]{T-87b}, \cite{Sg-62} ou \cite{Sg-68}.

\medskip
\par 3) L'endomorphisme $s'_\beta$ du groupe ${\g{T}}_A(k)$ d\'efini par $s'_\beta(t)=t.\beta^*(\beta(t))^{-1}$ est une involution, car $s'_\beta(s'_\beta(t))=t.\beta^*(\beta(t))^{-1}.\beta^*(\beta(t.\beta^*(\beta(t))^{-1}))^{-1}=t.\beta^*(\beta(t))^{-1}.\beta^*(\beta(t))^{-1}.\beta^*(\beta(\beta^*(\beta(t))))$ $=t$ (car $\beta\circ\beta^*$ est l'\'el\'evation au carr\'e). En fait pour $t=\gamma^*(r)$ on a $s'_\beta(t)=s'_\beta(\gamma^*(r))=\gamma^*(r).\beta^*(\beta(\gamma^*(r)))^{-1}=\gamma^*(r).\beta^*(r^{-\beta(\gamma^\vee)})$ ( $=(\gamma^*-\beta(\gamma^\vee)\beta^*)(r)$ avec une notation additive pour Hom$(k^*,\overline{\g{St}}_A(k))$ ). En particulier ($\overline{\mathrm KMT}$5) et (KMT5) sont semblables.

\medskip
\par 4) La relation ($\overline{\mathrm KMT}$4)\;: \; $t.\g x_\alpha(r).t^{-1}=\g x_\alpha(\alpha(t)r)$, pour $\alpha\in\Phi$, $r\in k$ et $t\in{\g{T}}_A(k)$ est cons\'equence des relations ci-dessus:

\par En effet pour $\alpha$, $\beta$ racines simples, $r\in k^*$, $r'\in k$, on a : $\beta^*(r).\g x_\alpha(r').\beta^*(r)^{-1}=\tilde s_\beta^{-1}.\tilde s_\beta(r^{-1}).\g x_\alpha(r').\tilde s_\beta(r^{-1})^{-1}.\tilde s_\beta=\tilde s_\beta^{-1}.\g x_\gamma(\epsilon.r^{\alpha(\beta^\vee)}.r').\tilde s_\beta=\g x_\alpha(r^{\alpha(\beta^\vee)}.r')=\g x_\alpha(\alpha(\beta^*(r)).r')$ (avec $\gamma=s_\beta(\alpha)$
 et $s^*_\beta(e_\alpha)=\epsilon.e_\gamma$).

\par Cette relation s'\'etend au cas $\alpha$ non simple d'apr\`es ($\overline{\mathrm KMT}$5), ($\overline{\mathrm KMT}$7) et ($\overline{\mathrm KMT}$8): si $\alpha$, $\beta$ sont simples et $\gamma$ satisfait \`a  ($\overline{\mathrm KMT}$4),on a:

\par $\beta^*(r).\g x_{s_\alpha(\gamma)}((-1)^{\gamma(\alpha^\vee)}.\epsilon.r').\beta^*(r)^{-1}=\beta^*(r).\tilde s_\alpha^{-1}.\g x_\gamma(r').\tilde s_\alpha.\beta^*(r)^{-1}$
\par\noindent$=\tilde s_\alpha^{-1}.\beta^*(r).\alpha^*(r^{-\alpha(\beta^\vee)}).\g x_\gamma(r').\alpha^*(r^{\alpha(\beta^\vee)}).\beta^*(r)^{-1}.\tilde s_\alpha$
\par\noindent$=\tilde s_\alpha^{-1}.\g x_\gamma(r^{\gamma(\beta^\vee)}.(r^{-\alpha(\beta^\vee)})^{\gamma(\alpha^\vee)}.r').\tilde s_\alpha=\g x_{s_\alpha(\gamma)}((-1)^{\gamma(\alpha^\vee)}.\epsilon.r^{\gamma(\beta^\vee)-\alpha(\beta^\vee).\gamma(\alpha^\vee)}.r')$
\par et $\gamma(\beta^\vee)-\alpha(\beta^\vee).\gamma(\alpha^\vee)=s_\alpha(\gamma)(\beta^\vee)$. Donc $s_\alpha(\gamma)$ satisfait aussi \`a ($\overline{\mathrm KMT}$4).

\medskip
\par 5) On peut appliquer ($\overline{\mathrm KMT}$7) \`a $\qg=\pm{}\qa$. On sait que $s^*_\qa(e_{\pm{}\qa})=e_{\mp\qa}$. Donc $\widetilde s_\qa(r).x_{\pm{}\qa}(r')$ $.\widetilde s_\qa(r)^{-1}=x_{\mp\qa}(r^{\mp2}.r')$ et
 $\widetilde s_\qa(r)=\widetilde s_\qa(r).\widetilde s_\qa(r).\widetilde s_\qa(r)^{-1}=x_{-\qa}(r^{-1}).x_\qa(r).x_{-\qa}(r^{-1})=\widetilde s_{-\qa}(r^{-1})$.

\end{remas*}

\begin{prop}
\label{1.8}

\par 1) L'homomorphisme canonique de $\g{St}_A$ dans $\g G_\shs$ se factorise par $\overline{\g{St}}_A$, autrement dit les relations ($\overline{\mathrm KMT}$5), ($\overline{\mathrm KMT}$7) et ($\overline{\mathrm KMT}$8) sont satisfaites dans $\g G_\shs$.

\par 2) Pour un anneau $k$, $\g G_\shs(k)$ est le quotient du produit libre $\overline{\g{St}}_A(k)*\g T_\shs(k)$ par les 2 relations suivantes (pour $\alpha$ racine simple, $\beta\in\Phi$, $r\in k$ et $t\in\g T_\shs(k)$) :

\par(KMT4) \quad  $t.\g x_\beta(r).t^{-1}=\g x_\beta(\beta(t).r)$,

\par ($\overline{\mathrm KMT}$6) \quad $\alpha^*(r)=\alpha^\vee(r)$ si $r\in k^*$.

\par 3) La relation ($\overline{\mathrm KMT}$6) induit un homomorphisme de foncteurs en groupes $\psi\,:\,\g T_A\rightarrow\g T_\shs$ dont le noyau $\g Z_\shs$ est central dans $\g{St}_A$. Le foncteur $\g G_\shs$ est le quotient (comme foncteur en groupes) du produit semi-direct $(\overline{\g{St}}_A/\g Z_\shs)\rtimes\g T_\shs$ par $\g T_A/\g Z_\shs$ antidiagonal. En particulier le quotient (comme foncteur en groupes) $\overline{\g{St}}_A/\g Z_\shs$ s'injecte dans $\g G_\shs$.

\end{prop}

\begin{proof}

\par Par d\'efinition de $\alpha^*(r)$ et (KMT6), la relation ($\overline{\mathrm KMT}$6) est bien satisfaite dans  $\g G_\shs(k)$. On en d\'eduit aussit\^ot la relation ($\overline{\mathrm KMT}$8) et que l'image de $\g T_A(k)$ est dans $\g T_\shs(k)$. D'autre part les deux involutions $s'_\beta$ et $s_\beta$ co\"{\i}ncident sur cette image d'apr\`es les calculs de \ref{1.7}.3; ainsi ($\overline{\mathrm KMT}$5) est une cons\'equence de (KMT5). Enfin $\tilde s_\alpha(r).\g x_\gamma(r').\tilde s_\alpha(r)^{-1}=\tilde s_\alpha.\alpha^*(r^{-1}).\g x_\gamma(r').\alpha^*(r).\tilde s_\alpha^{-1}=\tilde s_\alpha.\g x_\gamma(\gamma(\alpha^*(r^{-1})).r').\tilde s_\alpha^{-1}=\tilde s_\alpha.\g x_\gamma(r^{-\gamma(\alpha^\vee)}.r').\tilde s_\alpha^{-1}=\g x_\delta(\epsilon.r^{-\gamma(\alpha^\vee)}.r')$, d'o\`u ($\overline{\mathrm KMT}$7).

\par On a ainsi montr\'e 1) et que les relations (KMT4), ($\overline{\mathrm KMT}$6) sont v\'erifi\'ees dans $\g G_\shs(k)$. Inversement il faut montrer que (KMT5), (KMT6) et (KMT7) sont des cons\'equences de (KMT4), ($\overline{\mathrm KMT}$5), ($\overline{\mathrm KMT}$6), ($\overline{\mathrm KMT}$7) et ($\overline{\mathrm KMT}$8). La relation (KMT7) est un cas particulier de ($\overline{\mathrm KMT}$7), (KMT6) d\'ecoule de ($\overline{\mathrm KMT}$6) et de la d\'efinition de $\alpha^*$. Pour $\alpha$ simple et $t\in\g T_\shs(k)$, on a : $t^{-1}.\tilde s_\alpha.t.\tilde s_\alpha^{-1}=t^{-1}.\g x_{\alpha}(1).\g x_{-\alpha}(1).\g x_{\alpha}(1).t.\tilde s_\alpha^{-1}=\g x_{\alpha}(\alpha(t)^{-1}).\g x_{-\alpha}(\alpha(t)).\g x_{\alpha}(\alpha(t)^{-1}).\tilde s_\alpha^{-1}=\tilde s_\alpha(\alpha(t)^{-1}).\tilde s_\alpha^{-1}=\tilde s_\alpha.\alpha^*(\alpha(t)).\tilde s_\alpha^{-1}=\alpha^*(\alpha(t)^{-1})=\alpha^\vee(\alpha(t)^{-1})=t^{-1}.s_\alpha(t)$ ; en effet pour $t=\lambda(r)$ avec $\lambda\in Y$ et $r\in k^*$ on a : $s_\alpha(\lambda(r))=(\lambda-\alpha(\lambda)\alpha^\vee)(r)=\lambda(r).\alpha^\vee(r^{-\alpha(\lambda)})=\lambda(r).\alpha^\vee(\alpha(\lambda(r))^{-1})$. Ainsi (KMT5) est satisfaite.

\par 3) On a vu en \ref{1.6}.1 que $\g T_\shs$ s'injecte dans $\g G_\shs$, la relation ($\overline{\mathrm KMT}$6) n'induit donc aucun quotient dans $\g T_\shs$. Par contre, comme $\g T_A$ est engendr\'e par les $\alpha^*$, cette relation se traduit par un homomorphisme $\psi\,:\,\g T_A\rightarrow\g T_\shs$ ; celui-ci est compatible avec les morphismes $\gamma\,:\,\g T_A(k)\rightarrow k^*$ et $\gamma\,:\,\g T_\shs(k)\rightarrow k^*$ (pour $\gamma\in\Phi$), d'apr\`es ($\overline{\mathrm KMT}$8c) et la relation correspondante dans $\g T_\shs$.
 Ainsi la relation ($\overline{\mathrm KMT}$4) (\ref{1.7}.4) montre que le noyau $\g Z_\shs$ est central dans $\g{St}_A$. La relation (KMT4) montre que $\g G_\shs$ est quotient du produit semi-direct indiqu\'e par des relations induites par ($\overline{\mathrm KMT}$6). La relation ($\overline{\mathrm KMT}$4) montre que ce quotient se limite \`a l'action antidiagonale de $\g T_A/\g Z_\shs$.
\end{proof}

\begin{coro}\label{1.9}

\par 1) Pour le SGR simplement connexe $\shs_A$, on a $\g T_A=\g T_{\shs_A}$ et $\overline{\g{St}}_A=\g G_{\shs_A}$, que l'on notera aussi $\g G_A$ ou $\g G^A$.

\par 2) Si $k$ est un corps, l'homomorphisme canonique de $\g U_{\shs_A}^{\pm{}}(k)$ dans $\g U_{\shs}^{\pm{}}(k)$ est un isomorphisme.

\end{coro}

\begin{proof} 1) Comme $\g T_{\shs_A}(k)$ est produit direct des groupes $\alpha^\vee(k^*)$ (pour $\alpha$ racine simple) et s'injecte dans $\g G_{\shs_A}(k)$, la relation ($\overline{\mathrm KMT}$6) permet d'identifier $\g T_A(k)$ et $\g T_{\shs_A}(k)$. Alors (KMT4) se r\'eduit \`a ($\overline{\mathrm KMT}$4) dont on sait qu'il est v\'erifi\'e dans $\overline{\g{St}}_A(k)$ (\ref{1.7}.4), on a donc bien l'identification propos\'ee.

\par 2) On sait dans $\g G_{\shs_A}(k)$ que $\g T_{\shs_A}(k)$ et $\g U_{\shs_A}^{\pm{}}(k)$ forment un produit semi-direct (\ref{1.6}.5) le quotient par $\g Z_\shs(k)$ n'affecte donc pas $\g U_{\shs_A}^{\pm{}}(k)$.
\end{proof}

\subsection{Fonctorialit\'e}\label{1.10}

\par Si $\varphi\,:\,\shs\rightarrow\shs'$ est une extension commutative de SGR, la proposition \ref{1.8} permet de d\'efinir un morphisme $\g G_\varphi\,:\,\g G_\shs\rightarrow\g G_{\shs'}$ de foncteurs. En effet $\shs$ et $\shs'$ correspondent \`a la m\^eme matrice $A$, donc les m\^emes $\Phi$, $\g g_A$, ${\g{St}}_A(k)$ et $\overline{\g{St}}_A(k)$.
 La compatibilit\'e de $\varphi$ avec les racines et coracines permet  de d\'efinir $\g G_\varphi(k)$ comme le passage au quotient de ${\mathrm Id}*\g T_\varphi(k)$. Pour tout anneau $k$, les homomorphismes $\g G_\varphi(k)$ et $\g T_\varphi(k)$ ont m\^eme noyau; le groupe $\g G_{\shs'}(k)$ est engendr\'e par $\g T_{\shs'}(k)$ et le sous-groupe (distingu\'e) image de $\g G_\varphi(k)$  (\ref{1.8}.3);
on a $\g N_{\shs'}(k)=\varphi(\g N_{\shs}(k)).\g T_{\shs'}(k)$ .
De plus  l'image r\'eciproque de $\g N_{\shs'}(k)$ (resp. $\g T_{\shs'}(k)$) dans $\g G_{\shs}(k)$ est \'egale \`a $\g N_{\shs}(k)$ (resp. $\g T_{\shs}(k)$).
 Le noyau Ker$\g G_\qf$ est central dans $\g G_\shs$ (mais trivial si $\shs$ est un sous-SGR).
Par construction $\g G_\varphi$ induit un isomorphisme de $\g U^+_{\shs}(k)$ sur $\g U^+_{\shs'}(k)$ et de $\g U^-_{\shs}(k)$ sur $\g U^-_{\shs'}(k)$, pour tout corps $k$ (\ref{1.9}.2).

\par Sur un corps $k$, on voit facilement qu'il est possible d'identifier les immeubles combinatoires  $\shi^v_\pm$ de $\g G_{\shs}$ et $\g G_{\shs'}$ avec leurs facettes et leurs appartements. Les actions de $\g G_{\shs}(k)$ et $\g G_{\shs'}(k)$ sont compatibles (via $\g G_\varphi(k)$).

\par On va \'etudier quelques cas particuliers int\'eressants.

\begin{coro}\label{1.11}

\par Si le SGR $\shs'$ est extension semi-directe (resp. directe) du SGR $\shs$, alors $\g G_{\shs'}$ est produit semi-direct (resp. direct) de $\g G_{\shs}$ par un tore $\g T_{Z}$ o\`u $Z$ est un suppl\'ementaire de $Y$ dans $Y'$ et, pour un anneau $k$,  l'action de $\g T_{Z}(k)$ sur $\g G_{\shs}(k)$ est donn\'ee par les relations \quad $t.t'.t^{-1}=t'$ \;et\; $t.\g x_\alpha(r).t^{-1}=\g x_\alpha(\alpha(t).r)$\quad pour $t\in \g T_{Z}(k)$, $t'\in \g T_{\shs}(k)$, $\alpha\in\Phi$ et $r\in k$.

\end{coro}

\begin{proof}

\par On a $Y'=Y\oplus Z$, donc $\g T_{Y'}=\g T_{Y}\times\g T_{Z}\subset\g G_{\shs'}$. Dans $\overline{\g{St}}_A*(\g T_{Y}\times\g T_{Z})$ la relation ($\overline{\mathrm KMT}$6) n'implique que $\overline{\g{St}}_A*\g T_{Y}$ et (KMT4) d\'ecrit l'action de $\g T_{Z}$ sur $\overline{\g{St}}_A$, d'o\`u le r\'esultat.
\end{proof}

\begin{coro}\label{1.12}

\par Si le SGR $\shs$ est extension centrale torique du SGR $\shs'$, alors $\g G_{\shs'}$ est quotient (comme foncteur en groupes) de $\g G_{\shs}$ par un sous-tore facteur direct $\g T'$ de $\g T_\shs$, central dans $\g G_{\shs}$. Si l'extension centrale est scind\'ee, alors $\g G_{\shs}$ est isomorphe au produit direct de $\g G_{\shs'}$ et $\g T'$.

\end{coro}

\begin{proof}

\par Comme $\varphi\,:\,Y\rightarrow Y'$ est surjective, $Z={\mathrm Ker}\varphi$ est facteur direct dans $Y$ et contenu dans les noyaux Ker$\beta$ pour $\beta\in\Phi$. On pose $\g T'=\g T_Z$ et la premi\`ere assertion est alors une cons\'equence directe de \ref{1.8}. Si l'extension est scind\'ee, $Z$ a un suppl\'ementaire $Y''$ contenant les coracines et $\varphi$ induit un isomorphisme de $\shs''$ sur $\shs'$, d'o\`u le r\'esultat.
\end{proof}

\begin{coro}\label{1.13}

\par Si $\varphi\,:\,\shs\rightarrow\shs'$ est une extension centrale finie, alors pour un anneau $k$, $\g T_\varphi(k)\,:\,\g T_\shs(k)\rightarrow\g T_{\shs'}(k)$ a pour noyau le groupe fini $\g Z(k)=\{\,z\in\g T_\shs(k)\mid\chi\circ\g T_\varphi(z)=1\;\forall\chi\in X'\,\}$ contenu dans tous les noyaux de racines de $\Phi$. Ce groupe $\g Z(k)$ est central dans $\g G_\shs(k)$, le quotient $\g G_\shs(k)/\g Z(k)$ est un sous-groupe distingu\'e de $\g G_{\shs'}(k)$ et $\g G_{\shs'}(k)=(\g G_\shs(k)/\g Z(k)).\g T_{\shs'}(k)$.

\end{coro}

\begin{NB} Sauf si $k$ est un corps alg\'ebriquement clos, $\g G_{\shs'}(k)$ n'est pas forc\'ement \'egal \`a $\g G_\shs(k)/\g Z(k)$, car $\g T_\varphi(k)$ n'est pas forc\'ement surjective, voir \ref{1.2}.

\end{NB}

\begin{proof}

\par L\`a encore ($\overline{\mathrm KMT}$6) dans $\g G_{\shs'}$ est cons\'equence de la m\^eme relation dans $\g G_{\shs}$ et (KMT4) montre que $\g T_{\shs'}(k)$ normalise l'image de $\g G_{\shs}(k)$. L'injection de $\g T_{\shs'}(k)$ dans $\g G_{\shs'}(k)$ permet de contr\^oler la situation.
\end{proof}


\section{Alg\`ebre enveloppante enti\`ere}\label{s2}

\par Cette alg\`ebre "enveloppante" est la base de la construction par J. Tits de son groupe de Kac-Moody \'etudi\'e ci-dessus. On va  prouver ci-dessous un th\'eor\`eme de Poincar\'e-Birkhoff-Witt et introduire des "exponentielles tordues" dans le compl\'et\'e de cette alg\`ebre.
 On g\'en\'eralise ainsi des r\'esultats de H. Garland dans le cas affine (\cf \cite[Th. 5.8]{Ga78}), pr\'ecis\'es par D. Mitzman \cite{Mn-85}, voir \ref{2.12} ci-dessous.

\subsection{Rappels}\label{2.1} \cf \cite{T-87b}, \cite{Ry-02a}, \cite{M-88a}, \cite[VIII {\S{}} 12]{B-Lie}.

\medskip
\par 1) Dans l'alg\`ebre enveloppante $\shu_\C(\g g_\shs)$ de l'alg\`ebre de Lie complexe $\g g_\shs$, J. Tits
 \cite[{\S{}}  4]{T-87b} a introduit la sous$-\Z-$alg\`ebre $\shu_\shs$ engendr\'ee par les \'el\'ements suivants:
 \par\quad $e_i^{(n)}=\left.\begin{array}{c}e_i^n \\\hline n!\end{array}\right.$\qquad pour $i\in I$ et $n\in\N$,
 \par\quad $f_i^{(n)}=\left.\begin{array}{c}f_i^n \\\hline n!\end{array}\right.$\qquad pour $i\in I$ et $n\in\N$,
 \par\quad $\left(\begin{array}{c}h \\ n\end{array}\right)$ \qquad pour $h\in Y$ et $n\in\N$.

 \par La sous-alg\`ebre $\shu_\shs^+$ (resp. $\shu_\shs^-$, $\shu_\shs^0$) est engendr\'ee par les \'el\'ements du premier type (resp. du second type, du troisi\`eme type). On oubliera souvent l'indice $\shs$ dans les notations pr\'ec\'edentes, d'ailleurs $\shu_\shs^+$ et $\shu_\shs^-$ ne d\'ependent que de $A$ et non de $\shs$.

 \par On sait que $\shu_\shs$ (resp. $\shu_\shs^+$, $\shu_\shs^-$, $\shu_\shs^0$) est une $\Z-$forme de l'alg\`ebre enveloppante correspondante $\shu_\C(\g g_\shs)$ (resp $\shu_\C(\g n^+_\shs)$, $\shu_\C(\g n^-_\shs)$, $\shu_\C(\g h_\shs)$) et que l'on a la d\'ecomposition unique $\shu_\shs=\shu_\shs^+.\shu_\shs^0.\shu_\shs^-$.

 \medskip
 \par 2) La graduation de l'alg\`ebre de Lie $\g g_\shs$ par $Q$ induit une graduation de l'alg\`ebre associative $\shu_\shs$ par $Q$; le sous-espace de poids $\alpha\in Q$ est not\'e $\shu_{\shs\alpha}$. La sous-alg\`ebre $\shu_\shs^+$ (resp. $\shu_\shs^-$) est gradu\'ee par $Q^+$ (resp. $Q^-$), la sous-alg\`ebre $\shu_\shs^0$ est de poids $0$.

 \par L'alg\`ebre $\shu_\shs$ est un sous$-\shu_\shs-$module pour la repr\'esentation adjointe $ad$ de $\shu_\C(\g g_\shs)$ sur elle-m\^eme. En effet $(ad\,e_i^{(n)})(u)=\left.\begin{array}{c}(ad(e_i))^n \\\hline n!\end{array}\right. .(u)=\sum_{p+q=n}\,(-1)^q.e_i^{(p)}.u.e_i^{(q)}$ et, pour $h\in Y$, $ad\left(\begin{array}{c}h \\ n\end{array}\right)$ agit sur $\shu_\C(\g g_\shs)_\alpha$ par multiplication par $\left(\begin{array}{c}\alpha(h) \\ n\end{array}\right)\in\Z$.

 \par La filtration de $\shu_\C(\g g_\shs)$ induit une filtration sur $\shu_\shs$. Le niveau $1$ de cette filtration est donc somme directe de $\Z$ et de $\g g_{\shs\Z}=\g g_{\shs}\cap\shu_\shs$. L'alg\`ebre de Lie $\g g_{\shs\Z}$ est un sous$-\shu_\shs-$module pour l'action adjointe, elle a une d\'ecomposition triangulaire $\g g_{\shs\Z}=\g n^+_{\shs\Z}\oplus Y\oplus\g n^-_{\shs\Z}$ (o\`u $\g n^{\pm}_{\shs\Z}=\g n^{\pm}_{\shs}\cap\shu_\shs^\pm$ et $Y=\g h_{\shs\Z}=\g H_\shs\cap\shu_\shs^0$) et une d\'ecomposition radicielle $\g g_{\shs\Z}=Y\oplus(\oplus_{\alpha\in\Delta}\,\g g_{\shs\alpha\Z})$ (o\`u $\g g_{\shs\alpha\Z}=\g g_{\alpha}\cap\shu_\shs$).

 \medskip
 \par 3) L'alg\`ebre $\shu_\shs$ (ou $\shu_\shs^+$, $\shu_\shs^-$, $\shu_\shs^0$) a une structure de $\Z-$big\`ebre cocommutative, co-inversible, elle est non commutative (sauf $\shu_\shs^0$). Sa comultiplication $\nabla$, sa co-unit\'e $\epsilon$ et sa co-inversion (ou antiautomorphisme principal) $\tau$ respectent la graduation et la filtration: elles sont donn\'ees sur les g\'en\'erateurs par les formules suivantes (pour $i\in I$, $h\in Y$ et $n\in\N$):

\par $\nabla\left(\begin{array}{c}h \\ n\end{array}\right)=\sum_{p+q=n}\,\left(\begin{array}{c}h \\ p\end{array}\right)\otimes\left(\begin{array}{c}h \\ q\end{array}\right)$, $\epsilon\left(\begin{array}{c}h \\ n\end{array}\right)=0$ pour $n>0$

\par\noindent et $\tau\left(\begin{array}{c}h \\ n\end{array}\right)=\left(\begin{array}{c}-h \\ n\end{array}\right)=(-1)^n.\left(\begin{array}{c}h+n-1 \\ n\end{array}\right)=(-1)^n.\sum_{p+q=n}\,\left(\begin{array}{c}n-1 \\ p\end{array}\right)\left(\begin{array}{c}h \\ q\end{array}\right)$

\par $\nabla e_i^{(n)}=\sum_{p+q=n}\,e_i^{(p)}\otimes e_i^{(q)}$, $\epsilon e_i^{(n)}=0$ pour $n>0$ et $\tau e_i^{(n)}=(-1)^n e_i^{(n)}$

\par $\nabla f_i^{(n)}=\sum_{p+q=n}\,f_i^{(p)}\otimes f_i^{(q)}$, $\epsilon f_i^{(n)}=0$ pour $n>0$ et $\tau f_i^{(n)}=(-1)^n f_i^{(n)}$.

\medskip
4) Les automorphismes $s_i^*$ de \ref{1.4} s'\'etendent \`a $\shu_\C(\g g_\shs)$ et stabilisent $\shu_\shs$. On a une action de $W^*$ sur $\shu_\shs$. Ainsi, pour une racine r\'eelle $\alpha$ et un $e_\alpha$ base (sur $\Z$) de $\g g_{\alpha\Z}$, l'\'el\'ement $e_\alpha^{(n)}=\left.\begin{array}{c}e_\alpha^n \\\hline n!\end{array}\right.$ est dans $\shu_\shs$; c'est une base de $\shu_{\shs n\alpha}\cap\shu_\C(\g g_\alpha)$. On a comme ci-dessus $\nabla e_\alpha^{(n)}=\sum_{p+q=n}\,e_\alpha^{(p)}\otimes e_\alpha^{(q)}$, $\epsilon e_\alpha^{(n)}=0$ pour $n>0$ et $\tau e_\alpha^{(n)}=(-1)^n e_\alpha^{(n)}$.

\par L'automorphisme $\exp(ad(e_\alpha))$ de $\shu_\C(\g g_\shs)$ stabilise $\shu_\shs$, puisqu'on a $\left.\begin{array}{c}(ad(e_\alpha))^n \\\hline n!\end{array}\right. .(u)=\sum_{p+q=n}\,(-1)^q.e_\alpha^{(p)}.u.e_\alpha^{(q)}$. On note $s^*_\alpha$ l'\'el\'ement  suivant  $\exp(ad(e_\alpha)).\exp(ad(f_\alpha)).\exp(ad(e_\alpha))=\exp(ad(f_\alpha)).\exp(ad(e_\alpha)).\exp(ad(f_\alpha))$; cet \'el\'ement de $W^*$ d\'epend du choix de $e_\alpha$ comme base de $\g g_{\alpha\Z}$, si on change $e_\alpha$ en $-e_\alpha$, on change $s^*_\alpha$ en $(s^*_\alpha)^{-1}$.

\par Plus g\'en\'eralement le foncteur en groupes $\g G_\shs$ admet une repr\'esentation adjointe $Ad$ dans les automorphismes de $\shu_\shs$ respectant sa filtration et donc aussi dans les automorphismes de $\g g_{\shs\Z}$ \cite[9.5]{Ry-02a}. Pour $\alpha\in\Phi$ et $r\in R$, $\g x_\alpha(r)$ agit sur $\shu_{\shs R}=\shu_\shs\otimes R$ comme $\sum_{n\geq0}\,ad(e_\alpha^{(n)})\otimes r^n$; le groupe $\g T_\shs$ agit selon la d\'ecomposition de $\shu_\shs$ ou $\g g_\Z$ selon les poids.

\medskip
\par 5)  L'alg\`ebre $\shu^+$ est gradu\'ee par $Q^+$. On peut la graduer par un {\it degr\'e total} dans $\N$ dont chaque facteur regroupe un nombre fini d'anciens facteurs. Ainsi sa compl\'etion positive (pour le degr\'e total) est ${\widehat\shu}^+=\prod_{\alpha\in Q^+}\,\shu^+_\alpha$; c'est une alg\`ebre pour le prolongement naturel de la multiplication de $\shu^+$. On peut de m\^eme consid\'erer l'alg\`ebre ${\widehat\shu}_k^+=\prod_{\alpha\in Q^+}\,\shu^+_\alpha\otimes k$ (en g\'en\'eral diff\'erente de ${\widehat\shu}^+\otimes k$) pour tout anneau $k$.

\par On peut consid\'erer la compl\'etion positive ${\widehat\shu}^p$ de $\shu$ selon le degr\'e total (dans $\Z$); c'est une alg\`ebre contenant ${\widehat\shu}^+$.  On d\'efinit de m\^eme ${\widehat\shu}^p_k$ qui contient ${\widehat\shu}_k^+$ . L'alg\`ebre ${\widehat\shu}^p_k$ contient l'alg\`ebre de Lie  ${\widehat{\g g}}^p_k=(\oplus_{\alpha<0}\,\g g_{\alpha k})\oplus\g h_k\oplus(\prod_{\alpha>0}\,\g g_{\alpha k})$.

\par L'action adjointe $ad$ de $\shu$ sur $\shu$ respecte la graduation. On peut donc la prolonger en une action de ${\widehat\shu}^p$  (ou ${\widehat\shu}^p_k$) sur elle m\^eme. L'alg\`ebre de Lie  ${\widehat{\g g}}^p_k$ est un sous$-ad({\widehat\shu}^p_k)-$module.

\par Les r\'esultats pr\'ec\'edents et suivants sont bien s\^ur encore valables si on remplace $\Delta^+$ par $\Delta^-$ ou par un conjugu\'e de $\Delta^\pm$ par un \'el\'ement du groupe de Weyl $W^v$. On peut en particulier d\'efinir des compl\'etions n\'egatives ${\widehat\shu}^-$, ${\widehat\shu}^n$ ou ${\widehat{\g g}}^n$ (selon l'oppos\'e du degr\'e total).

\begin{prop}\label{2.2} L'alg\`ebre de Lie $\g g_\Z$ (resp. $\g n^+_\Z$, $\g n^-_\Z$) est le $\shu-$module (resp. $\shu^+-$module, $\shu^--$module) engendr\'e par $\g h_{\shs\Z}=Y$ et les $e_i$, $f_i$ (resp. par les $e_i$, par les $f_i$) pour $i\in I$ (pour l'action adjointe de $\shu$). L'alg\`ebre de Lie $\g g_\Z$  contient la sous-alg\`ebre de Lie $\g g^1_\Z$ engendr\'ee par $\g h_{\shs\Z}$ et les $e_i$, $f_i$ pour $i\in I$.

\end{prop}

\begin{proof} Comme $\g g_\Z$ est une alg\`ebre de Lie, elle contient $\g g_\Z^1$; comme c'est un sous$-\shu-$module de $\shu$ elle contient le $\shu-$module $\g g_\Z^2$ engendr\'e par $\g h_{\shs\Z}$ et les $e_i$, $f_i$. Si jamais $\g g_\Z\not=\g g_\Z^2$, il existe un nombre premier $p$ tel que l'image de $\g g_\Z^2$ dans $\g g_\Z\otimes\F_p$ n'est pas tout $\g g_\Z\otimes\F_p$. Mais ceci contredit \cite[1.7.2 p308]{M-96}. Le raisonnement pour $\g n^\pm$ est analogue.
\end{proof}

\begin{coro}\label{2.3} a) Soit $K$ un corps de caract\'eristique $p$. Si, pour tout coefficient de la matrice de Kac-Moody $A$, on a $p>-a_{i,j}$ ou si $p=0$, alors l'alg\`ebre $\g g^1_K=\g g^1_\Z\otimes K$
 (engendr\'ee par $\g h_{\shs\Z}\otimes K$ et les $e_i$, $f_i$ pour $i\in I$) est \'egale \`a $\g g_K=\g g_\Z\otimes K$. De m\^eme l'alg\`ebre de Lie $\g n^+_K=\g n^+_\Z\otimes K$ (resp. $\g n^-_K=\g n^-_\Z\otimes K$) est engendr\'ee par  
 les $e_i$ (resp. les $f_i$).

 \par b) Dans le cas simplement lac\'e \ie si $\vert a_{i,j}\vert\leq{}1$ pour $i\not=j$, on a $\g g^1_\Z=\g g_\Z$.

\end{coro}

\begin{proof} Le a) est clair pour $p=0$; pour $p>0$ c'est \cite[1.7.3 p308]{M-96}. On en d\'eduit alors facilement le b) comme dans la d\'emonstration pr\'ec\'edente. On peut aussi utiliser la proposition \ref{2.2} et \cite[prop 1 p181]{MT-72} pour prouver a) et b). Voir aussi \cite[prop 1]{T-81b}.
\end{proof}

\begin{prop}\label{2.4} Soit $x\in\g g_\Z$. Il existe des \'el\'ements $x^{[n]}\in\shu$ pour $n\in\N$, tels que $x^{[0]}=1$, $x^{[1]}=x$ et (si on note $x^{(n)}=x^n/n!\in\shu_\C(\g g_\shs)$) l'\'el\'ement $x^{[n]}-x^{(n)}$ est de filtration $<n$ dans $\shu_\C(\g g_\shs)$. Si de plus $x\in\g g_{\alpha\Z}$, $\alpha\in Q$, on peut supposer $x^{[n]}\in\shu_{n\alpha}$.

\end{prop}

\begin{rema*} Pour $x\in Y=\g h_\Z$, on peut prendre  $x^{[n]}=\left(\begin{array}{c}h \\ n\end{array}\right)$. Pour $x\in\g g_{\alpha\Z}$, $\alpha\in\Phi$, on peut prendre $x^{[n]}=x^{(n)}$. Le probl\`eme se pose pour $x\in\g g_{\alpha\Z}$, $\alpha\in\Delta_{im}$. Des calculs explicites sont faits dans \cite{Mn-85}; on constate que $x^{[n]}$ n'est pas forc\'ement dans $\sum_{n\in\N}\,\Q x^n$.

\end{rema*}

\begin{proof}  Pour $a,b\in\Z$, $x,y\in\g g_\Z$, on a $(ax+by)^{(n)}=\sum_{p+q=n}\;a^p b^q x^{(p)} y^{(q)}$ modulo filtration $<n$. On d\'eduit ainsi l'existence des $(ax+by)^{[n]}$ de celle des $x^{[p]}$,  $y^{[q]}$. La d\'emonstration se r\'eduit donc au cas o\`u $x$ est homog\`ene et l'un des g\'en\'erateurs du $\Z-$module $\g g_\Z$. Le r\'esultat signal\'e en remarque pour $x=e_i$, $x=f_i$ ou $x\in Y$ et la proposition \ref{2.2} montrent qu'il suffit de d\'emontrer le lemme suivant (le cas n\'egatif est semblable).
\end{proof}

\begin{lemm}\label{2.5} Soient $\beta\in\Delta^+$ et $x\in\g g_{\beta\Z}$ tels qu'il existe des \'el\'ements $x^{[n]}\in\shu_{n\beta}$ satisfaisant aux conditions de la proposition pr\'ec\'edente. Pour $i\in I$, $k\in \N$, notons $y=(ad(e_i)^k/k!)(x)\in\g g_\Z$ et $\gamma=\beta+k\alpha_i$ ($\in\Delta^+$ si $y\not=0$). Alors il existe des \'el\'ements $y^{[n]}\in\shu_{n\gamma}$ (pour $n\in\N$) satisfaisant aux conditions de la proposition pr\'ec\'edente.

\end{lemm}

\begin{proof} On raisonne par r\'ecurrence sur $n$; c'est vrai pour $n=0$ ou $1$. Si $\beta$ et $\alpha_i$ sont colin\'eaires, on a $\beta=\alpha_i$, donc $y=0$ pour $k\geq{}1$. On peut donc supposer $\beta$ et $\alpha_i$ ind\'ependants. On travaille dans l'alg\`ebre ${\widehat\shu}^+_\C=\prod_{\alpha\in Q^+}\,\shu^+_{\alpha\C}$
et pour $z\in \g n^+$ on note $\exp\,z=\sum_{n\in\N}\,z^{(n)}$.

\par On a alors \qquad$(\exp\,e_i).(\exp\,x).(\exp-e_i)=\exp(\sum_{m\in\N}\;(ad(e_i)^m/m!)(x))$

\noindent et aussi \qquad
$(\exp\,e_i).z.(\exp-e_i)=\sum_{p,q}\,(-1)^qe_i^{(p)}ze_i^{(q)}=\sum_{n\in\N}\,((ad\,e_i)^n/n!)(z)$ pour $z\in\shu^+_\C$,

\noindent voir \cite[II {\S{}} 6 ex.1 p90 et VIII {\S{}} 12 lemme 2 p174]{B-Lie}.

\par Identifions les termes de poids $n\gamma$ dans les deux membres de la premi\`ere relation ci-dessus, en calculant modulo filtration $<n$.

\par Dans le membre de gauche on trouve $\sum_{p+q=nk}\,(-1)^q e_i^{(p)}x^{(n)}e_i^{(q)}$ qui n'est diff\'erent de $\sum_{p+q=nk}\,(-1)^q e_i^{(p)}x^{[n]}e_i^{(q)}\in\shu_{n\gamma}$ que par le terme de poids $n\gamma$ de $(\exp\,e_i).z.(\exp-e_i)$ (pour $z=x^{[n]}-x^{(n)}\in\shu_{n\gamma\C}$ de filtration $<n$) c'est-\`a-dire par $\sum_{p+q=nk}\,(-1)^q e_i^{(p)}ze_i^{(q)}$. D'apr\`es la seconde relation ci-dessus ce dernier terme est de filtration $<n$. Ainsi le terme de poids $n\gamma$ du membre de gauche est dans $\shu_{n\gamma}$ modulo filtration $<n$.

\par Dans le membre de droite les termes de poids dans $n\beta+\N\alpha_i$ sont ceux de la somme $(\sum_{m\in\N}\,((ad\,e_i)^m/m!)(x)\,)^n/n!$, \'egale \`a $\sum_{\sum p_j=n}\;\prod_{j\in\N}\,(((ad\,e_i)^j/j!)(x))^{(p_j)}$ modulo des termes de filtration $<n$ et de m\^emes poids.
Pour s\'electionner les termes de poids $n\gamma$ dans cette derni\`ere somme, on a deux conditions sur les $p_j$: $n=\sum\,p_j$ et $nk=\sum\,jp_j$.  Ainsi $p_j=0$ pour $j>nk$, $p_j\leq{}n$ et $jp_j\leq{}nk$; donc si $p_j=n$ on a $j=k$.
Ainsi la partie de poids $n\gamma$ du membre de droite est, modulo filtration $<n$, la somme de $y^{(n)}$ et de produits $\prod_{j=o}^{j=kn}\,(((ad\,e_i)^j/j!)(x))^{(p_j)}$ avec des $p_j<n$.
Par r\'ecurrence le $j$-i\`eme terme de l'un de ces produits est dans $\shu$ modulo filtration $<p_j$, donc chacun de ces produits est dans $\shu_{n\gamma}$ modulo filtration $<n$.

\par Par comparaison des membres de gauche et de droite, on obtient que $y^{(n)}$ est dans $\shu_{n\gamma}$ modulo filtration $<n$.
\end{proof}

\subsection{Poincar\'e-Birkhoff-Witt pour $\shu$}\label{2.6}

\par Pour $\alpha\in\Delta\cup\{0\}$, soit $\shb_\alpha$ une base de $\g g_{\alpha\Z}$ sur $\Z$. Pour $\alpha=\alpha_i$ racine simple, on choisit $\shb_\alpha=\{e_i\}$ et $\shb_{-\alpha}=\{f_i\}$. Par conjugaison par $W^v$, pour $\alpha\in\Phi$, on peut choisir $\shb_\alpha=\{e_\alpha\}$ et $\shb_{-\alpha}=\{f_\alpha\}$ de fa\c{c}on que $[e_\alpha,f_\alpha]=-\alpha^\vee$. On ordonne la base $\shb=\cup\,\shb_\alpha$ de $\g g_{\Z}$ de mani\`ere quelconque.

\par Pour $x\in\shb$ on choisit des \'el\'ements $x^{[n]}$ comme en \ref{2.4}. Pour $N=(N_x)_{x\in\shb}\in\N^{(\shb)}$, on d\'efinit $[N]=\prod_{x\in\shb}\,x^{[N_x]}$ (ce produit est pris dans l'ordre choisi et est en fait fini car presque tous les $N_x$ sont nuls). Le poids de $[N]$ (et, par d\'efinition, de $N$) est la somme pond\'er\'ee $pds([N])=\sum\,N_xpds(x)\in Q$ des poids des $x\in\shb$.

\begin{prop*} Ces \'el\'ements $[N]$ forment une base de $\shu$ sur $\Z$.
\end{prop*}

\begin{remas*} 1) L'alg\`ebre gradu\'ee $gr(\shu)$ d\'efinie par la filtration de $\shu$ est donc une alg\`ebre de polyn\^omes divis\'es \`a coefficients dans $\Z$ et dont les ind\'etermin\'ees sont les \'el\'ements de $\shb$.

\par 2) Bien s\^ur ces r\'esultats sont encore valables sur un anneau $R$ quelconque. Si $R$ contient $\Q$, on peut remplacer $x^{[n]}$ par $x^{(n)}$ pour tout $x\in\shb$.

\end{remas*}

\begin{proof} C'est une cons\'equence imm\'ediate de \ref{2.4} d'apr\`es \cite[VIII {\S{}} 12 n\degres{}3 prop. 1]{B-Lie} qui se g\'en\'eralise sans probl\`eme \`a la dimension infinie.
\end{proof}

\begin{prop}\label{2.7}  Pour $x\in\g g_{\alpha\Z}$, $\alpha\in\Delta\cup\{0\}$, on peut modifier les $x^{[n]}$ de la proposition \ref{2.4} de fa\c{c}on \`a satisfaire aux relations suppl\'ementaires suivantes:

\par\qquad $\nabla(x^{[n]})=\sum_{p+q=n}\,x^{[p]}\otimes x^{[q]}$\quad et \quad $\epsilon(x^{[n]})=0$ pour $n>0$.

\end{prop}

\begin{defi*} Pour $x\in\g g_{\alpha\Z}$, $\alpha\in\Delta\cup\{0\}$, on dit que $(x^{[n]})_{n\in\N}$ est une {\it suite exponentielle} associ\'ee \`a $x$ si elle satisfait aux conditions des propositions \ref{2.4} (y compris $x^{[n]}\in\shu_{n\alpha}$) et \ref{2.7}.

\par Plus g\'en\'eralement, si $x\in\g g_k$ pour $k$ un anneau, il existe des {\it suites exponentielles inhomog\`enes} associ\'ees \`a $x$; on peut prendre $(\lambda y+\mu z)^{[n]}=\sum_{p+q=n}\,\lambda^p\mu^q y^{[p]}z^{[q]}$.

\end{defi*}

\begin{rema*} $\shu $ (resp. $gr(\shu)$) est donc isomorphe comme cog\`ebre gradu\'ee (resp. big\`ebre gradu\'ee) \`a une alg\`ebre de polyn\^omes divis\'es, plus pr\'ecis\'ement \`a l'alg\`ebre sym\'etrique \`a puissances divis\'ees $S^{pd}_\Z(\g g_\Z)$.

\end{rema*}

\begin{proof} Quitte \`a soustraire $\epsilon(x^{[n]})$, on peut supposer $\epsilon(x^{[n]})=0$, pour $n>0$.  On raisonne par r\'ecurrence sur $n$; c'est clair pour $n=0,1$. Supposons le r\'esultat vrai pour $m<n$, alors pour $N\in\N^{(\shb)}$ et $d^\circ(N):=\sum_x\,N_x<n$ on a $\nabla({[N]})=\sum_{P+Q=N}\,{[P]}\otimes {[Q]}$ (avec des $P,Q\in\N^{(\shb)}$). La propri\'et\'e des $x^{(n)}$ vis \`a vis de $\nabla$ et la relation entre $x^{(n)}$ et $x^{[n]}$ montre qu'il existe des $a^n_{P,R}\in\C$ tels que:
\par
$\nabla(x^{[n]})=\sum_{p+q=n}\,x^{[p]}\otimes x^{[q]}+\sum_{d^\circ(P+R)<n}\,a^n_{P,R}[P]\otimes[R]$.

\par D'apr\`es la propri\'et\'e de la co-unit\'e $\epsilon$,  $a^n_{P,R}=0$ si $P$ ou $R=0$. Par co-commutativit\'e,  $a^n_{P,R}=a^n_{R,P}$. Enfin $\nabla$ \'etant compatible \`a la graduation,  $a^n_{P,R}=0$ si le poids de $P+R$ (\ie de $[P].[R]$) est diff\'erent de $n\alpha$. D'apr\`es la co-associativit\'e et la formule ci-dessus pour $\nabla({[N]})$ quand $d^\circ(N)<n$ on a:
\medskip

\par\noindent$\nabla(x^{[n]})\otimes1+\sum_{p+q+r=n}^{r>0}\,x^{[p]}\otimes x^{[q]}\otimes x^{[r]}+\sum_{d^\circ(P+Q+R)<n}\,a^n_{P+Q,R}[P]\otimes[Q]\otimes[R]$
\par\noindent$=1\otimes\nabla(x^{[n]})+\sum_{p+q+r=n}^{p>0}\,x^{[p]}\otimes x^{[q]}\otimes x^{[r]}+\sum_{d^\circ(P+Q+R)<n}\,a^n_{P,Q+R}[P]\otimes[Q]\otimes[R]$
\medskip

\par Apr\`es remplacement de $\nabla(x^{[n]})$ par son expression et simplification, on obtient:

\par\noindent$\sum_{d^\circ(P+Q)<n}\,a^n_{P,Q}[P]\otimes[Q]\otimes1+\sum_{d^\circ(P+Q+R)<n}\,a^n_{P+Q,R}[P]\otimes[Q]\otimes[R]$
\par\noindent$=\sum_{d^\circ(Q+R)<n}\,a^n_{Q,R}1\otimes[Q]\otimes[R]+\sum_{d^\circ(P+Q+R)<n}\,a^n_{P,Q+R}[P]\otimes[Q]\otimes[R]$
\medskip
\par On identifie alors les coefficients de $[P]\otimes[Q]\otimes[R]$ dans les deux membres. Pour $P=0$, $Q=0$ ou $R=0$, on retrouve des relations connues. Pour $P,Q,R\not=0$, on trouve que $a^n_{P+Q,R}=a^n_{P,Q+R}$. Sachant de plus que $a^n_{P,R}=a^n_{R,P}$, on en d\'eduit que $a^n_{P,R}$ (pour $P,R\not=0$) ne d\'epend que de $P+R$, on le note donc $b^n_{P+R}$. (On laisse au lecteur cet exercice \'el\'ementaire, mais n\'ecessitant l'examen de plusieurs cas.)

\par L'\'el\'ement $x^{\{n\}}=x^{[n]}-\sum_{d^\circ(S)<n}^{S\not=0}\,b^n_S[S]$ est bien \'egal \`a $x^{(n)}$ modulo filtration $<n$ et, comme $b^n_S=0$ pour $pds(S)\not=n\alpha$, son poids est bien $n\alpha$. On a:

\par $\nabla x^{\{n\}}=\sum_{p+r=n}\,x^{[p]}\otimes x^{[r]}+\sum_{d^\circ(P+R)<n}\,a^n_{P,R}[P]\otimes[R]$
\par\qquad\qquad$-\sum_{d^\circ(S)<n}^{S\not=0}\,b^n_{S}([S]\otimes1+1\otimes[S])-\sum\,b^n_{S}[P]\otimes[R]$

\par\noindent o\`u la derni\`ere somme porte sur les $P,R$ tels que $P+R=S$, $d^\circ(S)<n$ et $P,R\not=0$. Elle est donc \'egale \`a la seconde somme puisque $a^n_{P,R}=b^n_{P+R}$ ou $0$ selon que $P$ et $R\not=0$ ou non. Ainsi :

\par\qquad $\nabla x^{\{n\}}=x^{\{n\}}\otimes1+1\otimes x^{\{n\}}+\sum_{p+r=n}^{p,r\not=0}\,x^{[p]}\otimes x^{[r]}$
\end{proof}

\subsection{Exponentielles tordues}\label{2.8}

\par On raisonne dans ${\widehat\shu}^+_R$ pour une racine $\alpha\in\Delta^+$ (ou ${\widehat\shu}^-_R$ pour $\alpha\in\Delta^-$) et un anneau $R$.

\par Si $x\in\g g_{\alpha\Z}$, $\alpha\in\Delta^+$ et $\lambda\in R$, {\it l'exponentielle tordue} $[\exp]\lambda x= \sum_{n\geq{}0}\,\lambda^nx^{[n]}\in{\widehat\shu}^+_R$ n'est d\'efinie que modulo le choix de la suite exponentielle $x^{[n]}$.

\par Ses propri\'et\'es multiplicatives sont mauvaises: $[\exp](\lambda+\mu)x= \sum_{n\geq{}0}\,(\lambda+\mu)^nx^{[n]}$ est en g\'en\'eral diff\'erent de $[\exp]\lambda x.[\exp]\mu x$. Pour $\lambda\in\Z$ il faudrait poser $[\exp]\lambda x=([\exp]x)^\lambda$ (voir ci-dessous pour $\lambda<0$), mais l'extension pour  $\lambda\in R$ pose probl\`eme. Pour tous $x,y\in\g g_\Z$, $\lambda,\mu\in k$, $[\exp](\lambda x).[\exp](\mu y)$ est un choix possible d'exponentielle tordue (inhomog\`ene) pour $\lambda x+\mu y$.

\par Par contre $\epsilon([\exp]\lambda x)=1$ et pour le coproduit $\nabla [\exp]\lambda x=[\exp]\lambda x\widehat\otimes[\exp]\lambda x$ : l'exponentielle tordue est un \'el\'ement de type groupe. La co-inversion $\tau$ v\'erifie $prod\,\circ(Id\widehat\otimes\tau)\circ\nabla=\epsilon=prod\,\circ(\tau\widehat\otimes Id)\circ\nabla$. Donc $([\exp]\lambda x).(\tau[\exp]\lambda x)=1=(\tau[\exp]\lambda x).([\exp]\lambda x)$: l'exponentielle tordue est inversible dans ${\widehat\shu}^+_R$ et son inverse vaut $\tau[\exp]\lambda x=\sum_{n\geq{}0}\,\lambda^n\tau x^{[n]}$. Bien s\^ur $\tau x^{[n]}$ est l'une des suites exponentielles associ\'ees \`a $\tau x=-x$, comme la suite $(-1)^n x^{[n]}$, mais elle n'est pas clairement li\'ee (en g\'en\'eral) \`a cette derni\`ere.

\par Pour $R$ contenant $\Q$, ou $\alpha\in\Phi$, on peut consid\'erer les exponentielles classiques $exp(x)\in{\widehat\shu}^\pm_R$ qui ont d'excellentes propri\'et\'es.

\subsection{Unicit\'e des exponentielles tordues}\label{2.9}

\par 1) Pour $x\in\g g_{\alpha\Z}$, cherchons s'il existe une suite exponentielle $x^{\{n\}}$ autre que $x^{[n]}$. D'apr\`es la d\'emonstration de \ref{2.7}, il suffit de se poser la question quand on change $x^{[n+1]}$ sans changer $x^{[m]}$ pour $m\leq{}n$. On a alors $x^{\{n+1\}}=x^{[n+1]}+y$ avec $y\in\shu_{(n+1)\alpha\Z}$ de filtration $\leq{}n$. Sachant que $\nabla x^{[n+1]}=\sum_{p+q=n+1}\,x^{[p]}\otimes x^{[q]}$, on a $\nabla x^{\{n+1\}}=\sum_{p+q=n+1}\,x^{\{p\}}\otimes x^{\{q\}}$ si et seulement si $\nabla y=y\otimes 1+1\otimes y$, autrement dit si $y$ est primitif dans la cog\`ebre $\shu$, c'est-\`a-dire si $y\in\g g_{(n+1)\alpha\Z}$ \cite[II {\S{}} 1 n\degres{} 5 cor p13]{B-Lie}.

\par Ainsi la suite $x^{[n]}$ n'est unique qu'\`a des \'el\'ements $x_2\in\g g_{2\alpha\Z},\cdots,x_m\in\g g_{m\alpha\Z},\cdots$ pr\`es. Autrement dit un autre choix pour $[\exp]x$ peut s'\'ecrire $[\exp]'x=[\exp]x.[\exp]x_2.[\exp]x_3.\cdots.[\exp]x_m.\cdots$.

\par 2) Pour $\alpha$ racine r\'eelle, il y a unicit\'e puisque $\g g_{m\alpha\Z}=0$ pour $m\geq{}2$, on a donc toujours $x^{[n]}=x^{(n)}$ et $[\exp]x=\exp\,x$. Par contre il n'y a pas unicit\'e des $x^{[n]}$ pour $\alpha=0$ ou $\alpha$ racine imaginaire. On verra en \ref {2.12} dans des cas affines les bons choix propos\'es par D. Mitzman.

\par 3) Pour $x\in\g g_{\alpha\Z}$ mais en raisonnant dans ${\widehat\shu}^+_\Q$, on obtient qu'il existe
 $y_2\in\g g_{2\alpha\Q},\cdots,y_m\in\g g_{m\alpha\Q},\cdots$ tels que $[\exp]x=\exp(x).\exp(y_2).\cdots.\exp(y_m).\cdots$. En particulier $([\exp]x)^{-1}=\tau[\exp]x=\cdots.\exp(-y_m).\cdots.\exp(-y_2).\exp(-x)$. Si les espaces $\g g_{m\alpha}$ commutent, les exponentielles ci-dessus commutent.

 \par Mais l'on a naturellement $[\exp]\lambda x=\exp(\lambda x).\exp(\lambda^2y_2).\cdots.\exp(\lambda^my_m).\cdots$ et non pas $[\exp]\lambda x=\exp(\lambda x).\exp(\lambda y_2).\cdots.\exp(\lambda y_m).\cdots$ (qui d'ailleurs n'est en g\'en\'eral pas dans ${\widehat\shu}^+_\Q$). Il para\^{\i}t donc difficile de faire de $[\exp]$ un sous-groupe \`a un param\`etre pour $\alpha\in\Delta^{im}$.

 \par 4) {\bf Remarque}: l'expression ci-dessus de $[\exp]x$ (pour $x\in\g g_{\alpha\Z}$, $\alpha\in\Delta$) en terme d'exponentielles classiques, montre que, $\forall n,m\in\N$, $x^{[n]}\in\shu_\Q(\oplus_{r\geq{}1}\,\g g_{r\alpha\Q})\cap\shu_{n\qa}$ et $x^{[n]}.x^{[m]}=\left(\begin{array}{c}n+m \\ n\end{array}\right)x^{[n+m]}$ modulo $\shu_\Q(\oplus_{r\geq{}2}\,\g g_{r\alpha\Q})$. Ce n'\'etait pas clair auparavant pour $\alpha$ imaginaire.

 \par Ainsi la base de $\shu$ sur $\Z$ construite en \ref{2.6}, qui est compatible avec la graduation par les poids dans $Q$, est aussi compatible avec la d\'ecomposition de $\shu_\C$ ou $\shu_\Q$ en produit tensoriel correspondant \`a la d\'ecomposition en somme directe $\g g=\g h\oplus(\oplus_{\qa\in\QD}\,\g g_\qa)$, \`a condition de regrouper une racine imaginaire et ses multiples positifs. Pour $\alpha$ imaginaire, cette base est compatible \`a la filtration par les $(\oplus_{n\geq{}N}\,\g g_{n\qa})$.

 \subsection{Exponentielles tordues en poids nul}\label{2.10}

 \par On peut essayer de g\'en\'eraliser \ref{2.8} et \ref{2.9} pour $\alpha=0$ \ie $x=h\in Y=\g h_\Z$. On travaille donc dans $\shu^0$ et, pour simplifier, on se place en rang $1$: $Y$ est de base $h$ et $\shu^o=\sum_{i\geq{}0}\,\Z\left(\begin{array}{c}h \\ i\end{array}\right)$. On raisonne dans ${\widehat\shu}^0_R=\prod_{i\geq{}0}\,R\left(\begin{array}{c}h \\ i\end{array}\right)$ pour un anneau $R$; ce n'est pas le compl\'et\'e de $\shu^0_R$ pour la graduation par $Q$; la multiplication de $\shu^0_R$ se prolonge car
 $\left(\begin{array}{c}h \\ p\end{array}\right)\left(\begin{array}{c}h \\ q\end{array}\right)\in\sum_{i=sup(p,q)}^{i=p+q}\;\Z\left(\begin{array}{c}h \\ i\end{array}\right)$. Ainsi ${\widehat\shu}^0_R$ est une alg\`ebre commutative avec une comultiplication $\nabla: {\widehat\shu}^0_R\rightarrow{\widehat\shu}^0_R{\widehat\otimes}{\widehat\shu}^0_R$, mais malheureusement sans co-inversion et d\'ependant du choix de $h$ (au lieu de $-h$), voir ci-dessous.

 \par On sait que $\shu^0$ est la big\`ebre des distributions \`a l'origine du groupe multiplicatif $\g{Mult}$. La big\`ebre affine de ce groupe est $\Z[\g{Mult}]=\Z[T,T^{-1}]$ et la dualit\'e entre ces deux big\`ebres est donn\'ee par $\left(\begin{array}{c}h \\ i\end{array}\right)T^m=\left(\begin{array}{c}m \\ i\end{array}\right)$ (en fait $\left(\begin{array}{c}h \\ i\end{array}\right)$ est ${{\partial}\over{\partial T^i}}={{1}\over{i!}}({{\partial}\over{\partial T}})^i$ calcul\'e en l'\'el\'ement neutre) \cf \cite[II {\S{}} 4 5.10]{DG-70}.

 \par Pour $x\in R$, on peut poser $[x]=\sum_{i\geq{}0}\,(x-1)^i\left(\begin{array}{c}h \\ i\end{array}\right)\in{\widehat\shu}^0_R$ (c'est moralement $(1+(x-1))^h$). On a bien $[x]T^n=x^n$ pour $n\geq{}0$; par contre pour $n<0$, ceci n'a un sens que si $x-1$ est nilpotent (donc $x\in R^*$).

 \par On a $\nabla[x]=[x]\otimes[x]\in{\widehat\shu}^0_R{\widehat\otimes}{\widehat\shu}^0_R$: l'\'el\'ement $[x]$ est de type groupe.

 \par Le prolongement naturel de $\tau$ devrait donner:

 \par$\tau[x]=\sum_{i\geq{}0}\,(x-1)^i\left(\begin{array}{c}-h \\ i\end{array}\right)=\sum_{p,q}\,(1-x)^{p+q}\left(\begin{array}{c}p+q-1 \\ p\end{array}\right)\left(\begin{array}{c}h \\ q\end{array}\right)$

 \qquad$=\sum_{q\geq{}0}\,(1-x)^{q}(\sum_{p\geq{}0}\,(1-x)^{p}\left(\begin{array}{c}p+q-1 \\ p\end{array}\right))\left(\begin{array}{c}h \\ q\end{array}\right)$.

 \par Mais ce calcul n'a un sens dans ${\widehat\shu}^0_R$ que si $1-x$ est nilpotent, auquel cas $\tau[x]$ est bien \'egal \`a $[x^{-1}]=\sum_{q\geq{}0}\,(1-x)^qx^{-q}\left(\begin{array}{c}h \\ q\end{array}\right)$, vu le calcul ci-dessus de $x^{-q}=[x]T^{-q}$.

 \par Les exponentielles tordues en poids nul semblent donc trop imparfaites.

 \subsection{Polyn\^omes de Mitzman}\label{2.11} \cite[p114-116]{Mn-85}

 \par On consid\`ere des ind\'etermin\'ees ${\underline Z}=(Z_s)_{s\geq{}1}$ et $\zeta$. Le polyn\^ome $\Lambda_s({\underline Z})$ est le coefficient de $\zeta^s$ dans le d\'eveloppement de l'expression formelle $\exp(\sum_{j\geq{}1}\,Z_j{{\zeta^j}\over{j}})$. C'est un polyn\^ome en les $Z_j$ \`a coefficients positifs dans $\Q$. Son poids total est $s$ si $Z_j$ est de poids $j$.

 \par On a: $\Lambda_0=1$, $\Lambda_1=Z_1$, $\Lambda_2=Z_1^{(2)}+(1/2)Z_2$, $\Lambda_3=Z_1^{(3)}+(1/2)Z_1Z_2+(1/3)Z_3$, $\cdots$

 \begin{enonce*}{Cas particuliers} Pour certaines sp\'ecialisations des ind\'etermin\'ees, on a:

 \par si $Z_i=0$ pour $i\geq{}2$, alors $\Lambda_n=Z_1^{(n)}$,
 \par si $Z_i=Z^i$ pour $i\geq{}1$, alors $\Lambda_n=Z^n$,
 \par si $Z_i=t^iZ$ pour $i\geq{}1$ et $t\in R$, alors $\Lambda_n=t^n\left(\begin{array}{c}Z+n-1 \\ n\end{array}\right)=(-t)^n\left(\begin{array}{c}-Z \\ n\end{array}\right)$.
\end{enonce*}

 \begin{lemm*} a) $n\Lambda_n=\sum_{p+q=n}^{p\geq{}1}\,Z_p\Lambda_q$, si $n\geq{}1$.

 \par b) $\Lambda_n(\underline Z+\underline Z')=\sum_{p+q=n}\,\Lambda_p(\underline Z).\Lambda_q(\underline Z')$.

 \par c) $\Lambda_n({\underline Z})$ est congru \`a $Z_1^{(n)}$ modulo des polyn\^omes de degr\'e total $\leq{}n-1$. Pour $n\geq{}1$ $\Lambda_s({\underline Z})$ n'a pas de terme constant.

\end{lemm*}

\begin{proof} La premi\`ere assertion r\'esulte de la troisi\`eme ligne de la d\'emonstration de \cite[4.1.14]{Mn-85} p.115 puisque $\exp(\sum_{j\geq{}1}\,Z_j\zeta^j/j)=\sum_{n\geq{}0}\,\Lambda_n({\underline Z})\zeta^n$. Pour la seconde assertion, on a $\sum_{n\geq{}0}\,\Lambda_n({\underline Z}+{\underline Z'})\zeta^n=\exp(\sum_{j\geq{}1}\,(Z_j+Z_j'){{\zeta^j}\over{j}})=\exp(\sum_{j\geq{}1}\,Z_j{{\zeta^j}\over{j}}).\exp(\sum_{j\geq{}1}\,Z_j'{{\zeta^j}\over{j}})=(\sum_{p\geq{}0}\,\Lambda_p({\underline Z})\zeta^p)$\goodbreak\noindent.$(\sum_{q\geq{}0}\,\Lambda_q({\underline Z})\zeta^q)$, d'o\`u le r\'esultat. Enfin la derni\`ere assertion est claire.
\end{proof}

 \subsection{Exponentielles tordues dans les alg\`ebres affines}\label{2.12}

 \par Dans une $\Q-$alg\`ebre de Lie gradu\'ee consid\'erons une suite $\underline x=(x_i)_{i\geq{}1}$ d'\'el\'ements commutant deux \`a deux et de poids $pds(x_i)=i.pds(x_1)$. Dans l'alg\`ebre enveloppante on peut donc calculer $x^{\{n\}}=\Lambda_n(\underline x)$ qui est de poids $n.pds(x_1)$. D'apr\`es le lemme la suite des $x^{\{n\}}$ est exponentielle. On peut donc lui associer l'exponentielle tordue
 $\{\exp\}x=\sum_n\,x^{\{n\}}=\exp(\sum_j\,x_j/j)=\prod_j\,\exp(x_j/j)$.

 \par Revenons \`a la situation de \ref{2.9}: $\alpha\in\Delta$ et $x\in\g g_{\alpha\Z}$. Supposons que les espaces $\g g_{\alpha}$, $\g g_{2\alpha}$,$\cdots$, $\g g_{i\alpha}$,$\cdots$ commutent deux \`a deux. C'est vrai dans le cas des alg\`ebres de Kac-Moody affines (ou semi-simples); mais malheureusement ce sont essentiellement les seuls cas o\`u cela soit v\'erifi\'e. Pour une suite exponentielle $x^{[n]}$, il existe des \'el\'ements $y_i\in \g g_{i\alpha\Q}$ pour $i\geq{}2$ tels que $[\exp]x=\exp(x).\exp(y_2).\cdots.\exp(y_m).\cdots=\exp(x+y_2+\cdots+y_m+\cdots)$. Posons $x_1=x$, $x_2=2y_2$,$\cdots$, $x_m=my_m$,$\cdots$. On a alors $[\exp]x=\exp(\sum_{j\geq{}1}\,x_j/j)$. En identifiant les termes de poids $n\alpha$, on obtient $x^{[n]}=\Lambda_n(\underline x)\in\shu$. Gr\^ace \`a \ref{2.11}.a on en d\'eduit facilement par r\'ecurrence que $x_i\in\g g_\Z$, $\forall i$; mais les exemples ci-dessous montrent que $y_i=x_i/i$ n'est en g\'en\'eral pas dans $\g g_\Z$.

 \par David Mitzman a calcul\'e explicitement une base de l'alg\`ebre enveloppante enti\`ere $\shu_\shs$ dans le cas affine (et pour un certain SGR $\shs$) \cite[4.2.6 et 4.4.12]{Mn-85}. Elle s'exprime comme la base de \ref{2.6}; avec $h^{[n]}=\Lambda_n(h,h,\cdots)=(-1)^n\left(\begin{array}{c}-h \\ n\end{array}\right)$ pour $h\in\g h_\Z$ et $x^{[n]}=\Lambda_n(x,0,\cdots)=x^{(n)}$ pour $x$ base de $\g g_\alpha$, $\alpha\in\Phi$. Si $\alpha\in\Delta_{im}$ et $x\in\shb_\alpha$ le choix de $x^{[n]}$ est $x^{[n]}=\Lambda_n(x,x_2,\cdots,x_m,\cdots)$ avec des $x_i\in\g g_{i\alpha\Z}$; donc $[\exp]x=\exp(x).\prod_{j=2}^\infty\,\exp({1\over j}x_j)$, par d\'efinition des $\QL_n$. On va  d\'ecrire ces $x_j$ dans le cas le plus facile, \cf \lc 4.2.5: $\g g$ est  simplement lac\'ee (donc non tordue, de type $\widetilde A$, $\widetilde D$ ou $\widetilde E$). Voir aussi le corollaire \ref{2.3}.b.

 \par Une alg\`ebre de Lie $\g g$ de l'un de ces types peut s'\'ecrire comme extension centrale (avec centre de dimension $1$) d'un produit semi-direct (avec $\C$) d'une alg\`ebre de lacets $\g g'=\g g_1\otimes\C[t,t^{-1}]$ o\`u $\g g_1$ est une alg\`ebre de Lie simple complexe de type $A$, $D$ ou $E$. Les racines imaginaires sont les multiples entiers d'une racine $\delta$ et, pour $n\not=0$, on a $\g g_{n\delta}=\g h_1\otimes t^n$ o\`u $\g h_1$ est une sous-alg\`ebre de Cartan de $\g g_1$. Si $h_1,\cdots,h_\ell$ est une base de coracines dans $\g h_1$, l'espace $\g g_{n\delta\Z}$ admet pour base $\shb_{n\delta}=\{\,h_{i,n}=h_i\otimes t^n\mid i=1,\ell\,\}$. Le bon choix (de Mitzman) pour $(h_{i,n})^{[p]}$ est $\Lambda_p(h_{i,n},h_{i,2n},\cdots,h_{i,jn},\cdots)$, \cf \lc 3.9.1 et 4.2.3.

 \par On peut remarquer que $(h_{i,n})^{[p]}=\Lambda_p(h_{i}\otimes t^n,h_{i}\otimes t^{2n},\cdots,h_{i}\otimes t^{jn},\cdots)$ a pour image $(h_{i})^{[p]}\otimes t^{np}=\left(\begin{array}{c}h+p-1 \\ p\end{array}\right)\otimes t^{np}=\left(\begin{array}{c}-h \\ p\end{array}\right)\otimes (-t^n)^p$ dans le quotient $\shu_\C(\g g_1)\otimes\C[t,t^{-1}]$ de $\shu_\C(\g g')$ (troisi\`eme cas particulier de \ref{2.11}).

\medskip
 \par {\bf L'exemple de $\widetilde{SL}_2$}: Pour le type $\widetilde A_1$, on peut pr\'eciser encore plus. On a $\g g'=\g{sl}_2\otimes\C[t,t^{-1}]\subset \End_{\C[t,t^{-1}]}(\C[t,t^{-1}]^2)$. Si $h=diag(1,-1)\in\g{sl}_2$, alors, pour $n\not=0$, $\g g_{n\delta\Z}$ a pour base $h_n=t^nh$. La repr\'esentation naturelle de $\g g'$ sur $\C[t,t^{-1}]^2$ induit un homomorphisme $\pi$ de $\shu_\C(\g g')$ dans $\End_{\C[t,t^{-1}]}(\C[t,t^{-1}]^2)$ ou de $\widehat\shu^p_\C$ dans $\End_{\C(\!(t)\!)}(\C(\!(t)\!)^2)$.
 On a $\pi(h_n^{[p]})=(-t^n)^p\left(\begin{array}{c}-h \\ p\end{array}\right)=t^{np}\left(\begin{array}{c}h+p-1 \\ p\end{array}\right)\in \End_{\Z(\!(t)\!)}(\Z(\!(t)\!)^2)$. On en d\'eduit que $\pi(\shu)\subset \End_{\Z[t,t^{-1}]}(\Z[t,t^{-1}]^2)$ et que, pour $\lambda$ dans un anneau $k$, $\pi([\exp]\lambda h_n)=diag(u,v)\in \End_{k(\!(t)\!)}(k(\!(t)\!)^2)$ avec $u=\sum^\infty_{p=0}\,(\lambda t^n)^p\left(\begin{array}{c}p \\ p\end{array}\right)=1+\lambda t^n+\lambda^2t^{2n}+\cdots$ et $v=\sum^\infty_{p=0}\,(-\lambda t^n)^p\left(\begin{array}{c}1 \\ p\end{array}\right)=1-\lambda t^n=u^{-1}$. En particulier  $\pi([\exp]\lambda h_n)$ est bien inversible et m\^eme dans $SL_2(k(\!(t)\!))$.

 \subsection{Sous-alg\`ebres et compl\'etions}\label{2.13}
\medskip
\par 1) Pour $\alpha\in\Delta$, le module  $\shu^\alpha=\shu_\C(\oplus_{n\geq{}1}\,\g g_{n\alpha})\cap(\oplus_{n\geq{}0}\,\shu_{n\alpha})=\shu_\C(\oplus_{n\geq{}1}\,\g g_{n\alpha})\cap\shu$ est une sous$-\Z-$alg\`ebre de $\shu$.   Pour $\alpha\in\Phi$, $\shu^\alpha$ admet pour base sur $\Z$ les $e_\alpha^{(n)}$, $n\in\N$. Pour $\alpha\in\Delta_{im}$ et $n\in\N$, on choisit une base $\shb_{n\alpha}$ de $\g g_{n\alpha\Z}$ et des suites exponentielles $x^{[p]}$ pour $x\in\shb_{n\alpha}$. D'apr\`es \ref{2.9}.4 ci-dessus, ces $x^{[p]}$ sont dans $\shu^\alpha$ et, d'apr\`es \ref{2.6}, les $[N]=\prod_{x\in\shb_{(\alpha)}}\,x^{[N_x]}$ (pour $(\alpha)=\N^*\alpha\subset\Delta_{im}$, $\shb_{(\alpha)}=\cup_{\beta\in(\alpha)}\,\shb_\beta$, $N\in\N^{(\shb_{(\alpha)})}$) forment une base de $\shu^\alpha$. On en d\'eduit, en particulier, que $\shu^\alpha$ est une sous$-\Z-$big\`ebre co-inversible de $\shu$.

\par Si $\Psi\subset\Delta\cup\{0\}$ est un ensemble clos, $\g g_\Psi=\oplus_{\alpha\in\Psi}\,\g g_\alpha$ est une alg\`ebre de Lie; l'alg\`ebre $\shu(\Psi)$ engendr\'ee par les $\shu^\alpha$ pour $\alpha\in\Psi$ est une sous-big\`ebre co-inversible de $\shu$ de base les $[N]$ pour $N\in\N^{(\shb_\Psi)}$ avec $\shb_\Psi=\cup_{\alpha\in\Psi}\,\shb_\alpha$. L'alg\`ebre enveloppante de $\g g_\Psi$ est $\shu_\C(\Psi)=\shu(\Psi)\otimes\C$ et $\shu(\Psi)=\shu_\C(\Psi)\cap\shu$. Les alg\`ebres $\shu(\Psi)$ et $\g g_{\Psi\Z}=\g g_\Psi\cap\g g_\Z$ sont stables par l'action adjointe de   $\shu(\Psi)$. Pour $\Psi=\Delta^+$ (resp $\Psi=\Delta^+\cup\{0\}$) on a $\shu(\Psi)=\shu^+$ (resp. $\shu(\Psi)=\shu^0\otimes\shu^+$). On a $\shu^\alpha=\shu((\alpha))$.

\par Pour un anneau $k$, on note $\shu_k^\alpha=\shu^\alpha\otimes_\Z\,k$ et $\shu_k(\Psi)=\shu(\Psi)\otimes_\Z\,k$.
\medskip
\par 2) Si $\Psi\subset\Delta^+$ est clos, l'alg\`ebre $\shu(\Psi)$ est gradu\'ee par $Q^+$. On peut la graduer par le degr\'e total. Sa compl\'etion positive ${\widehat\shu}(\Psi)=\prod_{\alpha\in Q^+}\,\shu(\Psi)_\alpha$ est une sous$-\Z-$alg\`ebre  de $\widehat\shu^+$. On peut de m\^eme consid\'erer l'alg\`ebre ${\widehat\shu}_k(\Psi)=\prod_{\alpha\in Q^+}\,\shu(\Psi)_\alpha\otimes k$ (en g\'en\'eral diff\'erente de ${\widehat\shu}(\Psi)\otimes k$) pour tout anneau $k$.

\par Si $\Psi\subset\Delta\cup\{0\}$ est clos, on peut consid\'erer la compl\'etion positive ${\widehat\shu}^p(\Psi)$ de $\shu(\Psi)$ selon le degr\'e total (dans $\Z$); c'est une sous$-\Z-$alg\`ebre de ${\widehat\shu}^p={\widehat\shu}^p(\Delta\cup\{0\})$. On d\'efinit de m\^eme ${\widehat\shu}^p_k(\Psi)$. L'intersection de ${\widehat\shu}^p_k(\Psi)$ avec $\widehat{\g g}^p_k$ est ${\widehat{\g g}}^p_k(\Psi)=(\oplus_{\alpha\in\Psi}^{\alpha<0}\,\g g_{\alpha k})\oplus\g h_k\oplus(\prod_{\alpha\in\Psi}^{\alpha>0}\,\g g_{\alpha k})$. Bien s\^ur ${\widehat\shu}^p_k(\Psi)$ et ${\widehat{\g g}}^p_k(\Psi)$ sont des sous$-ad({\widehat\shu}^p_k(\Psi))-$modules de ${\widehat\shu}^p_k$

\par Si $\Psi'$ est un id\'eal de $\Psi$, ${\widehat\shu}^p_k(\Psi')$ (resp ${\shu}_k(\Psi')$) est un id\'eal de  ${\widehat\shu}^p_k(\Psi)$ (resp ${\shu}_k(\Psi)$).

\medskip
\par 3) Les r\'esultats pr\'ec\'edents et suivants sont bien s\^ur encore valables si on remplace $\Delta^+$ par $\Delta^-$ ou par un conjugu\'e de $\Delta^\pm$ par un \'el\'ement du groupe de Weyl $W^v$. On peut en particulier d\'efinir des compl\'etions n\'egatives ${\widehat\shu}^n(\Psi)$ (selon l'oppos\'e du degr\'e total).

 \subsection{Modules int\'egrables \`a plus haut poids}\label{2.14}

\par Soit $\lambda$ un poids dominant de $\g T_\shs$, \ie $\lambda\in X^+=\{\,\lambda\in X\mid \lambda(\alpha_i^\vee)\geq{}0\,,\,\forall i\in I\,\}$. Consid\'erons le $\g g_\shs-$module irr\'eductible $L(\lambda)$ de plus haut poids $\lambda$ \cite[chap. 9 et 10]{K-90}. Si $v_\lambda$ est un vecteur non nul de poids $\lambda$, le module $L(\lambda)$ est un quotient du module de Verma $V(\lambda)=\shu_\C(\g g_\shs)\otimes_{\shu_\C(\g b^+_\shs)}\,\C v_\lambda$ (dans le cas non sym\'etrisable, on prend plut\^ot pour $L(\lambda)$ le quotient int\'egrable maximal de $V(\lambda)$ comme en \cite[p28]{M-88a}). Le sous$-\shu_\shs-$module $L_\Z(\lambda)=\shu_\shs.v_\lambda=\shu^-_\shs.v_\lambda$ est une $\Z-$forme de $L(\lambda)$; on peut d\'efinir $L_k(\lambda)=L_\Z(\lambda)\otimes_\Z\,k$, pour tout anneau $k$. Ces modules sont int\'egrables au sens de \cite{T-81b} et gradu\'es par $Q^-$(\`a translation pr\`es) avec des espaces de poids de dimension finie.
En particulier on a une repr\'esentation diagonalisable de $\g T_\shs$ dans $L_\Z(\lambda)$.

\par Il est clair que, pour tout anneau $k$, le $k-$module $L_k(\lambda)$ est en fait un $\widehat\shu^p_{\shs,k}-$module.

\par On d\'efinit de m\^eme des modules int\'egrables \`a plus bas poids $L_k(\lambda)$ pour $\lambda\in -X^+$; ce sont des $\widehat\shu^n_{\shs,k}-$modules.


\section{Les groupes de Kac-Moody maximaux (\`a la Mathieu)}\label{s3}

\par On va \'etudier ci-dessous deux groupes de Kac-Moody "maximaux" dits {\it positif} et {\it n\'egatif} qui contiennent le groupe minimal \'etudi\'e dans la premi\`ere partie. Ils se d\'eduisent l'un de l'autre par l'\'echange de $\Delta^+$ et $\Delta^-$; on va donc essentiellement se concentrer sur le groupe positif. Ce groupe sera not\'e $\g G_\shs^{pma}$, les sous-groupes ou les alg\`ebres associ\'ees seront en g\'en\'eral not\'es avec un exposant $pma$ ou $ma+$. Les analogues pour $\Delta^-$ seront not\'es avec un exposant $nma$ ou $ma-$.

\subsection{Groupes (pro-)unipotents}\label{3.1}

\par Soit $\Psi\subset\Delta^+$ un ensemble clos de racines. On consid\`ere l'alg\`ebre $\shu(\Psi)$ de \ref{2.13}, qui ne d\'epend que de $A$ et $\Psi$ (et non de $\shs$). C'est une $\Z-$big\`ebre co-inversible, cocommutative, gradu\'ee par $Q^+$:  $\shu(\Psi)=\bigoplus_{\alpha\in\N\Psi}\,\shu(\Psi)_\alpha$; tous ses espaces de poids sont libres de dimension finie sur $\Z$. On consid\`ere son dual restreint $\Z[\g U^{ma}_\Psi]=\bigoplus_{\alpha\in\N\Psi}\,\shu(\Psi)_\alpha^*$. C'est une $\Z-$big\`ebre commutative et co-inversible (sa co-inversion est le dual de $\tau$ encore not\'e $\tau$, qui est un isomorphisme d'alg\`ebres); c'est-\`a-dire l'alg\`ebre d'un sch\'ema en groupes affine $\g U^{ma}_\Psi$ que l'on verra comme un foncteur en groupes: pour un anneau $R$, $\g U^{ma}_\Psi(R)={\mathrm Hom}_{\Z-alg}(\Z[\g U^{ma}_\Psi],R)$.  On note $\g U^{ma+}= \g U^{ma}_{\Delta^+}$. Pour $\alpha\in\Delta_{im}^+$ on note $\g U_{(\alpha)}=\g U^{ma}_{\{\,n\alpha\mid n\geq{}1\,\}}$, pour $\alpha\in\Phi^+$ on note $\g U_\alpha=\g U_{(\alpha)}=\g U^{ma}_{\{\alpha\}}$; alors $\g U^{ma}_\Psi$ et $\g U_{(\alpha)}$ sont des sous-sch\'emas en groupes de $\g U^{ma+}$.

\par La d\'efinition de $\Z[M]=\Z[\g U^{ma}_\Psi]$ dans \cite[p 19-21]{M-88a} est plus compliqu\'ee car elle ne se limite pas \`a $\Psi\subset\Delta^+$. Dans notre cas le lemme 3 de \lc dit bien que $\Q[M]$ est le $\Q-$dual restreint de $\shu_\Q(\Psi)$ et alors $\Z[M]$ (d\'efini page 21 de \lc) est bien le dual restreint de $\shu(\Psi)$.

\par La base de $\shu(\Psi)$ index\'ee par des $N\in\N^{(\shb_\Psi)}$ expliqu\'ee en \ref{2.13}.1 fournit par dualit\'e une base $(Z^N)_N$ de $\Z[\g U^{ma}_\Psi]$, index\'ee par les m\^emes $N$. Comme $\nabla[N]=\sum_{P+Q=N}\,[P]\otimes[Q]$, on a $Z^P.Z^Q=Z^{P+Q}$. Ainsi $\Z[\g U^{ma}_\Psi]$ est une alg\`ebre de polyn\^omes sur $\Z$ dont les ind\'etermin\'ees $Z_x$ sont index\'ees par les $x\in\shb_\Psi$.

\par On retrouve ainsi un r\'esultat d'Olivier Mathieu \cite[lemme 2 p41]{M-89}. Inversement si on admet ce r\'esultat, $\Z[\g U^{ma}_\Psi]$ admet comme base des mon\^omes, son dual restreint $\shu(\Psi)$ admet une base gradu\'ee qui a de bonnes propri\'et\'es pour le coproduit $\nabla$. Un raisonnement analogue au raisonnement d'unicit\'e de \ref{2.9} permet alors de construire des suites exponentielles.

\begin{prop}\label{3.2} Pour $\Psi$ comme ci-dessus et $R$ un anneau, $\g U^{ma}_\Psi(R)$ s'identifie au sous-groupe multiplicatif de ${\widehat\shu}_R(\Psi)$ form\'e des produits $\prod_{x\in\shb_\Psi}\,[\exp]\lambda_xx$, pour des $\lambda_x\in R$, le produit \'etant pris dans l'ordre (quelconque) choisi sur $\shb$. L'\'ecriture d'un \'el\'ement de $\g U^{ma}_\Psi(R)$ sous la forme d'un tel produit est unique.
\end{prop}

\begin{remas*} 1) Si $\Psi$ est fini et donc form\'e de racines r\'eelles, on retrouve le groupe de sommes formelles de \cite[9.2]{Ry-02a}, c'est-\`a-dire l'unique sch\'ema en groupes lisse connexe unipotent d'alg\`ebre de Lie $\g g_{\Psi\Z}=\oplus_{\alpha\in\Psi}\,\g g_{\alpha\Z}$, \cf \cite[prop. 1 p547]{T-87b} et aussi \ref{3.4} ci-dessous. Dans ce cas les exponentielles sont classiques et les produits sont finis.

\par 2) Pour $\alpha\in\Phi^+$, le choix de l'\'el\'ement de base $e_\alpha$ de $\g g_{\alpha\Z}$ d\'etermine un isomorphisme $\g x_\alpha$ du groupe additif $\g{Add}$ sur $\g U_\alpha$ donn\'e par $\g x_\alpha(r)=\exp(r.e_\alpha)$.

\par 3) Dans les produits infinis propos\'es et pour tout $n\in\N$, tous les facteurs sauf un nombre fini sont \'egaux \`a $1$ modulo des termes de degr\'e total $\geq{}n$. Ces produits infinis sont donc bien dans $\widehat\shu_R(\Psi)$.

\end{remas*}

\begin{proof} Le $R-$dual de $R[\g U^{ma}_\Psi]$ est $\widehat\shu_R(\Psi)$. Donc $\g U^{ma}_\Psi(R)$ s'identifie \`a l'ensemble des \'el\'ements $y\in\widehat\shu_R(\Psi)$ tels que $1=\epsilon(y)$ $(=y(1))$ et $y(\varphi.\varphi')=y(\varphi).y(\varphi')$ pour tous $\varphi,\varphi'\in R[\g U^{ma}_\Psi]$ \ie $y\,\circ\nabla^*=prod\,\circ(y\otimes y)$ ou encore $\nabla y=y\otimes y$. Ainsi $\g U^{ma}_\Psi(R)$  est l'ensemble des \'el\'ements de $\widehat\shu_R(\Psi)$ de terme constant $1$ et de type groupe. Parmi ces \'el\'ements il y a les produits infinis de l'\'enonc\'e. Mais $R[\g U^{ma}_\Psi]$ est une alg\`ebre de polyn\^omes d'ind\'etermin\'ees $Z_x$, $x\in\shb_\Psi$. Par d\'efinition $Z_x(\prod_y\,[\exp]\lambda_yy)=\lambda_x$. Ainsi ces produits infinis fournissent une et une seule fois toutes les applications de $\shb_\Psi$ dans $R$, donc tous les homomorphismes de $R[\g U^{ma}_\Psi]$ dans $R$.
\end{proof}

\begin{lemm}\label{3.3} Soient $\Psi'\subset\Psi\subset\Delta^+$ des sous-ensembles clos de racines.

\par a) $\g U^{ma}_{\Psi'}$ est un sous-groupe ferm\'e de $\g U^{ma}_{\Psi}$ et $\Z[\g U^{ma}_{\Psi}/\g U^{ma}_{\Psi'}]$ est une alg\`ebre de polyn\^omes d'ind\'etermin\'ees index\'ees par $\shb_\Psi\setminus\shb_{\Psi'}$.

\par b) Si $\Psi\setminus\Psi'$ est \'egalement clos, alors on a une d\'ecomposition unique $\g U^{ma}_{\Psi}=\g U^{ma}_{\Psi'}.\g U^{ma}_{\Psi\setminus\Psi'}$.

\par c) Si $\Psi'$ est un id\'eal de $\Psi$, alors $\g U^{ma}_{\Psi'}$ est un sous-groupe distingu\'e de $\g U^{ma}_{\Psi}$ et on a un produit semi-direct si, de plus, $\Psi\setminus\Psi'$ est clos.

\par d) Si  $\Psi\setminus\Psi'$ est r\'eduit \`a une racine $\alpha$, alors le quotient $\g U^{ma}_{\Psi}/\g U^{ma}_{\Psi'}$ est isomorphe au groupe additif $\g{Add}$ pour $\alpha$ r\'eelle et au groupe unipotent commutatif $\g{Add}^{mult(\alpha)}$ pour $\alpha$ imaginaire.

\par e) Si $\qa,\qb\in\Psi$, $\qa+\qb\in\Psi$ implique $\qa+\qb\in\Psi'$, alors $\g U^{ma}_{\Psi}/\g U^{ma}_{\Psi'}$ est commutatif. Le groupe $\g U^{ma}_{\Psi}/\g U^{ma}_{\Psi'}(R)$ est isomorphe au groupe additif $\prod_{\qa\in\Psi\setminus\Psi'}\,\g g_{R\qa}$.

\par f) Si $\Psi''\subset\Psi'$ est un autre sous-ensemble clos et si $\qa\in\Psi'$, $\qb\in\Psi$, $\qa+\qb\in\QD$ implique $\qa+\qb\in\Psi''$, alors $\g U^{ma}_{\Psi''}$ contient le groupe de commutateurs $(\g U^{ma}_{\Psi},\g U^{ma}_{\Psi'})$
\end{lemm}

\begin{proof} a) $\shu(\Psi')$ est une sous-big\`ebre co-inversible de $\shu(\Psi)$, donc $\g U^{ma}_{\Psi'}$ est un sous-groupe de $\g U^{ma}_{\Psi}$. Ce sous-groupe est ferm\'e, d\'efini par l'annulation des ind\'etermin\'ees $Z_x$ pour $x\in\shb_\Psi\setminus\shb_{\Psi'}$. Pour la derni\`ere assertion il suffit dans \ref{3.2} d'ordonner $\shb$ de fa\c{c}on que $\shb_{\Psi'}$ soit en dernier (\`a droite).

\par L'assertion b) r\'esulte aussit\^ot de la proposition \ref{3.2}

\par c) Pla\c{c}ons nous d'abord sur $\C$ (ou un anneau contenant $\Q$). On peut alors remplacer dans  \ref{3.2} les exponentielles tordues par les exponentielles habituelles. Des calculs classiques montrent alors que $\g U^{ma}_{\Psi'}(\C)$ est distingu\'e dans $\g U^{ma}_{\Psi}(\C)$. En fait il est clair que $\g U^{ma}_{\Psi}(\C)$ ou $\g U^{ma}_{\Psi'}(\C)$ est le groupe introduit dans \cite[6.1.1]{Kr-02} presque sous le m\^eme nom et on peut donc faire r\'ef\'erence \`a \cite[6.1.2]{Kr-02}. Le normalisateur de $\g U^{ma}_{\Psi'}$ dans $\g U^{ma}_{\Psi}$ est ferm\'e dans $\g U^{ma}_{\Psi}$ \cite[II {\S{}} 1 3.6b]{DG-70}. Comme $\Z[\g U^{ma}_\Psi]$ est une alg\`ebre de polyn\^omes et que l'on vient de voir que ce normalisateur est \'egal \`a $\g U^{ma}_\Psi$ sur $\C$, il en est de m\^eme sur $\Z$.

\par d) Ordonnons $\shb_\Psi$ de mani\`ere que $\shb_\alpha$ soit en premier. Alors, d'apr\`es \ref{2.6} et \ref{3.2}, un \'el\'ement de $\g U^{ma}_\Psi(R)$ s'\'ecrit de mani\`ere unique sous la forme
 $(\prod_{x\in\shb_\alpha}\,[\exp]\lambda_xx)z$ avec des $\lambda_x\in R$ et $z\in\g U^{ma}_{\Psi'}(R)$. Soit $\Theta=\{\,n\alpha\in\Delta\mid n\geq{}2\,\}\subset\Psi'$ (vide si $\alpha$ est r\'eelle); il reste \`a montrer que $(\prod_x\,[\exp]\lambda_xx).(\prod_x\,[\exp]\lambda'_xx)=\prod_x\,[\exp](\lambda_x+\lambda'_x)x$ (modulo $\g U^{ma}_\Theta(R)$).
 Mais on a vu dans la remarque \ref{2.9}.4 que, pour $x,y\in\shb_\alpha$, $x^{[n]},y^{[n]}\in\shu^\alpha\cap\shu_{n\qa}$ (ces \'el\'ements commutent donc modulo $\shu(\Theta)$), et que $x^{[n]}.x^{[m]}=\left(\begin{array}{c}n+m \\ n\end{array}\right)x^{[n+m]}$ modulo $\shu(\Theta)$. Le r\'esultat est alors clair.

 \par e) Il est clair que les \'el\'ements de $\widehat\shu^p_R(\Psi)$ commutent modulo l'id\'eal $\widehat\shu^p_R(\Psi')$. Vus \ref{3.2} et les bonnes propri\'et\'es de la base des $[N]$ (\ref{2.9}.4), $\g U^{ma}_{\Psi}(R)/\g U^{ma}_{\Psi'}(R)$ est commutatif isomorphe au produit direct des groupes $\g U^{ma}_{(\qa)}(R)/\g U^{ma}_{(\qa)\setminus\{\qa\}}(R)$ pour $\qa\in\Psi\setminus\Psi'$.

 \par f) Comme en c) on a le r\'esultat sur $\C$. Le centralisateur de $\g U^{ma}_{\Psi'}/\g U^{ma}_{\Psi''}$ dans $\g U^{ma}_{\Psi}/\g U^{ma}_{\Psi''}$ est donc \'egal \`a $\g U^{ma}_{\Psi}/\g U^{ma}_{\Psi''}$ sur $\C$. Mais ce centralisateur est ferm\'e dans $\g U^{ma}_{\Psi}/\g U^{ma}_{\Psi''}$ \cite[II {\S{}} 1 3.6c]{DG-70}. Comme $\Z[\g U^{ma}_{\Psi}/\g U^{ma}_{\Psi''}]$ est un anneau de polyn\^omes, ce centralisateur est \'egal \`a $\g U^{ma}_{\Psi}/\g U^{ma}_{\Psi''}$.
\end{proof}

\subsection{Filtrations de $\g U^{ma}_\Psi$}\label{3.4}

\par Soient $\Psi$ un sous-ensemble clos de $\Delta^+$ et $n$ le plus petit des degr\'es totaux ($\deg(\alpha)=\sum n_i$) des $\alpha=\sum\,n_i\alpha_i\in\Psi$. Choisissons $\alpha\in\Psi$ de degr\'e total $n$, alors $\Psi'=\Psi\setminus\{\alpha\}$ est un id\'eal de $\Psi$. Ainsi $\g U^{ma}_{\Psi'}$ est un sous-groupe distingu\'e de $\g U^{ma}_{\Psi}$ et le quotient $\g U^{ma}_{\Psi}/\g U^{ma}_{\Psi'}$ est isomorphe au groupe additif $\g{Add}$ ou \`a  $\g{Add}^{mult(\alpha)}$. Si $\alpha$ est r\'eelle on a un produit semi-direct.

\par Ce proc\'ed\'e fournit une num\'erotation de $\Psi$ par $\N$ (ou un intervalle $[0,N]$ de $\N$): $\Psi=\{\beta_i\}$. Pour  $n\in\N$, on note $\Psi_n=\{\,\beta_i\mid\,i\geq{}n\,\}$. D'apr\`es le lemme \ref{3.3} le groupe $\g U^{ma}_{\Psi_n}$ est distingu\'e dans $\g U^{ma}_{\Psi}$ et $\g U^{ma}_{\Psi_n}/\g U^{ma}_{\Psi_{n+1}}$ est isomorphe \`a $\g{Add}^{mult(\beta_n)}$. La proposition \ref{3.2} montre que $\g U^{ma}_{\Psi}$ est limite projective des $\g U^{ma}_{\Psi}/\g U^{ma}_{\Psi_n}$. Ainsi $\g U^{ma}_{\Psi}$ est un groupe pro-unipotent (ou m\^eme unipotent si $\Psi$ est fini).

\par Pour $d\in\N$, notons $\Psi(d)$ l'ensemble des racines de $\Psi$ de degr\'e total $\geq{}d$. Alors les $\g U^{ma}_{\Psi(d)}$ sont distingu\'es dans $\g U^{ma}_{\Psi}$ et $\g U^{ma}_{\Psi(d)}/\g U^{ma}_{\Psi(d+1)}(R)$ est isomorphe au groupe additif $\bigoplus_{\qa\in\Psi(d)\setminus\Psi(d+1)}\,\g g_{R\qa}$.
Le groupe $\g U^{ma}_{\Psi}(R)$ est limite projective des $\g U^{ma}_{\Psi}/\g U^{ma}_{\Psi(d)}(R)$.

\subsection{Groupes de Borel et paraboliques minimaux}\label{3.5}

\par Le {\it groupe de Borel} $\g B_\shs^{ma+}$ est par d\'efinition le produit semi-direct $\g T_\shs\ltimes\g U^{ma+}$ pour l'action suivante de $\g T_\shs$ sur $\g U^{ma+}$. Pour un anneau $R$, $\g T_\shs(R)$ agit sur $\shu_{\shs R}$ par des automorphismes de big\`ebre: sur $\shu_{\shs\alpha R}$, l'action $Ad(t)$ de $t\in\g T_\shs(R)$ est la multiplication par $\alpha(t)\in R^*$. On en d\'eduit aussit\^ot l'action sur $\g U^{ma+}(R)$: $Int(t).[\exp]\lambda x=[\exp]\alpha(t)\lambda x$ si $x\in \g g_{\alpha R}$. Il est clair que les groupes $\g U^{ma}_{\Psi}$ de \ref{3.3} sont stables par l'action de $\g T_\shs$.

\par Si $\alpha=\alpha_i$ est une racine simple on a $\g U^{ma+}_{}=\g U_{\alpha}\ltimes\g U^{ma}_{\Delta^+\setminus\{\alpha\}}$. On d\'efinit comme ci-dessus $\g U_{-\alpha}$ avec $\Z[\g U_{-\alpha}]$ dans le dual de $\shu_\C(\g g_{-\alpha})\cap\shu_\shs$ et isomorphe \`a $\g{Add}$ par $\g x_{-\alpha}:\g{Add}\rightarrow\g U_{-\alpha}$, $\g x_{-\alpha}(r)=\exp(rf_\alpha)$. O. Mathieu d\'efinit un sch\'ema en groupes affine  {\it parabolique minimal} $\g P_i^{pma}=\g P_\alpha^{pma}$ (contenant $\g B_\shs^{ma+}$ et associ\'e \`a $\Delta^+\cup\{0,-\alpha\}$) et un sous-groupe ferm\'e $\g A_\alpha^Y=\g A_i^Y$ (associ\'e \`a $\{0,\pm\alpha\}$); il montre que $\g P_\alpha^{pma}=\g A_\alpha^Y\ltimes\g U^{ma}_{\Delta^+\setminus\{\alpha\}}$ \cite[lemme 8 p26]{M-88a}.

\par Le sch\'ema en groupes $\g A_\alpha^Y$ contient les sous-groupes ferm\'es $\g T_\shs$, $\g U_{\alpha}$ et $\g U_{-\alpha}$, plus pr\'ecis\'ement $\Z[\g A_\alpha^Y]$ est contenu dans le dual de $\shu_\shs\cap\shu_\C(\g g_\alpha\otimes\g h\otimes\g g_{-\alpha})$ et sa restriction \`a $\shu_\shs\cap\shu_\C(\g g_{\pm\alpha})$ (resp. $\shu_\shs\cap\shu_\C(\g h)$) est $\Z[\g U_{\pm\alpha}]$ (resp. $\Z[\g T_\shs]$). Il est clair (\cf \eg [\lc lemme 7.2]) que $\g A_\alpha^Y$ est le groupe r\'eductif de SGR $(\,(2),Y,\alpha,\alpha^\vee)$ il agit sur $\g g_\Z$ et est isog\`ene au produit de $SL_2$ par un tore d\'eploy\'e.

\par L'\'el\'ement $\widetilde s_i=\g x_{\alpha}(1).\g x_{-\alpha}(1).\g x_{\alpha}(1)$ de $\g A_\alpha^Y(\Z)$ normalise $\g T_\shs$, v\'erifie $\widetilde s_i^2=\alpha^\vee(-1)\in\g T_\shs(\Z)$ et \'echange $\g x_{\alpha}$, $\g x_{-\alpha}$; il agit sur $\g g_\Z$ comme $s_i^*$ (\cf \ref{1.4}). Enfin la double classe (grosse cellule) $\g C_i=\g B_\shs^{ma+}.\widetilde s_i.\g B^{ma+}_\shs$ de $\g P_i^{pma}$ est un ouvert dense et se d\'ecompose de mani\`ere unique sous la forme $\g C_i=\g U_{\alpha_i}.\widetilde s_i.\g B_\shs^{ma+}=\widetilde s_i.\g U_{-\alpha_i}.\g B_\shs^{ma+}$.

\begin{rema*}  Mathieu se place dans le cas d'un SGR $\shs$ libre, colibre, sans cotorsion et de dimension $2\vert I\vert-rang(A)$. Bien s\^ur ces hypoth\`eses sont inutiles pour la plupart de ses r\'esultats (c'est particuli\`erement clair pour la derni\`ere). On se placera cependant dans ce cadre (depuis l'alin\'ea pr\'ec\'edent) jusqu'en \ref{3.18}, 
 o\`u on examinera la g\'en\'eralisation gr\^ace aux r\'esultats de \ref{1.3}.
\end{rema*}

\par On abandonne en g\'en\'eral dans la suite (et jusqu'en \ref{3.17b}) l'indice $\shs$ et on abr\`ege parfois $\g B^{ma+}$ en $\g B$, $\g P_i^{pma}$ en $\g P_i$, $\g A_i^Y$ en $\g A_i$, etc.

\subsection{La construction de Mathieu}\label{3.6} \cf \cite[XVIII {\S{}} 2]{M-88a}, \cite{M-88b}, \cite[I et II]{M-89}

\par On consid\`ere des sch\'emas sur $\Z$, m\^eme si on n'aura essentiellement besoin de raisonner que sur des (sous-anneaux de) corps, \'eventuellement de caract\'eristique positive.

\par Soient $w\in W^v$ et $\widetilde w=s_{i_1}.\cdots.s_{i_n}$ une d\'ecomposition r\'eduite de $w$. On consid\`ere le sch\'ema $\g E(\widetilde w)=\g P_{i_1}\times^{\g B}\g P_{i_2}\times^{\g B}\cdots\times^{\g B}\g P_{i_n}$ et le sch\'ema de Demazure $\g D(\widetilde w)=\g E(\widetilde w)/\g B$. 

\par Le sch\'ema $\g B(w)$ est l'affinis\'e de $\g E(\widetilde w)$, c'est-\`a-dire $Spec(\Z[\g E(\widetilde w)])$, il ne d\'epend que de $w$  \cite[p 40, l -15 \`a -1]{M-89}. En particulier $\g B(s_i)=\g P_i$.

\par Si $w'\leq{}w$ (pour l'ordre de Bruhat-Chevalley), il existe une d\'ecomposition r\'eduite $\widetilde w'$ de $w'$ extraite de $\widetilde w$. Si dans l'\'ecriture de $\g E(\widetilde w)$, on remplace par $\g B$ les $\g P_{i_j}$ correspondant \`a des $s_{i_j}$ absents de $\widetilde w'$, on obtient un sous-sch\'ema ferm\'e de $\g E(\widetilde w)$ clairement isomorphe \`a $\g E(\widetilde w')$. L'immersion ferm\'ee $\g E(\widetilde w')\rightarrow \g E(\widetilde w)$ induit un morphisme $\g B(w')\rightarrow \g B(w)$ entre les affinis\'es; c'est aussi une immersion ferm\'ee, ind\'ependante des choix de $\widetilde w$ et $\widetilde w'$ \cite[bas de page 255]{M-88a}.

\par On a donc un syst\`eme inductif de sch\'emas $\g B(w)$ et on note $\g G^{pma}$ le ind-sch\'ema limite inductive de ce syst\`eme; c'est le {\it groupe de Kac-Moody positivement maximal} associ\'e \`a $\shs$. Chaque morphisme $\g B(w)\rightarrow\g G^{pma}$ est une immersion ferm\'ee. On verra essentiellement $\g G^{pma}$ comme un foncteur sur la cat\'egorie des anneaux:  $\g G^{pma}(R)= \underrightarrow{lim} \,  \g B(w)(R)$.

\par En fait $\g G^{pma}$ est un ind-sch\'ema en groupes : pour des d\'ecompositions r\'eduites $\widetilde w=s_{i_1}.\cdots.s_{i_n}$ et $\widetilde w'=s_{j_1}.\cdots.s_{j_m}$, il existe un morphisme naturel de $\g P_{i_1}\times\cdots\times\g P_{i_n}\times\g P_{j_1}\times\cdots\times\g P_{j_m}$ dans $\g B(\psi(w,w'))$ pour un certain $\psi(w,w')\leq{}w.w'$ \cite[p 41 et 45]{M-89}. Ce morphisme se factorise donc par $\g B(w)\times\g B(w')$. Ceci permet de d\'efinir la multiplication dans $\g G^{pma}$. La d\'efinition de l'inversion est claire. Pour cette structure de groupe le compos\'e $\g P_{i_1}\times\cdots\times\g P_{i_n}\rightarrow\g E(\widetilde w)\rightarrow\g B(w)\hookrightarrow\g G^{pma}$ est simplement la multiplication.

\subsection{Modules \`a plus haut poids et repr\'esentation adjointe}\label{3.7}

\par 1) Pour $\lambda\in X^+$, on a d\'efini en \ref{2.14} un $\shu_\shs-$module \`a plus haut poids $L_\Z(\lambda)$ qui est int\'egrable, donc un $\g T_\shs-$module. Pour un anneau $k$ le module $L_k(\lambda)$ est stable par  $\widehat\shu^{p}_{\shs k}$ donc par $\g U^{ma+}(k)$. On obtient ainsi une repr\'esentation de $\g B^{ma+}$ dans $L_\Z(\lambda)$. L'int\'egrabilit\'e de $L_\Z(\lambda)$ \cite{T-81b} montre en fait que $L_\Z(\lambda)$ est un $\g A_i-$module donc aussi un $\g P_i-$module ($\forall i\in I$). Ce module est "localement fini", plus pr\'ecis\'ement pour tout
sous$-\Z-$module $M$ de rang fini de $L_\Z(\lambda)$, il existe un sous$-\Z-$module facteur direct  de rang fini $M'$ de $L_\Z(\lambda)$ contenant $M$, tel que $M'\otimes k$ soit un $\g P_i(k)-$module pour tout anneau $k$. On en d\'eduit aussit\^ot un morphisme de $\g B(w)$ dans $\g{GL}(L_\Z(\lambda))$, $\forall w\in W$. On obtient ainsi une repr\'esentation de $\g G^{pma}$ dans $\g{GL}(L_\Z(\lambda))$.

\par 2) Soit $M=L_\Z(\lambda)$ le module \`a plus haut poids pr\'ec\'edent: $M=\oplus_{\nu\in Q^+}\, M_{\lambda-\nu}$. Pour un anneau $k$, on d\'efinit le compl\'et\'e n\'egatif $\widehat M^n_k$ de $M\otimes k$ : $\widehat M^n_k=\prod_{\nu\in Q^+}\, M_{\lambda-\nu}\otimes k$. Il est clair que $\widehat M^n_k$ est un $\widehat\shu^n_k-$module gradu\'e; en particulier c'est un $\g U^{ma-}(k)-$module.

\par 3) Pour un anneau $k$, l'action adjointe $ad$ de l'alg\`ebre  $\widehat\shu^p_k$ sur elle-m\^eme (\ref{2.1}.5) induit une action, dite adjointe et not\'ee $Ad$, du groupe $\g U^{ma+}(k)\subset \widehat\shu^p_k$ sur $\widehat\shu^p_k$. L'action de $\g T_\shs(k)$ est d\'efinie par les poids, d'o\`u la repr\'esentation adjointe $Ad$ de $\g B^{ma+}(k)$ sur $\widehat\shu^p_k$. Cette action stabilise \'evidemment  $\widehat{\g g}^p_k$ et $\widehat\shu^+_k$.

\par Pour $\alpha\in\Phi$, $r\in k$ et $u\in\widehat\shu^p_k$, $Ad(\g x_\alpha(r))(u)= \sum_{n\geq{}0}\,(ad\,e_\alpha^{(n)}).r^nu= \sum_{p,q\geq{}0}\,e_\alpha^{(p)}.r^nu.e_\alpha^{(p)}$.

\subsection{Le cas classique}\label{3.8}

\par Supposons la matrice $A$ de Cartan \ie $\Delta=\Phi$ fini et $W^v$ fini. Soit $\g G$ le groupe r\'eductif d\'eploy\'e sur $\Z$ de SGR $\shs$; d'apr\`es l'hypoth\`ese de Mathieu il est semi-simple simplement connexe. Le groupe $\g B^{ma+}_\shs=\g T_\shs\ltimes \g U^{ma+}$ est le sous-groupe de Borel $\g B$ de $\g G$ (\cf \ref{3.2}.1) et $\g P_{\alpha_i}^{pma}=\g A_{\alpha_i}^Y\ltimes\g U^{ma}_{\Delta^+\setminus\{\alpha_i\}}$ est un sous-groupe parabolique minimal de $\g G$.

\par Pour une d\'ecomposition r\'eduite $\widetilde w=s_{i_1}.\cdots.s_{i_n}$ dans $W^v$, le produit d\'efinit un morphisme $\pi\,:\,\g E(\widetilde w)=\g P_{\alpha_{i_1}}\times^{\g B}\cdots\times^{\g B}\g P_{\alpha_{i_n}}\rightarrow \g G$. Si on restreint $\pi$ aux grosses cellules, on obtient un isomorphisme de $\Omega(\widetilde w)=\g C_{{i_1}}\times^{\g B}\cdots\times^{\g B}\g C _{{i_n}}$ sur son image $\g U_{\alpha_{i_1}}.\widetilde s_{{i_1}}.\cdots.\g U_{\alpha_{i_n}}.\widetilde s_{{i_n}}.\g B=\widetilde s_{{i_1}}.\cdots.\widetilde s_{{i_n}}.\g U_{-s_{{i_n}}.\cdots.s_{{i_2}}(\alpha_{i_1})}.\cdots.\g U_{-s_{{i_n}}(\alpha_{i_{n-1}})}.\g U_{-\alpha_{i_n}}.\g B$.

\par Si $w=w_0$ est l'\'el\'ement de plus grande longueur de $W^v$, cette image est un ouvert dense: la grosse cellule $\g C(\g G)$ de $\g G$. On a donc la suite de morphismes:

\qquad\qquad\qquad\qquad\qquad\qquad \xymatrix{ \Omega(\widetilde w_0)\ar@{^{(}->}[r]^i&\g E(\widetilde w_0)\ar[r]^\pi&\g G\\}

\parni avec $i$ et $\pi\circ i$ des immersions ouvertes. De plus l'image $\g C(\g G)$ de $\pi\circ i$ rencontre toutes les fibres $\g G_{\F_p}$ de $\g G$ pour $p$ premier.

\begin{prop}\label{3.9} Dans ces conditions $\pi$ identifie  $\Z[\g G]$ \`a $\Z[ \g E(\widetilde w_0)]$. Ainsi $\g G$ s'identifie \`a $\g B(w_0)$ lui m\^eme \'egal \`a $\g G^{pma}$.
\end{prop}

\par{\bf N.B.} Si $w\not=w_0$, on trouve de la m\^eme mani\`ere que $\g B(w)$ est la sous-vari\'et\'e de Schubert de $\g G$ associ\'ee \`a $w$ (au moins sur $\C$).

\begin{proof} Regardons d'abord sur $\C$. Comme $\g G$ est un groupe affine, $\pi$ se factorise par $\g B(\widetilde w_0)$. On a donc la suite de morphismes:

\qquad \xymatrix{ \Omega(\widetilde w_0)_\C\ar@{^{(}->}[r]^i&\g E(\widetilde w_0)_\C\ar@/_1pc/[rr]_\pi\ar[r]^(.4){\Aff}&\g B(\widetilde w_0)_\C=\g G_\C^{pma}\ar[r]^(.7){\tilde\pi}&\g G_\C\\}

\par\noindent Ce compos\'e est une immersion ouverte d'image dense. Pour les alg\`ebres de fonctions, comme $i$ est d'image dense $i^*$ est injectif et $\Aff^*$ est l'identit\'e par d\'efinition. Donc $(\Aff\circ i)^*$ est injectif et $\Aff\circ i$ est dominant. Ainsi l'image de $\Aff\circ i$ contient un ouvert dense $\Omega'$ de $\g B(\widetilde w_0)(\C)$. Soit $\gamma\in{\mathrm Ker}\widetilde \pi$, $\gamma\Omega'\cap\Omega'\not=\emptyset$, il existe donc $g',g''\in\Omega'$ tels que $\gamma g'=g''$, ainsi $\widetilde \pi(g')=\widetilde \pi(g'')$. Mais $\widetilde\pi\circ{\Aff}\circ i$ est injectif, on en d\'eduit que $g'=g''$, $\gamma=1$ \ie $\widetilde \pi$ est injectif. Ainsi $\g G^{pma}$ est un sous-groupe de $\g G$ contenant un ouvert dense: on a $\g G^{pma}=\g G$.

\par On a le diagramme:\quad
\xymatrix{
\Z[ \Omega(\widetilde w_0)]\ar@{^{(}->}[d] &\; \Z[ \g E(\widetilde w_0)] \ar@{^{(}->}[d] \ar@{_{(}->}[l]_{i^*} & \;\Z[\g G]\ar@{^{(}->}[d] \ar@{_{(}->}[l]_(.4){\pi^*} \cr
\C[ \Omega(\widetilde w_0)] &\; \C[ \g E(\widetilde w_0)] \ar@{_{(}->}[l]_{i^*}& \; \C[\g G]\ar@{_{(}->}[l]_(.4){=} \cr
}


\par Les fl\`eches horizontales sont injectives  car $i^*$ et $(\pi\circ i)^*$ le sont. Le sch\'ema $ \Omega(\widetilde w_0)$ est produit de groupes additifs et multiplicatifs, la fl\`eche verticale de gauche est donc injective. Ainsi toutes les fl\`eches sont injectives.

\par Pour montrer que $\Z[ \g E(\widetilde w_0)] =\pi^* \Z[\g G]$ il suffit donc de montrer que $\Z[ \Omega(\widetilde w_0)] \cap \C[\g G]=\Z[\g G]$. Ceci est clairement vrai si on remplace $\Z$ par $\Q$. Donc si ce n'est pas vrai pour $\Z$, il existe $n\geq{}2$ et $\varphi\in\Z[ \Omega(\widetilde w_0)] \cap \Q[\g G]$ tels que $\varphi\notin\Z[\g G]$ mais $n\varphi\in\Z[\g G]$. Soit $p$ un diviseur premier de $n$; quitte \`a modifier $\varphi$ on peut supposer $n=p$. Mais alors $\overline{p\varphi}$ est une fonction non identiquement nulle sur $\g G_{\F_p}$ alors qu'elle est nulle sur $ \Omega(\widetilde w_0)_{\F_p}$. C'est absurde car $ \Omega(\widetilde w_0)_{\F_p}$ est un ouvert dense de $\g G_{\F_p}$.
\end{proof}

\subsection{Application au cas non classique}\label{3.10}

\par Soit $J\subset I$. On note $\Delta(J)=\Delta\cap(\oplus_{i\in J}\Z\alpha_i)$, $\Delta^\pm(J)=\Delta(J)\cap\Delta^\pm$, $\Delta^\pm_J=\Delta^\pm\setminus\Delta^\pm(J)$, $\g U^{ma+}(J)=\g U^{ma}_{\Delta^+(J)}$ et $\g U^{ma+}_J=\g U^{ma}_{\Delta^+_J}$. D'apr\`es \ref{3.3} on a $\g U^+=\g U^{ma+}(J)\ltimes\g U^{ma+}_J$.

\par Supposons $J$ de type fini et notons $\g G(J)$ le groupe r\'eductif d\'eploy\'e sur $\Z$ de SGR $\shs(J)=(A\vert_J,Y,(\alpha_i)_{i\in J},(\alpha^\vee_i)_{i\in J})$. Il est clair que son groupe unipotent maximal positif est $\g U^{ma+}(J)=:\g U^{+}(J)$ et son sous-groupe de Borel positif $ \g B^{+}(J)=\g T_Y\ltimes \g U^{+}(J)$. Enfin, pour $j\in J$, son parabolique minimal associ\'e \`a $j$ est $\g P_j(J)=\g A_j\ltimes\g U^{ma}_{\Delta^+(J)\setminus\{j\}}$, il v\'erifie $\g P_j=\g P_j(J)\ltimes\g U^{ma+}_J$.

\par Le groupe de Weyl (fini) $W^v(J)=\langle s_i\mid i\in J\rangle\subset W^v$ admet un  \'el\'ement de plus grande longueur $w_0^J$.

\begin{coro*} $\g B(w_0^J)$ est isomorphe au sous-groupe parabolique $\g P(J)=\g G(J)\ltimes\g U^{ma+}_J$ de $\g G^{pma}$.
\end{coro*}

\begin{rema*} Si $J$ n'est pas forc\'ement de type fini, la d\'emonstration ci-dessous prouve que le ind-sch\'ema en groupes limite inductive des $\g B(w)$ pour $w\in W^v(J)$ est le sous-groupe parabolique $\g P(J)=\g G^{pma}(J)\ltimes\g U^{ma+}_J$ produit semi-direct par $\g U^{ma+}_J$ du groupe de Kac-Moody $\g G^{pma}(J)=\g G^{pma}_{\shs(J)}$.
\end{rema*}

\begin{proof} On a $\g E(\widetilde w_0)=\g P_{j_1}\times^{\g B^{ma+}}\g P_{j_2}\times^{\g B^{ma+}}\cdots\times^{\g B^{ma+}}\g P_{j_n}$, $\g P_{j_k}=\g P_{j_k}(J)\ltimes\g U^{ma+}_J$ et $\g B^{ma+}= \g B^{+}(J)\ltimes \g U^{ma+}_J$. On en d\'eduit facilement que $\g E(\widetilde w_0)=\g P_{j_1}(J)\times^{\g B^{+}(J)}\g P_{j_2}(J)\times^{\g B^{+}(J)}\cdots\times^{\g B^{+}(J)}\g P_{j_n}(J)\ltimes \g U^{ma+}_J$, et, comme l'affinis\'e d'un produit est le produit des affinis\'es, $\g B(w_0^J)=\g G(J)\ltimes\g U^{ma+}_J$.
\end{proof}

\subsection{Sous-groupes radiciels n\'egatifs}\label{3.11}

\par On a d\'efini $\g U_\alpha$ et $\g x_{\alpha}:\g{Add}\rightarrow\g U_\alpha$ pour $\alpha\in\Phi^+$ (en \ref{3.2}.2) ou pour $-\alpha=\alpha_i$ simple (en \ref{3.5}).

\par Pour $i\in I$, $\widetilde s_i=\g x_{\alpha_i}(1).\g x_{-\alpha_i}(1).\g x_{\alpha_i}(1)\in\g A_{\alpha_i}^Y(\Z)\subset\g P_{\alpha_i}^{pma}(\Z)\subset\g G^{pma}(\Z)$ est  un \'el\'ement du normalisateur dans $\g G^{pma}$ de $\g T_Y$; il induit sur $\g T_Y$ la r\'eflexion $s_i$. De plus $\widetilde s_i$ agit sur $\g g_\Z$ comme $s_i^*$ (\cf \ref{1.4} et \ref{3.7}.3).

\par Si $\alpha\in\Phi$ et $e_\alpha$ est une base de $\g g_{\alpha\Z}$, il existe $w\in W^v$ tel que $w^{-1}\alpha=\alpha_j$ est une racine simple. On consid\`ere une d\'ecomposition  $w=s_{i_1}.\cdots.s_{i_n}$, $\overline w=\widetilde s_{i_1}.\cdots.\widetilde s_{i_n}$ et on a $e_\alpha=\overline w(\epsilon e_j)$ avec $\epsilon=\pm1$. On pose alors $\g U_\alpha=\overline w.\g U_{\alpha_j}.\overline w^{-1}$ et $\g x_\alpha(r)=\overline w.\g x_{\alpha_j}(\epsilon r).\overline w^{-1}$, pour $r$ dans un anneau $R$.

\begin{lemm*}Ces d\'efinitions de $\g U_\alpha$ et $\g x_\alpha$ ne d\'ependent que de $\alpha$ et $e_\alpha$ (et non des autres choix effectu\'es). Pour $\alpha\in\Phi^+$ (ou $-\alpha$ simple) elles co\"{\i}ncident avec celles de \ref{3.1},  \ref{3.2} ou  \ref{3.5}.
\end{lemm*}

\begin{proof} Supposons que $e_\alpha=\widetilde s_{i_1}.\cdots.\widetilde s_{i_n}(\epsilon e_{\alpha_j})=\widetilde s_{i'_1}.\cdots.\widetilde s_{i'_m}(\epsilon' e_{\alpha_{j'}})$, donc $e_{\alpha_{j'}}=\overline w(\epsilon\epsilon'e_{\alpha_{j}})$ pour $\overline w=\widetilde s_{i'_m}^{-1}.\cdots.\widetilde s_{i'_1}^{-1}.\widetilde s_{i_1}.\cdots.\widetilde s_{i_n}$. On doit montrer que $\g x_{\alpha_j}(r)$ est conjugu\'e de $\g x_{\alpha_{j'}}(\epsilon\epsilon'r)$ par $\overline w$. Mais on sait que les $s^*_i=\widetilde s_i$ v\'erifient les relations de tresse (\cite[(d) p551]{T-87b} ou calcul classique dans le groupe r\'eductif $\g G(\{i,j\})$ quand $s_is_j$ est d'ordre fini), que $\widetilde s^2_i=\alpha_i^\vee(-1)\in\g T_Y(\Z)$ et on conna\^{\i}t les relations de commutation entre les $\widetilde s_i$ et les \'el\'ements du tore. Ainsi, pour toute expression r\'eduite $w=s_{j_1}.\cdots.s_{j_p}$ de $ w= s_{i'_m}.\cdots. s_{i'_1}. s_{i_1}.\cdots. s_{i_n}$, on peut r\'eduire l'expression ci-dessus de $\overline w$ sous la forme $\overline w=\widetilde s_{j_1}.\cdots.\widetilde s_{j_p}.t$ avec $t\in\g T_Y(\Z)$.

\par Il suffit donc de montrer que si $w(\alpha_j)=\alpha_{j'}$, il existe une d\'ecomposition r\'eduite $w=s_{j_1}.\cdots.s_{j_p}$ de $w$  pour laquelle un $\overline w$ comme ci-dessus v\'erifie  la relation $e_{\alpha_{j'}}=\overline w\epsilon''e_{\alpha_j}$ et $\g x_{\alpha_{j'}}(r)=\overline w.\g x_{\alpha_j}(\epsilon''r).\overline w^{-1}$. D'apr\`es \cite[3.3.1]{T-87b} ou, pour un \'enonc\'e plus pr\'ecis \cite[2.17 p85]{He-91}, on est ramen\'e \`a l'un des deux cas suivants: $j_1,\cdots,j_p\in\{j,j'\}=J$ de type fini ou $j=j'$ et $j_1,\cdots,j_p\in\{j,j''\}=J$ de type fini. Tout se passe alors dans le groupe r\'eductif $\g G(J)$ et le r\'esultat est connu.

\par Si $-\alpha$ est simple, la co\"{\i}ncidence des deux d\'efinitions de $\g x_\alpha$ et $\g U_\alpha$ est claire. Si $\alpha\in\Phi^+$, il existe une suite $i_1,\cdots,i_n\in I$ telle que $s_{i_n}.\cdots.s_{i_1}\alpha$ est une racine simple et $\forall j$ $s_{i_j}.\cdots.s_{i_1}\alpha\in\Phi^+$. Pour v\'erifier la co\"{\i}ncidence des deux d\'efinitions de $\g x_\alpha$ et $\g U_\alpha$, on doit v\'erifier que, si $i\in I$, $\alpha\in\Phi^+$, $\beta=s_i\alpha\in\Phi^+$ et $e_\beta=\widetilde s_i\epsilon e_\alpha$, alors $\exp(re_\beta)=\widetilde s_i.\exp(\epsilon re_\alpha).\widetilde s_i^{-1}$; c'est la d\'efinition pr\'ecise du produit semi-direct $\g P_i^{pma}=\g A_i^Y\ltimes\g U^{ma}_{\Delta^+\setminus\{\alpha_i\}}$ de \ref{3.5}.
\end{proof}

\subsection{Comparaison avec le groupe \`a la Tits}\label{3.12}

\par On a d\'efini en \ref{1.6} le groupe de Kac-Moody minimal $\g G=\g G_\shs$. On choisit toujours $\shs$ comme \`a la remarque \ref{3.5} et on prend les $e_\alpha$, $f_\alpha$ de \ref{1.4} dans la base $\shb$ de \ref{2.6}.

\par Pour un anneau $k$, les g\'en\'erateurs de $\g G(k)$ sont les $\g x_\alpha(r)$, $\alpha\in\Phi$, $r\in k$ et $t\in\g T_Y(k)$, on les envoie sur les \'el\'ements de m\^eme nom dans $\g G^{pma}(k)$.

\begin{prop*} On obtient ainsi un homomorphisme de foncteurs en groupes $i:\g G\rightarrow\g G^{pma}$.
\end{prop*}

\begin{proof} Il s'agit de v\'erifier les relations (KMT3) \`a (KMT7) dans $\g G^{pma}(k)$. La relation (KMT3) esp\'er\'ee entre $\g x_\alpha(r)$ et $\g x_\beta(r')$ provient de la relation constat\'ee entre $\exp(ad\,re\alpha)$ et $\exp(ad\,r'e_\beta)$ dans Aut$(\g g_A)$. Comme $\{\alpha,\beta\}$ est une paire pr\'enilpotente, il existe $\Psi$ finie close dans $\Phi$ contenant $\alpha$ et $\beta$. D'apr\`es \ref{3.11} on peut supposer $\Psi\subset\Phi^+$. D'apr\`es la remarque \ref{3.2}.1 cette relation est en fait constat\'ee dans $\g U_\Psi^{ma}$, d'o\`u le r\'esultat.

\par Les relations (KMT4) \`a (KMT6) n'impliquent que des \'el\'ements du groupe r\'eductif $\g A_\alpha^Y$ (qui est de rang $1$), elles sont trivialement satisfaites.

\par Enfin (KMT7) est clair d'apr\`es \ref{3.11}.
\end{proof}

\begin{remas*}
\par 1) L'homomorphisme $i$ envoie aussi $\widetilde s_\alpha$ sur l'\'el\'ement de m\^eme nom de $\g G^{pma}(\Z)$ (pour $\alpha$ racine simple). L'image $i(\g N)$ de $\g N$ est donc engendr\'ee par $\g T_\shs$ et les $\widetilde s_\alpha$ pour $\alpha$ racine simple. Par construction $i$ est un isomorphisme sur $\g T_\shs$ et $\widetilde s_{i_1}.\cdots.\widetilde s_{i_n}\not=1$ si la d\'ecomposition $ s_{i_1}.\cdots. s_{i_n}$ est r\'eduite (second alin\'ea de \ref{3.8}). Donc $i$ est un isomorphisme de $\g N$ sur $i(\g N)$ (que l'on note encore $\g N$).

\par 2) En conjuguant par un \'el\'ement $n$ de $\g N(\Z)$ tel que $^v\nu(n)=w\in W^v$, on trouve dans $\g G^{pma}$ un sous-groupe analogue \`a $\g U^{ma+}$ mais associ\'e \`a $w\Delta^+$. Ainsi $\g G^{pma}$ ne d\'epend pas du choix de $\Delta^+$ dans sa classe de $W^v-$conjugaison. Par contre l'alg\`ebre $\widehat\shu^p$ d\'epend a priori de ce choix et si on change $\Delta^+$ en $\Delta^-$, le groupe obtenu $\g G^{nma}$ est diff\'erent.

\par 3) L'homomorphisme $i$ induit un isomorphisme  du groupe $\g U_\alpha$ sur le sous-groupe de m\^eme nom de $\g G^{pma}$: c'est clair par d\'efinition pour $\alpha\in\Phi^+$, on s'y ram\`ene pour $\alpha\in\Phi^-$ gr\^ace \`a \ref{3.11}.
\end{remas*}

\begin{prop}\label{3.13} Si $k$ est un corps, alors $i_k$ est injectif.
\end{prop}

\begin{NB} Si $k$ est un corps, on notera souvent avec la lettre romaine correspondante l'ensemble des points sur $k$ d'un sch\'ema. La proposition nous permet donc d'identifier $G=\g G(k)$ \`a un sous-groupe de $G^{pma}=\g G^{pma}(k)$.
\end{NB}

\begin{proof} D'apr\`es \cite[prop. 8.4.1]{Ry-02a} $G$ est muni d'une donn\'ee radicielle jumel\'ee enti\`ere et donc d'un syst\`eme de Tits $(G,B^+,N,S)$ avec $B^+\cap N=T$. Alors $B^+.{\mathrm Ker}(i)$ est un sous-groupe parabolique de $G$, mais par construction il ne peut contenir aucun $U_{-\alpha}$ pour $\alpha$ simple, donc $B^+.{\mathrm Ker}(i)=B^+$ et Ker$(i)$ est contenu dans $B^+$ et tous ses conjugu\'es. D'apr\`es \cite[1.5.4 et 1.2.3.v]{Ry-02a}, l'intersection de tous ces conjugu\'es est r\'eduite \`a $T$. Mais par construction $i\vert_T$ est injectif, d'o\`u le r\'esultat.
\end{proof}

\subsection{Comparaison des repr\'esentations}\label{3.14}

\par 1) On a vu en \ref{3.7}.3 que, pour un anneau $k$, l'alg\`ebre $\widehat\shu^p_k$ est un $\g B^{ma+}(k)-$module. D'apr\`es \ref{2.1}.4 la restriction \`a $\g B^+(k)$ de cette repr\'esentation stabilise $\shu_k$ (et $\g g_k$) et y induit la m\^eme repr\'esentation que la repr\'esentation adjointe de $\g G(k)$.

\par 2) Pour $\lambda\in X^+$, on a construit en \ref{3.7}.1 le module \`a plus haut poids $M=L(\lambda)$ pour $\g G^{pma}$. Il est clair que la restriction \`a $\g G$ (via $i$) de cette repr\'esentation est la repr\'esentation classique \`a plus haut poids de $\g G$.

\par 3) En  \ref{3.7}.2 on a construit une repr\'esentation de $\g U^{ma-}$ sur le compl\'et\'e n\'egatif $\widehat M^n_k$ de $M$. Il est clair que la restriction \`a $\g U^-$ (via $i$) de cette repr\'esentation est le compl\'et\'e de la repr\'esentation de $\g G$ ci-dessus (restreinte \`a $\g U^-$).

\begin{prop}\label{3.15} Soit $\widetilde w=s_{i_1}.\cdots.s_{i_m}$ une d\'ecomposition r\'eduite de  $w\in W^v$ et $\mu : \g F(\widetilde w):=\g P_{\alpha_{i_1}}\times\cdots\times\g P_{\alpha_{i_m}}\rightarrow \g B(w)$ le morphisme de multiplication. Si $k$ est un corps, alors $\mu_k$ est surjectif.
\end{prop}

\begin{proof} Il faut revenir un peu sur les d\'efinitions de \ref{3.6}. Le groupe $(\g B^{ma+})^{m-1}$ agit sur $\g F(\widetilde w)$ par $(b_1,\cdots,b_{m-1}).(p_1,\cdots,p_m)=(p_1b_1^{-1},b_1p_2b_2^{-1},\cdots,b_{m-1}p_m)$ et $\g E(\widetilde w)$ est le quotient sch\'ematique; on note $\pi$ le morphisme de passage au quotient. Le sch\'ema $\g B(w)$ est l'affinis\'e de $\g E(\widetilde w)$, on note $\nu :\g E(\widetilde w)\rightarrow\g B(w)$ le morphisme naturel. On a donc $\mu=\nu\circ\pi
$. On se place maintenant sur le corps $k$ (sans forc\'ement le noter dans les noms).

\par Soit $\lambda\in X^+$ un poids dominant  entier r\'egulier et $L(\lambda)$  la repr\'esentation int\'egrable de plus haut poids $\lambda$ (\cf \cite[p246]{M-88a} ou \ref{3.7} ci-dessus); les groupes $\g B^{ma+}$ et $\g P_i$ agissent dessus. On choisit un vecteur de plus haut poids $e_\lambda$; pour $w\in W^v$, $e_{w\lambda}=we_\lambda\in L(\lambda)_{w\lambda}$ est d\'efini au signe pr\`es et le sous-espace vectoriel $E_w(\lambda)=\shu^+_k.e_{w\lambda}$ est de dimension finie. De plus $E_v(\lambda)\subset E_w(\lambda)$ si $v\leq{}w$ pour l'ordre de Bruhat-Chevalley, mais si $v<w$, $E_v(\lambda)$ a une composante nulle sur $ke_{w\lambda}$.
On note $\Sigma_w$ l'adh\'erence de $\g B^{ma+}(k).ke_{w\lambda}$ dans $E_w(\lambda)$; c'est une sous-vari\'et\'e affine ferm\'ee de $E_w(\lambda)$. On a $\Sigma_w=\g B(w)(k).ke_\lambda$
il contient donc $\Sigma_v$ si $v\leq{}w$ (\cf \cite[p64]{M-88a} formul\'e en caract\'eristique $0$, mais qui se g\'en\'eralise). Ainsi $\Sigma^c_w:= \Sigma_w\setminus \bigcup_{v<w}\Sigma_v$,
contient l'ouvert de $\Sigma_w$ d\'efini par la non nullit\'e de la coordonn\'ee sur $e_{w\lambda}$.

\par On a un morphisme $\zeta$ de $\g F(\widetilde w)$ dans $\Sigma_w$ donn\'e par la "formule na\"{\i}ve" $\zeta(p_1,\cdots,p_m)=p_1.p_2.\cdots.p_m.e_\lambda$, il se factorise en un morphisme $\eta$ de $\g E(\widetilde w)$ dans $\Sigma_w$ (\cf \cite[p65, 66]{M-88a} en caract\'eristique $0$) puis, comme $\Sigma_w$ est affine et $\nu$ une affinisation, en un morphisme $\xi:\g B(w)\rightarrow\Sigma_w$, simplement donn\'e par l'action de $\g B(w)\subset\g G^{pma}$ sur $e_\lambda$ \cf \ref{3.7}.1. On a donc le diagramme:

\qquad\qquad \xymatrix{\zeta\,:\; \g F(\widetilde w)\ar[r]^(.6)\pi\ar@/_1pc/[rr]_\mu&\g E(\widetilde w)\ar[r]^(.5){\nu}\ar@/^1pc/[rr]^\eta&\g B( w)\ar[r]_(.5){\xi}&\Sigma_w\\}

\par Si $v\leq{}w$, il existe une d\'ecomposition r\'eduite $\widetilde v$ de $v$ "extraite" de $\widetilde w$. Le sous-sch\'ema ferm\'e de $\g F(\widetilde w)$ form\'e des \'el\'ements dont le j-\`eme facteur est $1$ si $s_{i_j}$ est omis dans $\widetilde v$ est isomorphe \`a $\g F(\widetilde v)$. On obtient ainsi un diagramme dont toutes les fl\`eches verticales sont injectives:

\qquad\qquad\xymatrix{\g F(\widetilde w)\ar[r]^(.5)\pi\ar@/^1pc/[rr]^\mu&\g E(\widetilde w)\ar[r]^(.5){\nu}&\g B( w)\ar[r]_(.5){\xi}&\Sigma_w\\
\g F(\widetilde v)\ar@{^{(}->}[u]\ar[r]^(.5)\pi&\g E(\widetilde v)\ar@{^{(}->}[u]\ar[r]^(.5){\nu}&\g B( v)\ar@{^{(}->}[u]\ar[r]_(.5){\xi}&\Sigma_v\ar@{^{(}->}[u]\\}

\par On note $\g C(w)=\g B(w)\setminus\bigcup_{v<w}\g B(v)$ qui contient $\xi^{-1}(\Sigma^c_w)$. Par r\'ecurrence il suffit de montrer que $\mu(\g F(\widetilde w))\supset\g C(w)$.

\par On note $\g F^c(\widetilde w)=\g C_{i_1}\times\cdots\times\g C_{i_m}$ ouvert de $\g F(\widetilde w)$ satur\'e pour le quotient par $(\g B^{ma+})^{m-1}$. Mais $\g C_{i_j}=\g B^{ma+}.s_{i_j}.\g B^{ma+}=\g U_{i_j}.s_{i_j}.\g B^{ma+}$ et cette derni\`ere d\'ecomposition est unique. Donc $\g U_{i_1}.s_{i_1}\times\cdots\times\g U_{i_m}.s_{i_m}.\g B^{ma+}$ est un syst\`eme de repr\'esentants de $\g E^c(\widetilde w):=\g F^c(\widetilde w)/(\g B^{ma+})^{m-1}=\pi(\g F^c(\widetilde w))$ et $(\g B^{ma+})^{m-1}$ agit librement sur $\g F^c(\widetilde w)$. Ainsi $\g E^c(\widetilde w)\simeq\g U_{i_1}\times\cdots\times\g U_{i_m}\times\g B^{ma+}$ est un ouvert affine de $\g E(\widetilde w)$.

\par D'autre part $\zeta(u_{i_1}.s_{i_1},\cdots,u_{i_m}.s_{i_m}.b)=s_{i_1}.\cdots.s_{i_m}.u'_1.\cdots.u'_mbe_\lambda$ avec $u'_j\in\g U_{\beta_j}$ pour $\beta_j=s_{i_1}.\cdots.s_{i_j}(\alpha_{i_j})\in\Phi$. En particulier la coordonn\'ee de cet \'el\'ement sur $e_{w\lambda}$ est non nulle; il appartient donc \`a $\Sigma_w^c$: on a $\zeta(\g F^c(\widetilde w))=\eta(\g E^c(\widetilde w))\subset\Sigma^c_w$; donc $\mu(\g F^c(\widetilde w))\subset\xi^{-1}(\Sigma^c_w)\subset\g C(w)$.

\par Le compl\'ementaire de $\g F^c(\widetilde w)$ est r\'eunion de sous-sch\'emas ferm\'es obtenus en rempla\c{c}ant dans $\g F(\widetilde w)$ l'un des facteurs $\g P_i$ par $\g B^{ma+}$, ce qui, dans l'image par $\pi$ revient \`a supprimer ce facteur. Alors le processus de multiplication d\'ecrit dans \cite[p45]{M-89} montre que l'image par $\mu$ de ce sous-sch\'ema ferm\'e est dans un $\g B(v)$ pour $v<w$. Ainsi $\g F^c(\widetilde w)\supset\mu^{-1}(\g C(w))$ et finalement $\g F^c(\widetilde w)=\mu^{-1}(\g C(w))$. De plus
$\pi(\g F^c(\widetilde w))=\g E^c(\widetilde w)$. Donc $\g E^c(\widetilde w)=\nu^{-1}(\g C(w))$.

\par Soit $\g B(w)_f$ un ouvert principal de $\g B(w)$ contenu dans $\g C(w)$. Alors $\nu^{-1}(\g B(w)_f)=\g E(\widetilde w)_{f\circ \nu}=\g E^c(\widetilde w)_{f\circ \nu}$, ouvert principal donc affine de $\g E^c(\widetilde w)$. On a donc $k[\g B(w)_f]=k[\g B(w)][f^{-1}]$ $=k[\g E(\widetilde w)][(f\circ\nu)^{-1}]\subset k[\g E^c(\widetilde w)][(f\circ\nu)^{-1}]=k[\g E^c(\widetilde w)_{f\circ\nu}]=k[\g E(\widetilde w)_{f\circ\nu}]$. Mais $\g E(\widetilde w)$ est r\'eunion de $2^m$ ouverts affines (d\'emonstration par r\'ecurrence selon l'esquisse des pages 54, 55 de \cite{M-88a}); donc $k[\g E(\widetilde w)_{f\circ\nu}]\subset k[\g E(\widetilde w)][(f\circ\nu)^{-1}]$. Ainsi $k[\g B(w)_f]=k[\g E(\widetilde w)_{f\circ\nu}]$. Cela prouve que $\nu$ induit un isomorphisme de $\nu^{-1}(\g B(w)_f)$ sur $\g B(w)_f$ et donc aussi de $\g E^c(\widetilde w)=\nu^{-1}(\g C(w))$ sur $\g C(w)$. On sait de plus que $\pi$ induit un isomorphisme du sous-sch\'ema ferm\'e $\g U_{i_1}.s_{i_1}\times\cdots\times\g U_{i_m}.s_{i_m}.\g B$ de $\g F(\widetilde w)$ sur $\g E^c(\widetilde w)$. Donc $\pi_k$ est surjectif.
\end{proof}

\subsection{Structure de $BN-$paire raffin\'ee sur $\g G^{pma}$}\label{3.16}

\par On note $\g U^+$ (resp. $\g U^-$) le sous-foncteur en groupes de $\g G^{pma}$ engendr\'e par les sous-groupes $\g U_\alpha$ pour $\alpha\in\Phi^+$ (resp. $\alpha\in\Phi^-$); il s'identifie via $i$ (au moins sur un corps) avec le sous-groupe de m\^eme nom de $\g G$. On a $\g U^+\subset\g U^{ma+}$. On rappelle que $S=\{s_i\mid i\in I\}$ est le syst\`eme de g\'en\'erateurs de $W^v$ et que $\g N$ est l'image par $i$ du groupe de m\^eme nom de $\g G$ (remarque \ref{3.12}.1).

\begin{prop*} Si $k$ est un corps, le sextuplet $(G^{pma},N,U^{ma+},U^-,T,S)$ est une $BN-$paire raffin\'ee (refined Tits system).
\end{prop*}

\begin{rema*} On a donc (\cf \eg \cite[1.2]{Ry-02a}) un syst\`eme de Tits $(G^{pma},B^{ma+}=TU^{ma+},N,S)$, une d\'ecomposition de Bruhat $G^{pma}=\coprod_{n\in N}U^{+}nU^{ma+}=\coprod_{n\in N}U^{ma+}nU^{ma+}$, une d\'ecomposition de Birkhoff $G^{pma}=\coprod_{n\in N}U^{-}nU^{ma+}$ et $U^-\cap NU^{ma+}=N\cap U^{ma+}=\{1\}$.
\end{rema*}

\begin{proof} V\'erifions les axiomes de \cite{KP-85}, voir aussi \cite{Ry-02a}.

\par (RT1): le groupe $G^{pma}$ admet bien $N$, $U^{ma+}$, $U^-$ et $T$ comme sous-groupes. Le groupe $T$ normalise les $U_\alpha$ ($\alpha\in\Phi$) donc aussi $U^-$ et $U^{ma+}$, il est distingu\'e dans $N$ et $N/T=W^v$ est engendr\'e par les \'el\'ements de $S$ qui sont d'ordre $2$. Par construction $G^{pma}$ est r\'eunion des $B(w)$ qui sont engendr\'es par des sous-groupes $P_i$ (d'apr\`es \ref{3.15}). On a vu que $P_i$ est engendr\'e par $U_{-\alpha_i}=\widetilde s_i.U_{\alpha_i}.\widetilde s_i^{-1}$ et $B^{ma+}=TU^{ma+}$. Donc $G^{pma}$ est engendr\'e par $N$ et $U^{ma+}$.

\par (RT2) Pour $s=s_i\in S$, $U_s=U^{ma+}\cap sU^-s$ est dans $G$ qui est muni d'une donn\'ee radicielle jumel\'ee $(G,(U_\alpha)_{\alpha\in\Phi},T)$ \cite [prop 8.4.1]{Ry-02a} et d'apr\`es le th\'eor\`eme 1.5.4 de  \lc on a $U_{s_i}=U_{\alpha_i}$. Un r\'esultat classique dans le groupe $A_i$ donne donc (RT2a); tandis que (RT2b) est clair. Enfin (RT2c) r\'esulte de \ref{3.3} (voir aussi \ref{3.5}).

\par (RT3) Soient $u^-\in U^-$, $u^+\in U^{ma+}$ et $n\in N$ tels que $u^-nu^+=1$. Alors $n$ et $u^-$sont dans $G$ ainsi donc que $u^+$. On d\'eduit de l'axiome (RT3) dans $G$ \cite[1.5.4]{Ry-02a} que $u^-=u^+=n=1$.
\end{proof}

\begin{coro}\label{3.17} Si $k$ est un corps, $U^+=\g U^+(k)=\g G(k)\cap\g U^{ma+}(k)$ et $B^+=\g B^+(k)=\g G(k)\cap\g B^{ma+}(k)$.
\end{coro}

\begin{rema*} Le sextuplet $(G,N,U^{+},U^-,T,S)$ est aussi une $BN-$paire raffin\'ee (\ref{1.6}.5 et \cite[1.5.4]{Ry-02a}).
\end{rema*}

\begin{proof} On a deux syst\`emes de Tits $(G^{pma},B^{ma+},N,S)$ et $(G,B^{+},N,S)$, \cf \ref{1.6}.5. Donc $G^{pma}=\coprod_{w\in W^v}B^{ma+}wB^{ma+}$ et $G=\coprod_{w\in W^v}B^+wB^{+}$. Comme $B^+\subset B^{ma+}$, on a $B^{ma+}\cap G=B^+$. Mais $B^{ma+}=T\ltimes U^{ma+}$ et $B^{+}=T\ltimes U^{+}$ donc $B^+\cap U^{ma+}=U^{+}$.
\end{proof}

\begin{coro}\label{3.17b} Si $k$ est un corps, l'injection $i$ induit un isomorphisme simplicial $\g G(k)-$\'equi\-variant de l'immeuble $\shi^v_+$ associ\'e \`a $\g G$ sur $k$ sur l'immeuble $\widehat\shi^v_+$ associ\'e au syst\`eme de Tits $(G^{pma},B^{ma+},N,S)$. Plus pr\'ecis\'ement ces deux immeubles ont les m\^emes facettes et les appartements de $\shi^v_+$ sont des appartements de $\widehat\shi^v_+$, mais, en g\'en\'eral, $\widehat\shi^v_+$ a plus d'appartements.
\end{coro}

\begin{proof} Les appartements sont isomorphes car associ\'es au m\^eme groupe $N$. D'apr\`es \ref{3.17} $G/B^+$ s'injecte dans $G^{pma}/B^{ma+}$, d'o\`u une injection $\shi^v_+\hookrightarrow\widehat\shi^v_+$. Mais, par construction, si $\alpha$ est simple, c'est le m\^eme groupe $U_\alpha$ qui est transitif sur les chambres (de $\shi^v_+$ ou $\widehat\shi^v_+$) $s_\alpha-$adjacentes \`a la chambre fondamentale $C^v_f$ (associ\'ee \`a $B^+$), ce sont donc les m\^emes chambres. Par r\'ecurrence sur la longueur de galeries, on en d\'eduit que  $\shi^v_+=\widehat\shi^v_+$.
\end{proof}

\subsection{Changement de SGR}\label{3.18}

\par 1) On a d\'ecrit ci-dessus le groupe \`a la Mathieu $\g G_{\shs}^{pma}$ associ\'e \`a un SGR $\shs$ libre, colibre, sans cotorsion et v\'erifiant de plus une certaine condition de dimension (\cf \ref{3.5}). ll est clair que cette condition n'intervient pas dans les constructions de Mathieu, on peut donc s'en affranchir. On a d'ailleurs vu en \ref{1.3} que si $\shs$ est libre, colibre, sans cotorsion il est extension semi-directe d'un SGR $\shs^{mat}$ v\'erifiant de plus cette condition de dimension. On en a d\'eduit en \ref{1.11} que $\g G_{\shs}$ est produit semi-direct de $\g G_{\shs^{mat}}$ par un tore $\g T_1$. De m\^eme $\g G_{\shs}^{pma}$ est produit semi-direct de $\g G_{\shs^{mat}}^{pma}$ par ce tore $\g T_1$. Ces d\'ecompositions en produit semi-direct proviennent de la m\^eme d\'ecomposition $\g T_\shs=\g T_{\shs^{mat}}\times\g T_1$ et sont donc compatibles avec les morphismes de foncteurs $i$ de \ref{3.12}.

\par 2) Soit $\shs$ un SGR quelconque. D'apr\`es \ref{1.3} il existe une extension centrale torique $\shs^{sc}$ de $\shs$ et une extension semi-directe $\shs^{sc\ell}$ de $\shs^{sc}$ qui est libre, colibre et sans cotorsion. D'apr\`es les corollaires \ref{1.11} et \ref{1.12} $\g G_{\shs}$ est quotient de $\g G_{\shs^{sc}}$ par un sous-tore central $\g T_2$ de $\g T_{\shs^{sc}}$ et $\g G_{\shs^{sc\ell}}$ est produit semi-direct de $\g G_{\shs^{sc}}$ par un sous-tore $\g T_3$ de $\g T_{\shs^{sc\ell}}$. Inversement $\g G_{\shs^{sc}}$ est le noyau d'un homomorphisme $\theta$ de $\g G_{\shs^{sc\ell}}$ sur $\g T_3$.

\par L'homomorphisme $\theta$ s'\'etend au groupe $\g G_{\shs^{sc\ell}}^{pma}$ (construit ci-dessus): on prend $\theta$ trivial sur $\g U^{ma+}_{\shs^{sc\ell}}$, on conna\^{\i}t $\theta$ sur $\g T_{\shs^{sc\ell}}$ et les groupes $\g A_i$ de \ref{3.5}, donc $\theta$ est d\'efini sur $\g B^{ma+}_{\shs^{sc\ell}}$ et les paraboliques $\g P_i$; son extension \`a $\g G_{\shs^{sc\ell}}^{pma}$ d\'ecoule alors de la construction de \ref{3.6}. On d\'efinit donc $\g G_{\shs^{sc}}^{pma}$ comme le noyau de $\theta$. Enfin $\g G_{\shs}^{pma}$ est le quotient (comme foncteur en groupes) de $\g G_{\shs^{sc}}^{pma}$ par le tore $\g T_2$.

\par On a ainsi d\'efini $\g G_{\shs}^{pma}$ et un morphisme de foncteurs $i_\shs:\g G_\shs\rightarrow\g G_{\shs}^{pma}$ pour tout SGR $\shs$.

\par Si $\shs$ est libre, colibre et sans cotorsion on a ainsi deux d\'efinitions de $\g G_{\shs}^{pma}$. Mais alors, d'apr\`es \ref{1.3}, \ref{1.11} et \ref{1.12}, $\g G_{\shs^{sc}}$ est produit direct de $\g G_{\shs}$ par $\g T_2$ et $\g G_{\shs^{sc\ell}}$ produit semi-direct de $\g G_{\shs^{sc}}$ par $\g T_3$. Il est clair que $\g G_{\shs^{sc}}^{pma}$ et $\g G_{\shs^{sc\ell}}^{pma}$ comme d\'efinis en 1) s'\'ecrivent de la m\^eme mani\`ere comme produit semi-direct. D'o\`u l'identification avec les groupes d\'efinis comme ci-dessus.

\par 3) {\bf Fonctorialit\'e.} La construction de Mathieu est clairement fonctorielle en $\shs$ pour les SGR libres, colibres, sans cotorsion et les extensions commutatives de SGR. On a remarqu\'e en \ref{1.3}.1 que les constructions $\shs\mapsto\shs^{sc}$ et $\shs^{sc}\mapsto\shs^{sc\ell}$ sont fonctorielles. Une extension commutative $\varphi:\shs\rightarrow\shs_*$ d\'efinit donc un morphisme de foncteurs en groupes $\g G_{\shs^{sc\ell}}^{pma}\rightarrow\g G_{\shs_*^{sc\ell}}^{pma}$. Ce morphisme est compatible avec les homomorphismes $\theta$ sur les tores $\g T_3$ correspondants, il induit donc un morphisme $\g G_{\shs^{sc}}^{pma}\rightarrow\g G_{\shs_*^{sc}}^{pma}$. Ce dernier morphisme \'echange les tores $\g T_2$ correspondants, ce qui donne le morphisme $\g G_{\varphi}^{pma}:\g G_{\shs}^{pma}\rightarrow\g G_{\shs_*}^{pma}$ cherch\'e.

\par Le groupe $\g G_{\shs}^{pma}$ est donc fonctoriel en $\shs$ pour les extensions commutatives de SGR (de mani\`ere compatible avec la fonctorialit\'e des groupes $\g G_\shs$ et les morphismes $i_\shs$). Par construction, si $\varphi:\shs\rightarrow\shs'$ est un tel morphisme, $\g G_{\varphi}^{pma}$ induit un isomorphisme de $\g U_{\shs}^{ma+}$ sur $\g U_{\shs'}^{ma+}$ et le morphisme $\g T_\varphi$ de $\g T_\shs$ dans $\g T_{\shs'}$. Si $k$ est un corps, $\g G_{\varphi k}^{pma}$ induit l'isomorphisme $\g G_{\varphi k}$ de  $\g U_{\shs}^{\pm{}}(k)$ sur $\g U_{\shs'}^{\pm{}}(k)$ \cf \ref{1.10} et \ref{3.12}.

\begin{enonce*}{\quad4) Proposition} Les actions de $\g G_\shs(k)$ par automorphismes int\'erieurs sur $\g G_{\shs k}^{}$, $\g G_{\shs k}^{pma}$ ou $\g G_{\shs k}^{nma}$ et ses actions adjointes sur $\shu_{\shs k}^{}$,  $\widehat\shu_{\shs k}^{p}$ ou  $\widehat\shu_{\shs k}^{n}$ se factorisent via $\g G_{\qf^{ad}k}=:Ad^\ell_k$ en des actions du groupe adjoint minimal
$\g G_{\shs_{Am}}(k)=\g G_{\shs^{lad}}(k)=:\g G_{\shs}^{lad}(k)$.

\par Si $k$ est un corps, ces actions de $\g G_{\shs^{lad}}(k)$ sont fid\`eles, en particulier Ker$Ad^\ell_k$ est le centre de $\g G_\shs(k)$.
 \end{enonce*}

 \begin{proof} D'apr\`es \ref{1.10} le noyau de $Ad^\ell_k$ est central dans $\g G_\shs(k)$ et $\g G_{\shs}^{lad}(k)=Ad^\ell_k(\g G_{\shs}^{}(k))$ $.\g T_{\shs^{lad}}(k)$.
  Comme les actions de $\g T_{\shs_{}}(k)$ sur les groupes et alg\`ebres propos\'es se factorisent \'evidemment en des actions de $\g T_{\shs^{lad}}(k)=:\g T_{\shs}^{lad}(k)$, on a clairement des actions de $\g G_{\shs}^{lad}(k)$ qui permettent de factoriser via $Ad^\ell_k$ les actions de $\g G_\shs(k)$ (voir \ref{1.8}.2).

  \par Si $k$ est un corps, l'intersection des normalisateurs dans $\g G_\shs(k)$ des groupes radiciels $\g U_\qa(k)$ est $\g T_{\shs_{}}(k)$ \cite[8.4.1 ii)]{Ry-02a}.
  Donc le noyau de l'action par conjugaison de $\g G_{\shs}^{lad}(k)$ est dans $\g T_{\shs}^{lad}(k)$.
  Mais on conna\^{\i}t bien l'action de ce tore sur les groupes $\g U_{\qa_i}(k)$, pour $i\in I$.
  On en d\'eduit que les actions par conjugaison de $\g G_{\shs}^{lad}(k)$ sont fid\`eles.

  \par Pour les actions adjointes, le stabilisateur de $\shu_{\shs k}^{+}$ ou $\widehat\shu_{\shs k}^{+}$ contient $B^+_{\shs^{lad}}$ donc est un sous-groupe parabolique; comme il ne contient aucun \'el\'ement non trivial de $W^v$, il est r\'eduit \`a $B^+_{\shs^{lad}}$.
  En raisonnant de m\^eme avec $\shu_{\shs k}^{-}$ ou $\widehat\shu_{\shs k}^{-}$, on trouve que le noyau de l'action est dans $B^+_{\shs^{lad}}\cap B^-_{\shs^{lad}}=\g T_{\shs^{lad}}(k)$.
  Mais l'action de ce dernier tore sur les $e_i$ (pour $i\in I$) est bien connue, donc le noyau est trivial.
 \end{proof}

\subsection{Comparaison avec les groupes  de Kac-Moody \`a la Kumar}\label{3.19}

\par Dans \cite[6.1.6 et 1.1.2]{Kr-02} S. Kumar consid\`ere un SGR v\'erifiant les m\^emes conditions que celles demand\'ees par O. Mathieu (\cf \ref{3.5}). On reste donc dans ce cadre et
bien s\^ur on se place, comme Kumar, sur le corps $\C$.

\par Pour $\Theta$ clos dans $\Delta^+$, les groupes $U_\Theta$, $T$ et $N$ d\'efinis dans \lc 6.1 s'identifient clairement \`a $U_\Theta^{pma}$, $T$ et $N$ d\'efinis ci-dessus. De m\^eme pour une partie de type fini $J$ de $I$, le groupe parabolique $P_J$ de Kumar s'identifie au groupe $P(J)=G(J)\ltimes U^{ma+}_J$ de \ref{3.10}, en particulier les deux d\'efinitions des paraboliques minimaux $P_i$ co\"{\i}ncident.

\par On a vu que $(G^{pma},B^{ma+}=TU^{ma+},N,S)$ est un syst\`eme de Tits. D'apr\`es [\lc 5.1.7] $G^{pma}$ est le produit amalgam\'e des sous-groupes $N$ et $P_i$. Mais c'est exactement la d\'efinition de \lc 6.1.16 du groupe de Kac-Moody \`a la Kumar. Donc sur $\C$ les groupes \`a la Kumar et \`a la Mathieu co\"{\i}ncident.

\par L'injection $i_\C$ de \ref{3.12} permet d'identifier $\g G(\C)$ et le groupe minimal \`a la Kumar \cite[7.4.1]{Kr-02}.


\section{Appartement affine et sous-groupes parahoriques}\label{s4}

\par On consid\`ere dor\'enavant un SGR fixe  $\shs$ qui est libre et le groupe de Kac-Moody $\g G_\shs$ associ\'e sur un corps $K$. On a vu en \ref{1.3} et \ref{1.11} que si $\shs$ est un SGR non libre, $\g G_{\shs}$ est un sous-groupe d'un groupe $\g G_{\shs^{\ell}}$ avec $\shs^\ell$ libre. Toute action de $\g G_{\shs^\ell}(K)$ sur une masure induira donc une action de $\g G_{\shs}(K)$ sur celle-ci. On peut noter que l'essentialis\'e de la masure de $\g G_{\shs^\ell}(K)$ ne d\'epend pas du choix de $\shs^\ell$: les appartements sont des espaces affines sous Hom$_\Z(Q,\R)$. C'est le choix fait par B. R\'emy pour ses r\'ealisations coniques d'immeubles jumel\'es.

\par On abr\'egera fr\'equemment $\g G_{\shs}(K)$ en $G_\shs$ ou m\^eme $G$. On notera avec la lettre romaine correspondante les points sur $K$ d'un foncteur en groupes not\'e avec une lettre gothique; de m\^eme pour les morphismes de foncteurs.

\par \`A partir de \ref{4.2}, le corps $K$ sera suppos\'e muni d'une valuation r\'eelle $\omega$.

\subsection{Comparaison avec \cite{Ru-06} et \cite{Ru-10}}\label{4.1}

\par Le quintuplet $(V,W^v,(\alpha_i)_{i\in I},(\alpha_i^\vee)_{i\in I},\Delta_{im})$ de \ref{1.1} satisfait aux conditions de \cite[1.1 et 1.2.1]{Ru-10}, on est dans le cas Kac-Moody (mod\'er\'ement imaginaire). On peut en particulier d\'efinir dans $V$ des facettes vectorielles et des c\^ones de Tits positif ou n\'egatif $\pm\sht$  \cite[1.3]{Ru-10}. Ces notions sont \'egalement introduites dans \cite[1.1 \`a 1.3]{Ru-06} o\`u $\sht$ est not\'e $\A^v$. On note $C^v_f=\{\,v\in V\mid\alpha(v)>0,\forall\alpha\in\Phi^+\}$ la chambre vectorielle fondamentale, donc $\sht=W^v.\overline C^v_f$. Une facette vectorielle est dite sph\'erique si son fixateur (dans $W^v$) est fini \ie si elle est contenue dans $\pm\sht^\circ$ (int\'erieur du c\^one de Tits).

\par D'apr\`es \ref{1.6} le groupe de Kac-Moody $G$ contient un sous-groupe $T$ et des sous-groupes radiciels $U_\alpha$ pour $\alpha\in\Phi$, chacun isomorphe au groupe additif $(K,+)$ par un isomorphisme $x_\alpha$. Le triplet $(G,(U_\alpha)_{\alpha\in\Phi},T)$ est une donn\'ee radicielle de type $\Phi$ au sens de \cite[1.4]{Ru-06}. Le groupe $N$ d\'efini dans \cite[1.5.3]{Ru-06} co\"{\i}ncide avec celui de \ref{1.6}.4, il est muni d'une application surjective $\nu^v:N\rightarrow W^v$ de noyau $T$.

\par On reprend, sauf exception signal\'ee, les notations de \cite[{\S{}} 1]{Ru-06}. On a en particulier un immeuble $\shi^v_+$ dont les appartements sont isomorphes \`a $\A^v=\sht$, mais aussi un immeuble $\shi^v_-$ dont les appartements sont isomorphes \`a $-\sht$. Ces deux immeubles sont jumel\'es \cite{Ry-02a}; comme dans \cite{Ru-06} on ne garde que les facettes sph\'eriques dans $\shi^v_\pm$. On a d\'ej\`a vu ces immeubles de mani\`ere combinatoire en \ref{1.6}.5, mais on les verra d\'esormais sous la forme de ces r\'ealisations coniques. En \ref{3.17b} on a vu que $G^{pma}$ (resp $G^{nma}$) agit sur $\shi^v_+$ (resp. $\shi^v_-$).

\subsection{L'appartement affine t\'emoin $\A$}\label{4.2}

\par 1) On suppose dor\'enavant le corps $K$ muni d'une valuation r\'eelle non triviale $\omega$, et ainsi $\Lambda=\omega(K^*)$ est un sous-groupe non trivial de $\R$. On note $\sho$ l'anneau des entiers de $(K,\omega)$ et $ \g m$ son id\'eal maximal.

\par On note $\A$ l'espace vectoriel $V$ consid\'er\'e comme espace affine; on note cependant toujours $0$ l'\'el\'ement neutre de $V$.

\par Le groupe $T=\g T(K)$ agit sur $\A$ par translations: si $t\in T$, $\nu(t)$ est l'\'el\'ement de $V$ tel que $\chi(\nu(t))=-\omega(\chi(t))$, $\forall\chi\in X$. Cette action est $W^v-$\'equivariante.
D'apr\`es le raisonnement de \cite[2.9.2]{Ru-06} on a:

\begin{enonce*}{\quad\ 2) Lemme}
  Il existe une action affine $\nu$ de $N$ sur $\A$ qui induit la pr\'ec\'edente sur $T$ et telle que, pour $n\in N$, l'application affine $\nu(n)$ admet $\nu^v(n)$ comme application lin\'eaire associ\'ee.
\end{enonce*}

\par 3) L'image de $N$ est $\nu(N)=W_{Y\Lambda}=W^v\ltimes(Y\otimes\Lambda)$. Son noyau est $H=$Ker$(\nu)=\sho^*\otimes Y=\g T(\sho)$.

\medskip
\par 4) Par construction le point $0$ est fix\'e par les \'el\'ements $m(x_{\alpha_i}(\pm{}1))=\widetilde s_{-\alpha_i}^{\pm1}$ et donc (par conjugaison: (KMT7) ) par tous les $\widetilde s_{\alpha}$, pour $\alpha\in\Phi$. D'apr\`es (KMT6) pour $u\not=0$, $m(x_{\alpha}(u))=\widetilde s_{-\alpha}(u^{-1})=\widetilde s_{-\alpha}.(-\alpha)^\vee(u)=\alpha^\vee(u).\widetilde s_{-\alpha}$ et $\alpha(\nu(\alpha^\vee(u)))=-\omega(\alpha(\alpha^\vee(u)))=-\omega(u^2)=-2\omega(u)$. Donc $\nu(m(x_{\alpha}(u)))$ est la r\'eflexion $r_{\alpha,\omega(u)}$ par rapport \`a l'hyperplan $M(\alpha,\omega(u))=\{y\in V\mid \alpha(y)+\omega(u)=0\}$, c'est-\`a-dire l'application affine d'application lin\'eaire associ\'ee $s_\alpha$ et d'ensemble de points fixes $M(\alpha,\omega(u))$.

\medskip
\par 5) {\bf D\'efinitions.} L'{\it appartement affine t\'emoin} $\A$ est l'espace affine $\A$ (sous l'espace vectoriel $V$ muni de $W^v$, $\Phi$ et $\Delta$) muni de l'ensemble $\shm$ des {\it murs (r\'eels)} $M(\alpha,k)=\{y\in V\mid\alpha(y)+k=0\}$ pour $\alpha\in\Phi$ et $k\in\Lambda$ et de l'ensemble $\shm^i$ des {\it murs imaginaires} $M(\alpha,k)$ pour $\alpha\in\Delta_{im}$ et $k\in\Lambda$. L'espace $D(\alpha,k)=\{y\in V\mid\alpha(y)+k\geq{}0\}$, d\'efini pour $\alpha\in X$ et $k\in\R$ est un {\it demi-appartement} si $\alpha\in\Phi$ et $k\in\Lambda$. On note $D^\circ(\alpha,k)=\{y\in V\mid\alpha(y)+k>0\}$.

\par Ainsi $\A$ satisfait aux conditions de \cite[1.4]{Ru-10} (\`a une variante pr\`es pour $\shm^i$, \cf \lc 1.6.2); il est semi-discret si et seulement si la valuation $\omega$ est discr\`ete. On choisit $x_0=0$ comme point sp\'ecial privil\'egi\'e. On reprend les notations de \cite{Ru-10} (sauf pour $P^\vee$, $Q^\vee$ et $cl$, voir ci-dessous). On notera les ressemblances et diff\'erences avec \cite[{\S{}} 2]{Ru-06}. Le cas $\Lambda=\Z$ est trait\'e dans \cite[{\S{}} 2 et 3.1]{GR-08}.

\par On note $Q^\vee=\sum_{i\in I}\Z\alpha_i^\vee$ et $P^\vee=\{y\in V\mid \alpha_i(y)\in\Z,\forall i\}$. Donc $Q^\vee\subset Y\subset P^\vee\subset V$.

\par On consid\`ere des filtres de parties de $\A$ comme d\'efinis dans \cite[I {\S{}} 6]{B-TG} et utilis\'es dans \cite{Ru-10} ou \cite{GR-08}.
Pour un tel filtre $F$ de parties de $\A$, on a introduit dans \lc  plusieurs variantes d'enclos: si $\shp\subset X\setminus\{0\}$, $cl_\shp(F)$ est le filtre form\'e des sous-ensembles de $\A$ contenant un \'el\'ement de $F$ de la forme $\cap_{\alpha\in\shp}\,D(\alpha,k_\alpha)$ avec pour chaque $\alpha$, $k_\alpha\in\Lambda\cup\{+\infty\}$ (en particulier chaque $D(\alpha,k_\alpha)$ contient $F$). On note $cl(F)=cl_\Delta(F)$, $cl^{si}(F)=cl_\Phi(F)$. On d\'efinit aussi  (selon Charignon \cite{Cn-10}) $cl^{\#}(F)$ comme le filtre form\'e des sous-ensembles de $\A$ contenant un \'el\'ement de $F$ de la forme $\bigcap_{i=1}^nD(\alpha_i,k_i)$ o\`u $\alpha_1,\cdots,\alpha_n\in\Phi$ et $k_i\in\Lambda\cup\{\infty\}$. On a $cl^{\#}(F)\supset cl^{si}(F)\supset cl(F)\supset \overline{\conv}(F)$ (enveloppe convexe ferm\'ee de $F$); $cl(F)$ et $cl^{si}(F)$ co\"{\i}ncident avec ceux d\'efinis dans \cite{Ru-10}. Un filtre $F$ est dit {\it clos} si $F=cl(F)$ et, quand on ne pr\'ecise pas, l'enclos de $F$ d\'esigne $cl(F)$.

\par Une {\it facette-locale} de $\A$ est associ\'ee \`a un point $x$ de $\A$ et une facette vectorielle $F^v$ dans $V$; c'est le filtre $F^\ell(x,F^v)= \germ_x(x+F^v)$. La {\it facette} associ\'ee \`a $F^\ell(x,F^v)$ est le filtre $F(x,F^v)$ form\'e des ensembles contenant une intersection de demi-espaces $D(\alpha,k_\alpha)$ ou $D^\circ(\alpha,k_\alpha)$ [un seul $k_\alpha\in\Lambda\cup\{+\infty\}$ (resp. $k_\alpha\in\R\cup\{+\infty\}$) pour chaque $\alpha\in\QF$ (resp. $\alpha\in\QD^{im}$)], cette intersection devant contenir $F^\ell(x,F^v)$ \ie un voisinage de $x$ dans $x+F^v$. La facette ferm\'ee $\overline F(x,F^v)$ est l'adh\'erence de $F(x,F^v)$ et aussi $cl(F^\ell(x,F^v))=cl(F(x,F^v))$. Les facettes $F^\ell(x,F^v)$, $F(x,F^v)$ et $\overline F(x,F^v)$ sont dites {\it sph\'eriques} si $F^v$ l'est.

\par Ces d\'efinitions de facette, facette-locale ou facette ferm\'ee co\"{\i}ncident avec celles de \cite{Ru-10}.
S'il existe $\alpha\in\Delta_{im}$ avec $\alpha(x)\notin\Lambda$, il se pourrait que $F(x,F^v)$ soit l\'eg\`erement plus petite que la facette d\'efinie dans \cite{GR-08}; les facettes locales et ferm\'ees sont les m\^emes, \cf \cite[1.6.2]{Ru-10}.

\medskip
\par 6) {\bf Le groupe de Weyl.} Pour $\alpha\in\Phi$, il est clair que le groupe de Weyl (affine) $W$, engendr\'e par les r\'eflexions par rapports aux murs,  contient les translations de vecteur dans $\Lambda\alpha^\vee$. Ainsi $Q^\vee\otimes\Lambda$ (not\'e $Q^\vee$ dans \cite{Ru-10}) est un groupe de translations de $V$, contenu dans $W$ et invariant par $W^v$. On a $W=W^v\ltimes(Q^\vee\otimes\Lambda)$.

\par Le groupe $P^\vee_\Lambda=\Lambda.P^\vee=\{y\in V\mid \alpha_i(y)\in\Lambda,\forall i\}$ est not\'e $P^\vee$ dans  \cite{Ru-10}. Ainsi $W_P=W^v\ltimes P_\Lambda^\vee$ est le plus grand sous-groupe de $W^e_\R=W^v\ltimes V$ stabilisant $\shm$ (et $\shm^i$).

\par Comme $N$ est engendr\'e par $T$ et les $\widetilde s_\alpha$, il est clair que $\nu(N)=W_Y=W^v\ltimes(Y\otimes\Lambda)$. On a : $W\subset W_Y\subset W_P \subset W_\R^e$.

\medskip
\par 7) {\bf Valuation des groupes radiciels.} Pour $\alpha\in\Phi$ et $u\in U_\alpha$, on pose $\varphi_\alpha(u)=\omega(r)$ si $u=x_\alpha(r)$ avec $r\in K$. Les $\varphi_\alpha$ forment une valuation de la donn\'ee radicielle $(G,(U_\alpha),T)$ \cite[2.2]{Ru-06}. Voir aussi \cite{Cn-10}

\par On consid\`ere le mono\"{\i}de ordonn\'e $\widetilde\R$ de \cite[6.4.1]{BtT-72} form\'e d'\'el\'ements $r$, $r^+$ (pour $r\in\R$) et $\infty$ v\'erifiant $r<r^+<s<s^+<\infty$ si $r<s$. On note $\widetilde\Lambda$ l'ensemble des \'el\'ements de $\widetilde\R$ qui sont borne inf\'erieure d'une partie minor\'ee de $\Lambda$. On a $\widetilde\Lambda=\Lambda\cup\{\infty\}$ si $\Lambda$ est discret et $\widetilde\Lambda=\Lambda\cup\{\infty\}\cup\{r^+\mid r\in\R\}$ sinon.
 Les id\'eaux fractionnaires de $\sho$ sont index\'es par $\widetilde\Lambda$:  \`a $\lambda\in\widetilde\Lambda$ on associe $K_\lambda=\{r\in K\mid\omega(r)\geq{}\lambda\}$.

 \par Pour $\lambda\in\widetilde\Lambda$ et $\alpha\in\Phi$, $U_{\alpha,\lambda}=\{u\in U_\alpha\mid\varphi_\alpha(u)\geq\lambda\}$ est un sous-groupe de $U_\alpha$; on a $U_{\alpha,\infty}=\{1\}$.

 \subsection{Premiers groupes associ\'es \`a un filtre $\Omega$ de parties de $\A$}\label{4.3}

 \par On consid\`ere un filtre $\Omega$ de parties de $\A$, comme en \ref{4.2}.5.

 \par 1) Pour $\alpha\in X$, soit $f_\Omega^\Lambda(\alpha)=f_\Omega(\alpha)=
 inf\{\lambda\in\Lambda\mid\Omega\subset D(\alpha,\lambda)\}=
inf\{\lambda\in\Lambda\mid\alpha(\Omega)+\lambda\subset[0,+\infty[\}\in\widetilde\Lambda$. Pour $\alpha\in\Delta$ (resp. $\Phi$, $X$) $f_\Omega(\alpha)$ ne d\'epend que de $cl(\Omega)$ (resp. $cl^{si}(\Omega)$ ou ${cl^\#(\Omega)}$, $\overline{\conv}(\Omega)\supset\overline\Omega$).
Si $\Omega$ est un ensemble, $f_\Omega(\alpha)\not=\lambda^+$, $\forall\alpha\in X$, $\forall\lambda\in\Lambda$. Cette fonction $f_\Omega$  est {\it concave} \cite{BtT-72}: $\forall\alpha,\beta\in X$, $f_\Omega(\alpha+\beta)\leq f_\Omega(\alpha)+f_\Omega(\beta)$ et  $f_\Omega(0)=0$.

\par On dit que $\Omega$ est {\it presque-ouvert} si $\forall\alpha\in\Phi$, $f_\Omega(\alpha)+f_\Omega(-\alpha)>0$. On dit que $\Omega$ est {\it \'etroit} si aucun mur ne s\'epare $\Omega$ en deux, autrement dit si $\forall\alpha\in\Phi$, $f_\Omega(\alpha)+f_\Omega(-\alpha)=0$ ou $0^+$ ou (dans le cas discret) $inf\{\lambda\in\Lambda\mid\lambda>0\}$ et s'il n'existe pas de $\lambda\in\Lambda$ avec $f_\Omega(\alpha)=\lambda^+$ et $f_\Omega(-\alpha)=(-\lambda)^+$.

\par 2) Pour $\alpha\in\Phi$, on note $U_{\alpha,\Omega}=U_{\alpha,f_\Omega(\alpha)}$; $U_\Omega^{(\alpha)}$ est le groupe engendr\'e par $U_{\alpha,\Omega}$ et $U_{-\alpha,\Omega}$
 et on note $N_\Omega^{(\alpha)}=N\cap U_\Omega^{(\alpha)}$.

\par On d\'efinit $U_\Omega$ comme le groupe engendr\'e par les  $U_{\alpha,\Omega}$ (pour $\alpha\in\Phi$) et $U^\pm_\Omega=U_\Omega\cap U^\pm$; on verra (\ref{4.12}) que $U^{+}_\Omega$ (resp. $U^{-}_\Omega$)n'est pas forc\'ement \'egal au groupe $U^{++}_\Omega$ (resp. $U^{--}_\Omega$) engendr\'e par les  $U_{\alpha,\Omega}$ pour $\alpha\in\Phi^+$ (resp. $\alpha\in\Phi^-$).

\par Le groupe $N^u_\Omega$ ($\subset N\cap U_\Omega$) est engendr\'e par tous les $N_\Omega^{(\alpha)}$ pour $\alpha\in\Phi$.

\par Tous ces groupes sont normalis\'es par $H$ et on peut donc d\'efinir $N_\Omega^{min}=HN_\Omega^{u}$ et $P_\Omega^{min}=HU_\Omega$. Ces groupes ne d\'ependent que de l'enclos $cl(\Omega)$ de $\Omega$ et m\^eme que de $cl^\#(\Omega)$. Si $\Omega\subset\Omega'$, on a $U_{\Omega}\supset U_{\Omega'}$, etc.

 \begin{enonce*}{\quad3) Lemme}
\par Soient $\alpha\in\Phi$ et $\Omega$ un filtre comme ci-dessus.

\par a)$U_\Omega^{(\alpha)}=U_{\alpha,\Omega}.U_{-\alpha,\Omega}.N_\Omega^{(\alpha)}=U_{-\alpha,\Omega}.U_{\alpha,\Omega}.N_\Omega^{(\alpha)}$.

\par b) Si $f_\Omega(\alpha)+f_\Omega(-\alpha)>0$, alors $N_\Omega^{(\alpha)}\subset H$. Si $f_\Omega(\alpha)=-f_\Omega(-\alpha)=k\in\Lambda$, alors $\nu(N_\Omega^{(\alpha)})=\{r_{\alpha,k},1\}$.

\par c) $N_\Omega^{(\alpha)}$ fixe $\Omega$, \ie $\forall n\in N_\Omega^{(\alpha)}$, il existe $S\in\Omega$ fixe point par point par $\nu(n)$.
\end{enonce*}

\begin{proof}
\cf \cite[lemma 3.3]{GR-08}
\end{proof}

\par 4) Le groupe $W_\Omega^{min}=\nu(N_\Omega^{min})=N_\Omega^{min}/H$ est engendr\'e par les r\'eflexions $r_{\alpha,k}$ pour $\alpha\in\Phi$ et $k\in\Lambda$ tels que $\Omega\subset M(\alpha,k)$. Il est contenu dans le fixateur $W_\Omega$ de $\Omega$ dans $W$.
 Il y a bien s\^ur \'egalit\'e si $\Omega$ est r\'eduit \`a un point sp\'ecial, mais aussi si $\Omega$ est une facette sph\'erique et si $\Lambda$ est discret. En effet dans ce dernier cas l'image de $W_\Omega$ dans $W^v$ fixe une facette vectorielle sph\'erique et on est ramen\'e au cas classique o\`u $W$ est produit semi-direct d'un groupe de Weyl fini par un groupe discret de translations; ce cas est trait\'e dans \cite[7.1.10]{BtT-72} et ne se g\'en\'eralise pas si $\QL$ est dense.
  D'apr\`es \ref{4.12}.4 ci-dessous il ne se g\'en\'eralise pas non plus au cas non classique (\ie non sph\'erique).

 \subsection{Alg\`ebres et modules associ\'es au filtre $\Omega$ }\label{4.4}

\par 1) Pour $\alpha\in\Delta$ et $\lambda\in\widetilde \Lambda$, on pose $\g g_{\alpha,\lambda}=\g g_{\alpha\Z}\otimes K_\lambda$ et $\g g_{\alpha,\Omega}=\g g_{\alpha,f_\Omega(\alpha)}$. De m\^eme $\g h_\sho=\g h_\Z\otimes\sho$. Par concavit\'e de $f_\Omega$, il est clair que $\g g_\Omega=\g h_\sho\oplus(\oplus_{\alpha\in\Delta}\,\g g_{\alpha,\Omega})$ est une sous$-\sho-$alg\`ebre de Lie de $\g g_K=\g g_\Z\otimes K$.

\par On peut de m\^eme consid\'erer les compl\'etions positive et n\'egative: $\widehat{\g g}_\Omega^p=(\oplus_{\alpha<0}\,\g g_{\alpha,\Omega})\oplus\g h_\sho\oplus(\prod_{\alpha>0}\,\g g_{\alpha,\Omega})$ et $\widehat{\g g}_\Omega^n=(\oplus_{\alpha>0}\,\g g_{\alpha,\Omega})\oplus\g h_\sho\oplus(\prod_{\alpha<0}\,\g g_{\alpha,\Omega})$ si $\Omega$ est un ensemble, sinon $\widehat{\g g}_{\Omega}^{p}=\cup_{\Omega'\in\Omega}\,\widehat{\g g}_{\Omega'}^p$ et $\widehat{\g g}_{\Omega}^n=\cup_{\Omega'\in\Omega}\,\widehat{\g g}_{\Omega'}^n$.

\par Ces alg\`ebres ne d\'ependent que de $cl(\Omega)$. Par contre $cl^{si}(\Omega)$ ou $cl^\#(\Omega)$ pourraient \^etre insuffisants pour les d\'eterminer.

\medskip
\par 2) L'alg\`ebre enveloppante enti\`ere $\shu_\shs$ est gradu\'ee par $Q\subset X$. On peut donc d\'efinir $\shu_\Omega=\bigoplus_{\alpha\in Q}\,\shu_{\alpha,\Omega}$ avec $\shu_{\alpha,\Omega}=\shu_{\shs,\alpha}\otimes_\Z\,K_{f_\Omega(\alpha)}$. C'est une sous$-\sho-$big\`ebre gradu\'ee de $\shu_{\shs,K}=\shu_\shs\otimes_\Z\,K$. Le terme de niveau $1$ de la filtration induite par celle de $\shu_{\shs,K}$ est $\sho\oplus\g g_\Omega$. On peut avoir $\shu_\Omega\not=\shu_{cl(\Omega)}$.

\par On peut consid\'erer la compl\'etion positive $\widehat\shu^p_\Omega$ (resp. n\'egative $\widehat\shu^n_\Omega$) de $\shu_\Omega$ pour le degr\'e total (resp. son oppos\'e); c'est une sous-alg\`ebre de $\widehat\shu^p_K$ (resp.  $\widehat \shu^n_K$).

\medskip
\par 3) Soit $\lambda$ un poids dominant de $\g T_\shs$ (\ie $\lambda\in X^+$). On a construit en \ref{2.14} une forme enti\`ere $L_\Z(\lambda)$ du module int\'egrable $L(\lambda)$ de plus haut poids $\lambda$. Tout poids de $M=L(\lambda)$ est de la forme $\mu=\lambda-\nu$ avec $\nu\in Q^+$. Pour $\mu\in X$ et $r\in\widetilde\Lambda$, on pose $M_{\mu,r}=L_\Z(\lambda)_\mu\otimes_\Z\,K_{r}$ puis $M_{\mu,\Omega}=M_{\mu,f_\Omega(\mu)}$ et $M_\Omega=\bigoplus_{\mu\in X}\,M_{\mu,\Omega}$;  en particulier $M_{\lambda,\Omega}=K_{f_\Omega(\lambda)}.v_\lambda$.
Par concavit\'e de $f_\Omega$, il est clair que $M_\Omega$ est un sous$-\widehat\shu^p_\Omega-$module gradu\'e de $M\otimes K$. On peut avoir $M_\Omega\not=M_{cl(\Omega)}$, mais $M_\Omega$ ne d\'epend que de $cl_{\shp(M)}(\Omega)$ o\`u $\shp(M)$ est l'ensemble des poids de $M$.

\par On peut consid\'erer la compl\'etion n\'egative de $M_\Omega$: $\widehat M^n_\Omega=\prod_{\nu\in Q^+}\,M_{\lambda-\nu,\Omega}$ si $\Omega$ est un ensemble et $\widehat M^n_\Omega=\cup_{\Omega'\in\Omega}\,\widehat M^n_{\Omega'}$ sinon. Il est clair que $\widehat M^n_\Omega$ est un sous$-\widehat\shu^p_\Omega-$module gradu\'e de $\widehat{M\otimes K}$.

\par Bien s\^ur on construit aussi $M_\Omega$ et sa compl\'etion positive $\widehat M^p_\Omega$ pour un module int\'egrable $M$ \`a plus bas poids $\lambda\in-X^+$.

 \subsection{Groupes (pro-)unipotents associ\'es au filtre $\Omega$ }\label{4.5}

 \par 1) Soit $\Psi\subset\Delta^+$ un ensemble clos de racines. On lui a associ\'e une $\Z-$big\`ebre $\shu(\Psi)$ (\ref{2.13}.2 et \ref{3.1}). On peut donc d\'efinir $\shu(\Psi)_\Omega=\bigoplus _{\alpha\in Q^+}\,\shu(\Psi)_{\alpha,\Omega}$ (avec $\shu(\Psi)_{\alpha,\Omega}=\shu(\Psi)_\alpha\otimes K_{f_\Omega(\alpha)}$); c'est une sous$-\sho-$big\`ebre de $\shu_\Omega$ et de $\shu(\Psi)\otimes K$ (plus pr\'ecis\'ement leur intersection).

 \medskip
 \par 2) Le plus simple pour associer \`a $\shu(\Psi)_\Omega$ un sous-groupe de $\g U^{ma}_\Psi(K)$ semble \^etre de proc\'eder comme suit:

 \par Si $\Omega$ est un ensemble on consid\`ere la compl\'etion $\widehat\shu(\Psi)_\Omega=\prod_{\alpha\in Q^+}\,\shu(\Psi)_{\alpha,\Omega}$; c'est une sous-alg\`ebre de $\widehat\shu^p_\Omega$ et de $\widehat\shu_K(\Psi)$ et elle contient $[\exp]\lambda x$ pour $x\in\g g_{\alpha\Z}$ et $\lambda\in K_{f_\Omega(\alpha)}$. Le mono\"{\i}de $U_\Omega^{ma}(\Psi)$ est l'intersection de $\widehat\shu(\Psi)_\Omega$ ou $\widehat\shu^p_\Omega$ et de $\g U^{ma}_\Psi(K)$ (qui est un sous-groupe des \'el\'ements inversibles de $\widehat\shu_K(\Psi)$). On sait que l'inverse d'un \'el\'ement de $\g U^{ma}_\Psi(K)$ est son image par la co-inversion $\tau$ et il est clair que $\widehat\shu(\Psi)_\Omega$ est stable par $\tau$ (car $\tau$ est gradu\'ee et stabilise $\shu(\Psi)$). Ainsi $U_\Omega^{ma}(\Psi)$ est un sous-groupe de $\g U^{ma}_\Psi(K)$. D'apr\`es \ref{3.2} il est clair que $U_\Omega^{ma}(\Psi)$ est form\'e des produits $\prod_{x\in\shb_\Psi}\,[\exp]\lambda_x.x$ pour $\lambda_x\in K_{f_\Omega(pds(x))}$ (il suffit de raisonner par r\'ecurrence sur le degr\'e total).

 \par Si $\Omega$ est un filtre, le groupe $U_\Omega^{ma}(\Psi)$ est la r\'eunion des $U_{\Omega'}^{ma}(\Psi)$ pour $\Omega'\in\Omega$, donc encore l'intersection de  $\g U^{ma}_\Psi(K)$ et de $\widehat\shu(\Psi)_\Omega=\cup_{\Omega'\in\Omega}\,\widehat\shu(\Psi)_{\Omega'}$.

\medskip
\par 3) On sait que $G$ et $\g U^{ma}_\Psi(K)$ s'injectent dans $\g G^{pma}(K) $ (\ref{3.3}a, \ref{3.6} et \ref{3.13}). On note donc $U_\Omega^{pm}(\Psi)=G\cap U_\Omega^{ma}(\Psi)$. On a aussi $U_\Omega^{pm}(\Psi)=U^+\cap U_\Omega^{ma}(\Psi)$ puisque $U^+=G\cap U^{ma+}$ (\ref{3.17}). Pour un filtre on a $U_\Omega^{pm}(\Psi)=\cup_{\Omega'\in\Omega}\,U_{\Omega'}^{pm}(\Psi)$. Bien s\^ur $U_\Omega^{pm+}:=U_\Omega^{pm}(\Delta^+)$ contient $U_\Omega^{+}$ (et $U_\Omega^{nm-}:=U_\Omega^{nm}(\Delta^-)\supset U_\Omega^{-}$).

\medskip
\par 4) {\bf Remarques}: a) Le sous-groupe $H=\g T(\sho)$ de $T=\g T(K)$ stabilise les alg\`ebres et modules construits en \ref{4.4} ou ci-dessus. Ainsi $H$ normalise $U^{ma}_\Omega(\Psi)$ et $U^{pm}_\Omega(\Psi)$.

\par b) Comme $\widehat\shu(\Psi)_\Omega$ est une sous-alg\`ebre de $\widehat\shu^p_\Omega$ et de $\widehat\shu_K(\Psi)$, il est clair que les groupes $U_\Omega^{ma}(\Psi)$ et $U_\Omega^{pm}(\Psi)$ stabilisent $\widehat{\shu}^p_\Omega$ et $\widehat{\g g}^p_\Omega$ pour l'action adjointe de $U^{ma}_\Psi(K)$ sur $\widehat\shu^p_K$, \cf \ref{3.7}.3.

\par c) Pour $\alpha\in\Phi$ et $\Psi=\{\alpha\}$, le groupe $U_\Omega^{ma}(\Psi)$ est \'egal \`a $U_{\alpha,\Omega}\subset U_\alpha\subset G$ et isomorphe \`a $K_{f_\Omega(\alpha)}$ par $x_\alpha:r\mapsto \exp(re_\alpha)$ (d\'efini en \ref{4.2}.7).

\par d) Si $\Omega\subset\Omega'$, on a $U^{ma}_{\Omega}(\Psi)\supset U^{ma}_{\Omega'}(\Psi)$ et $U^{pm}_{\Omega}(\Psi)\supset U^{pm}_{\Omega'}(\Psi)$. Si $\Omega$ est la r\'eunion d'une famille $(\Omega_i)_{i\in N}$ de filtres, alors $f_{\Omega}$ est le sup (pris dans $\widetilde\Lambda$) des $f_{\Omega_i}$ et $\shu(\Psi)_{\Omega}=\cap_{i\in N}\,\shu(\Psi)_{\Omega_i}$. Ainsi $U^{ma}_{\Omega}(\Psi)=\cap_{i\in N}\,U^{ma}_{\Omega_i}(\Psi)$ et $U^{pm}_{\Omega}(\Psi)=\cap_{i\in N}\,U^{pm}_{\Omega_i}(\Psi)$.

\par e) Pour $\Omega=\{0\}\subset\A$, on a $\shu_\Omega=\shu_\shs\otimes_\Z\sho$, $\shu(\Psi)_\Omega=\shu_\shs(\Psi)\otimes_\Z\sho$, $U_\Omega^{ma}(\Psi)=\g U^{ma}_\Psi(\sho)$.

\par f)  $U_\Omega^{ma}(\Psi)$ et $U_\Omega^{pm}(\Psi)$ ne d\'ependent que de $cl(\Omega)$ (cela pourrait \^etre faux pour $cl^{si}(\Omega)$ ou $cl^\#(\Omega)$).

 \begin{enonce*}{\quad5) Lemme}
Soient $\Psi'\subset\Psi\subset\Delta^+$ des sous-ensembles clos de racines.

\par a) $ U^{ma}_\Omega({\Psi'})$ (resp. $ U^{pm}_\Omega({\Psi'})$) est un sous-groupe  de $U^{ma}_\Omega({\Psi})$ (resp. $ U^{pm}_\Omega({\Psi})$).

\par b) Si $\Psi\setminus\Psi'$ est \'egalement clos, alors on a des d\'ecompositions uniques $ U^{ma}_\Omega({\Psi})= U^{ma}_\Omega({\Psi'}). U^{ma}_\Omega({\Psi\setminus\Psi'})$ et $ U^{pm}_\Omega({\Psi})= U^{pm}_\Omega({\Psi'}). U^{pm}_\Omega({\Psi\setminus\Psi'})$.

\par c) Si $\Psi'$ est un id\'eal de $\Psi$, alors $ U^{ma}_\Omega({\Psi'})$ (resp. $ U^{pm}_\Omega({\Psi'})$)  est un sous-groupe distingu\'e de  $ U^{ma}_\Omega({\Psi})$ (resp. $ U^{pm}_\Omega({\Psi})$) et on a un produit semi-direct si, de plus, $\Psi\setminus\Psi'$ est clos.

\par d) Si  $\alpha\in\Delta^+$ est une racine simple, on a $ U^{ma+}_\Omega=U_{\alpha,\Omega}\ltimes U^{ma}_\Omega(\Delta^+\setminus\{\alpha\})$ et $ U^{pm+}_\Omega=U_{\alpha,\Omega}\ltimes  U^{pm}_\Omega(\Delta^+\setminus\{\alpha\})$. Les groupes $ U^{ma}_\Omega(\Delta^+\setminus\{\alpha\})$ et  $ U^{pm}_\Omega(\Delta^+\setminus\{\alpha\})$ sont normalis\'es par $HU_\Omega^{(\alpha)}$.
\end{enonce*}

\begin{proof} Le a) est clair. Pour le b) et le c) on a des d\'ecompositions analogues dans $\g U^{ma}_\Psi(K)$ (\ref{3.3}) et, par exemple, $U^{ma}_\Omega(\Psi')=\g U^{ma}_{\Psi'}(K)\cap\widehat\shu^p_\Omega$, d'o\`u les r\'esultats par intersection (puisque $\shu^p_\Omega$ est gradu\'ee). La premi\`ere partie de d) r\'esulte de c). On sait donc que $U^{ma}_\Omega(\Delta^+\setminus\{\alpha\})$ et $U^{pm}_\Omega(\Delta^+\setminus\{\alpha\})$ sont normalis\'es par $H$ et $U_{\alpha,\Omega}$. Le m\^eme raisonnement avec $\Delta^+$ remplac\'e par $s_\alpha(\Delta^+)=\Delta^+\cup\{-\alpha\}\setminus\{\alpha\}$ montre qu'ils sont aussi normalis\'es par $U_{-\alpha,\Omega}$, d'o\`u la conclusion.
\end{proof}

\par 6) On a vu en \ref{3.5} que $\g U^{ma}_{\Delta^+\setminus\{\alpha\}}(K)$ est normalis\'e par $G^{(\alpha)}=\g A_\alpha^Y(K)=\langle T,U_\alpha,U_{-\alpha}\rangle$. Cela signifie en particulier que $\g U_{-\alpha}\ltimes\g U^{ma}_{\Delta^+\setminus\{\alpha\}}=s_\alpha(\g U^{ma}_{\Delta^+})$. On peut donc ainsi d\'efinir dans $\g G^{pma}$ un groupe $\g U^{ma}(w\Delta^+)$ pour tout $w\in W^v$. 

\medskip
\par 7) Dans toutes les notations pr\'ec\'edentes un $+$ (resp. un $-$) peut remplacer $\Psi$ quand $\Psi=\Delta^+$ (resp. $\Psi=\Delta^-$).

\par On consid\`ere aussi le groupe de Kac-Moody n\'egativement maximal $\g G^{nma}$  construit comme $\g G^{pma}$ en \'echangeant $\Delta^+$ et $\Delta^-$. Plus g\'en\'eralement on peut changer $p$ en $n$ et $\pm$ en $\mp$ dans ce qui pr\'ec\`ede pour obtenir des groupes similaires (avec des propri\'et\'es similaires) dans le cas n\'egatif.

\begin{prop}\label{4.6} Soit $\Omega$ un filtre de parties de $\A$. Il y a trois sous-groupes de $G$ associ\'es \`a $\Omega$ et ind\'ependants du choix de $\Delta^+$ dans sa classe de $W^v-$conjugaison.

\par 1) Le groupe $U_\Omega$ (engendr\'e par tous les $U_{\alpha,\Omega}$) est \'egal \`a $U_\Omega^{}=U_\Omega^{-}.U_\Omega^{+}.N_\Omega^{u}=U_\Omega^{+}.U_\Omega^{-}.N_\Omega^{u}$.

\par 2) Le groupe $U_\Omega^{pm}$ (engendr\'e par  les groupes $U_\Omega^{pm+}$ et $U_\Omega$) est \'egal \`a $U_\Omega^{pm}=U_\Omega^{pm+}.U_\Omega^{-}.N_\Omega^{u}$.

\par 3) Sym\'etriquement le groupe $U_\Omega^{nm}$ (engendr\'e par  les groupes $U_\Omega^{nm-}$ et $U_\Omega$) est \'egal \`a $U_\Omega^{nm}=U_\Omega^{nm-}.U_\Omega^{+}.N_\Omega^{u}$.

\par 4) On a:
$$
\begin{array}{ll}
i)\ \ \,U_\Omega\cap N=N_\Omega^u & ii)\ U_\Omega^{pm}\cap N=N_\Omega^u\\
iii)\ U_\Omega\cap (N.U^\pm)=N_\Omega^u.U^\pm _\Omega & iv)\ U_\Omega^{pm}\cap (N.U^+)=N_\Omega^u.U_\Omega^{pm+}\\
v)\ \ \,U_\Omega\cap U^\pm =U^\pm _\Omega & vi)\ U_\Omega^{pm}\cap U^+=U_\Omega^{pm+}
\end{array}
$$ et sym\'etriquement pour $U_\Omega^{nm}$.

\end{prop}

\begin{proof} On refait la preuve de \cite[prop. 3.4]{GR-08} en utilisant \ref{4.5}.5, \ref{4.3}.3 et la remarque \ref{3.16}.
\end{proof}

\begin{remas*} a) Le groupe $H$ normalise $U_\Omega$, $U_\Omega^+$, $U_\Omega^-$, $N_\Omega^u$, $U_\Omega^{pm+}$, $U_\Omega^{nm-}$, $U_\Omega^{pm}$ et $U_\Omega^{nm}$. Le groupe $N_\Omega^{min}=HN_\Omega^{u}$ est contenu dans le fixateur $\widehat{N}_\Omega$ de $\Omega$. On note $P_\Omega^{min}=HU_\Omega^{}$, $P_\Omega^{pm}=HU_\Omega^{pm}$ et $P_\Omega^{nm}=HU_\Omega^{nm}$. Ces trois groupes v\'erifient les m\^emes d\'ecompositions que 1), 2) et 3) ci-dessus en rempla\c{c}ant $N^u_\Omega$ par $N_\Omega^{min}$.  On verra en \ref{5.11}.2 que $P_\Omega^{pm}=P_\Omega^{nm}$ quand $\Omega$ est un point sp\'ecial ou quand $\Omega$ est une facette sph\'erique et la valuation est discr\`ete.

\par b)  L'\'egalit\'e $U^{+}_\Omega=U^{pm+}_\Omega$ est \'equivalente \`a $U^{pm}_\Omega=U^{}_\Omega$; on verra en \ref{5.7}.3 qu'elle n'est pas toujours satisfaite.
 En dehors du cas fini ci-dessous, il ne semble pas facile de d\'eterminer quand cette \'egalit\'e a lieu.
 Voir aussi \cite[Remark 3.4]{GR-12} pour le cas affine non tordu.

\par c) On montre \'egalement au cours de la preuve de \ref{4.6} que $U_\Omega^{pm+}.U_\Omega^{nm-}.N_\Omega^{u}$ ou $U_\Omega^{nm-}.U_\Omega^{pm+}.N_\Omega^{u}$ ne d\'epend pas du choix de $\Delta^+$ dans sa classe de $W^v-$conjugaison.

\par d) Dans le cas classique des groupes r\'eductifs, on a $G= G^{pma}=G^{nma}$, $U^{++}_\Omega=U^{+}_\Omega=U^{ma+}_\Omega=U^{pm+}_\Omega$ et $U^{--}_\Omega=U^{-}_\Omega=U^{ma-}_\Omega=U^{nm-}_\Omega$. De plus $U^{}_\Omega$ ($=U^{pm}_\Omega=U^{nm}_\Omega$) est le m\^eme que le groupe d\'efini en \cite[6.4.2 et 6.4.9]{BtT-72}. Le groupe $P^{min}_\Omega$ est not\'e $P_\Omega$ par Bruhat et Tits.
\end{remas*}

\begin{prop}\label{4.7} D\'ecomposition d'Iwasawa

\par Supposons $\Omega$ \'etroit, alors $G=U^+NU_\Omega$.

\par Supposons de plus $\Omega$ presque-ouvert, alors l'application naturelle de $\nu(N)=W_{Y\Lambda}=N/H$ sur $U^+\backslash G/HU_\Omega$ est bijective
\end{prop}

\begin{proof} Voir la preuve de \cite[prop. 3.6]{GR-08}
\end{proof}

\begin{remas*}

\par 1) On a aussi $G=U^-NU_\Omega$ et, de la m\^eme fa\c{c}on avec les groupes maximaux $G^{pma}=U^{ma+}NU_\Omega$ et $G^{nma}=U^{ma-}NU_\Omega$.

\par 2) Ainsi quand $\Omega$ est \'etroit, tout sous-groupe $P$ de $G$ contenant $U_\Omega$ peut s'\'ecrire $P=(P\cap U^+N).U_\Omega$. Si de plus $P\cap U^+N=U^+_P.N_P$ avec $U^+_P=P\cap U^+$ et $N_P=P\cap N$ on a $P=U^+_P.N_P.U_\Omega$ et, si $N_P\subset\widehat N_\Omega$,  $P=U^+_P.U^-_\Omega.N_P$  \cf \ref{4.10}.

\end{remas*}

 \subsection{Repr\'esentations des groupes associ\'es \`a $\Omega$ }\label{4.8}

\par 1) Supposons que $\Omega$ est un ensemble. D'apr\`es \ref{4.5}.4b
le groupe $U_\Omega^{ma+}$ stabilise $\widehat\shu^p_\Omega$ et  $\widehat{\g g}^p_\Omega$ pour l'action de $U^{ma+}$ sur $\widehat\shu^p_K$. Le groupe $U^+$ stabilise $\shu_K$ et $\g g_K$ pour l'action adjointe de $G$ et $U^+$ \cf \ref{2.1}.4. Les deux actions co\"{\i}ncident sur $U^+\subset G\cap U^{ma+}$. Donc $U_\Omega^{pm+}=G\cap U_\Omega^{ma+}=U^+\cap U_\Omega^{ma+}$ (\ref{4.5}.3) stabilise $\shu_\Omega=\widehat\shu^p_\Omega\cap\shu_K$ et $\g g_\Omega=\widehat{\g g}^p_\Omega\cap\g g_K$ pour l'action adjointe de $G$.

\par De m\^eme $U_\Omega^{nm-}$ stabilise $\shu_\Omega$ et $\g g_\Omega$. Ainsi $\shu_\Omega$ et $\g g_\Omega$ sont stables par $U_{\alpha,\Omega}$ pour tout $\alpha\in\Phi$ et donc par $U_\Omega$.

\par Finalement on sait que $\shu_\Omega$ et $\g g_\Omega$ sont stables par les groupes $U_\Omega$, $U_\Omega^{pm}$, $U_\Omega^{nm}$ et $H$ (qui normalise les trois autres) (\ref{4.5}.4a et \ref{4.6}a). Ce r\'esultat est encore vrai si $\Omega$ est un filtre, car alors tous ces objets sont r\'eunion filtrante des objets correspondant \`a $\Omega'\in\Omega$.

\par 2) Soit $M=L_\Z(\lambda)$ un $\shu-$module \`a plus haut poids comme en \ref{4.4}.3. Supposons dans un premier temps que $\Omega$ est un ensemble. Comme $M_\Omega$ est un $\widehat\shu^p_\Omega-$module, le groupe $U_\Omega^{ma+}\subset(\widehat\shu^p_\Omega)^*$ stabilise $M_\Omega$ pour l'action de $G^{pma}$ \cf \ref{3.7}. Ce module $M_\Omega$ est en particulier stable par $U_\Omega^{pm+}$ et les groupes $U_{\alpha,\Omega}$ pour $\alpha\in\Phi^+$.

\par Comme $\widehat{M}^n_\Omega$ est un $\widehat\shu^n_\Omega-$module, le groupe $U_\Omega^{ma-}$ stabilise $\widehat{M}^n_\Omega$, pour la repr\'esentation de $U^{ma-}$ sur $\widehat{M}^n_K$ (\cf \ref{3.7}) qui induit sur $U^-$ la m\^eme action que $G$ (\ref{3.14}.3). Comme $G$ stabilise $M_K$, le groupe $U_\Omega^{nm-}=G\cap U_\Omega^{ma-}=U^-\cap U_\Omega^{ma-}$ stabilise $M_\Omega=M_K\cap\widehat{M}^n_\Omega$.

\par Ainsi le sous-module $M_\Omega$ de $M_K$ est stable par l'action de $U_\Omega^{pm+}$ et $U_\Omega^{nm-}$. Ce r\'esultat est encore vrai quand $\Omega$ est un filtre.

\par Le groupe $U_\Omega$ est engendr\'e par les $U_{\alpha,\Omega}$ pour $\alpha\in\Phi$ qui sont dans $U_\Omega^{pm+}$ ou $U_\Omega^{nm-}$. Ainsi $M_\Omega$ est stable par $U_\Omega^{}$, $U_\Omega^{\pm}$, $U_\Omega^{pm}$ et $U_\Omega^{nm}$. On sait aussi qu'il est stable par $H$.

\par 3) Ces r\'esultats sont encore vrais, mutatis mutandis, pour des modules \`a plus bas poids.

\begin{defi}\label{4.9}  Si $\Omega$ est un ensemble, on consid\`ere le sous-groupe $\widetilde{P}_\Omega$ de $G$ form\'e des \'el\'ements stabilisant $\shu_\Omega$ et $M_\Omega$ pour tout module $M$ \`a plus haut ou plus bas poids (comme en \ref{4.4}.3). On note $\widetilde{N}_\Omega=N\cap\widetilde{P}_\Omega$.

\end{defi}

\par D'apr\`es \ref{4.8} le groupe $\widetilde{P}_\Omega$ contient $U_\Omega^{}$, $U_\Omega^{\pm}$, $U_\Omega^{pm}$,  $U_\Omega^{nm}$ et $H$. D'apr\`es la remarque \ref{4.7}.2 il est utile de d\'eterminer $\widetilde{P}_\Omega\cap U^\pm N$, voir \ref{4.11}.

\par On note $\shp$ la r\'eunion de $\Delta$ et de tous les ensembles de poids des modules $M$ ci-dessus. Le groupe $\widetilde{P}_\Omega$ ne d\'epend que de $cl_\shp(\Omega)$ (et non seulement de $cl(\Omega)$).

 \subsection{Conjugaison par $N$}\label{4.10}

 \par Soit $n\in N$, on peut \'ecrire $n=n_0t$ avec $n_0$ dans $N_0=\g N(\Z)$ (le groupe engendr\'e par les $m(x_\alpha(\pm1))$ qui fixe $0$) \cf \ref{4.2}.4, $\nu^v(n)=w\in W^v$ et $t\in T$.

 \par 1) Pour $M=L_K(\lambda)$ un module \`a plus haut ou plus bas poids ou $M=\g g_K$ ou $M=\shu_K$, $\mu\in X$ et $r\in \R$, on a $nM_{\mu,r}=M_{w\mu,r+\omega(\mu(t))}$. Consid\'erons \'egalement l'action sur $\A$: $nD(\mu,r)=n_0D(\mu,r+\omega(\mu(t)))=D(w\mu,r+\omega(\mu(t)))$. Donc $r\geq{}f_\Omega(\mu)\Leftrightarrow\Omega\subset D(\mu,r)\Leftrightarrow\nu(n).\Omega\subset D(w\mu,r+\omega(\mu(t)))\Leftrightarrow r+\omega(\mu(t))\geq{}f_{\nu(n).\Omega}(w\mu)$ et ainsi $f_{\nu(n).\Omega}(w\mu)=f_\Omega(\mu)+\omega(\mu(t))$. Finalement $nM_\Omega =M_{\nu(n)\Omega}$.

 \par Ainsi $\widetilde{P}_{\nu(n).\Omega}=n\widetilde{P}_\Omega n^{-1}$; en particulier le fixateur $\widehat{N}_\Omega$ de $\Omega$ dans $N$ (pour l'action $\nu$) normalise $\widetilde{P}_\Omega$. On a vu en \ref{4.3}.4 que $N_\Omega^{min}=HN_\Omega^{u}\subset\widehat N_\Omega$.

 \par 2) On a aussi $nU_{\alpha, r}n^{-1}=U_{w\alpha,r+\omega(\mu(t))}$, donc $nU_\Omega n^{-1}=U_{\nu(n).\Omega}$, $nN^u_\Omega n^{-1}=N^u_{\nu(n).\Omega}$ et $\widehat{N}_\Omega$ normalise les groupes $U_\Omega$, $P_\Omega$, $N^u_\Omega$ et $N_\Omega^{min}$.

 \par 3) D'apr\`es 1) on a $Ad(n)\shu_\Omega(\Psi)=\shu_{\nu(n).\Omega}(w\Psi)$ pour $\Psi\subset\Delta^+$. Pour $n=m(x_\alpha(\pm1))$ ($\alpha$ racine simple) et $\Psi_\alpha=\Delta^+\cap s_\alpha(\Delta^+)=\Delta^+\setminus\{\alpha\}$, on peut passer aux compl\'et\'es: $Ad(n).\widehat\shu^p_\Omega(\Psi_\alpha)=\widehat\shu^p_{\nu(n).\Omega}(\Psi_\alpha)$. Donc $nU_\Omega^{ma}(\Psi_\alpha)n^{-1}=U_{\nu(n).\Omega}^{ma}(\Psi_\alpha)$; plus pr\'ecis\'ement pour $\beta\in\Psi_\alpha$, $x\in\g g_{\beta\Z}$ et $\lambda\in K$, $n.[\exp]\lambda x.n^{-1}$ est une exponentielle tordue associ\'ee \`a $\lambda'=\lambda\beta(t)\in K$ et $Ad(n_0).x\in\g g_{s_\alpha(\beta)\Z}$ et cela passe aux produits infinis.

 \par D'apr\`es 2) et \ref{4.5}.6 on a $nU_\Omega^{ma}(\Delta^+)n^{-1}=U_{\nu(n).\Omega}^{ma}(w\Delta^+)$ si $w=s_\alpha$. Ce r\'esultat s'\'etend alors au cas g\'en\'eral pour $w$. Par intersection avec $G$, on a $nU_\Omega^{pm}(\Delta^+)n^{-1}=U_{\nu(n).\Omega}^{pm}(w\Delta^+)$.

 \par D'apr\`es \ref{4.6}, on a $nU_\Omega^{pm}n^{-1}=U_{\nu(n).\Omega}^{pm}$ et $nU_\Omega^{nm}n^{-1}=U_{\nu(n).\Omega}^{nm}$. En particulier $\widehat{N}_\Omega$ normalise $U_\Omega^{pm}$ et $U_\Omega^{nm}$.

 \begin{lemm}\label{4.11} Si $\Omega$ est un ensemble, $\widetilde{P}_\Omega\cap U^+ N=U_\Omega^{pm+}.\widetilde{N}_\Omega$ et $\widetilde{P}_\Omega\cap U^- N=U_\Omega^{nm-}.\widetilde{N}_\Omega$. De plus $\widetilde{N}_\Omega$ est le stabilisateur (dans $N$ pour l'action $\nu$ sur $\A$) du $\shp-$enclos $cl_\shp(\Omega)$ de $\Omega$, il normalise $U_\Omega$, $U_\Omega^{pm+}$, $U_\Omega^{nm-}$,$\cdots$
\end{lemm}
\par {\bf N.B.} En particulier $\widetilde{N}_\Omega\supset\widehat{N}_\Omega\supset N^{min}_\Omega$.
\begin{proof} On refait la preuve de \cite[3.10]{GR-08} car il y a quelques changements substantiels.

\par a) Soient $n\in N$ et $u\in U^+$ tels que $un\in\widetilde{P}_\Omega$ et $w=\nu^v(n)$. Pour $\shm=M_\Omega$ ou $\g g_\Omega$ ou $\shu_\Omega$, $g\in \widetilde{P}_\Omega$ et $\mu,\mu'\in X$, on d\'efinit $_{\mu'}\vert g\vert_\mu$ comme la restriction de $g$ \`a $\shm_\mu$ suivie de la projection sur $\shm_{\mu'}$ (parall\`element aux autres espaces de poids). Pour tout $\mu\in X$, $_{w\mu}\vert un\vert_\mu=\,_{w\mu}\vert n\vert_\mu$ et $n=\oplus_\mu\,_{w\mu}\vert n\vert_\mu$ (en un sens \'evident), donc $n\in\widetilde{N}_\Omega$. Ainsi $\widetilde{P}_\Omega\cap U^+ N=(\widetilde{P}_\Omega\cap U^+).\widetilde{N}_\Omega$. Il reste \`a d\'eterminer $\widetilde{N}_\Omega$ et $\widetilde{P}_\Omega\cap U^+$ (ainsi que $\widetilde{P}_\Omega\cap U^-$).

\par b) On a vu en \ref{4.8} que $U_\Omega^{pm+}\subset\widetilde{P}_\Omega\cap U^+ $. Inversement soit $u\in U^{ma+}$ stabilisant $\shu_\Omega$ et donc $\widehat\shu^p_\Omega$. On peut \'ecrire $u=\prod_{\alpha\in\Delta^+}\,u_\alpha$ avec les conventions suivantes: l'ordre des facteurs $u_\alpha$ est tel que la hauteur cro\^{\i}t de droite \`a gauche et, pour $\alpha\in\Delta^+$, $u_\alpha=[\exp]t_{\alpha1}e_{\alpha1}.\cdots.[\exp]t_{\alpha s}e_{\alpha s}$ avec $e_{\alpha1},\cdots,e_{\alpha s}$ une base de $\g g_{\alpha\Z}$ et les $t_{\alpha i}$ dans $K$ (si $\alpha\in\Phi^+$ on a $s=1$ et $[\exp]=\exp$). Nous allons montrer que $u\in U_\Omega^{ma+}$ et pour cela que, $\forall\alpha\in\Delta^+$, $\omega(t_{\alpha i})\geq{}f_\Omega(\alpha)$. Par r\'ecurrence on peut supposer que c'est vrai pour $u_{\alpha'}$ \`a droite de $u_\alpha$, alors ces $u_{\alpha'}$ sont dans $U_\Omega^{ma+}$ et stabilisent $\widehat\shu^p_\Omega$, on peut donc les supposer  \'egaux \`a $1$. Donc $u=(\prod^{\beta\not=\alpha}_{ht(\beta)\geq{}ht(\alpha)}\,u_\beta).u_\alpha$ et ainsi $_\alpha\vert u\vert_0=\,_\alpha\vert u_\alpha\vert_0$.

\par Choisissons un \'el\'ement $h\in\shb_0$ (base de $\g h_\Z$) tel que $\alpha(h)=m\not=0$; quitte \`a changer $h$ en $-h$ et donc \`a utiliser une autre base de $\shu_\Z$, on peut supposer $m>0$. Par ailleurs $Ad([\exp]t_{\alpha i}e_{\alpha i})=\sum^\infty_{p=0}\,t_{\alpha i}^p ad(e_{\alpha i}^{[p]})$ et  $Ad(u_\alpha)\left(\begin{array}{c}h \\ n\end{array}\right)$ a pour terme de poids $\alpha$ l'expression $\sum_{i=1}^s\,t_{\alpha i}ad(e_{\alpha i})\left(\begin{array}{c}h \\ n\end{array}\right)=\sum_{i=1}^s\,t_{\alpha i}(e_{\alpha i}\left(\begin{array}{c}h \\ n\end{array}\right)-\left(\begin{array}{c}h \\ n\end{array}\right)e_{\alpha i})=-\sum_{i=1}^s\,t_{\alpha i}e_{\alpha i}(\left(\begin{array}{c}h+\alpha(h) \\ n\end{array}\right)-\left(\begin{array}{c}h \\ n\end{array}\right))=-\sum_{i=1}^s\,t_{\alpha i}e_{\alpha i}(\sum_{q=0}^{n-1}\,\left(\begin{array}{c}h \\ q\end{array}\right)\left(\begin{array}{c}\alpha(h) \\ n-q\end{array}\right))$ et doit \^etre dans $\shu_{\alpha,\Omega}$. Comme les $e_{\alpha i}\left(\begin{array}{c}h \\ q\end{array}\right)$ pour $1\leq i\leq s$ et $q\in\N$ font partie d'une base de $\shu_\Z$, on doit avoir $\omega(t_{\alpha i}\left(\begin{array}{c}\alpha(h) \\ n-q\end{array}\right))\geq f_\Omega(\alpha)$ pour $1\leq i\leq s$ et $q\leq n-1$. Pour $n=m$ et $q=0$ on obtient le r\'esultat cherch\'e $\omega(t_{\alpha i})\geq f_\Omega(\alpha)$

\par c) Reprenons les notations de \ref{4.10}:  $nM_{\mu,r}=M_{w\mu,r+\omega(\mu(t))}$ et $nD(\mu,r)=D(w\mu,r+\omega(\mu(t)))$. Donc $n\in\widetilde{P}_\Omega$ si et seulement si, $\forall\mu\in\shp$, $f_\Omega(\mu)+\omega(\mu(t))=f_\Omega(w\mu)$ c'est-\`a-dire
$nD(\mu,f_\Omega(\mu))=D(w\mu,f_\Omega(w\mu))$; c'est \'equivalent au fait que $n$ stabilise l'ensemble $cl_\shp(\Omega)$.

\par Le groupe $\widetilde{N}_\Omega$  normalise $U_\Omega$, $U_\Omega^{pm+}$, $U_\Omega^{nm-}$,$\cdots$ car ces groupes ne d\'ependent que de $cl_\shp(\Omega)$.
\end{proof}

\begin{exems}\label{4.12}  1) Soit $x$ un point sp\'ecial de $\A$ et $\Omega=\{x\}$, alors $cl_\shp(x)$ est une partie born\'ee de $\{y\in\A\mid\alpha(y)=\alpha(x),\forall\alpha\in\Phi\}=cl(x)$, car tout \'el\'ement de $X$ est combinaison lin\'eaire \`a coefficients dans $\Q^+$ d'\'el\'ements de $\shp$. Donc $\widetilde{N}_x=\widehat{N}_x=N_x^{min}$ (car $\nu^v(N_x^{min})=W^v$ \cf \ref{4.3}). D'apr\`es \ref{4.7} et \ref{4.11}, on a $\widetilde{P}_x=
U_x^{pm+}.N_x^{min}.U_x^{}=U_x^{pm+}.U_x.N_x^{min}=P_x^{pm}=
P_x^{nm}$.

\medskip
\par 2) Si de plus $x=0$ est l'origine de $\A$, on a $\g g_0=\g g_\Z\otimes\sho$ et $M_0=M_\Z\otimes\sho$ pour tout module \`a plus haut ou plus bas poids. Il est clair que ces modules sont stables par $\g G(\sho)$ (\ref{3.7} et \ref{3.14}). Donc $\widetilde{P}_0\supset\g G(\sho)$. On a $H,U_0\subset\g G(\sho)$ par construction et aussi $U_0^{++}=\g U^+(\sho)$. Comme $U_0^{ma+}= \g U^{ma+}(\sho)$ (\ref{3.2} et \ref{4.5}.2), on a $U_0^{pm+}=\g U^+(K)\cap\g U^{ma+}(\sho)$.

\par On se demande si cette derni\`ere intersection est \'egale \`a $\g U^+(\sho)\subset\g U^+_\#(\sho)$ (resp. $\g U^+_\#(\sho)$) (voir \ref{1.6}.5). Cela prouverait que $\widetilde{P}_0=P_0^{pm}=P_0^{nm}=\g G(\sho)$ (resp. peut-\^etre $\g G_\#(\sho)$). Un r\'esultat analogue a \'et\'e prouv\'e dans \cite[3.14]{GR-08} quand le corps r\'esiduel de $K$ contient $\C$. Mais il concerne le groupe minimal $\g G$ \`a la Kumar avec sa structure de ind-groupe (sur $\C$ mais \'etendue aux $\C-$alg\`ebres). Il pourrait donc ne pas avoir de rapport avec le r\'esultat cherch\'e ici, car \ref{3.19} ne donne une identification que pour les points rationnels sur $\C$ des groupes maximaux ou minimaux. Le foncteur en groupes minimal $\g G$ n'est pas forc\'ement le meilleur choix sur tous les anneaux, un foncteur $\g G_\#$, comme en \ref{1.6}.1, peut \^etre plus adapt\'e, voir 3) ci-dessous.

\par D'apr\`es \ref{3.17} et \ref{1.6}.1, \ref{1.6}.5 (KMG7), on a toujours $\g G(K)\cap\g U^{ma+}(K)=\g U^+(K)$ et $\g U^+_\#(\sho)=\g U^+_\#(K)\cap\g G_\#(\sho)$ mais cela ne permet pas de r\'epondre aux questions.

\par 3) Cas de $\widetilde{SL}_2$ : On consid\`ere le SGR $\shs=(\left(\begin{matrix}2&-2  \cr -2&2\cr \end{matrix}\right),Y=\Z h,\overline\alpha_0=-\overline\alpha_1,\alpha^\vee_0=-\alpha^\vee_1=-h)$. L'alg\`ebre de Kac-Moody correspondante est $\g g=\g{sl}_2(\C)\otimes_\C\C[t,t^{-1}]$ (qui est ni libre, ni colibre, voir \ref{2.12} o\`u cette alg\`ebre est not\'ee $\g g'$). Si on pose $\delta=\alpha_0+\alpha_1\in Q$, on a $\Delta^+_{im}=\N^*\delta$, $\Delta^+_{re}=\{\alpha_1+n\delta, \alpha_0+n\delta\mid n\in\N\}$ et les espaces propres sont, pour $n\in\Z$, $\g g_{n\delta}=\C\left(\begin{matrix}t^n&0  \cr 0&-t^n\cr \end{matrix}\right)$, $\g g_{\alpha_1+n\delta}=\C\left(\begin{matrix}0&t^n \cr 0&0\cr \end{matrix}\right)$ et $\g g_{\alpha_0+n\delta}=\C\left(\begin{matrix}0&0  \cr t^{n+1}&0\cr \end{matrix}\right)$.

\par Comme $\shs$ est non libre on consid\`ere la masure essentielle et l'appartement $\A$ correspondant du d\'ebut de cette section: $\qa_0$ et $\qa_1$ forment donc une base de l'espace vectoriel r\'eel $V^*$.

\par a) La repr\'esentation naturelle $\pi$ de $\g g_\Z$ sur $\End_{\Z[t,t^{-1}]}({\Z[t,t^{-1}]}^2)$ (\ref{2.12}) induit une identification de $G=\g G(K)$ et $SL_2({K[t,t^{-1}]})$ qui envoie $U^+=\g U^+(K)$ sur le groupe des matrices de $SL_2(K[t])$ qui sont triangulaires sup\'erieures strictes modulo $t$. Le foncteur en groupes $\g G_\#$ de \ref{1.6}.1 peut \^etre choisi tel que $\g G_\#(k)=SL_2({k[t,t^{-1}]})$ pour tout anneau $k$.
 On sait que le groupe $U^+$ est produit libre des sous-groupes $u^s(K[t])=\left(\begin{matrix}1&K[t] \cr 0&1\cr \end{matrix}\right)=\langle U_{\alpha_1+n\delta}(K)\mid n\in\N\rangle$
 et $u^i(tK[t])=\left(\begin{matrix}1&0\cr tK[t]&0\cr \end{matrix}\right)=\langle U_{\alpha_0+n\delta}(K)\mid n\in\N\rangle$, \cf \cite[3.10d]{T-87b}; et on a :

 \par\noindent$\left(\begin{matrix}1&0  \cr \varpi&1\cr \end{matrix}\right)\left(\begin{matrix}1&t  \cr 0&1\cr \end{matrix}\right)\left(\begin{matrix}1&0  \cr -\varpi&1\cr \end{matrix}\right)=\left(\begin{matrix}1-\varpi t&t  \cr -\varpi^2t&1+\varpi t\cr \end{matrix}\right)=\left(\begin{matrix}1&\varpi^{-1}  \cr 0&1\cr \end{matrix}\right)\left(\begin{matrix}1&0  \cr -\varpi^2t&1\cr \end{matrix}\right)\left(\begin{matrix}1&-\varpi^{-1}  \cr 0&1\cr \end{matrix}\right)$

 \par\noindent o\`u on note $\varpi$ un \'el\'ement non nul de l'id\'eal maximal $\g m$ de $\sho$ (\eg une uniformisante dans le cas discret).

 \par Cet \'el\'ement $g\in G$ est dans $U_0$ (d'apr\`es l'expression de gauche) et dans $U^+$ (expression du milieu ou de droite) donc dans $U^+_0$. Cependant l'expression de droite montre que $g$ n'est pas dans $U_0^{++}$, puisqu'il y a unicit\'e des d\'ecompositions dans le produit  libre $U^+$. On a donc une inclusion stricte: $\g U^+(\sho)=U_0^{++}\subsetneqq U_0^{+}$.

 \par b) La repr\'esentation naturelle $\pi$ permet d'identifier $G^{pma}=\g G^{pma}(K)$ \`a $SL_2(K(\!(t)\!))$; voir \cite[13.2.8]{Kr-02} pour un r\'esultat voisin. Dans cette identification $\g U^{ma+}(K)$ (resp. $\g U^{ma+}(\sho)$) devient le groupe des matrices de $SL_2(K[[t]])$ (resp. $SL_2(\sho[[t]])$) qui sont triangulaires sup\'erieures strictes modulo $t$ (l'ingr\'edient essentiel de la d\'emonstration  a \'et\'e expliqu\'e en \ref{2.12}). On s'est demand\'e en 2) si $U_0^{pm+}=\g U^{ma+}(\sho)\cap\g U^+(K)=\g U^{ma+}(\sho)\cap SL_2({K[t,t^{-1}]})$ est \'egal \`a  $\g U^+(\sho)$ (c'est impossible d'apr\`es a)\,) ou \'eventuellement \`a  $\g U^+_\#(\sho)=\g U^+(K)\cap\g G_\#(\sho)$. Ce dernier groupe est form\'e des matrices de $SL_2(\sho[t])$ qui sont triangulaires sup\'erieures strictes modulo $t$, il est donc bien  \'egal \`a $U_0^{pm+}=\g U^{ma+}(\sho)\cap\g G(K)$.

 \par On a donc:
 \par\noindent $\g U^+(\sho)=U_0^{++}\subsetneqq U_0^{+}=U_0\cap U^+\subset\g G(\sho)\cap U^+\subset \g G_\#(\sho)\cap U^+= \g U^+_\#(\sho)=U_0^{pm+}$.

\par En particulier $\g U^+\not=\g U^+_\#$. Mais en fait $U_0= \g G_\#(\sho)$ et  $U_0^{pm+}=U_0^+=\g U^+_\#(\sho)$, car $U_0$ est le groupe engendr\'e par les matrices \'el\'ementaires de $SL_2({\sho[t,t^{-1}]})$ et on sait que celui-ci est \'egal \`a $SL_2({\sho[t,t^{-1}]})= \g G_\#(\sho)$ \cite[th. 3.1]{Cu-84} (\footnotemark[2]). \footnotetext[2]{Merci \`a Leonid Vaserstein pour cette r\'ef\'erence}

\par c) On consid\`ere $\QO=\{0,z\}\subset\A$ avec $z$ d\'etermin\'e par $\qd(z)=0$ et $\qa_1(z)=p\in\N$.
 Alors $U_\QO^{ma+}$ est topologiquement engendr\'e par $u^s(\sho[[t]])$ et $u^i(\varpi^pt\sho[[t]])$ dans $SL_2(K[[t]])$. On en d\'eduit assez facilement que $U_\QO^{ma+}$ (resp. $U_\QO^{pm+}$) est contenu dans  (en fait \'egal \`a) l'ensemble des matrices
$ \left(\begin{matrix}a&b  \cr c&d\cr \end{matrix}\right)\in SL_2(\sho[[t]])$ (resp. $SL_2(\sho[t])$) telles que $a\equiv1$, $d\equiv1$ et $c\equiv0$ modulo $\varpi^pt$.
 De m\^eme $U_\QO^{ma-}$ est topologiquement engendr\'e par $u^s(t^{-1}\sho[[t^{-1}]])$ et $u^i(\varpi^p\sho[[t^{-1}]])$ dans $SL_2(K[[t^{-1}]])$; donc $U_\QO^{ma-}$ (resp. $U_\QO^{nm-}$) est contenu dans  (en fait \'egal \`a) l'ensemble des matrices
$ \left(\begin{matrix}a&b  \cr c&d\cr \end{matrix}\right)\in SL_2(\sho[[t^{-1}]])$ (resp. $SL_2(\sho[t])$) telles que $a,d\equiv1$ modulo $\varpi^pt^{-1}$, $b\equiv0$ modulo $t^{-1}$ et $c\equiv0$ modulo $\varpi^p$.
 Ainsi $U_\QO^{pm+}$ et $U_\QO^{nm-}$ sont contenus dans le groupe $V_\QO$ des matrices
 $ \left(\begin{matrix}a&b  \cr c&d\cr \end{matrix}\right)\in SL_2(\sho[t,t^{-1}])$ telles que  $a,d\equiv1$ et  $c\equiv0$ modulo $\varpi^p$.
  D'autre part (si $p\geq{}1$) $\widehat N_\QO$ est form\'e des matrices diagonales avec $d^{-1}=a=u.t^n$ pour $u\in\sho^*$ et $n\in\Z$.
  On voit facilement que, si $p\geq{}2$, la matrice
  $ \left(\begin{matrix}1+\varpi^{p-1}t&1  \cr -\varpi^{2p-2}t^2&1-\varpi^{p-1}t\cr \end{matrix}\right)$
   est dans $G_\QO$ mais pas dans $V_\QO.\widehat N_\QO$.
   Donc $\QO$ ne satisfait \`a aucune des conditions (GF$\pm$) de \ref{5.3} ci-dessous (si $p\geq{}2$).

\par 4) On consid\`ere toujours $\widetilde{SL}_2$ avec son appartement $\A$ d'espace vectoriel associ\'e $V$ et $\qa_1$, $\qa_0=\qd-\qa_1$ base du dual $V^*$. On suppose $\QL=\Z$.
 Les murs ont alors pour \'equations $(\pm{}\qa_1+n\qd)(x)+m=0$ avec $n,m\in\Z$ et les coracines v\'erifient $\qa_0^\vee=-\qa_1^\vee$.
  Le groupe de Weyl (affine) associ\'e est $W=W^v\ltimes\Z.\qa_1^\vee$. Le groupe de Weyl vectoriel $W^v$ contient la transvection $\qs$ donn\'ee par $\qs(x)=x-\qd(x).\qa_1^\vee$.
  Si l'on consid\`ere la translation $\qt_{\qa_1^\vee}$ de vecteur $\qa_1^\vee$, l'\'el\'ement $\qt_{\qa_1^\vee}\circ\qs$ de $W$ fixe tous les points de l'hyperplan affine d'\'equation $\qd(x)=1$.
  Cependant le point $y$ tel que $\qd(y)=1$ et $\qa_1(y)=\frac{1}{2}$ n'est dans aucun mur. Ainsi $W_y^{min}=\{1\}$ alors que $W_y$ est infini.
  Ce ph\'enom\`ene dispara\^{\i}t si on consid\`ere un espace vectoriel $V$ o\`u $\qa_0^\vee$ et $\qa_1^\vee$ sont ind\'ependants.

\end{exems}

\begin{defi}\label{4.13} Soit $\Omega$ un sous-ensemble de $\A$. On d\'efinit le groupe $\widehat{P}_\Omega$ comme l'intersection des groupes $\widetilde{P}^L_{\Omega'}$ o\`u $\Omega'$ est une partie non vide de $\overline\Omega$, $L/K$ une extension valu\'ee et  $\widetilde{P}^L_{\Omega'}$ est le groupe construit de mani\`ere analogue \`a $\widetilde{P}_{\Omega'}$ dans $\g G(L)$. (On sait que $\g G(K)$ s'injecte fonctoriellement dans $\g G(L)$ \cf \ref{1.6}.)

\par Si $\Omega$ est un filtre de parties de $\A$, on d\'efinit $\widehat{P}_\Omega=\cup_{\Omega'\in\Omega}\,\widehat{P}_{\Omega'}$.

\par Si $\Omega$ est un point ou une facette de $\A$, on dira que $\widehat{P}_\Omega$ est le "fixateur" de $\Omega$,  voir ci-dessous \ref{5.1} et \ref{5.7}.2.
\end{defi}

\begin{enonce*}[definition]{Propri\'et\'es}
\par 1) On a une fonctorialit\'e \'evidente en le corps des diff\'erents groupes ou alg\`ebres  d\'efinis jusqu'en \ref{4.6} (\ie sauf $\widetilde{P}_\Omega$ et $\widehat{P}_\Omega$). En particulier le sous-groupe $\widehat{P}_\Omega$ de $\widetilde{P}^L_{\Omega}$ contient les groupes $P_\Omega^{pm}$, $P_\Omega^{nm}$, $P_\Omega$ et $\widehat{N}_\Omega$. D'apr\`es \ref{4.10}, pour $n\in N$ on a $\widehat{P}_{\nu(n)\Omega}=n\widehat{P}_{\Omega}n^{-1}$. Bien s\^ur si $\Omega\subset\Omega'$, on a $\widehat{P}_{\Omega}\supset \widehat{P}_{\Omega'}$.

\par 2) Soient $\Omega$ un sous-ensemble de $\A$ et $x$ un point de $\overline\Omega$. Il existe une extension valu\'ee $L/K$ telle que $x$ soit un point sp\'ecial de $\A^L$ (\eg si $\omega(L^*)=\R$).
D'apr\`es \ref{4.12}.1 on a $N(L)\cap\widetilde{P}^L_{x}=\widehat{\g N(L)}_x$, donc $N\cap\widehat{P}_\Omega=\widehat{N}_{\overline\Omega}=\widehat{N}_{\Omega}$. Cette relation est encore valable si $\Omega$ est un filtre. Par contre, m\^eme pour un ensemble, on ne sait pas si l'inclusion $\cap_{x\in\Omega}\,\widehat{P}_x\subset \widehat{P}_\Omega$ est toujours une \'egalit\'e.

\par 3) Si $\Omega$ est un ensemble, $U_\Omega^{pm+}.\widehat{N}_\Omega\subset \widehat{P}_\Omega\cap U^+N\subset \cap_{\Omega',L}\, \widetilde{P}^L_{\Omega'}\cap\g U^+(L)\g N(L)= \cap_{\Omega',L}\, \g U_{\Omega'}^{pm+}(L)\widetilde{N}^L_{\Omega'}=  (\cap_{\Omega',L}\, \g U_{\Omega'}^{pm+}(L)).(\cap_{\Omega',L}\, \widetilde{N}^L_{\Omega'})=U_\Omega^{pm+}.\widehat{N}_\Omega$, d'apr\`es 1), 2) ci-dessus, \ref{4.11} et l'unicit\'e dans les  d\'ecompositions de \ref{3.16}. Donc $\widehat{P}_\Omega\cap U^+N=U_\Omega^{pm+}.\widehat{N}_\Omega$, $\widehat{P}_\Omega\cap U^+=U_\Omega^{pm+}$ et ceci est encore valable pour un filtre.

\par 4) Si $\Omega$ est \'etroit on a, d'apr\`es \ref{4.7}, $\widehat{P}_\Omega=U_\Omega^{pm+}.\widehat{N}_\Omega.U_\Omega=U_\Omega^{pm+}.U_\Omega.\widehat{N}_\Omega=U_\Omega^{pm+}.U_\Omega^-.\widehat{N}_\Omega=U_\Omega^{nm-}.U_\Omega^+.\widehat{N}_\Omega$. Ainsi $\widehat{P}_\Omega=P_\Omega^{pm}.\widehat{N}_\Omega=P_\Omega^{nm}.\widehat{N}_\Omega$.

\par 5) On ne va utiliser $\widehat{P}_\Omega$ que si $\Omega$ est un point ou une facette et, dans ce cas, on ne d\'efinira que plus tard (\ref{5.11}) le sous-groupe parahorique $P_\Omega$ (avec $P_\Omega^{pm}=P_\Omega^{nm}=P_\Omega\subset \widehat{P}_\Omega$).

\par Les propri\'et\'es des groupes $\widehat{P}_x$ sont r\'esum\'ees dans la proposition suivante. Il n'est pas exclus que, pour tout $x$ dans $\A$, $\widehat{P}_x=U_x.\widehat{N}_x$ \ie $U_x^{pm+}=U_x^{+}$ et $U_x^{nm-}=U_x^{-}$ \cf \ref{4.12}.3b. En \ref{5.7}.3 on verra que ceci ne peut se g\'en\'eraliser \`a tout ensemble $\QO$ \`a la place de $x$.
\end{enonce*}

\begin{prop}\label{4.14} Les groupes $\widehat{P}_x$ associ\'es aux points de $\A$ satisfont aux propri\'et\'es suivantes:

\par (P1)  $\widehat{P}_x\cap N=\widehat{N}_x$ (le fixateur de $x$ dans $N$).

\par (P2) $n\widehat{P}_xn^{-1}=\widehat{P}_{\nu(n)x}$.

\par (P3) $\widehat{P}_x=U_x^{pm+}.U_x^{nm-}.\widehat{N}_x=U_x^{nm-}.U_x^{pm+}.\widehat{N}_x$ avec $U_x^{pm+}=\widehat{P}_x\cap U^+$ et $U_x^{nm-}=\widehat{P}_x\cap U^-$.
\end{prop}

\par On a en fait (P3') $\widehat{P}_x=U_x^{pm+}.U_x^{-}.\widehat{N}_x=U_x^{nm-}.U_x^{+}.\widehat{N}_x$

\subsection{Comparaison avec les raisonnements de \cite{GR-08}}\label{4.15}

\par La g\'en\'eralisation des r\'esultats de \cite{GR-08} (en particulier le lemme \ref{4.11}) \`a la caract\'eristique r\'esiduelle positive a n\'ecessit\'e, sans surprise, le remplacement de l'alg\`ebre de Lie $\g g_\Omega$ par l'alg\`ebre "enveloppante" $\shu_\Omega$.

\par Les raisonnements de  \cite[sect. 3.7]{GR-08} sont compliqu\'es, entach\'es d'une ou deux erreurs et utilisent la sym\'etrisabilit\'e. C'\'etait sans doute trop ambitieux de d\'efinir \`a ce stade les groupes $P_\Omega$ au lieu des groupes $\widehat{P}_\Omega$.  On a pu simplifier et g\'en\'eraliser en restreignant l'ambition et en utilisant des extensions de corps (comme sugg\'er\'e par les travaux de Cyril Charignon \cite{Cn-10} ou \cite{Cn-10b}).


\section{La masure affine ordonn\'ee}\label{s5}

\par Dans cette section on ne va utiliser que les propri\'et\'es (P1) \`a (P3) des groupes $\widehat{P}_x$ indiqu\'ees dans la proposition \ref{4.14}. La mention qui sera faite des groupes $\widehat{P}_\Omega$ plus g\'en\'eraux n'est pas indispensable au raisonnement.

\begin{defi}\label{5.1} La {\it masure} $\shi=\shi(\g G,K)$ de $\g G$ sur $K$ est le quotient de l'ensemble $G\times\A$ par la relation:

\par $(g,x)\sim(h,y)\Leftrightarrow$ il existe $n\in N$ tel que $y=\nu(n).x$ et $g^{-1}hn\in\widehat{P}_x$.

\par Il est clair que $\sim$ est une relation d'\'equivalence \cite[7.4.1]{BtT-72}.
\end{defi}

\par L'application $\A\rightarrow\shi$, $x\mapsto cl(1,x)$ est injective d'apr\`es (P1). Elle identifie $\A$ \`a son image $A(T)=A(T,K)$, l'{\it appartement} de $T$ dans $\shi$.

\par L'action \`a gauche de $G$ sur $G\times\A$ induit une action de $G$ sur $\shi$. Les appartements de $\shi$ sont les $g.A(T)$ pour $g\in G$. L'action de $N$ sur $A(T)$ se fait via $\nu$; en particulier $H$ fixe (point par point) $A(T)$. Par construction le fixateur de $x\in\A$ est $G_x=\widehat{P}_x$ et, pour $g\in G$ on a $gx\in\A\Leftrightarrow g\in N\widehat{P}_x$.

\par Comme $\widehat{P}_x$ contient $U_x$, il est clair que, $\forall\alpha\in\Phi$, $\forall r\in K$, $x_\alpha(r)$ fixe $D(\alpha,\omega(r))$. Donc, pour $k\in\R$, le groupe $HU_{\alpha,k}$ fixe $D(\alpha,k)$.

\subsection{Les groupes $G_\Omega$}\label{5.2}

\par Pour $\Omega$ une partie de $\A$, on note $G_\Omega=\cap_{x\in\Omega}\,\widehat{P}_x$ et pour $\Omega$ un filtre, $G_\Omega=\cup_{\Omega'\in\Omega}\,G_{\Omega'}$. Le groupe $G_\Omega$ est le fixateur de $\Omega$ (pour l'action de $G$ sur $\shi$ qui contient $\A$).

\par Par construction de  $\widehat{P}_\Omega$, on a  $G_\Omega\supset\widehat{P}_\Omega$ (donc $G_\Omega\supset U_\Omega^{pm+},U_\Omega^{nm-}$) et $G_\Omega\cap N= \widehat{N}_\Omega=\widehat{P}_\Omega\cap N$.

\par D'apr\`es \ref{4.5}.4d et (P3) on a $G_\Omega\cap U^+= U_\Omega^{pm+}$ et $G_\Omega\cap U^-= U_\Omega^{nm-}$.

\begin{lemm*} Le sous-ensemble $G(\Omega\subset\A)$ de $G$ consistant en les $g\in G$ tels que $g\Omega\subset\A$ est $G(\Omega\subset\A)=\cup_{\Omega'\in\Omega}(\cap_{x\in\Omega'}\,N\widehat{P}_x)$.
\end{lemm*}

\begin{proof}
$g\Omega\subset\A\Leftrightarrow\exists\Omega'\in\Omega, g\Omega'\subset\A \Leftrightarrow\exists\Omega'\in\Omega, \forall x\in\Omega', gx\in\A\Leftrightarrow\exists\Omega'\in\Omega, \forall x\in\Omega', g\in N\widehat{P}_x$.
\end{proof}

\begin{defi}\label{5.3} On consid\`ere les conditions suivantes:

\par\qquad (GF+) $G_\Omega=U_\Omega^{pm+}.U_\Omega^{nm-}.\widehat{N}_\Omega$.

\par\qquad (GF$-$) $G_\Omega=U_\Omega^{nm-}.U_\Omega^{pm+}.\widehat{N}_\Omega$.

\par\qquad\; (TF) \;$G(\Omega\subset\A)=NG_\Omega$.

\par On dit que $\Omega$ a un {\it fixateur transitif} s'il satisfait (TF).

\par On dit que $\Omega$ a un {\it assez bon fixateur} s'il satisfait (TF) et (GF+) ou (GF$-$).

\par On dit que $\Omega$ a un {\it  bon fixateur} s'il satisfait (TF), (GF+) et (GF$-$).

\end{defi}

\par D'apr\`es la remarque \ref{4.6}c, cette d\'efinition ne d\'epend pas du choix de $\Delta^+$ dans sa classe de $W^v-$conjugaison. La condition (GF+) ou (GF$-$) implique que $G_\Omega=\widehat{P}_\Omega$ (\ref{4.13}.1); le groupe $N$ permute les filtres satisfaisant (GF+), (GF$-$) ou (TF) et les fixateurs correspondants  (\cf \ref{4.13}.1).

\par D'apr\`es (P3) et \ref{5.1} un point a toujours un bon fixateur. Mais il existe des ensembles ne v\'erifiant ni (GF+) ni (GF$-$), \cf  \ref{4.12}.3c.

\begin{lemm}\label{5.4} Soit $\Omega$ un filtre de parties de $\A$. Si $\Omega$ a un fixateur transitif, alors $G_\Omega$ est transitif sur les appartements contenant $\Omega$.
\end{lemm}

\begin{proof} C'est classique et dans \cite[Rem. 4.2]{GR-08}.
\end{proof}

\begin{enonce*}[definition]{Cons\'equence}
 Alors $G_\Omega$ et tous ses sous-groupes normaux ne d\'ependent pas du choix de l'appartement contenant $\Omega$.
\end{enonce*}

\begin{prop} \label{5.5}

\par 1)  Supposons $\Omega\subset\Omega'\subset
cl(\Omega)$. Si $\Omega$ dans $\mathbb A$ a un bon (ou assez bon) fixateur, alors c'est \'egalement vrai pour $\Omega'$  et $G_{\Omega}=\widehat{N}_{\Omega}.G_{\Omega'}$,
$N.G_{\Omega}=N.G_{\Omega'}$. En particulier tout appartement contenant
$\Omega$ contient aussi son enclos $cl(\Omega)$.

\par Inversement si $\mathbb A=supp(\Omega)$, le plus petit espace affine contenant $\Omega$, (ou si $supp(\Omega')=supp(\Omega)$, donc $\widehat{N}_{\Omega'}=\widehat{N}_{\Omega}$), $\Omega$ a un assez bon fixateur et  $\Omega'$ a un bon fixateur, alors $\Omega$ a un bon fixateur.

\par 2) Si un filtre $\Omega$ dans $\mathbb A$ est engendr\'e par une famille $\mathcal F$ de filtres (\ie $S\in\Omega\Leftrightarrow\exists F\in\shf,S\in F$)
avec des bons (ou assez bons) fixateurs, alors  $\Omega$ a un bon (ou assez bon) fixateur $G_{\Omega}=\bigcup_{F\in\mathcal F}\;G_{F}$.

\par 3) Supposons que le filtre $\Omega$ dans $\mathbb A$ est la r\'eunion d'une suite croissante $(F_i)_{i\in\mathbb N}$  (\ie $S\in\Omega\Leftrightarrow S\in F_i,\forall i$) de filtres avec de bons (ou assez bons) fixateurs et que, pour un certain $i$, l'espace ${\mathrm supp}(F_i)$ a un fixateur fini  $W_0$ dans $\nu(N)=W_{Y\Lambda}$, alors
$\Omega$ a un bon (ou assez bon) fixateur $G_{\Omega}=\bigcap_{i\in\mathbb N}\;G_{F_i}$.

\par 4)Soient $\Omega$ et $\Omega'$ deux filtres dans $\mathbb A$. Supposons que $\Omega'$ satisfait \`a
(GF+) (resp. (GF+) et (TF)) et qu'il existe un nombre fini de chambres vectorielles ferm\'ees positives $\overline {C^v_1},\cdots$, $\overline {C^v_n}$ telles que:
$\Omega\subset\cup_{i=1,n}\;\Omega'+\overline {C^v_i}$. Alors $\Omega\cup\Omega'$ satisfait \`a
(GF+) (resp. (GF+) et (TF)) et $G_{\Omega\cup\Omega'}=G_{\Omega}\cap G_{\Omega'}$.

\end{prop}

\begin{proof} Voir \cite[prop. 4.3]{GR-08}
\end{proof}

\begin{rema}\label{5.6}

\par Dans 4) ci-dessus, les m\^emes r\'esultats sont vrais si on change $+$ en $-$.

\par  Si $\Omega'$ a un bon fixateur, $\Omega\subset\cup_{i=1,n}\;\Omega'+\overline
{C^v_i}$ et $\Omega\subset\cup_{i=1,n}\;\Omega'-\overline {C^v_i}$, alors $\Omega\cup\Omega'$ a un bon fixateur.

 Si $\Omega$ satisfait \`a (GF$-$), $\Omega'$ satisfait \`a (GF+), $\Omega$ ou $\Omega'$ satisfait \`a (TF), $\Omega\subset\cup_{i=1,n}\;\Omega'+\overline {C^v_i}$ et
$\Omega'\subset\cup_{i=1,n}\;\Omega-\overline {C^v_i}$, alors $\Omega\cup\Omega'$ a un bon fixateur.

\end{rema}

\subsection{Exemples de filtres avec de bons fixateurs}\label{5.7}
Pour les preuves manquantes ci-dessous, voir \cite[sect. 4.2]{GR-08}.

\par 1) Si $x\leq{}y$ ou $y\leq{}x$ dans $\A$ \ie si $y-x\in\pm{}\sht$ (c\^one de Tits), alors $\{x,y\}$, $[x,y]$ et $cl(\{x,y\})$ ont de bons fixateurs ( \ref{5.6} et (P3)). De plus, si $x\not=y$, $]x,y]=[x,y]\setminus\{x\}$ a un bon fixateur et le germe de segment $[x,y)=\germ_x([x,y])$ ou le germe d'intervalle $ ]x,y)=\germ_x(]x,y])$ est dit {\it pr\'eordonn\'e} et a un bon fixateur (\ref{5.5}.2).

\par Si $x\,{\bc\le}\,y$ ou $y\,{\bc\le}\,x$ dans $\A$ \ie si $y-x\in\pm{}\sht^\circ$ (int\'erieur du c\^one de Tits), la demi-droite $\delta$ d'origine $x$ et contenant $y$ est dite {\it g\'en\'erique} et a un bon fixateur. En effet $\delta$ est r\'eunion croissante des segments $[x,x+n(y-x)]$ pour $n\in\N$ et, si $n>0$, ce segment a un fixateur fini dans $W$ (car $y-x\in\pm{}\sht^\circ$); on peut donc appliquer \ref{5.5}.3. De m\^eme la droite contenant $x$ et $y$ a aussi un bon fixateur.

\par 2) Une facette locale $F^\ell(x,F^v)$
, une facette $F(x,F^v)$ ou une facette ferm\'ee $\overline F(x,F^v)$
 a un bon fixateur. Ici et dans la suite $F^v$ d\'esigne une facette vectorielle et $C^v$ une chambre vectorielle de $V$.

\par 3) Un quartier $\g q=x+C^v$ a un bon fixateur.

\par On notera que le fixateur du quartier $\g q=\g q_{x,+\infty}=x+C^v_f$ est $G_\g q=HU^{pm}_{\g q}=HU^{pm+}_{\g q}$, alors que $HU^{nm}_{\g q}=HU^{+}_{\g q}=HU^{++}_{x}$, car $U^{nm-}_{\g q}=\{1\}$ et $\widehat N_\g q=H$. On a vu en \ref{4.12}.3 que $U^{++}_{x}$ peut \^etre plus petit que $U^{+}_{x}\subset U^{pm+}_{\g q}$. On peut donc avoir $\widehat P_\g q=G_\g q=HU^{pm}_{\g q}\not=HU^{nm}_{\g q}=HU_{\g q}=HU_{\g q}^+=U_{\g q}^+.\widehat N_\g q$.

\par 4) Un germe de quartier $\g Q=\germ_\infty(x+C^v)$ a un bon fixateur (\ref{5.5}.2). Le fixateur de $\g Q_{\pm\infty}=\germ_\infty(x\pm{}C^v_f)$ est $HU^\pm$ car tout \'el\'ement de $U^\pm$ est un produit fini d'\'el\'ements de groupes $U_\alpha$ pour $\alpha\in\Phi^\pm$. Par contre $U^{ma+}$ n'est pas la r\'eunion des $U^{ma+}_\Omega$ pour $\Omega\in\g Q_{\infty}$, on n'a sans doute pas d'action de $G^{pma}$ sur $\shi$.

\par 5) L'appartement $\A$ lui-m\^eme a un bon fixateur $G_\A=H$. Par d\'efinition le stabilisateur de $\A$ est donc $G(\A\subset\A)=N$. Comme le corps $K$ est infini, il r\'esulte alors de \ref{1.6}.4 que les appartements de $\shi$ sont en bijection avec les sous-tores d\'eploy\'es maximaux de $\g G$.

\par 6) Un mur $M(\alpha,k)$ a un bon fixateur: soit $x\in M(\alpha,k)$ et $\xi$ dans une cloison vectorielle de Ker$\alpha$, alors $M(\alpha,k)$ est r\'eunion croissante des $cl(\{x-n\xi,x+n\xi\})$ dont le support $M(\alpha,k)$ a un fixateur fini ($\{1,r_{\alpha,k}\}$) dans $W_{Y\Lambda}$, on conclut gr\^ace \`a 1) ci-dessus et \ref{5.5}.3. Ce r\'esultat s'\'etend au cas du support d'une facette sph\'erique. Le (bon) fixateur de $M(\alpha,k)$ est $U_{\alpha,k}.U_{-\alpha,k}.\{1,r_{\alpha,k}\}.H$.

\par 7) Un demi-appartement $D(\alpha,k)$ est l'enclos de deux quartiers dont des cloisons sont dans $M(\alpha,k)$ et oppos\'ees, d'apr\`es 3) ci-dessus et \ref{5.5}.1,  \ref{5.5}.4, il a un bon fixateur, qui est \'egal \`a $U_{\alpha,k}.H$.

\par 8) Si $x_1\,{\bc\le}\,x_2$ \ie si $x_2-x_1\in\sht^\circ$, alors $cl(F(x_1,F_1^v),F(x_2,F_2^v))$ a un bon fixateur pour toutes facettes vectorielles $F_1^v$ et $F_2^v$ de signes quelconques: d'apr\`es 2) ci-dessus on peut appliquer la remarque \ref{5.6} aux facettes locales $F^\ell(x_1,F_1^v)$ et $F^\ell(x_2,F_2^v)$, on conclut gr\^ace \`a \ref{5.5}.1.

\par 9) L'enclos $cl(F(x,F_1^v),F(x,F_2^v))$ a un bon fixateur  d\`es que les facettes  vectorielles $F_1^v$ et $F_2^v$ sont de signe oppos\'es ou si l'une d'elles est sph\'erique. En effet le premier cas est clair par 2) ci-dessus et la remarque \ref{5.6}. Si la facette vectorielle $F_1^v$ est sph\'erique, elle est contenue dans le c\^one de Tits ouvert $\sht^\circ$ (ou $-\sht^\circ$), il existe donc une chambre $C^v$ et un $\xi\in F_1^v$ tels que, $\forall t>0$, $\overline{C^v}$ contient $t\xi-\shv_t$ pour un voisinage $\shv_t$ de $0$ dans $F_2^v$. Alors $F^\ell(x,F_1^v)\subset x+\overline{C^v}$ et $F^\ell(x,F_2^v)\subset F^\ell(x,F_1^v)-\overline{C^v}$. D'apr\`es 2) ci-dessus et la remarque \ref{5.6} $F^\ell(x,F_2^v)\cup F^\ell(x,F_1^v)$ a un bon fixateur et on conclut par \ref{5.5}.1.

\subsection{Intersection d'appartements}\label{5.7b}

\par Soient $A=g.A(T)=g.\A$ un appartement de $\shi$ et $x,y\in A$, la relation $g^{-1}x\leq{}g^{-1}y$ (resp. $g^{-1}x\,{\bc\leq{}}\,g^{-1}y$) dans $\A$ ne d\'epend pas du choix de $g$ d'apr\`es \ref{5.7}.5, car $N$ stabilise le c\^one de Tits et son int\'erieur; on note cette relation $x\leq{}_A\, y$ (resp. $x\,{\bc\leq{}}_A \,y$).

\par Soient $A_1$, $A_2$ deux appartements de $\shi$ et $x,y\in A_1\cap A_2$ tels que $x\leq{}_{A_1}y$. Alors $x\leq{}_{A_2}y$ (car $\{x,y\}$ a un bon fixateur). On d\'efinit donc ainsi une relation $\leq{}$ (et aussi $\bc\leq{}$) sur $\shi\times\shi$; on verra en \ref{5.14} que c'est un pr\'eordre.  Il r\'esulte de \ref{5.5}.1 que $A_1\cap A_2$ contient $cl(\{x,y\})$ (calcul\'e dans $A_1$ ou $A_2$) et en particulier le segment $[x,y]$ (qui est le m\^eme dans $A_1$ et $A_2$); on dit qu'une intersection d'appartements est {\it convexe pour le pr\'eordre} $\leq{}$.

\subsection{Extension centrale finie du groupe}\label{5.8}

\par Contrairement aux conventions depuis 4.0, on va modifier le SGR.

\par Soit $\varphi:\shs\rightarrow\shs'$ un morphisme de SGR libres qui est une extension centrale finie. On abr\'egera $\g G_\shs$ en $\g G$, $\g G_{\shs'}$ en $\g G'$, $\g G_\varphi$ en $\varphi$, etc.

\par Alors $Y$ est un sous-module de $Y'$ de m\^eme rang, donc $V=V'$ et, comme $I=I'$, on a $\g g=\g g'$, $\Phi=\Phi'$. Le morphisme $\varphi: \g G\rightarrow\g G'$ est d\'ecrit en \ref{1.10} et \ref{1.13}, en particulier $G'=T'.\varphi(G)$, $\varphi^{-1}(N')=N$ et le noyau de $\varphi$ est dans $T$ et m\^eme dans $H$ puisqu'il est fini.

\par Par construction (\ref{4.2}) les appartements affines $\A$ et $\A'$ sont identiques, les actions $\nu$ et $\nu'$ de $N$ et $N'$ sont compatibles et on a les m\^emes murs puisque pour $\alpha\in\Phi$, $\g G_\varphi\circ\g x_\alpha=\g x'_\alpha$ et $\g G_\varphi(\g N)\subset\g N'$. Les images $\nu(N)=W_{Y\Lambda}=W^v\ltimes(Y\otimes\Lambda)$ et $\nu'(N')=W_{Y'\Lambda}=W^v\ltimes(Y'\otimes\Lambda')$ sont en g\'en\'eral diff\'erentes (\ref{4.2}.6), mais on a $T=\varphi^{-1}(T')$ et $N'=\varphi(N)T'$.

\par Pour $\Omega\subset\A$, les alg\`ebres de Lie $\g g_{\Omega}$ et $\g g_{\Omega}'$ ont les m\^emes composantes de poids non nul, \cf \ref{2.2} (ce n'est pas vrai en poids nul: si $K$ est de caract\'eristique positive l'application de $\g h_K$ dans $\g h'_K$ peut m\^eme ne pas \^etre injective).
Ainsi l'isomorphisme $\varphi$ de $U^{ma+}$ sur $U'^{ma+}$ (\ref{3.18}.3) induit un isomorphisme  de $U^{ma+}_\Omega$ sur $U'^{ma+}_\Omega$ et aussi de $U^{pm+}_\Omega$ sur $U'^{pm+}_\Omega$, on identifie ces groupes.
On a les r\'esultats sym\'etriques pour $U^{nm-}_\Omega$ et $U'^{nm-}_\Omega$, etc. Comme les actions de $N$ et $N'$ sur $\A$ sont compatibles et Ker$\varphi\subset T$, on a $\widehat{N}_{\Omega}=\varphi^{-1}(\widehat{N}'_{\Omega})$.
Pour $x\in \A$, on a donc $\widehat{P}'_x=U_x^{pm+}.U_x^{nm-}.\widehat{N}'_x=\varphi(\widehat{P}_x).\widehat{N}'_x$ et $\varphi^{-1}(\widehat{P}'_x)=\widehat{P}_x$ car Ker$\varphi\subset H\subset \widehat{N}_x$.

\par Les relations ci-dessus permettent facilement de montrer que l'application $\varphi\times Id:G\times\A\rightarrow G'\times\A$ passe au quotient en une bijection de la masure $\shi$ sur la masure $\shi'$. Cette bijection est compatible avec $\varphi$ et les actions de $G$ et $G'$; elle \'echange les facettes. Comme $G'=\varphi(G)T'$ et $N=\varphi^{-1}(N')$ les appartements de $\shi$ et $\shi'$ sont les m\^emes. On identifie ces deux masures.

\begin{prop}\label{5.9} Soient $\varphi:\shs\rightarrow\shs'$  une extension centrale finie de SGR libres et $\Omega$ un filtre de parties de $\A$. Alors $\Omega$ a un bon (ou assez bon) fixateur $G_\Omega$ pour $G=G_\shs$ si et seulement si il en a un $G'_\Omega$ pour $G'=G_{\shs'}$. Dans ce cas  $G'_\Omega=\varphi(G_\Omega).\widehat{N}'_\Omega$.
\end{prop}

\begin{proof} Comme les actions de $G$ et $G'$ sur $\shi$ sont compatibles, on a $\varphi^{-1}(G'_\Omega)=G_\Omega$.

\par Supposons que $G'_\Omega$ est un (assez) bon fixateur (pour le signe $+$). On a $G'_\Omega=U_\Omega^{pm+}.U_\Omega^{nm-}.\widehat{N}'_\Omega$, donc $G_\Omega=\varphi^{-1}(G'_\Omega)=U_\Omega^{pm+}.U_\Omega^{nm-}.\varphi^{-1}(\widehat{N}'_\Omega)=U_\Omega^{pm+}.U_\Omega^{nm-}.\widehat{N}_\Omega$, d'o\`u (GF+) pour $\Omega$ et $G$.
Soit $\Omega'\in\Omega$, $G(\Omega'\subset\A)=\cap_{x\in\Omega'}\,N\widehat{P}_x\subset\varphi^{-1}(\cap_{x\in\Omega'}\,N'\widehat{P}'_x)\subset\varphi^{-1}(N'G'_{\Omega})=\varphi^{-1}(N'.U_\Omega^{nm-}.U_\Omega^{pm+})=\varphi^{-1}(N').U_\Omega^{nm-}.U_\Omega^{pm+}=N.G_{\Omega}$. Donc (TF) est satisfait par le filtre $\Omega$ et le groupe $G$.

\par Supposons que $G_\Omega$ est un (assez) bon fixateur (pour le signe $+$). Soit $g'\in G'(\Omega\subset\A)$, quitte \`a multiplier \`a gauche $g'$ par un \'el\'ement de $T'$ on peut supposer $g'\in\varphi(G)$. Donc il existe
$\Omega'\in \Omega$ tel que $g'\in (\cap_{x\in\Omega'}\,N'\widehat{P}'_x)\cap\varphi(G)=\cap_{x\in\Omega'}\,(N'\widehat{P}'_x\cap\varphi(G))=\cap_{x\in\Omega'}\,(N'\cap\varphi(G)).\varphi(\widehat{P}_x)=\cap_{x\in\Omega'}\,\varphi(N).\varphi(\widehat{P}_x)=\varphi(\cap_{x\in\Omega'}\,N\widehat{P}_x)=\varphi(N.G_\Omega)=\varphi(N).U_\Omega^{nm-}.U_\Omega^{pm+}\subset N'G'_\Omega$. D'o\`u (TF) pour $\Omega$ et $G'$.

\par Pour $g'\in G'_\Omega$, ce calcul montre que $g'\in T'.\varphi(N).U_\Omega^{nm-}.U_\Omega^{pm+}=N'.U_\Omega^{nm-}.U_\Omega^{pm+}$. Mais $U_\Omega^{nm-}$, $U_\Omega^{pm+}$ sont dans $G'_\Omega$, donc $g'\in(N'\cap G'_\Omega).U_\Omega^{nm-}.U_\Omega^{pm+}=\widehat{N}'_\Omega.U_\Omega^{nm-}.U_\Omega^{pm+}=\widehat{N}'_\Omega.\varphi(G_\Omega)$. D'o\`u (GF+) pour $\Omega$ et $G'$, plus la derni\`ere assertion de l'\'enonc\'e.
\end{proof}

\subsection{Passage au simplement connexe}\label{5.10}

\par 1) D'apr\`es \ref{1.3} on a une suite d'extensions commutatives de SGR $\shs_A\rightarrow \shs^1\hookrightarrow \shs^s\rightarrow\shs$ qui sont successivement centrale torique, semi-directe et centrale finie. On note $G^A$, $G^1$, $G^s$, $G$ les groupes correspondants et $\psi$ le compos\'e des morphismes $G^A\rightarrow G^1\hookrightarrow G^s\rightarrow G$.

\par D'apr\`es \ref{5.8}, $G^s$ et $G$ ont la m\^eme masure $\shi$. Par contre $\shs_A$ et $\shs^1$ ne sont en g\'en\'eral pas libres, on consid\`ere les actions de $G^A$ et $G^1$ sur $\shi$ via leurs morphismes dans $G$.

\par 2) On a $G=\psi(G^A).T$ (\ref{1.8}.3); on en d\'eduit aussit\^ot que le groupe simplement connexe $G^A$ permute transitivement les appartements de $\shi$. Le stabilisateur de $\A$ dans $G^A$ est le groupe $\psi^{-1}(N)=N^A=N_{\shs_A}$ (\ref{1.10}) et, d'apr\`es la construction de \ref{4.2}, $\nu(N^A)=W=W^v\ltimes(Q^\vee\otimes\Lambda)$. Ainsi tout appartement de $\shi$ est muni d'une unique structure d'appartement de type $\A$ telle que $G^A$ induise des isomorphismes d'appartements (au sens de \cite[1.13]{Ru-10}).

\par 3) Dans \cite[6.1 et 6.2]{Ru-10} on se place dans le cas $\shs=\shs^s$ (libre) et le groupe $G_1$ qui y est d\'efini est \'egal \`a $G^1.H=\psi(G^A).H$.

\subsection{Paires conviviales}\label{5.7c}

\par 1) On dit qu'une paire $(F_1,F_2)$ form\'ee de deux filtres de parties de $\shi$ est {\it conviviale} s'il existe un appartement  contenant ces deux filtres et si deux appartements $A,A'$ contenant $F_1,F_2$ sont isomorphes par un isomorphisme fixant l'enclos de $F_1$ et $F_2$ (calcul\'e dans $A$ ou $A'$).

\par On ne consid\'erera ici que des paires $G-${\it conviviales}, \ie  telles que les isomorphismes d'appartements soient induits par des \'el\'ements de $G$.
Une paire $(F_1,F_2)$ est donc $G-$conviviale d\`es que $\QO=F_1\cup F_2$ est contenu dans un appartement et y a un assez bon fixateur (car $G_\QO=U_\QO^{pm+}.U_\QO^{nm-}.\widehat N_\QO$ et $U_\QO^{pm+},\,U_\QO^{nm-}\subset\psi(G^A)$ induisent des isomorphismes d'appartements).

\begin{enonce*}{\quad2) Lemme} Soient $F_1,F_2$ deux filtres de parties de $\A$.

\par Supposons $G=G_{F_1}.N.G_{F_2}$, alors pour tous $g_1,g_2\in G$, $g_1F_1$ et $g_2F_2$ sont contenus dans un m\^eme appartement.

\par Si $G=G_{F_1}.N.G_{F_2}$ ou si $F_1$ ou $F_2$ a un fixateur transitif, alors $F_1$ et $F_2$ sont conjugu\'es par $G$ si et seulement si ils le sont par $N$.
\end{enonce*}

\begin{proof} Les cons\'equences de $G=G_{F_1}.N.G_{F_2}$ sont classiques, \cf \eg \cite[3.6 et 3.7]{Ru-06}. La derni\`ere assertion r\'esulte de \ref{5.4} et \ref{5.7}.5.
\end{proof}

\par 3) Il r\'esulte donc de la d\'ecomposition d'Iwasawa \ref{4.7} et de \ref{5.7}.4 qu'un filtre \'etroit dans un appartement et un germe de quartier sont toujours dans un m\^eme appartement. D'apr\`es \ref{5.7} et \ref{5.5}.4 une facette (ou un germe de segment, un germe d'intervalle, une facette locale, une facette ferm\'ee) et un germe de quartier forment une paire $G-$conviviale.

\par 4) Par convexit\'e ordonn\'ee un appartement contenant un point $x$ et un germe de quartier $\g Q=\germ_\infty(y+C^v)$, contient le quartier $\g q=x+C^v$ et son enclos donc son adh\'erence $x+\overline C^v$. On en d\'eduit qu'une facette $F(x,F^v)$ et un filtre \'etroit $F$ contenant $x$ sont toujours dans un m\^eme appartement. Ainsi, d'apr\`es \ref{5.7}.9, deux facettes $F(x,F_1^v)$ et $F(x,F_2^v)$ en le m\^eme point forment une paire $G-$conviviale d\`es que $F_1^v$ et $F_2^v$ sont de signes oppos\'es ou si l'une d'elles est sph\'erique.

\par 5) Soient $C=F(x,C^v)$ une chambre de $\A$ et $M(\alpha,k)$ l'un de ses murs (avec $C\subset D(\alpha,k)$). Alors $U_{\alpha,k}=U_{\alpha,C}$ agit transitivement sur les chambres $C'\not=C$ adjacentes \`a $C$ le long de $M(\alpha,k)$; en particulier n'importe laquelle de ces chambres $C'$ est dans un m\^eme appartement que le demi-appartement $D(\alpha,k)$, \cf \cite[4.3.4]{GR-08}. De plus $C'$ et $D(\alpha,k)$ forment une paire conviviale: si $C'\subset\A$, $C'\cup D(\alpha,k)$ a un bon fixateur $H.U_{\alpha,k^+}$ (\ref{5.7}.7).

\par 6) On a vu en \cite[6.10]{GR-08} que deux points de la masure $\shi$ ne sont pas toujours  contenus dans un m\^eme appartement. C'est pour cette raison que l'on a abandonn\'e le nom d'immeuble pour $\shi$; une d\'efinition abstraite des masures affines n'est pas possible dans des termes approchant ceux de \cite{Ru-08}, d'o\`u la d\'efinition de \cite{Ru-10} inspir\'ee de \cite{T-86a}. On reproduit ci-dessous cette d\'efinition dans une formulation utilisant la notion de paire conviviale.

Une paire de points ou une paire de facettes de $\shi$ n'est pas toujours conviviale. De plus dans l'exemple \ref{4.12}.3c de $\widetilde{SL}_2$, on a trouv\'e deux points de $\A$ dont le fixateur n'est pas assez bon. Par contre il est peut-\^etre possible que le fixateur de deux points de $\A$ soit toujours transitif (donc qu'une paire de points d'un m\^eme appartement soit toujours conviviale).

\begin{defi}\label{5.7d}
Une {\it masure affine} de type $\A$ est un ensemble $\mathcal I$ muni d'un recouvrement par un ensemble $\mathcal A$ de sous-ensembles appel\'es {\it appartements} tel que:

\par (MA1) Tout $A\in \mathcal A$ est muni d'une structure d'appartement de type $\A$ (au sens de \cite[1.13]{Ru-10} \ie est "isomorphe" \`a $\A$).

\par (MA2) Si $F$ est un point, un germe d'intervalle pr\'eordonn\'e, une demi-droite g\'en\'erique ou une chemin\'ee solide d'un appartement $A$, alors $(F,F)$ est une paire conviviale.

\par (MA3+4) Si $\mathfrak R$ est un germe de chemin\'ee \'evas\'ee, si $F$  est une facette  ou un germe de chemin\'ee solide, alors $(\mathfrak R,F)$ est une paire conviviale.

\qquad(La notion de chemin\'ee, utilis\'ee ci-dessus, ne sera d\'efinie qu'en \ref{5.12}.1.)

\noindent La masure affine $\mathcal I$ est dite {\it ordonn\'ee} si elle v\'erifie l'axiome suppl\'ementaire suivant:

\par (MAO) Soient $x,y$ deux points de $\mathcal I$ et $A,A'$ deux appartements les contenant; si $x\le  y$ (dans $A$), alors les segments $[x,y]_A$ et $[x,y]_{A'}$ d\'efinis par $x$ et $y$ dans $A$ et $A'$ sont \'egaux.

\par La masure affine $\mathcal I$ est dite {\it \'epaisse} si toute cloison de $\mathcal I$ est dans l'adh\'erence d'au moins trois chambres.

\end{defi}

\subsection{Sous-groupes parahoriques et types de facettes}\label{5.11}

\par 1) Si $\Omega\subset\A$ a un assez bon fixateur, on a $G_\Omega=(\psi(G^A)\cap G_\Omega)\widehat N_\Omega$, puisque $U_\Omega^{pm+}$, $U_\Omega^{nm-}$, ... ne d\'ependent pas du SGR \`a extension commutative pr\`es. En particulier le fixateur $G_\Omega^A$ de $\Omega$ dans $G^A$ agit transitivement sur les appartements de $\shi$ contenant $\Omega$. On note $\widehat P^{sc}_\Omega=\psi(G^A_\Omega)H$. Ainsi $G_\Omega=\widehat P_\Omega=\widehat P^{sc}_\Omega.\widehat N_\Omega$. Le groupe $\widehat N^{sc}_\Omega=N\cap \widehat P^{sc}_\Omega=\widehat N_\Omega\cap \widehat P^{sc}_\Omega$ a pour image via $\nu$ le fixateur $W_\Omega$ de $\Omega$ dans $W$. On a aussi $\widehat P^{sc}_\Omega=P_\Omega^{pm}.\widehat N^{sc}_\Omega=P_\Omega^{nm}.\widehat N^{sc}_\Omega$.

\par  Il est clair que les facettes $F=F(x,F^v)$ et $\overline F=\overline F(x,F^v)$ ont le m\^eme fixateur dans $W$, donc $\widehat P^{sc}_{F}=\widehat P^{sc}_{\overline F}$.

\par 2) Quand $W_\Omega=W_\Omega^{min}$, on note $P_\QO=\widehat P^{sc}_\Omega$ (qui est alors aussi \'egal \`a $P_\Omega^{pm}$ et $P_\Omega^{nm}$). Cela se produit si $\Omega$ est r\'eduit \`a un point sp\'ecial ou si $\Omega$ est une facette sph\'erique et si la valuation est discr\`ete (d'apr\`es  \ref{4.3}.4); c'est a priori rare en dehors de ces deux cas d'apr\`es \cite[7.1.10.2]{BtT-72} et \ref{4.12}.4.

\par Pour g\'en\'eraliser la d\'efinition de \cite[1.5.1]{BtT-72}, on d\'efinit un {\it sous-groupe parahorique} comme le groupe $P_F=P_{\overline F}$ associ\'e \`a une facette sph\'erique $F$ ou $\overline F$. On parle de {\it sous-groupe d'Iwahori} si $F$ ou $\overline F$ est une chambre.



\par 3) Le {\it type global affine} d'un filtre de $\shi$ qui a un assez bon fixateur (en particulier une facette) est son orbite sous $G^A$. Dans un appartement $A$ de $\shi$ deux filtres (assez bien fix\'es) ont m\^eme type si et seulement si ils sont conjugu\'es par le groupe de Weyl $W_A$ de $A$ (\cf le 1), \ref{5.7c}.2 et \ref{5.10}.2).
 Pour les facettes (ou les filtres \'etroits assez bien fix\'es) la relation "avoir le m\^eme type global affine" est la relation engendr\'ee par transitivit\'e \`a partir des relations pr\'ec\'edentes dans les appartements. En effet, \'etant donn\'ees deux facettes $F_1,F_2$ il existe toujours un quartier $\g q$ et des appartements $A_1,A_2$ contenant $\g q$ et respectivement $F_1,F_2$ (\cf \ref{5.7c}.3); de plus $\g q$ contient forc\'ement un $W_{A_i}-$conjugu\'e de $F_i$.

 \par Le type global affine d'une facette co\"{\i}ncide donc avec le type habituel dans le cas d'un immeuble affine discret. Mais en valuation dense et si on suppose $\A$ essentiel, le type global affine d'une facette locale $F^\ell(x,F^v)$ est la donn\'ee de l'orbite $W.x$ de $x$ et de l'orbite de $F^v$ sous $W_x$ (c'est-\`a-dire du type (vectoriel) de $F^v$ si $x$ est sp\'ecial).

\subsection{Chemin\'ees}\label{5.12}

\par 1) Une {\it chemin\'ee} dans $\A$ est associ\'ee \`a une facette $F=F(x,F^v_0)$ (sa base) et une facette vectorielle $F^v$ (sa direction), c'est le filtre $\g r(F,F^v)=cl(F+F^v)=cl(\overline F+F^v)=cl(F^\ell(x,F^v_0)+F^v)\supset \overline F+\overline F^v$.

\par La chemin\'ee $\g r(F,F^v)$ est dite {\it \'evas\'ee}  si $F^v$ est sph\'erique, son signe est alors celui de $F^v$. Cette chemin\'ee est dite {\it solide} (resp. {\it pleine}) si la direction de tout sous-espace affine la contenant a un fixateur fini dans $W^v$ (resp. est $V$). Une chemin\'ee \'evas\'ee est solide. L'enclos d'un quartier est une chemin\'ee pleine.

\par Si $F_0^v=F^v$ la chemin\'ee $\g r(F(x,F^v),F^v)$ est l'enclos de la face de quartier $x+F^v$; cette face est dite sph\'erique si $F^v$ est sph\'erique..

\par Un raccourci de la chemin\'ee $\g r(F,F^v)$ est d\'efini par un \'el\'ement $\xi\in\overline F^v$, c'est la chemin\'ee $cl(F+\xi+F^v)$. Le {\it germe} de la chemin\'ee $\g r(F,F^v)$ est le filtre $\g R(F,F^v)=\germ_\infty(\g r(F,F^v))$ form\'e des parties de $\A$ contenant un de ses raccourcis.

\par Si $F^v$ est la facette vectorielle minimale, on a $\g r(F,F^v)=\g R(F,F^v)=F$.

\medskip
\par 2) \`A la facette vectorielle $F^v$ est associ\'e un sous-groupe parabolique $P(F^v)$ de $G$ avec une d\'ecomposition de Levi $P(F^v)=M(F^v)\ltimes U(F^v)$. Le groupe $M(F^v)$ est engendr\'e par $T$ et les $U_\alpha$ pour $\alpha\in\Phi$ et $\alpha(F^v)=0$, \cf \cite[6.4]{Ru-10} ou \cite[6.2]{Ry-02a}.

\par Si $F^v=F^v(J)=\{v\in \overline C^v_f\mid \alpha_j(v)=0,\forall j\in J\}$ pour $J\subset I$ (\cf \ref{4.1}), le groupe $M(F^v)$ est clairement un quotient du groupe de Kac-Moody minimal $\g G_{\shs(J)}(K)=G(J)$ (notations de \ref{3.10}). D'apr\`es la remarque \ref{3.10} et \ref{3.13} on a en fait \'egalit\'e: $M(F^v)=G(J)$.
Ce groupe $G(J)$ induit sur $\A$ une structure d'appartement $\A(J)$ dont les murs sont les $M(\alpha,k)$ pour $\alpha\in \Delta(J)$ et $k\in\Lambda$.

\par L'enclos correspondant $cl_J(F)$ de la facette $F$ est une facette ferm\'ee de $\A(J)$, on note $M(F,F^v)$ son fixateur dans $M(F^v)=G(J)$. Cette construction se r\'ealise aussi si $F^v$ est une autre facette vectorielle ou si $F$ est r\'eduit \`a un point $F=\{x\}$ (ou $F=F(x,F^v)$).

\par Le sous-groupe "parahorique" de $G$ associ\'e \`a $\g r$ est le produit semi-direct $P^\mu(\g r)=M(F,F^v)\ltimes U(F^v)$. Il ne d\'epend en fait que du germe $\g R(F,F^v)$, on le note donc aussi $P^\mu(\g R)$.

\par Les r\'esultats suivants se d\'emontrent comme dans \cite[{\S{}} 6]{Ru-10}.

\begin{enonce*}[plain]{\quad3) Proposition} Le groupe $P^\mu(\g R)$ fixe (point par point) le germe de chemin\'ee $\g R$.
\end{enonce*}

\begin{enonce*}[plain]{\quad4) Proposition} Soient $\g R_1$ et $\g R_2$ deux germes de chemin\'ees de $\A$ avec $\g R_1$ \'evas\'e, alors  $G=P^\mu(\g R_1).N.P^\mu(\g R_2)$.
\end{enonce*}

\begin{enonce*}[plain]{\quad5) Proposition} Une chemin\'ee solide et son germe ont de bons fixateurs. Il en est de m\^eme pour une face de quartier sph\'erique $x+F^v$ et son germe (\`a l'infini); le fixateur (point par point) de $\germ_\infty(x+F^v)$ est $M(x,F^v).U(F^v)$.
\end{enonce*}

\begin{enonce*}[plain]{\quad6) Proposition} Soient $\g R_1$ un germe de chemin\'ee \'evas\'ee et $\g R_2$ un germe de chemin\'ee solide ou une facette dans $\A$, alors $\g R_1\cup\g R_2$ a un assez bon fixateur (et m\^eme un bon fixateur si $\g R_2$ est \'evas\'e).
\end{enonce*}

\begin{enonce*}[definition]{\quad7) Remarque} On obtient donc de nouvelles paires $G-$conviviales dans $\shi$: un germe de chemin\'ee \'evas\'ee et un germe de chemin\'ee solide ou un germe de chemin\'ee \'evas\'ee et une facette. Cela permet en particulier de d\'efinir des r\'etractions de $\shi$ sur un appartement $A$ avec pour centre un germe de chemin\'ee \'evas\'ee pleine de $A$ (par exemple un germe de quartier) \cite[2.6]{Ru-10}.
\end{enonce*}

\begin{theo}\label{5.13} La masure affine $\shi$ construite en \ref{5.1} est une masure affine ordonn\'ee \'epaisse au sens de \cite{Ru-10} ou \ref{5.7d}. Elle est semi-discr\`ete si la valuation $\omega$ de $K$ est discr\`ete.
\end{theo}

\begin{proof} L'axiome (MA1) r\'esulte de \ref{5.10}.2 et (MA2) r\'esulte de ce que l'on a vu en \ref{5.7} ou \ref{5.12}.5: les filtres impliqu\'es dans (MA2) ont de bons fixateurs. Enfin (MA3+4) r\'esulte   de \ref{5.12}.7. Les derni\`eres assertions se montrent comme dans \cite[6.11]{Ru-10}. On notera cependant que l'\'epaisseur d'une cloison est le cardinal du corps r\'esiduel de $K$ et donc peut \^etre finie dans notre cas, plus g\'en\'eral que celui de \cite{GR-08}.
\end{proof}

\begin{remas}\label{5.14}
 1) La masure $\shi$ a donc toutes les propri\'et\'es d\'emontr\'ees dans \cite{Ru-10}, un certain nombre d'entre elles ont d\'ej\`a \'et\'e prouv\'ees ci-dessus. On obtient cependant au moins deux propri\'et\'e int\'eressantes suppl\'ementaires: les relations $\leq{}$ et $\bc\le$ de \ref{5.7b} sont des pr\'eordres (elles sont transitives) et les r\'esidus en chaque point de $\shi$ ont une structure d'immeubles jumel\'es.

\par La transitivit\'e de la relation $\leq{}$ et son invariance par $G$ permettent de d\'efinir un sous-semi-groupe $G^+$ de $G$, qui admet une d\'ecomposition de Cartan, voir \cite[1.6]{GR-12} pour plus de d\'etails.
 Ceci constitue une g\'en\'eralisation des d\'ecompositions de Cartan de \cite{Gd-95} \'etablies pour les groupes de lacets.

 \par 2) On notera les propri\'et\'es de l'action du groupe $G$: il est transitif sur les appartements et tous les isomorphismes entre appartements dont l'existence est exig\'ee par les axiomes de \cite{Ru-10} ou \ref{5.7d} sont induits par des \'el\'ements de $G$. On peut dire que l'action de $G$ est {\it fortement transitive}. Dans le cas classique d'une action de groupe sur un immeuble affine discret \'epais, cette notion de forte transitivit\'e entra\^{\i}ne clairement la notion classique (et lui est sans doute \'equivalente).
\end{remas}

\begin{prop}\label{5.15} Les immeubles jumel\'es \`a l'infini associ\'es \`a la masure $\shi$ en \cite[{\S{}} 3]{Ru-10} s'identifient (avec leurs appartements et leur action de $G$) aux immeubles jumel\'es de $G$ d\'efinis par J. Tits (\ref{1.6}.5).

\par L'immeuble microaffine positif (resp. n\'egatif) \`a l'infini associ\'e \`a la masure $\shi$ en \cite[{\S{}} 4]{Ru-10} s'identifie (avec ses appartements et son action de $G$) \`a l'immeuble microaffine de \cite{Ru-06} (dans sa r\'ealisation de Satake) associ\'e \`a $G$ et \`a la valuation $\omega$ de $K$ (resp. l'analogue obtenu en changeant le c\^one de Tits en son oppos\'e).

\end{prop}

\begin{rema*} Cyril Charignon construit directement dans \cite{Cn-10} un objet immobilier contenant $\shi$, les immeubles microaffines ci-dessus et d'autres masures affines associ\'ees aux sous-groupes paraboliques non sph\'eriques de $G$. C'est la g\'en\'eralisation au cas Kac-Moody de la compactification de Satake (ou compactification poly\'edrique) des immeubles de Bruhat-Tits. Cette derni\`ere a \'et\'e construite abstraitement (sans l'aide d'un groupe) dans \cite{Cn-08}.
\end{rema*}

\begin{proof} On identifie facilement l'appartement canonique (et son action de $N$) d'un immeuble de \cite{Ru-10} avec celui qui doit lui correspondre, \cf \cite[3.3.3 et 4.4]{Ru-10}. Il reste donc \`a identifier les fixateurs de points de ces appartements.

\par Le fixateur de la facette vectorielle (sph\'erique) $F^v$ dans l'immeuble de \ref{1.6}.5 est le sous-groupe parabolique $P(F^v)$ engendr\'e par $T$ et les $U_\alpha$ pour $\alpha(F^v)\geq{}0$. Il est clair qu'un \'el\'ement de chacun de ces groupes transforme une face de quartier de $\A$ de direction $F^v$ en une face parall\`ele (dans $\shi$). Ainsi $P(F^v)$ fixe la classe de parall\'elisme des faces de quartier correspondant \`a $F^v$.

\par Un point de l'appartement $\A^s$ de l'immeuble de \cite{Ru-06} est de la forme $x^s=(F^v,\widetilde y)$ avec $F^v$ une facette vectorielle sph\'erique positive de $V$, $\langle F^v\rangle$ l'espace vectoriel engendr\'e et $\widetilde y\in V/\langle F^v\rangle$ [\lc, 4.2]. On lui associe le germe (\`a l'infini) de la face de quartier $y+F^v$ pour $y\in\widetilde y$.
Le fixateur de $x^s$ est engendr\'e par $U(F^v)$, le groupe $M(x,F^v)$ et le sous-groupe de $T$ induisant dans $\A$ des translations de vecteur dans $\langle F^v\rangle$ [\lc, 4.2.3]. Les deux premiers groupes engendrent le fixateur point par point du germe de la face de quartier $y+F^v$ (\ref{5.12}.5) et le troisi\`eme groupe stabilise \'evidemment ce germe.

\par Ces inclusions entre les fixateurs de points permettent de d\'efinir des applications surjectives et $G-$\'equivariantes des immeubles de \ref{1.6}.5 ou \cite{Ru-06} vers les immeubles correspondants de \cite{Ru-10}. Mais un point $x$ et son transform\'e $gx$ sont toujours dans un m\^eme appartement et les applications ci-dessus sont injectives sur les appartements. Les inclusions entre fixateurs sont donc des \'egalit\'es et on a bien les identifications d'immeubles annonc\'ees.
\end{proof}

\subsection{Immeubles et groupe de Kac-Moody maximal}\label{5.16}

\par Le groupe $G^{pma}$ n'agit pas sur la masure affine $\shi$. Par contre, d'apr\`es \ref{3.17b} il agit sur l'immeuble vectoriel $\shi^v_+$.

\par Le groupe $G$ agit sur l'immeuble microaffine positif $\shi_+^{\qm s}$ de \cite[{\S{}} 4]{Ru-06} (dans sa r\'ealisation de Satake) qui est r\'eunion disjointe d'immeubles index\'es par les facettes sph\'eriques positives.
 Plus pr\'ecis\'ement, \`a une telle facette $F^v$ on associe l'immeuble de Bruhat-Tits (non \'etendu) $\shi(F^v)$ du groupe r\'eductif $M(F^v)$ sur $K$. Le groupe $U(F^v)$ agit trivialement sur $\shi(F^v)$ et $G$ permute les $\shi(F^v)$ selon son action sur $\shi^v_+$:
 Un germe de chemin\'ee ou de face de quartier $\g R=\g R(F,F^v)$ correspond \`a une facette ou un point $\g R_\infty$ de $\shi(F^v)$. Le fixateur (point par point) de cette facette ou ce point $\g R_\infty$ pour l'action de $G$ sur $\shi_+^{\qm s}$ est $ZM(F^v).M(F,F^v).U(F^v)$ avec les notations de \ref{5.12}.2 et $ZM(F^v)$ le centre de $M(F^v)$.
 Ce fixateur de $\g R_\infty$ est donc plus grand que $P^\qm(\g R)$ qui est le fixateur (point par point) de $\g R$ (\ref{5.12}.3) et (par d\'efinition) le {\it fixateur strict} de $\g R_\infty$.

 \par Le fixateur de la facette sph\'erique positive  $F^v$ pour l'action de $G^{pma}$ sur $\shi^v_+$ est $P^{pma}(F^v)=M(F^v)\ltimes U^{ma+}(F^v)$ (\cf \ref{3.10} et \ref{3.17b}).
 On peut donc prolonger l'action de $G$ sur $\shi_+^{\qm s}$ en une action de $G^{pma}$:

 \par $P^{pma}(F^v)$ agit sur $\shi(F^v)$ via $M(F^v)$ \ie $U^{ma+}(F^v)$ agit trivialement.
 Le fixateur dans $G^{pma}$ de $\g R_\infty$ comme ci-dessus est $ZM(F^v).M(F,F^v).U^{ma+}(F^v)$; il contient le groupe $P^\qm_{pma}(\g R)$ $=M(F,F^v).U^{ma+}(F^v)$
que l'on appelle encore  fixateur strict de $\g R_\infty$ pour l'action de $G^{pma}$ (m\^eme si ce n'est pas justifi\'e par une action sur $\shi$ qui contient le filtre $\g R$).

 \par Pour deux germes de chemin\'ees \'evas\'ees positives $\g R_1$ et $\g R_2$, la d\'emonstration dans \cite[6.7]{Ru-10} de la proposition \ref{5.12}.4 se g\'en\'eralise: \qquad$G^{pma}=P^\qm_{pma}(\g R_1).N.P^\qm_{pma}(\g R_2)$.

 \begin{NB} Dans  \cite[4.2.1]{Ru-06} on peut modifier  le  quotient en envoyant  $F^v\times Y_\R$ sur $Y_\R$ par la seconde projection. 
  On obtient  une r\'ealisation g\'eom\'etrique $\shi_+^{\qm S}$, de Satake au sens fort, de l'immeuble microaffine, sur laquelle $G$ et $G^{pma}$ agissent.
  Alors le fixateur strict $P^\qm_{}(\g R)$ (resp. $P^\qm_{pma}(\g R)$) est le fixateur  de $\g R\subset Y_\R\subset \shi_+^{\qm S}$ pour l'action de $G$ (resp. $G^{pma}$).
 \end{NB}


\section{Appendice: Comparaison et simplicit\'e}\label{s6}

Pour la construction des masures nous avons plong\'e le groupe de Kac-Moody $\g G_\shs$ dans le groupe de Kac-Moody maximal \`a la Mathieu $\g G_\shs^{pma}$. Pla\c{c}ons nous sur un corps quelconque $k$ et notons $G=\g G_\shs(k)$, $G^{pma}=\g G_\shs^{pma}(k)$, etc. Alors $G^{pma}$ appara\^{\i}t comme le compl\'et\'e de $G$ pour une certaine filtration. D'autres groupes maximaux ont \'et\'e d\'efinis par R\'emy et Ronan \cite{RR-06} ou Carbone et Garland \cite{CG-03} \`a l'aide d'autres filtrations, on va essayer de les comparer en tirant parti des r\'esultats des sections \ref{s2} et \ref{s3}.

Par ailleurs  Moody, dans un preprint non publi\'e \cite{My-82}, a d\'emontr\'e la simplicit\'e d'un groupe analogue \`a $G^{pma}$ en caract\'eristique $0$. L\`a encore on va utiliser les sections \ref{s2} et \ref{s3} pour montrer dans certains cas la simplicit\'e d'un sous-quotient de $G^{pma}$  (th\'eor\`eme \ref{6.19}).

\subsection{Le groupe compl\'et\'e \`a la R\'emy-Ronan}\label{6.1}

\par 1) Ce groupe $G^{rr}$  est l'adh\'erence de l'image de $G$ dans le groupe des automorphismes de son immeuble positif $\shi^v_+$ \cite{RR-06}. Il contient donc le quotient de $G$ par le noyau $Z'(G)=\cap_{g\in G}\,gB^+g^{-1}$ de l'action de $G$ sur cet immeuble. Ce groupe $Z'(G)$ est
en fait le centre de $G$: $Z'(G)=Z(G)=\{t\in T\mid\qa_i(t)=1,\forall i\in I\,\}$ (\ref{1.6}.5 et \cite[lemma 1B1]{RR-06}).

\par On va s'int\'eresser plut\^ot \`a une variante de $G^{rr}$ d\'efinie par Caprace et R\'emy \cite[1.2]{CR-09}. Ce groupe $G^{crr}$ contient $G$ et le groupe $G^{rr}$ en est un quotient.

\par 2) Pour $r\in\N$, soit $D(r)$ (resp. $C(r)$) la r\'eunion des chambres ferm\'ees de l'immeuble $\shi^v_+$ (resp. de son appartement standard $\A^v$) qui sont \`a distance num\'erique $\leq{}r$ de la chambre fondamentale $C_0$.
Les fixateurs $U^+_{D(r)}$ et $U^+_{C(r)}$ de $D(r)$ et $C(r)$ dans $U^+=\g U^+(k)$ forment deux filtrations de $U^+$ qui sont exhaustives ($U^+_{D(0)}=U^+_{C(0)}=U^+$) et s\'epar\'ees (le syst\`eme de Tits $(G,B^+,N,S)$ est satur\'e \ie le fixateur de $\A^v$ dans $G$ est $T$ \cite[1.9]{Ru-06} et $T\cap U^+=\{1\}$:   \ref{1.6}.5). De plus $U^+_{D(r)}$ est le plus grand sous-groupe distingu\'e de $U^+$ contenu dans $U^+_{C(r)}$.

\par Les deux filtrations sont en fait tr\`es li\'ees car, d'apr\`es \cite[2.1]{CR-09}, il existe une fonction croissante $r:\N\to\N$ tendant vers l'infini telle que, pour $\qa\in\QF^+$, $U_\qa\subset U^+_{C(R)}\Rightarrow U_\qa\subset U^+_{D(r(R))}$.

\par 3) On d\'efinit de la m\^eme mani\`ere des fixateurs $U^{ma+}_{D(r)}$ et $U^{ma+}_{C(r)}$, pour l'action de $U^{ma+}$ sur $\shi^v_+$ (\ref{3.17b}). Les filtrations correspondantes de $U^{ma+}$ sont plus diff\'erentes: l'intersection des $U^{ma+}_{C(r)}$ contient les groupes $\g U^{ma}_{\N^*\qa}(k)$ pour $\qa\in\QD^+_{im}$ ( le syst\`eme de Tits $(G^{pma},B^{ma+},N,S)$ est rarement satur\'e) et celle des $U^{ma+}_{D(r)}$ est souvent triviale (\cf  \ref {6.5}).

\par 4) Le groupe $U^{rr+}$ est le compl\'et\'e de $U^+$ pour sa filtration par les $U^{+}_{D(r)}$. Plus g\'en\'eralement le groupe $G^{crr}$ est le compl\'et\'e de $G$ pour sa filtration (non exhaustive) par les $U^{+}_{D(r)}$; il est muni de la topologie correspondante. En particulier  $U^{rr+}$ est ouvert dans $G^{crr}$.

Le groupe $G^{rr}$ est le s\'epar\'e-compl\'et\'e de $G$ pour la filtration par les fixateurs $G_{D(r)}\subset B^+$. C'est un quotient de $G^{crr}$ par un sous-groupe $Z'(G^{crr})$ contenant $Z(G)$ (\cf 1) ); il  contient $U^{rr+}$. Si le corps $k$ est fini, $Z'(G^{crr})=Z(G)$ \cite[prop.1]{CR-09}.

\subsection{Le groupe compl\'et\'e \`a la Carbone-Garland}\label{6.2}

Soit $\ql\in X^+$ un poids dominant r\'egulier (\ie $\ql(\qa^\vee_i)>0$, $\forall i\in I$). On consid\`ere la repr\'esentation $\qp_\ql$ de $G$ ou $G^{pma}$ (de plus haut poids $\ql$) dans $V^\ql=L_\Z(\ql)\otimes k$ (\ref{3.7}.1). Le groupe $G^{cg\ql}$ d\'efini par Carbone et Garland \cite{CG-03} est le s\'epar\'e-compl\'et\'e de $G$ pour la filtration d\'efinie par les fixateurs de parties finies de $V^\ql$, \ie par les fixateurs de sous-espaces vectoriels de dimension finie de $V^\ql$.
Pour $n\in\N$ soit $V(n)$ la somme des sous-espaces de $V^\ql$ de poids $\ql-\qa$ avec $\qa\in Q^+$ et $\deg(\qa)\leq{}n$. Ainsi $G^{cg\ql}$ est le s\'epar\'e-compl\'et\'e de $G$ pour la filtration d\'efinie par les fixateurs $G_{V(n)}$ de $V(n)$.

Cette filtration n'est pas s\'epar\'ee: $\cap_n\,G_{V(n)}=$ Ker$\qp_\ql$. Si $v_\ql$ est un vecteur de plus haut poids et si $g\in $ Ker$\qp_\ql$, il fixe $v_\ql$ et $\widetilde s_i(v_\ql)$ $\forall i\in I$, donc est de la forme $g=tu$ avec $t\in T$, $u\in U^+$ et $\ql(t)=\widetilde s_i(\ql)(t)=1$.
Mais $\widetilde s_i(\ql)=\ql-\ql(\qa^\vee_i)\qa_i$ et $\ql(\qa^\vee_i)\not=0$, donc $\ql(t)=\qa_i(t)=1$ $\forall i\in I$: $t\in Z(G)\cap$Ker$(\ql)$ (\ref{1.6}.5). Inversement $Z(G)\cap$Ker$(\ql)$ agit trivialement sur $V^\ql$, donc  Ker$\qp_\ql=(Z(G)\cap$Ker$(\ql)).(U^+\cap$ Ker$\qp_\ql)$.
On verra plus loin que $U^+\cap$ Ker$\qp_\ql=\{1\}$ (\ref{6.3}.4).

Si l'on veut une filtration s\'epar\'ee et donc plonger $G$ dans son compl\'et\'e, on modifie la construction en faisant agir $G$ sur la somme directe $V^{\ql_1}\oplus\cdots\oplus V^{\ql_r}$ o\`u $\ql_1,\cdots,\ql_r$ engendrent $X$; on note alors $V(n)$ la somme des $V(n)$ correspondant aux diff\'erents facteurs. Le compl\'et\'e de Carbone-Garland modifi\'e $G^{cgm}$ ainsi obtenu est donc d\'efini par une filtration contenue dans $U^+$, celle des $U^+_{V(n)}=U^+\cap G_{V(n)}$.

Pour comparer $G^{cgm}$ et $G^{crr}$ on va comparer cette filtration sur $U^+$ avec celle des $U^+_{D(r)}$.

 On note $U^{cg+}$ le s\'epar\'e compl\'et\'e de $U^+$ pour la filtration par les $U^+_{V(n)}$.

\subsection{Comparaison}\label{6.3}

1) D'apr\`es \ref{3.2},  \ref{3.3} et  \ref{3.4}, $U^{ma+}$ est complet pour la filtration par les sous-groupes distingu\'es $U^{ma+}_n=\g U^{ma}_{\Psi(n)}(k)$ o\`u $\Psi(n)=\{\qa\in\QD^+\mid \deg(\qa)\geq{}n\}$. On note $\overline U^+$ l'adh\'erence de $U^+$ dans $U^{ma+}$ pour la topologie associ\'ee; c'est aussi le compl\'et\'e de $U^+$ pour la filtration par les $U^+_n=U^+\cap U^{ma+}_n$.

2) Pour $i\in I$ et $\qa\in\QD^+\setminus\{\qa_i\}$, on a $s_i(\qa)\in\QD^+$ et $\deg(s_i(\qa))\leq{}(1+M)\deg(\qa)$, o\`u $M$ est le maximum des valeurs absolues de coefficients non diagonaux de la matrice de Kac-Moody $A$. Ainsi $U^{ma+}_{n(1+M)}\subset \widetilde s_i(U^{ma+}_{n}) \subset U^{ma+}_{n/(1+M)}$ et $\widetilde s_i$ est un automorphisme du groupe topologique $\g U^{ma+}_{\QD^+\setminus\{\qa_i\}}(k)$.

3) D\`es que $\deg(\qa)\geq{}(1+M)^d$, on a $\deg(s_{i_1}.\cdots.s_{i_d}(\qa))\geq{}{{\deg(\qa)}\over{(1+M)^d}}$ pour toute suite $i_1,\cdots,i_d\in I$. On en d\'eduit que, pour $n\geq{}(1+M)^d$, $U^{ma+}_n$ est dans le conjugu\'e de $U^{ma+}$ par $s_{i_1}.\cdots.s_{i_d}$ et donc $U^{ma+}_n$ fixe $C(d)$ pour l'action de $G^{pma}$ sur $\shi^v_+$.
On a $U^{ma+}_n\subset U^{ma+}_{C(d)}$ et m\^eme $U^{ma+}_n\subset U^{ma+}_{D(d)}$, car $U^{ma+}_n$ est distingu\'e dans $U^{ma+}$. En particulier $U^+_n\subset U^+_{D(d)}\subset U^+_{C(d)}$.

4) Vue la forme de l'action de $U^{ma+}\subset\widehat\shu^+_k$ sur $V^\ql$, on a $U^{ma+}_{n+1}\subset U^{ma+}_{V(n)}$. Par ailleurs, pour $b\in B^{ma+}$ et $i_1,\cdots,i_d\in I$, le fixateur de $b.s_{i_1}.\cdots.s_{i_d}(v_\ql)$ dans $U^{ma+}$ est dans $U^{ma+}\cap\; ^{b.s_{i_1}.\cdots.s_{i_d}}B^{ma+}$, c'est-\`a-dire dans le fixateur (dans $U^{ma+}$) de $b.s_{i_1}.\cdots.s_{i_d}(C_0)$.
Comme en 2) ci-dessus on montre que le poids de $s_{i_1}.\cdots.s_{i_d}(v_\ql)$ est $\ql-\qa$ avec $\deg(\qa)\leq{}{{M'}\over{M}}((1+M)^d-1)$ si $M'=Max\{\ql(\qa^\vee_i)\mid i\in I\}$.
Donc, si $u\in U^{ma+}_{V(n)}$ avec $n\geq{}{{M'}\over{M}}(1+M)^d$, $u$ fixe toutes les chambres de la forme $b.s_{i_1}.\cdots.s_{i_d}(C_0)$, c'est-\`a-dire les chambres de $D(d)$.

Ainsi $U^{ma+}_{n+1}\subset U^{ma+}_{V(n)}\subset U^{ma+}_{D(d)}$ et $U^{+}_{n+1}\subset U^{+}_{V(n)}\subset U^{+}_{D(d)}$ pour $n\geq{}{{M'}\over{M}}(1+M)^d$.
Comme la filtration par les $U^{+}_{D(d)}$ est s\'epar\'ee, il en est de m\^eme de celle par les $U^{+}_{V(n)}$; en particulier $U^+\cap$ Ker$\qp_\ql=\{1\}$.

5) De ces comparaisons de filtrations on d\'eduit des homomorphismes continus:

\noindent$\phi$ :\quad \xymatrix{ \overline U^+\ar[r]^(.4){\qg}&U^{cg+}\ar[r]^(.4){\qr}&U^{rr+}}. Le probl\`eme de la comparaison des groupes $U^{ma+}$, $U^{cg+}$ et $U^{rr+}$ se traduit donc en deux questions:  \qquad A-t'on $U^{ma+}=\overline U^+$?\quad(voir \ref{6.10} et \ref{6.11})

\qquad$\phi$, $\qg$ et $\qr$ sont-ils des isomorphismes de groupes topologiques?\quad(voir \ref{6.7} \`a \ref{6.9})

Si $k$ est un corps fini, $U^{ma+}$ et donc $\overline U^+$ sont compacts. Comme $U^{cg+}$ et $U^{rr+}$ sont s\'epar\'es, les homomorphismes $\phi$, $\qg$ et $\qr$ sont surjectifs, ferm\'es et ouverts. Ce sont donc des isomorphismes de groupes  topologiques si et seulement si ils sont injectifs.

6) {\bf Remarques} a) La filtration $(U^{ma+}_{n})_{n\in\N}$ de $U^{ma+}$ permet de d\'efinir une m\'etrique invariante \`a gauche sur $G^{pma}$, pour laquelle $G^{pma}$ est complet et dont les boules ouvertes sont les $gU^{ma+}_{n}$. D'apr\`es la d\'ecomposition de Bruhat, le fait que
$U^{ma+}\vartriangleleft B^{ma+}$ et la relation de 2) impliquant les $\widetilde s_i$, cette m\'etrique est \'equivalente \`a celle, invariante \`a droite, dont les boules ouvertes sont les $U^{ma+}_{n}g$. Ainsi $G^{pma}$ est un groupe topologique, dans lequel $U^{ma+}$ est ouvert. L'adh\'erence $\overline G$ de $G$ dans $G^{pma}$ est le s\'epar\'e-compl\'et\'e de $G$ pour la filtration induite.

b) Comme $\overline G$, $G^{cgm}$ et $G^{crr}$ sont des s\'epar\'es-compl\'et\'es pour les filtrations dans $U^+$ de 4) ci-dessus, les homomorphismes de 5) ci-dessus se prolongent en
$\phi$ :\quad \xymatrix{ \overline G\ar[r]^(.4){\qg}&G^{cgm}\ar[r]^(.4){\qr}&G^{crr}}.
Par d\'efinition Ker$\phi= \overline U^+\cap Z'(G^{pma})$ o\`u $Z'(G^{pma})=\cap_{g\in G^{pma}}\,gB^{ma+}g^{-1}$ est le noyau de l'action de $G^{pma}$ sur l'immeuble $\shi^v_+$.

c) Le groupe $Z(G)$ s'envoie trivialement dans $G^{rr}$. D'apr\`es les calculs de \ref{6.2} on a donc un homomorphisme continu $G^{cg\ql}\to G^{rr}$. Celui-ci est surjectif, ferm\'e et ouvert si le corps est fini (r\'esultat de U. Baumgartner et B. R\'emy \cf \cite[Th. 2.6]{CER-08}).

\begin{prop}\label{6.4}  Les centres $Z(G)$, $Z(G^{pma})$  ainsi que $Z'(G^{pma})=\cap_{g\in G^{pma}}\,gB^{ma+}g^{-1}$ v\'erifient  $Z(G)\subset Z(G^{pma})\subset Z'(G^{pma})$ et $Z'(G^{pma})=Z(G).(Z'(G^{pma})\cap U^{ma+})$, $Z(G^{pma})=Z(G).(Z(G^{pma})\cap U^{ma+})$. De plus $Z'(G^{pma})\cap U^{ma+}$ est distingu\'e dans $G^{pma}$.
\end{prop}


\begin{proof} Comme $Z(G)$ centralise $U^{ma+}$ il est contenu dans $Z(G^{pma})$. Mais $B^{ma+}$ est \'egal \`a son normalisateur (car il fait partie d'un syst\`eme de Tits) on a donc $Z(G^{pma})\subset B^{ma+}$ et m\^eme $Z(G^{pma})\subset Z'(G^{pma})$.
Soit $h\in  Z'(G^{pma})$. On \'ecrit $h=tu$ avec $t\in T$ et $u\in U^{ma+}$. Soit $i\in I$, on peut alors \'ecrire $u=\exp(ae_i).v$ avec $a\in k$ et $v\in\g U^{ma}_{\QD^+\setminus\{\qa_i\}}(k)$.
Soient $\ql\in k$ et $y_\ql=\exp(\ql f_i)$, alors $y_\ql hy_\ql^{-1}=t.\exp((\qa_i(t)-1)\ql f_i).y_\ql.\exp(ae_i).y_\ql^{-1}.y_\ql.v.y_\ql^{-1}$ avec $y_\ql.v.y_\ql^{-1}\in\g U^{ma}_{\QD^+\setminus\{\qa_i\}}(k)$. Un calcul rapide dans $SL_2$ montre alors que $y_\ql uy_\ql^{-1}\in B^{ma+}$ $\forall\ql\in k$ si et seulement si $\qa_i(t)=1$ et $a=0$; donc $t\in Z(G)$ et $Z'(G^{pma})\subset Z(G).U^{ma+}_2$.

Cette derni\`ere assertion montre que le groupe $Z'(G^{pma})\cap U^{ma+}$ est normalis\'e par les $U_{-\qa_i}$ (pour $i\in I$); comme il est normalis\'e par $B^{ma+}$, il est distingu\'e dans $G^{pma}$.
\end{proof}

\subsection{GK-simplicit\'e}\label{6.5}

On dit que l'alg\`ebre de Lie $\g g_k=\g g_\shs\otimes_\Z k$ est simple au sens du th\'eor\`eme de Gabber-Kac (ou {\it GK-simple}) si tout sous$-\shu_k-$module gradu\'e non trivial de $\g g_k$ contenu dans $\g n^+_k$ est r\'eduit \`a $\{0\}$.

De m\^eme on dit que le groupe $G^{pma}$ est {\it GK-simple} si tout sous-groupe distingu\'e de $G^{pma}$ contenu dans $U^{ma+}$ est r\'eduit \`a $\{1\}$, \ie si $Z'(G^{pma})\cap U^{ma+}=\cap_{r\in\N}\,U^{ma+}_{D(r)}=\{1\}$ ou si $Z'(G^{pma})=Z(G)$. On esp\`ere que $G^{pma}$ est toujours GK-simple (voir  \ref{6.9}.1 ou \ref{6.8}). L'assertion correspondante pour $G$ est toujours satisfaite (\ref{6.1}.1).

\begin{rema*} En caract\'eristique $0$, $\g g_k$ est GK-simple dans le cas sym\'etrisable (et m\^eme conjecturalement dans tous les cas). Mais, comme me l'a indiqu\'e O. Mathieu, cela n'implique pas la GK-simplicit\'e en caract\'eristique $p$; voir ci-dessous l'exemple des lacets de $SL_n$ (\ref{6.8}).
On voit facilement que $\g g_k$ est GK-simple si et seulement si $\forall\qa\in\QD^+_{im}$ tout $X\in\g g_{k\qa}$ tel que, $\forall j\in I$, $\forall q\in\N^*$ $ad(f_j^{(q)})X=0$ est forc\'ement nul.
L'ensemble de ces $X\in\g g_{k\qa}$ est d\'etermin\'e par des \'equations lin\'eaires \`a coefficients entiers ind\'ependantes de $k$. Si $\g g_\C$ est GK-simple, un d\'eterminant d\'etermin\'e par ces \'equations est non nul. Ainsi la condition ci-dessus est v\'erifi\'ee pour $\g g_{k\qa}$ si la caract\'eristique $p$ de $k$ ne divise pas ce d\'eterminant.
Malheureusement $\QD^+_{im}$ est vide ou infini, on pourrait donc avoir $\g g_k$ non GK-simple quelque soit la caract\'eristique. Par contre dans le cas affine on n'a  qu'un nombre fini de d\'eterminants \`a consid\'erer car il y a p\'eriodicit\'e des espaces radiciels imaginaires.
En effet, pour le type affine $X_n^{(k)}$, $\g  g_A''=\g g_A/centre$ est r\'ealis\'e comme sous-alg\`ebre de Lie de $\g s\otimes\C[t,t^{-1}]$ o\`u $\g s$ est l'alg\`ebre simple de type $X_n$; d'apr\`es \cite{K-90} (resp. les calculs explicites de \cite{Mn-85}) $\g  g_A''$ (resp. $(\g  g_A'')_\Z$) est stable par multiplication par $t^{\pm{}k}$.

Si la matrice de Kac-Moody $A$ a tous ses facteurs de type affine (ou de type fini), alors $\g g_k$ est GK-simple si la caract\'eristique $p$ de $k$ est assez grande.
\end{rema*}

\begin{prop}\label {6.6} Supposons $\g g_k$  GK-simple et $k$ infini. Soit $u\in U^{ma+}_n\setminus U^{ma+}_{n+1}$ avec $n\geq{}1$. Alors il existe $j\in I$ et $y\in U_{-\qa_j}$ tels que $yuy^{-1}\notin B^{ma+}$ (si $n=1$) ou $yuy^{-1}\in U^{ma+}\setminus U^{ma+}_{n}$ (si $n\geq{}2$).
\end{prop}

\begin{proof} (\cf \cite[prop. 8]{My-82} en caract\'eristique $0$)

On calcule dans l'alg\`ebre $\widehat\shu^p_k(\QD^+)=\widehat\shu^+_k$ qui est gradu\'ee par les poids dans $Q^+$ et le degr\'e total dans $\N$ (\ref{2.13}.2) et aussi dans $\widehat\shu^p_k(s_j(\QD^+))$. Soit $u\in U^{ma+}_n\setminus U^{ma+}_{n+1}\subset \widehat\shu^p_k(\QD^+)$; sa composante $u_n$ de degr\'e $n$ est dans l'alg\`ebre de Lie $\g g_k$.
Soit $\qa\in\QD^+$ de degr\'e $n$ tel que la composante $u_\qa$ de $u_n$ sur $\g g_{k\qa}$ soit non nulle. Comme $\g g_k$  est GK-simple, il existe $j\in I$ et $q\geq{}1$ tels que $adf_j^{(q)}u_\qa\not=0$ (sinon $ad(\shu_k)u_\qa$ est de plus bas degr\'e $n$).

Si $n=1$, on a $\qa=\qa_j$, $u=\exp(u_\qa)u'$ avec $u'\in \g U^{ma}_{\QD^+\setminus\{\qa_j\}}(k)$. Soit $y=\exp f_j$, alors $yuy^{-1}=y.\exp u_\qa.y^{-1}.y.u'.y^{-1}$ avec $y.u'.y^{-1}\in  \g U^{ma}_{\QD^+\setminus\{\qa_j\}}(k)$ et $y.\exp u_\qa.y^{-1}\notin B^{ma+}$. Donc $yuy^{-1}\notin B^{ma+}$.

Si $n\geq{}2$, $u \in  \g U^{ma}_{\QD^+\setminus\{\qa_j\}}(k)$ et ce groupe est normalis\'e par $U_{-\qa_j}$ (contenu dans $\g U^{ma}_{s_j(\QD^+)}(k)$) \cf \ref{3.3}. Soient $\ql\in k$ et $y_\ql=\exp(\ql f_j)\in U_{-\qa_j}$, alors $y_\ql uy_\ql^{-1}\in U^{ma+}$ vaut $\sum_{m\in\N}\,\ql^madf_j^{(m)}u$ (voir la d\'emonstration de \ref{2.5}); sa composante de poids $\qa-q\qa_j$ est donc \'egale \`a la somme finie $\sum_{m\geq{}q}\,\ql^madf_j^{(m)}u_{\qa+(m-q)\qa_j}$, o\`u  $u_{\qa+(m-q)\qa_j}$ est la composante de poids $\qa+(m-q)\qa_j$ de $u$ dans $\widehat\shu^p_k(\QD^+)$.

Cette composante de $y_\ql uy_\ql^{-1}$ s'exprime comme un polyn\^ome en $\ql$ dont le terme de degr\'e $q$ est non nul. Comme $k$ est infini, il existe un $\ql\in k$ tel que ce polyn\^ome soit non nul et donc $y_\ql uy_\ql^{-1}\notin U^{ma+}_{n}$, puisque sa composante de poids $\qa-q\qa_j$ est non nulle.
\end{proof}

\begin{prop}\label {6.7} Supposons $\g g_k$  GK-simple et $k$ infini. Alors $\phi$, $\qg$ et $\qr$ sont des hom\'eomorphismes. Ainsi $\overline G$, $G^{cgm}$ et $G^{crr}$ sont des groupes topologiques isomorphes (et de m\^eme pour $\overline U^+$, $U^{cg+}$ et $U^{rr+}$).
\end{prop}

\begin{rema*} On verra en \ref{6.9}.4 que $\qg:\overline G\to G^{cgm}$ est un isomorphisme sous la seule hypoth\`ese que $\g g_k$  est GK-simple.
\end{rema*}

\begin{proof} Soit $u\in U^{ma+}_n\setminus U^{ma+}_{n+1}$ avec $n\geq{}1$. D'apr\`es la proposition \ref{6.6} il existe $m\leq{}n$, $j_1,\cdots,j_m\in I$ et $u_i\in U_{-\qa_{j_i}}$ tels que
$u_m.\cdots.u_1.u.(u_m.\cdots.u_1)^{-1}\notin B^{ma+}$.
Mais $B^{ma+}$ est le fixateur de la chambre fondamentale $C_0$, donc $u$ ne fixe pas la chambre $(u_m.\cdots.u_1)^{-1}C_0$ qui est \`a distance $\leq{}m$ de $C_0$ (on a une galerie $C_0, (u_1)^{-1}C_0,(u_2.u_1)^{-1}C_0,\cdots$). On a donc montr\'e que $U^{ma+}_{D(r)}\subset U^{ma+}_{r+1}$ pour $r\in\N$. Avec les inclusions inverses prouv\'ees en \ref{6.3}.4, on voit que les trois filtrations sont \'equivalentes et donc $\phi$, $\qg$ et $\qr$ sont des hom\'eomorphismes.
\end{proof}

\begin{exem}\label {6.8}
On consid\`ere la matrice de Kac-Moody $A$ de type $\widetilde A_{m-1}$ avec $m\geq{}3$. Pour un bon choix du SGR $\shs$, on obtient des alg\`ebres et groupes de lacets: $\g g_k=\g{sl}_m(k)\otimes_k k[t,t^{-1}]$, $G=SL_m(k[t,t^{-1}])$, $\widehat{\g g}_k=\g{sl}_m(k)\otimes k(\!(t)\!)$ et $G^{pma}=SL_m(k(\!(t)\!))$.

Si la caract\'eristique $p$ de $k$ divise $m$, $\g g_k$ contient toutes les matrices scalaires. Celles-ci sont clairement annul\'ees par $ad(\shu_k)$ et sont contenues dans $\g n^+_k$ si le scalaire est dans $tk[t]$. Donc $\g g_k$ n'est pas GK-simple d\`es que $p$ divise $m$; il est facile de voir que la r\'eciproque est vraie.

Notons $V=k^m$ et $\qe_1,\cdots,\qe_m$ sa base canonique. Soit $R=k[[t]]$, on consid\`ere les r\'eseaux $V_0=R\qe_1\oplus\cdots\oplus R\qe_m$, $V_1=R\qe_1\oplus\cdots\oplus R\qe_{m-1}\oplus Rt\qe_m$,$\cdots$, $V_{m-1}=R\qe_1\oplus Rt\qe_{2}\oplus\cdots\oplus Rt\qe_m$, $V_m=tV_0$ et de mani\`ere g\'en\'erale $V_{am+b}=t^aV_b$ pour $a,b\in\Z$.
Alors $B^{ma+}$ est le stabilisateur dans $G^{pma}$ de tous ces r\'eseaux et il est assez facile de voir que $U^{ma+}_n=\{g\in G^{pma}\mid (g-1)V_i\subset V_{i+n}\,\forall i\in\Z\}$ pour $n\geq{}1$. \'Egalement $U^{ma+}_{mn}$ est form\'e des matrices \`a coefficients dans $R$ congrues \`a l'identit\'e  modulo $t^n$ et \`a une matrice triangulaire sup\'erieure modulo $t^{n+1}$.

Le groupe $U^{ma+}$ ne contient pas de matrice scalaire. Mais si $p$ divise $m$, on peut trouver dans $U^{ma+}_{mn}\setminus U^{ma+}_{mn+1}$ une matrice diagonale $u$ dont tous les coefficients sont \'egaux modulo $t^{pn}$. Si $y=\exp(\ql f_j)$, alors $yuy^{-1}$ a les m\^emes coefficients que $u$ sauf un (colonne $j$, ligne $j+1$) qui est dans $t^{pn}R$, donc $yuy^{-1}\in U^{ma+}_{mn}$. La conclusion de la proposition \ref{6.6} n'est pas v\'erifi\'ee.

Les sommets de l'appartement standard $\A^v$ sont associ\'es aux r\'eseaux de la forme $Rt^{n_1}\qe_1\oplus\cdots\oplus Rt^{n_m}\qe_m$. Ainsi le fixateur d'une partie finie de $\A^v$ est form\'e des $g\in G^{pma}$ stabilisant un nombre fini de tels sous-modules. On en d\'eduit facilement que la filtration de $G^{pma}$ par les $G^{pma}_{C(r)}$ est \'equivalente \`a la filtration par les $B^{ma+}(n)=\{g=(g_{ij})\in SL_m(R)\mid g_{ij}\in t^nR,\forall i\not=j\}$.

Notons $U^{ma+}[n]=U^{ma+}\cap(\cap_{i=1}^{m-1}\,(\exp\,f_i)B^{ma+}(n)(\exp-f_i))$. On calcule facilement que $U^{ma+}[n]$ est form\'e des matrices $g=(g_{ij})\in SL_m(R)$ telles que $g_{ij}\in t^nR,\forall i\not=j$ et $g_{ii}-g_{i+1,i+1}\in t^nR, \forall i\leq{}m-1$.
Ainsi $g_{11}^m\in1+t^nR$ et $g_{11}\in 1+tR$. Si $p^e$ est la plus grande puissance de $p$ divisant $m$, on en d\'eduit que $g_{11}-1\in t^{n/p^e}R$; donc $U^{ma+}[n]\subset U^{ma+}_{mn/p^e}$.
Comme la filtration par les $U^{ma+}_{D(r)}$ est plus fine que celle par les  $U^{ma+}[n]$, on d\'eduit de l'inclusion pr\'ec\'edente et de celles prouv\'ees en \ref{6.3}.4 que toutes ces filtrations sont \'equivalentes. La conclusion de la proposition \ref{6.7} est v\'erifi\'ee m\^eme si $\g g_k$ n'est pas GK-simple. On a aussi $Z'(G^{pma})=Z(G)$: $G^{pma}$ est GK-simple.
\end{exem}

\begin{remas}\label{6.9} 1) On a vu au cours de la d\'emonstration de la proposition \ref{6.7} que, si $\g g_k$ est GK-simple et $k$ infini, les filtrations de  $U^{ma+}$ par les  $U^{ma+}_n$,  $U^{ma+}_{V(n)}$ ou $U^{ma+}_{D(d)}$ sont \'equivalentes. En particulier $G^{pma}$ est GK-simple: $Z'(G^{pma})=Z(G)$ (\cf \cite[prop. 8]{My-82} en caract\'eristique $0$).

2) Si $G^{pma}$ est GK-simple ou plus g\'en\'eralement si $Z'(G^{pma})\cap\overline G=Z(G)$, l'homomorphisme $\phi:\overline G\to G^{crr}$ est injectif (\ref{6.3}.6.b). Si de plus $k$ est fini, on en d\'eduit que $\phi:\overline U^+\to U^{rr+}$ et $\qg:\overline U^+\to U^{cg+}$ sont des isomorphismes de groupes topologiques (\ref{6.3}.5);
ainsi les groupes topologiques $\overline G$, $G^{cgm}$ et $G^{crr}$ sont isomorphes. Inversement pour $k$ fini, l'isomorphisme de $\overline G$ et $G^{crr}$ implique que $Z'(G^{pma})\cap\overline G=Z(G)$ (\ref{6.1}.4); si la caract\'eristique de $k$ est grande, cela implique que $G^{pma}$ est GK-simple (\ref{6.11}).

3) Si jamais $\g g_k$ n'est pas GK-simple en caract\'eristique $0$, on peut montrer que $G^{pma}$ n'est pas  GK-simple
; en particulier $G^{pma}$ et $G^{crr}$ ne sont pas isomorphes.

4) Supposons $\g g_k$ GK-simple. Soit $k'$ un corps contenant $k$. On consid\`ere l'immeuble $\shi^v_+(k')$ de $\g G$ sur $k'$, il contient $\shi^v_+$. On note $D'(r)$ la boule de $\shi^v_+(k')$ de centre $C_0$ et rayon $r$ et  $U^{ma+}_{D'(r)}$ son fixateur dans $U^{ma+}$. Les groupes $U^{ma+}_n$,  $U^{ma+}_{V(n)}$ et $U^{ma+}_{D'(r)}$ sont les intersections avec $U^{ma+}$ des groupes analogues \`a $U^{ma+}_n$,  $U^{ma+}_{V(n)}$ et $U^{ma+}_{D(r)}$ dans $\g U^{ma+}(k')$.
Si $k'$ est infini la d\'emonstration de \ref{6.7} prouve que $U^{ma+}_{D'(r)}\subset U^{ma+}_{r+1}\subset U^{ma+}_{V(r)}\subset U^{ma+}_{D'(d)}\subset U^{ma+}_{D(d)}$ si $r\geq{}{M'\over M}(1+M)^d$.
Notons $U^{rri+}$ (resp. $G^{crri}$) le compl\'et\'e de $U^+$ (resp. $G$) pour la filtration par les $U^{+}_{D'(r)}=U^+\cap U^{ma+}_{D'(r)}$. Alors les groupes $\overline G$, $G^{cgm}$ et $G^{crri}$ sont des groupes topologiques isomorphes (et de m\^eme pour $\overline U^+$, $U^{cg+}$ et $U^{rri+}$).
En particulier $U^{rri+}$ et $G^{crri}$ ne d\'ependent pas du choix du corps infini $k'$ contenant $k$.

5) Supposons $\g g_k$ GK-simple. S'il y a injection continue du groupe  $U^{rr+}$ associ\'e \`a $k$ dans celui associ\'e \`a l'extension infinie $k'$ (hypoth\`ese raisonnable mais non \'evidemment v\'erifi\'ee) alors $\phi$ est un isomorphisme de groupes topologiques et $G^{crr}=G^{crri}$, $U^{rr+}=U^{rri+}$.
\end{remas}

\subsection{Comparaison de $U^{ma+}$ et $\overline U^+$}\label{6.10}

Il s'agit de savoir si $U^+$ est dense dans $U^{ma+}$, c'est-\`a-dire si $U^{ma+}$ est topologiquement engendr\'e par les sous-groupes radiciels $U_\qa$ pour $\qa\in\QF^+$. On va r\'epondre par un contre-exemple et une proposition assez g\'en\'erale.

{\bf Contre-exemple.} On consid\`ere la matrice de Kac-Moody $A=\left(\begin{matrix}2&-m  \cr -m&2\cr \end{matrix}\right)$ avec $m\geq{}3$ et $G=\g G_{\shs_A}(\F_2)$. On note $\qa,\qb$ les racines simples, $e$ (resp. $f$) une base de $\g g_\qa$ (resp. de $\g g_\qb$) sur $k=\F_2$ et $a=\exp(e)$ (resp. $b=\exp(f)$) l'\'el\'ement non trivial de $U_\qa$ (resp. $U_\qb$).
 D'apr\`es \cite[exer. 5.25]{K-90} les plus petites racines de $\QF^+$ sont $\qa,\qb,\qa+m\qb,\qb+m\qa$ tandis que $\qb+\qa,\qb+2\qa,\cdots,\qb+(m-1)\qa,\qa+2\qb,\cdots,\qa+(m-1)\qb$ sont des racines imaginaires. De plus $\qb+r\qa$ ou $\qa+r\qb$ n'est pas une racine pour $r>m$.

Soit $\Psi$ l'id\'eal de $\QD^+$ form\'e des racines positives de la forme $r\qb+q\qa$ avec $r\geq{}2$ ou $r+q\geq{}4$; il contient toutes les racines r\'eelles positives sauf $\qa$ et $\qb$. Notons $U^{ma}_\Psi=\g U^{ma}_\Psi(k)$, c'est un sous-groupe ouvert de $U^{ma+}$.
D'apr\`es \ref{3.2} et \ref{3.3} $U^{ma+}/U^{ma}_\Psi$ est un groupe en bijection avec 
$\g g_{\qa}\oplus\g g_{\qb}\oplus\g g_{\qb+\qa}\oplus\g g_{\qb+2\qa}$. Si $U^+$ est dense dans $U^{ma+}$ ce groupe doit \^etre engendr\'e par les images de $a$ et $b$.

On fait des calculs dans le quotient de $\widehat\shu^+_k$ par l'id\'eal $\widehat\shu^+_{\N^*\Psi}\otimes k$. On utilise que $U^{ma+}\subset\widehat\shu^+_k$ et $U^{ma}_\Psi\subset1+\widehat\shu^+_{\N^*\Psi}\otimes k$. On note $e^{(n)}*f=ad(e^{(n)})f\in\g g_k$, qui est non nul pour $n\leq m$. 
Ainsi $\g g_{\qa}\oplus\g g_{\qb}\oplus\g g_{\qb+\qa}\oplus\g g_{\qb+2\qa}=\F_2e\oplus\F_2f\oplus\F_2(e*f)\oplus\F_2(e^{(2)}*f)$, $\widehat\shu^+_k$ admet $1,e,e^{(2)},e^{(3)},f,ef,e^{(2)}f,e*f,e(e*f),e^{(2)}*f$
comme base d'un suppl\'ementaire de $\widehat\shu^+_{\N^*\Psi}\otimes k$ et l'isomorphisme de $U^{ma}/U^{ma}_\Psi$ sur $\g g_{\qa}\oplus\g g_{\qb}\oplus\g g_{\qb+\qa}\oplus\g g_{\qb+2\qa}$ s'obtient par la projection \'evidente.

Dans cette alg\`ebre quotient $a=1+e+e^{(2)}+e^{(3)}$ et $b=1+f$. On peut calculer facilement les 17 mots en a et b de longueur $\leq{}8$. On constate que $(ab)^2=1+e*f+e^{(2)}*f=(ba)^2$ et $(ab)^4=(ba)^4=1$.
Ainsi ces 17 mots constituent toute l'image de $U^+$ dans cette alg\`ebre quotient et il y a au maximum 14 mots diff\'erents. Comme $U^{ma+}/U^{ma}_\Psi$ est de cardinal 16, il n'est pas \'egal au groupe engendr\'e par $a$ et $b$ (il contient 8 \'el\'ements et pas $[\exp]e^{(2)}*f$).

On en d\'eduit que $U^+$ n'est pas dense dans $U^{ma+}$.

\begin{prop}\label{6.11} Supposons la caract\'eristique $p$ du corps $k$ nulle ou strictement plus grande que $M$ (\ref{6.3}.2). Alors le groupe $U^+$ (resp. $G$) est  dense dans $U^{ma+}$ (resp. $G^{pma}$).
\end{prop}
\begin{rema*} Dans le cas simplement lac\'e, le r\'esultat est valable sans condition sur la caract\'eristique.
\end{rema*}
\begin{proof} D'apr\`es \ref{2.3} l'alg\`ebre de Lie $\g n^+_k$ est engendr\'ee par les $e_i$ pour $i\in I$. Montrons par r\'ecurrence sur $n$ que tout \'el\'ement $g\in U^{ma+}_n$ est congru \`a un \'el\'ement de $U^+$ modulo $U^{ma+}_{n+1}$, c'est clair pour $n=1$.
D'apr\`es \ref{3.3} $U^{ma+}_{n}/U^{ma+}_{n+1}$ est un groupe commutatif isomorphe \`a $\g g_n=\oplus_{\deg(\qa)=n}\,\g g_{k\qa}$. D'apr\`es \ref{3.2} et les bonnes propri\'et\'es de la base des $[N]$ (\ref{2.9}.4), l'isomorphisme $\qf_n$ est donn\'e comme suit: on plonge $U^{ma+}$ dans $\widehat\shu^+_k$ et l'\'el\'ement $g\in U^{ma+}_n$ est congru \`a $1+\qf_n(g)$ modulo $\widehat\shu^+_{k,\geq{}n+1}=\prod_{\deg(\qa)\geq{}n+1}\;\widehat\shu^+_{k\qa}$.
Il suffit de montrer le r\'esultat cherch\'e pour des \'el\'ements $g$ engendrant $U^{ma+}_{n}/U^{ma+}_{n+1}$, on va consid\'erer ceux tels que $\qf_n(g)$ soit de la forme $[e_i,v]$ avec $i\in I$ et $v\in \g g_{n-1}$.
Par hypoth\`ese il existe $h\in U^+_{n-1}$ tel que $\qf_{n-1}(h)=v$, donc $h$ est congru \`a $1+v$ modulo $\widehat\shu^+_{k,\geq{}n}$. Consid\'erons $h'=(\exp\,e_i).h.(\exp-e_i).h^{-1}\in U^+$. On calcule facilement ce produit dans $\widehat\shu^+_{k}/\widehat\shu^+_{k,\geq{}n+1}$, il est \'egal \`a $1+[e_i,v]$. Donc  $h'\in  U^{ma+}_n$ et $\qf_n(h')=\qf_n(g)$ \ie $g=h'$ modulo $U^{ma+}_{n+1}$.

Ainsi $U^+$ est dense dans $U^{ma+}$; mais $G^{pma}=G.U^{ma+}$ (\cf \ref{3.16}, \ref{3.17}) donc $G$ est dense dans $G^{pma}$.
\end{proof}

\begin{theo}\label{6.12} Supposons $\g g_k$ GK-simple et le corps $k$ infini de caract\'eristique nulle ou $>M$. Alors les groupes topologiques $G^{pma}$, $G^{cgm}$ et $G^{crr}$ sont isomorphes
\end{theo}

\begin{proof} Cela r\'esulte aussit\^ot des propositions \ref{6.7} et \ref{6.11}.
\end{proof}

\subsection{Simplicit\'e}\label{6.13}

On va maintenant g\'en\'eraliser le r\'esultat de simplicit\'e de Moody \cite{My-82}. On suit son sch\'ema de d\'emonstration en reproduisant m\^eme les parties inchang\'ees, \`a cause de la disponibilit\'e restreinte de cette r\'ef\'erence.

On note $G_{(1)}$ (resp. $G^{pma}_{(1)}$) le sous-groupe de $G$ (resp. $G^{pma}$) engendr\'e par les groupes radiciels $U_\qa$ pour $\qa\in\QF$ (resp. et par $U^{ma+}$). Il est normalis\'e par $T$ et les $\widetilde s_{\qa_i}$ donc est distingu\'e dans $G$ (resp. $G^{pma}$). D'apr\`es \ref{6.4} on a $Z'(G^{pma})\subset Z(G).G^{pma}_{(1)}$.

\begin{lemm}\label{6.14} 1) Si $\vert k\vert\geq{}4$, le groupe $G_{(1)}$ est parfait et \'egal au groupe d\'eriv\'e de $G$.

2) Le groupe $G^{pma}_{(1)}$ est topologiquement parfait (\ie \'egal \`a l'adh\'erence de son groupe d\'eriv\'e) et \'egal \`a l'adh\'erence du groupe d\'eriv\'e de $G^{pma}$, dans les cas suivants:

\quad a) La caract\'eristique $p$ de $k$ est nulle ou $>M$ et $\vert k\vert\geq{}4$.

\quad b) La matrice de Kac-Moody $A$ n'a pas de facteur de type affine et $k$ est infini.
\end{lemm}

\begin{rema*} Dans les conditions de 2)b ci-dessus, si de plus $k$ n'est pas extension alg\'ebrique d'un corps fini, alors $G^{pma}_{(1)}$ est parfait.
\end{rema*}

\begin{proof} 1) Soient $\qa\in\QF$ et $r,r'\in k^*$, alors $\qa^\vee(r)\in G_{(1)}$ et $(\qa^\vee(r),x_\qa(r'))=x_\qa((r^2-1)r')$. Donc $U_\qa\subset(G_{(1)},G_{(1)})\subset(G,G)\subset G_{(1)}$ puisque $G=T.G_{(1)}$.

2) Le cas a) r\'esulte de 1),  de la proposition \ref{6.11} et de ce que $G^{pma}=T.G^{pma}_{(1)}$. Pour le cas b) on peut utiliser le lemme \ref{6.15} ci-dessous. Par un calcul analogue au pr\'ec\'edent on en d\'eduit que $U^{ma+}_n$ est parfait modulo $U^{ma+}_{n+1}$.
\end{proof}

\begin{lemm}\label{6.15} On suppose la matrice de Kac-Moody $A$ sans facteur de type affine et le corps $k$ infini, alors $\forall \qa\in\QD$ il existe $t_\qa\in T\cap G_{(1)}$ tel que $\qa(t_\qa)\not=1$.
Si $k$ n'est pas extension alg\'ebrique d'un corps fini on peut supposer $t=t_\qa$ ind\'ependant de $\qa$.
\end{lemm}

\begin{proof} On peut supposer $A$ ind\'ecomposable et m\^eme de type ind\'efini, car le cas de type fini est clair. D'apr\`es \cite[4.3]{K-90} il existe alors $\ql=\sum\,a_j\qa_j^\vee\in Q^\vee\subset Y$ avec $a_j\in\N^*$ et $\qa_i(\ql)<0$ $\forall i,j\in I$. Soient $x\in k^*$ et $t=\ql(x)\in T$; comme $\ql\in Q^\vee$, $t\in G_{(1)}$.
Si $\qa=\sum\,n_i\qa_i\in\QD$, on a $\qa(t)=x^{\qa(\ql)}$; mais les $n_i$ sont tous de m\^eme signe et non tous nuls, donc $\qa(\ql)=\sum\,n_i\qa_i(\ql)\not=0$. Ainsi $\qa(t)\not=1$ si $x$ est  d'ordre $\geq{}\vert\qa(\ql)\vert$ dans $k^*$ ou, mieux, s'il est d'ordre infini.
\end{proof}

\begin{prop}\label{6.16} On suppose $G^{pma}_{(1)}$ topologiquement parfait et la matrice de Kac-Moody $A$ ind\'ecomposable. Soit $K$ un sous-groupe ferm\'e de $G^{pma}$ normalis\'e par $G^{pma} _{(1)}$ , alors $K\subset Z'(G^{pma})$ ou $K\supset G^{pma}_{(1)}$.
\end{prop}

\begin{rema*} Ainsi $G^{pma}_{(1)}/(Z'(G^{pma})\cap G^{pma}_{(1)})$ est topologiquement simple. Notons que, d'apr\`es \ref{3.18}, ce quotient ne d\'epend que de $A$ et $k$ (mais pas de $\shs$).
\end{rema*}

\begin{proof} Le syst\`eme de Tits $(G^{pma},B^{ma+},N)$, les groupes $U^{ma+}$ et $K$ satisfont aux conditions du th\'eor\`eme 5 de \cite[IV {\S{}} 2 n\degres{}7]{B-Lie} sauf peut-\^etre (2) ou (3) (avec $''U''=U^{ma+}$, $''H''=K$ et $''G_1''=G^{pma}_{(1)}$). On peut donc suivre la d\'emonstration de ce th\'eor\`eme jusqu'au point o\`u (2) et (3) sont utilis\'es.
Si $K\not\subset Z'(G^{pma})$ on obtient donc $G^{pma}_{(1)}\subset KU^{ma+}$. Les sous-groupes $KU^{ma+}_n$ sont ouverts donc ferm\'es; comme $G^{pma}_{(1)}$est topologiquement parfait, on a $G^{pma}_{(1)}\subset KU^{ma+}_n$, $\forall n\geq{}1$ (\cf \ref{3.4}).
Soit $g\in G^{pma}_{(1)}$, \'ecrivons $g=k_nu_n$ avec $k_n\in K$ et $u_n\in U^{ma+}_n$. La suite $u_n$ tend vers $1$, donc $k_n=gu_n^{-1}$ tend vers $g$ et comme $K$ est ferm\'e $g\in K$.
\end{proof}

\begin{lemm}\label{6.17} On suppose $G^{pma}_{(1)}$ topologiquement parfait et la matrice de Kac-Moody $A$ ind\'ecomposable. Soient $K$ un sous-groupe de $G^{pma}$ normalis\'e par $G^{pma}_{(1)}$ et $t\in T\cap G^{pma}_{(1)}$ tels que $K\not\subset Z'(G^{pma})$ et $\qa(t)\not=1$ $\forall\qa\in\QD$. Alors $t\in K$.
\end{lemm}

\begin{proof} D'apr\`es \ref{6.16} $\overline K$ contient $G^{pma}_{(1)}$. Soit $k_n$ une suite d'\'el\'ements de $K$ convergeant vers $t$. Pour $s$ assez grand $u=k_st^{-1}\in U^{ma+}$. On va montrer que $t$ est conjugu\'e dans $G^{pma}_{(1)}$ \`a $ut=k_s\in K$, donc $t\in K$.

Montrons qu'il existe des suites $u_n,v_n$ dans $U^{ma+}$ telles que, $\forall n\geq{}1$,

\qquad (1) $u_n,v_n\in U^{ma+}_n$, $u_1=u$,

\qquad (2) $v_nu_ntv_n^{-1}=u_{n+1}t$.

Si c'est le cas, la suite des produits $v_n.v_{n-1}.\cdots.v_2.v_1$ converge vers un \'el\'ement $v\in U^{ma+}$ et $vutv^{-1}=t$.

Supposons $u_1,\cdots,u_r$ et $v_1,\cdots,v_{r-1}$ d\'ej\`a construits. On sait que $U^{ma+}_r/U^{ma+}_{r+1}$ est isomorphe au groupe additif somme des $\g g_{k\qa}$ pour $\deg(\qa)=r$. Notons $\sum_\qa\,X_\qa$ l'\'el\'ement correspondant \`a la classe de $u_r$.
On d\'efinit $v_r$ comme un \'el\'ement de $U^{ma+}_r$ dont la classe modulo $U^{ma+}_{r+1}$ correspond \`a $\sum_\qa\,{{-1}\over{1-\qa(t)}}.X_\qa$. Soit $u_{r+1}=v_ru_rtv_r^{-1}t^{-1}$, sa classe modulo $U^{ma+}_{r+1}$ correspond \`a l'\'el\'ement $\sum_\qa\,({{-1}\over{1-\qa(t)}}+1+{{\qa(t)}\over{1-\qa(t)}}).X_\qa=0$. Donc $u_{r+1}\in U^{ma+}_{r+1}$.
\end{proof}

\begin{lemm}\label{6.18} Soient $G^{pma}$, $K$ et $t$ comme en \ref{6.17}. Si $u\in U^{ma+}$, il existe $v\in U^{ma+}$ tel que $vtv^{-1}t^{-1}=u$, en particulier $u\in K$. Donc $U^{ma+}\subset K$.
\end{lemm}

\begin{proof} Montrons qu'il existe deux suites $u_n,v_n$ dans $U^{ma+}$ telles que:

\qquad (1) $v_n\in U^{ma+}_n$, $u^{-1}u_n\in U^{ma+}_{n+1}$, pour $n\geq{}0$,

\qquad (2) $v_n.\cdots.v_1.t.v_1^{-1}.\cdots.v_n^{-1}=u_{n}t$, pour $n\geq{}1$.

Si c'est le cas, la suite des produits $v_n.v_{n-1}.\cdots.v_2.v_1$ converge vers un \'el\'ement $v\in U^{ma+}$, $u_n$ tend vers $u$ et $vtv^{-1}=ut$.

Posons $u_0=u,v_0=1$ et supposons $u_0,\cdots,u_{r-1},v_0,\cdots,v_{r-1}$ d\'ej\`a construits. La classe de $u^{-1}u_{r-1}$ dans $U^{ma+}_r/U^{ma+}_{r+1}$ correspond \`a un \'el\'ement $\sum_\qa\,X_\qa$ de la somme des $\g g_{k\qa}$ pour $\deg(\qa)=r$. On d\'efinit $v_r$ comme un \'el\'ement de $U^{ma+}_r$ dont la classe modulo $U^{ma+}_{r+1}$ correspond \`a $\sum_\qa\,{{-1}\over{1-\qa(t)}}.X_\qa$.
Alors $u^{-1}.v_r.\cdots.v_1.t.v_1^{-1}.\cdots.v_r^{-1}.t^{-1}=u^{-1}.v_r.u_{r-1}.t.v_r^{-1}.t^{-1}$ a la m\^eme image dans $U^{ma+}_r/U^{ma+}_{r+1}$ que $v_r.u^{-1}.u_{r-1}.t.v_r^{-1}.t^{-1}$ (\ref{3.3}.f) qui correspond \`a $\sum_\qa\,({{-1}\over{1-\qa(t)}}+1+{{\qa(t)}\over{1-\qa(t)}}).X_\qa=0$. Donc, si on pose $u_r=v_r.\cdots.v_1.t.v_1^{-1}.\cdots.v_r^{-1}.t^{-1}$, on a bien $u^{-1}u_r\in U^{ma+}_{r+1}$.
\end{proof}

\begin{theo}\label{6.19} On suppose $G^{pma}_{(1)}$ topologiquement parfait, la matrice de Kac-Moody $A$ ind\'ecomposable non de type affine et le corps $k$ de caract\'eristique $0$ ou de caract\'eristique $p$ mais non alg\'ebrique sur $\F_p$. Alors pour tout sous-groupe $K$ de $G^{pma}$ normalis\'e par $G^{pma}_{(1)}$, soit $K$ est contenu dans $Z'(G^{pma})$, soit il contient $G^{pma}_{(1)}$. En particulier $G^{pma}_{(1)}/(Z'(G^{pma})\cap G^{pma}_{(1)})$ est un groupe simple.
\end{theo}

\begin{proof} Si $K\not\subset Z'(G^{pma})$, les hypoth\`eses des lemmes \ref{6.17} et \ref{6.18} sont v\'erifi\'ees (gr\^ace au lemme \ref{6.15}). Donc $U^{ma+}\subset K$. Mais $G^{pma}_{(1)}$ contient les $\widetilde s_\qa$ (pour $\qa\in\QF$) et normalise $K$, donc $K$ contient aussi les $U_\qa$ pour $\qa\in\QF$, c'est-\`a-dire $K\supset G^{pma}_{(1)}$.
\end{proof}

\begin{remas}\label{6.20} 1) On a vu que $G^{pma}_{(1)}$ est souvent topologiquement parfait (\ref{6.14}) et que $Z'(G^{pma})$ est assez souvent \'egal aux centres $Z(G)$ et $Z(G^{pma})$ (\ref{6.9}.1 et \ref{6.8}).

2) La simplicit\'e de $G^{pma}_{(1)}/(Z'(G^{pma})\cap G^{pma}_{(1)})$ en caract\'eristique $0$ est le r\'esultat essentiel de R. Moody dans \cite{My-82}. Il consid\`ere en fait un groupe d'automorphismes du compl\'et\'e de l'alg\`ebre de Kac-Moody simple associ\'ee \`a une matrice de Kac-Moody $A$ (ind\'ecomposable non de type affine). Il n'y a donc pas pour lui d'hypoth\`ese de GK-simplicit\'e de $\g g_k$ (et $Z'(G^{pma})=Z(G)=Z(G^{pma})=\{1\}$ pour son choix de groupe).

3) Pour un corps fini (de cardinal $\geq{}4$) et une matrice $A$ $2-$sph\'erique ind\'ecomposable non de type affine, Carbone, Ershov et Ritter montrent la simplicit\'e de $G^{rr}_{(1)}$ \cite[th. 1.1]{CER-08}.
Ce r\'esultat vient d'\^etre g\'en\'eralis\'e par T. Marquis \cite{Mar13}: $G^{rr}_{(1)}$ est simple si le corps $k$ est fini et $A$ ind\'ecomposable non de type affine. Il utilise \'egalement les parties  \ref{s2} et \ref{s3}  ci-dessus, mais aussi des r\'esultats sur les groupes de contraction dus \`a Caprace, Reid et Willis.

\par Ce groupe $G^{rr}_{(1)}$ est l'adh\'erence de l'image de $G_{(1)}$ dans Aut$(\shi^v_+)$, c'est-\`a-dire, par compacit\'e, l'image de l'adh\'erence $\overline{G_{(1)}}$ de $G_{(1)}$ dans $G^{pma}$. Donc $G^{rr}_{(1)}=\overline{G_{(1)}}/(Z'(G^{pma})\cap \overline{G_{(1)}})$.
En grande caract\'eristique, \ref{6.11} montre que $\overline{G_{(1)}}=G^{pma}_{(1)}$; on retrouve alors un \'enonc\'e analogue \`a \ref{6.19}.

Si $k$ est fini mais assez grand, on sait aussi que $G_{(1)}/(Z(G)\cap G_{(1)})$ est simple \cite[th. 20]{CR-09}. On en d\'eduit facilement le m\^eme r\'esultat pour $k$ extension alg\'ebrique d'un corps fini.

4) Si $A$ est ind\'ecomposable de type affine non tordu, $G^{pma}_{(1)}/(Z'(G^{pma})\cap G^{pma}_{(1)})$ est le groupe des points sur $k(\!(t)\!)$ d'un groupe alg\'ebrique simple d\'eploy\'e, il est donc simple. Le m\^eme r\'esultat est sans doute encore vrai si on enl\`eve "non tordu" et "d\'eploy\'e".
\end{remas}

\bigskip
\parni{\bf Index des notations}
\medskip

\parni{\bf\ref{1.1}} : $A$, $\shs$, $Y$, $X$, $\overline\qa_i$, $\qa_i^\vee$, $Q$, $\qa_i$, $Q^{\pm}$, $Q^\vee$, $\g g_\shs$, $\g h_\shs$, $\g g_\qa$, $\QD$, $\QD^{\pm}$, $\g n^{\pm}_\shs$, $\g b^{\pm}_\shs$, $W^v$, $\QF=\QD_{re}$, $\QD_{im}$, $s_\qa$

\medskip
\parni{\bf\ref{1.2}} : $\g T_\shs$, $\g T_\qf$
\qquad\qquad{\bf\ref{1.3}} :  $\shs^s$, $\shs^{ad}$, $\shs^{sc}$, $\shs^{\ell}$, $\shs^{sc\ell}$, $\shs^{mat}$

\medskip
\parni{\bf\ref{1.4}} : $s_i^*$, $W^*$, $e_\qa$, $[\qa,\qb]$, $]\qa,\qb[$,
\qquad\qquad{\bf\ref{1.5}} : $\g{St}_A$, $\g U_\psi$, $\g U_\qa$, $\g x_\qa$, $\widetilde s_\qa$, $\qa^*$

\medskip
\parni{\bf\ref{1.6}} : $\g G_\shs$, $\g G_\#$, $\g U_\#^{\pm}$, $\g U^{\pm}_\shs$, $\g B^{\pm}_\shs$, $\g N_\shs$, $\qn^v$, $m(u)$, $\shi^v_\pm{}$ ($\shs$ facultatif)
\qquad\qquad{\bf\ref{1.7}} : $\overline{\g{St}}_A$

\medskip
\parni{\bf\ref{2.1}} : $e_i^{(n)}$, $f_i^{(n)}$, $\shu_\C$, $\shu_\shs$, $\shu^{\pm}_\shs$, $\shu^0_\shs$, $\shu_{\shs\qa}$, $\g g_{\shs\Z}$, $\g n^{\pm}_{\shs\Z}$ ($\shs$ facultatif) $\nabla$, $\qe$, $\widehat{\shu}^+$, $\widehat{\shu}^p$, $\widehat{\g g}^p$, $\widehat{\shu}^-$, $\widehat{\shu}^n$, $\widehat{\g g}^n$

\medskip
\parni{\bf\ref{2.4}} : $x^{[n]}$, \qquad\qquad{\bf\ref{2.6}} : $\shb_\qa$, $[N]$, \qquad\qquad{\bf\ref{2.8}} : $[\exp]$

\medskip
\parni{\bf\ref{2.13}} : $\shu^\qa$, $\g g_\Psi$, $\shu(\Psi)$, $\widehat\shu(\Psi)$, $\widehat\shu^p(\Psi)$, $\widehat\shu^n(\Psi)$, \qquad\qquad{\bf\ref{2.14}} : $L(\ql)$, $L_\Z(\ql)$

\medskip
\parni{\bf\ref{3.1}} : $\g U^{ma}_\Psi$, $\g U^{ma+}$, $\g U_{(\qa)}$, $\g U_\qa$,
\qquad\qquad{\bf\ref{3.2}} : $\g x_\qa$

\medskip
\parni{\bf\ref{3.5}} : $\g B_\shs^{ma+}$, $\g P_i^{pma}$, $\g P_i$, $\g A_i^Y$
\qquad\qquad{\bf\ref{3.6}} : $\g E(\widetilde w)$, $\g D(\widetilde w)$, $\g B(w)$, $\g G^{pma}$

\medskip
\parni{\bf\ref{3.11}} : $\g U_\qa$, $\g x_\qa$, \qquad {\bf\ref{3.12}} : $\g G^{nma}$
\qquad{\bf\ref{3.16}} : $G^{pma}$, $N$, $U^{ma+}$, $U^-$, $T$, $S$,
\qquad{\bf\ref{3.17}} : $U^+$, $B^+$, $G$

\medskip
\parni{\bf\ref{4.1}} : $\sht$, $C^v_f$, $\sht^\circ$, $U_\qa$, $x_\qa$,
\qquad\qquad{\bf\ref{4.2}} : $K$, $\qo$, $\QL$, $\sho$, $\A$, $\qn$, $M(\qa,k)$, $\shm$, $\shm^i$, $D(\qa,k)$, $D^\circ(\qa,k)$, $P^\vee$, $cl$, $F^\ell(x,F^v)$, $F(x,F^v)$, $W$, $\qf_\qa$, $\widetilde\R$, $U_{\qa,\ql}$

\medskip
\parni{\bf\ref{4.3}} : $f_\QO$, $U_{\qa,\QO}$, $U_\QO^{(\qa)}$, $N_\QO^{(\qa)}$, $U_\QO$, $U^{\pm}_\QO$, $U_\QO^{\pm{}\pm{}}$, $N^u_\QO$, $W_\QO^{min}$, $W_\QO$

\medskip
\parni{\bf\ref{4.4}} : $\g g_{\qa,\ql}$, $\g g_{\qa,\QO}$, $\g h_\sho$, $\g g_\QO$, $\widehat{\g g}_\QO^p$, $\widehat{\g g}_\QO^n$, $\shu_\QO$, $\shu_{\qa,\QO}$, $\widehat\shu_\QO^p$, $\widehat\shu_\QO^n$, $M_{\qm,\QO}$, $M_\QO$, $\widehat M_\QO^p$, $\widehat M_\QO^n$

\medskip
\parni{\bf\ref{4.5}} : $\shu(\Psi)_\QO$, $\widehat\shu(\Psi)_\QO$, $U^{ma}_\QO(\Psi)$, $U^{pm}_\QO(\Psi)$, $U_\QO^{ma+}$, $U_\QO^{pm+}$, $U_\QO^{nm-}$

\medskip
\parni{\bf\ref{4.6}} : $U_\QO^{pm}$,  $U_\QO^{nm}$, $N_\QO^{min}$, $\widehat N_\QO$, $P_\QO^{min}$, $P_\QO^{pm}$, $P_\QO^{nm}$

\medskip
\parni{\bf\ref{4.9}} : $\widetilde P_\QO$, $\widetilde N_\QO$, $\shp$
\qquad\qquad{\bf\ref{4.13}} : $\widehat P_\QO$, $\widehat P_x$

\medskip
\parni{\bf\ref{5.1}} : $\SHI=\SHI(\g G,K)$
\qquad\qquad{\bf\ref{5.2}} : $G_\QO$, $G(\QO\subset\A)$
\qquad\qquad{\bf\ref{5.7}} : $\leq{}$, ${\bc\le}$, $\g q$, $\g Q$

\medskip
\parni{\bf\ref{5.12}} : $\g r(F,F^v)$, $\g R(F,F^v)$, $P(F^v)$, $M(F^v)$, $U(F^v)$, $P^\qm(\g R)$

\medskip
\parni{\bf\ref{6.1}} : $G^{rr}$, $Z'(G)$, $G^{crr}$, $U^{rr+}$, $Z'(G^{crr})$
\qquad\qquad{\bf\ref{6.2}} : $G^{cg\ql}$, $G^{cgm}$

\medskip
\parni{\bf\ref{6.3}} : $U_n^{ma+}$, $U_n^+$, $M$, $\phi$, $\qg$, $\qr$
\qquad\qquad{\bf\ref{6.13}} : $G^{}_{(1)}$, $G^{pma}_{(1)}$


\medskip

Guy.Rousseau@univ-lorraine.fr

\parni 1. Universit\'e de Lorraine, Institut Elie Cartan de Lorraine, UMR 7502,  Vand\oe uvre-l\`es-Nancy, F-54506, France.
\parni 2. CNRS, Institut Elie Cartan de Lorraine, UMR 7502,  Vand\oe uvre-l\`es-Nancy, F-54506, France.

\section{Erratum 6 Juin 2025}

\textbf{Erratum} \`a ``Groupes de Kac-Moody d\'eploy\'es sur un corps local, II.  Masures ordonn\'ees''

\medskip

\par Auguste H\'ebert vient de d\'emontrer que tout filtre dans l'appartement $\A$ a un bon fixateur \cite{Heb25}.
Donc le paragraphe 4.12.3 (c) (de \cite{R12}) en particulier sa derni\`ere phrase, est faux.
Les lignes ci-dessous devraient remplacer ce paragraphe (c).
J'y explique l'erreur faite en rectifiant quelques r\'esultats et ajoutant quelques pr\'ecisions.

\medskip

\par (b2) Pour mieux comprendre l'action du groupe $G$ sur l'appartement $\A$, il faut consid\'erer le SGR libre et colibre $\SHS^ {mat}=(A=\left(\begin{matrix}2&-2  \cr -2&2\cr \end{matrix}\right);Y^ {mat}=\Z h\oplus\Z d\oplus\Z c;\qa_{0},\qa_{1};\qa_{1}^\vee=h, \qa_{0}^\vee=-h+ c)$ d\'efini par $\qa_{1}(h)=2,\qa_{1}(d)=\qa_{1}(c)=0,\qa_{0}=\qd-\qa_{1},\qd(h)=\qd(c)=0,\qd(d)=1$; c'est le ``Kac-Moody root datum'' $\cal{D}^A_{Kac}$ de \cite[7.10 et 7.14 (3)]{Mar18}.
On consid\`ere aussi le SGR libre $\SHS^ {\ell}=(A;Y^ {\ell}=\Z h\oplus\Z d;\qa_{0},\qa_{1};\qa_{0}^\vee=-\qa_{1}^\vee=-h)$ d\'efini par les m\^emes relations (avec $c=0$).
Alors, d'apr\`es \cite[\S 7.6 p. 166, 167]{Mar18}, $\g G_{\SHS^{\ell}}$ est le quotient de $\g G_{\SHS^{mat}}$ par son centre et s'identifie au produit semi direct $\g G_{\SHS}\rtimes\mathfrak{Mult}$; de plus l'action associ\'ee de $K^*=\mathfrak{Mult}(K)$ sur $G=\g G_{\SHS}(K)=SL_{2}(K[t,t^ {-1}])$ est donn\'ee par :
$R.\left(\begin{matrix}a(t)&b(t)  \cr c(t)&d(t)\cr \end{matrix}\right).R^ {-1}=\left(\begin{matrix}a(Rt)&b(Rt)  \cr c(Rt)&d(Rt)\cr \end{matrix}\right)$ pour $R\in K^*$ (voir aussi \cite[\S 6.1 eq. (6.1)]{BarHR25}).
Comme $\g{Mult}$ est dans le tore canonique $\g T_{\SHS^{\ell}}$ de $\g G_{\SHS^{\ell}}$, il agit sur l'appartement $\A=Y^ {\ell}\otimes\R=\R h\oplus \R d$ : l'action de $R\in K^*$ avec $\qo(R)=r\in\R$ est donn\'ee par $\qn_{r}(xh+yd)=xh+(y-r)d$.
Par ailleurs la matrice $\qt_{S}=\left(\begin{matrix}S&0  \cr 0&S^ {-1}\cr \end{matrix}\right)\in G$ avec $S\in K^*$ et $\qo(S)=p\in\R$ agit sur $\A$ par la translation $\qt_{2p}(xh+yd)=(x-2p)h+yd$.


\par Comme $0\in\A$ est un point sp\'ecial, son fixateur $G_{0}$ est \'egal \`a $U_{0}$ et donc $G_{0}=SL_{2}(\sho[t,t^ {-1}])$ (\cf fin de 4.12.3 (b)).
En conjuguant par $\qt_{S}$ et faisant agir  $R\in K^ {*}$ (avec $\qo(S)=q, \qo(R)=r$), on trouve que le fixateur $G_{2q,r}$ de $(2q,r)=2qh+rd$ est : $\qt_{S}^ {-1}R^ {-1}G_{0}R\qt_{S}$ $=
\set{\left(\begin{matrix}a(R^{-1}t)&S^ {-2}b(R^{-1}t)  \cr S^2c(R^{-1}t)&d(R^{-1}t)\cr \end{matrix}\right)\in SL_{2}(K[t,t^ {-1}]) \mid a(t),b(t),c(t),d(t)\in \sho[t,t^ {-1}]}
=\set{\left(\begin{matrix}a'(t)&b'(t)  \cr c'(t)&d'(t)\cr \end{matrix}\right)\in SL_{2}(K[t,t^ {-1}]) \mid a'(Rt),d'(Rt),S^2b'(Rt),S^ {-2}c'(Rt))\in \sho[t,t^ {-1}]}$.
Ceci est valable pour $q,r\in\qo(K^*)$.
Mais si $q,r\in\Q\qo(K^*)$, il existe une extension ramifi\'ee $(K',\qo')$ de $(K,\qo)$ d'indice de ramification $e(K'/K)=e$ telle que $eq,er\in2\qo(K^*)$.
Alors il existe $R,S\in K'$ tels que $\qo'(R)=r,\qo'(S)=q/2$, donc, si on note $\sho'$ l'anneau des entiers de $(K',\qo')$, on a $G_{q,r}=\set{\left(\begin{matrix}a'(t)&b'(t)  \cr c'(t)&d'(t)\cr \end{matrix}\right)\in SL_{2}(K[t,t^ {-1}]) \mid a'(Rt),d'(Rt),S^2b'(Rt),S^ {-2}c'(Rt))\in \sho'[t,t^ {-1}]}$.

\medskip
\par (b3) On a $U_{0}^ {ma+}=\g U^ {ma+}(\sho)$ d'apr\`es la d\'efinition en 4.5 (2) et la proposition 3.2 ; de m\^eme $U_{0}^ {ma-}=\g U^ {ma-}(\sho)$.
D'apr\`es l'exemple de $\widetilde{SL}_{2}$ en 2.12 et la proposition 3.2, on a 
$\g U^ {ma+}(\sho)=\set{\left(\begin{matrix}1&b  \cr 0&1\cr \end{matrix}\right)\left(\begin{matrix}x&0  \cr 0&x^ {-1}\cr \end{matrix}\right)\left(\begin{matrix}1&0  \cr tc&1\cr \end{matrix}\right)=\left(\begin{matrix}x+btcx^ {-1}&bx^ {-1}  \cr tcx^ {-1}&x^ {-1}\cr \end{matrix}\right) \mid b,c,x\in\sho[[t]],x\equiv 1\ \mathrm{mod}\ t}$.
Donc $U_{0}^ {ma+}=\g U^ {ma+}(\sho)=\set{\left(\begin{matrix}a&b  \cr c&d\cr \end{matrix}\right)\in SL_{2}(\sho[[t]]) \mid a,d \equiv 1, c\equiv 0 \ \mathrm{mod}\ t}$.
De m\^eme $U_{0}^ {ma-}=\g U^ {ma-}(\sho)=\set{\left(\begin{matrix}a&b  \cr c&d\cr \end{matrix}\right)\in SL_{2}(\sho[[t^ {-1}]]) \mid a,d \equiv 1, b\equiv 0 \ \mathrm{mod}\ t^ {-1}}$.
Les m\^emes \'egalit\'es sont vraies si on remplace $\sho$ par $K$ et donc $U_{0}^ {ma\pm}$ par $U_{}^ {ma\pm}$.

\medskip
\par (c2) On suppose pour simplifier $\qo$ discr\`ete d'uniformisante $\qpp$ (avec $\qo(\qpp)=1$),  $p\in\N^*$ et $K$ complet.
On note $z_{p}$ le point $(p,0)=\qt_{-p}(0)$ de $\A$ (\ie $\qd(z_{p})=0$ et $\qa_{1}(z_{p})=p$) et $\QO_{p}=\set{0,z_{p}}$.
L'enclos $cl(\QO_{p})$ est le segment $[0,z_{p}]$ ; par contre l'enclos (renforc\'e) $cl^ {\#}(\QO_{p})$ (d\'efini en 4.2.5) est la r\'eunion de $\QO_{p}$ et du filtre des voisinages dans $\A$ de l'intervalle ouvert $]0,z_{p}[$.

\par Comme ci-dessus (dernier paragraphe de (b2)), on trouve que $U_{z_{2p}}^ {ma+}=\qt_{\qpp^ {-p}}U_{0}^ {ma+}\qt_{\qpp^p}=\set{\left(\begin{matrix}a&b\qpp^ {-2p}  \cr c\qpp^{2p}&d\cr \end{matrix}\right) \mid \left(\begin{matrix}a&b  \cr c&d\cr \end{matrix}\right) \in U_{0}^ {ma+}}$ et donc $U_{\QO_{2p}}^ {ma+}=U_{0}^ {ma+}\cap U_{z_{2p}}^ {ma+}=\set{\left(\begin{matrix}a&b  \cr c&d\cr \end{matrix}\right)\in SL_{2}(\sho[[t]]) \mid a,d \equiv 1\ \mathrm{mod}\ t, c\equiv 0 \ \mathrm{mod}\ \qpp^{2p}t}$.
 De m\^eme $U_{\QO_{2p}}^ {ma-}=\set{\left(\begin{matrix}a&b  \cr c&d\cr \end{matrix}\right)\in SL_{2}(\sho[[t^ {-1}]]) \mid a,d \equiv 1\ \mathrm{mod}\ t^ {-1}, b\equiv 0 \ \mathrm{mod}\ t^ {-1} , c\equiv 0 \ \mathrm{mod}\ \qpp^{2p}}$.
 Comme \`a la fin de (b2) on voit que ces r\'esultats sont aussi valables avec $2p$ remplac\'e par $p$.

\par Ces groupes sont en fait plus grands que ce qui \'etait indiqu\'e dans 4.12.3 (c) (car $p\geq1$).
Ce r\'esultat faux \'etait ``justifi\'e'' par le fait que $U_{\QO_{p}}^ {ma+}$ aurait \'et\'e topologiquement engendr\'e par $u^s(\sho[[t]])$ et $u^i(\qpp^pt\sho[[t]])$ (d\'eduit abusivement d'une phrase de J. Tits dans [40], fin de 3.10 (d) page 555).
En fait on oubliait ainsi la contribution des racines imaginaires \`a $\g U^ {ma\pm}$.

\par On trouve les formules pour $U_{0}^ {pm+}, U_{\QO_{p}}^ {pm+}$ (\resp $U_{0}^ {nm-}, U_{\QO_{p}}^ {nm-}$) en intersectant avec $G=SL_{2}(K[t,t^ {-1}])$.
Ce sont les m\^emes que pour $U_{0}^ {ma+}, U_{\QO_{p}}^ {ma+}$ (\resp $U_{0}^ {ma-}, U_{\QO_{p}}^ {ma}$) en rempla\c cant $\sho[[t]]$ (\resp $\sho[[t^ {-1}]]$) par $\sho[t]$ (\resp $\sho[t^ {-1}]$).
En particulier $U_{0}^ {pm+}=\set{\left(\begin{matrix}a&b  \cr c&d\cr \end{matrix}\right)\in SL_{2}(\sho[t]) \mid a,d \equiv 1, c\equiv 0 \ \mathrm{mod}\ t}$,
$U_{\QO_{p}}^ {pm+}=\set{\left(\begin{matrix}a&b  \cr c&d\cr \end{matrix}\right)\in SL_{2}(\sho[t]) \mid a,d \equiv 1 \ \mathrm{mod}\ t, c\equiv 0 \ \mathrm{mod}\ \qpp^pt}$,
$U_{\QO_{p}}^ {nm-}=\set{\left(\begin{matrix}a&b  \cr c&d\cr \end{matrix}\right)\in SL_{2}(\sho[t^ {-1}]) \mid a,d \equiv 1 \ \mathrm{mod}\ t^ {-1}, b\equiv 0 \ \mathrm{mod}\ t^ {-1}, c\equiv 0 \ \mathrm{mod}\ \qpp^p}$.
On peut remarquer que $U_{\QO_{p}}^ {pm+}$ contient la matrice $g$ de 4.12.3 (a) si $p=1$ ou $2$.
 Si $\qo(S)=p$, on a \'egalement $G_{\QO_{p}}=G_{0}\cap G_{z_{p}}=G_{0}\cap \qt_{S}^ {-1}G_{0}\qt_{S}=\set{\left(\begin{matrix}a&b  \cr c&d\cr \end{matrix}\right)\in SL_{2}(\sho[t,t^ {-1}]) \mid  c\equiv 0 \ \mathrm{mod}\ \qpp^p}$.
 
\par Si $p\geq1$ une matrice $g=\left(\begin{matrix}a&b  \cr c&d\cr \end{matrix}\right)\in U_{\QO_{p}}^ {pm+}$ satisfait \`a   $c\equiv 0 \ \mathrm{mod}\ \qpp^p t$ et $a,d\equiv 1 \ \mathrm{mod}\ \qpp t$ (comme $ad-bc=1$ et $c\equiv 0 \ \mathrm{mod}\ \qpp$, on a $ad=1$ mod $\qpp$ et $a,d$ sont des constantes modulo $\qpp$).
D'apr\`es le dernier r\'esultat de (b2) on en d\'eduit que $g$ est dans $G_{q,r}$ pour tous $q,r\in\Q$ satisfaisant $0<q<p$ et $\vert r\vert$ assez petit (si $eq$ est dans $\Z$ il suffit que $1/(e\vert r\vert$) soit plus grand que les degr\'es de $a,b,c,d\in\sho[t])$).
Cela signifie (par convexit\'e) que $U_{\QO_{p}}^ {pm+}$ fixe l'enclos (renforc\'e) $cl^ {\#}(\QO_{p})$, \ie 
 $U_{\QO_{p}}^ {pm+}=U_{cl^ {\#}(\QO_{p})}^ {pm+}$. De m\^eme $U_{\QO_{p}}^ {nm-}=U_{cl^ {\#}(\QO_{p})}^ {nm-}$.
Plus g\'en\'eralement, si on remplace $\QO_{p}$ par n'importe quel filtre $\QO$ dans $\A$, ceci est  un r\'esultat d'Auguste H\'ebert \cite{Heb25}.

\medskip
\par {\bf Remarques} (i) Soit $A$ un appartement de $\SHI$ contenant $\QO_{p}$, d'apr\`es \cite[Th. 3.6]{Heb22} il existe un isomorphisme $\qf:\A\to A$ fixant $\A\cap A$ et $\A\cap A$ contient $cl^ {\#}(\QO_{p})$.
Mais $cl^ {\#}(\QO_{p})$ contient une chambre locale $C=F^\ell(z',C^ {\mathrm{v}})$ pour $z'\in]0,z_{p}[$.
Comme cette chambre a un bon fixateur (\cf 5.7.2) il existe $g\in G_{C}$ tel que $A=g\A$.
Ainsi $\qf^ {-1}\circ g$ est un automorphisme de $\A$ fixant $C$ : c'est  l'identit\'e.
Donc $g$ fixe $\QO_{p}$ et  $cl^ {\#}(\QO_{p})$.
On en d\'eduit que $G_{\QO_{p}}$ ou $G_{cl^ {\#}(\QO_{p})}$ satisfait \`a l'axiome (TF)  de 5.3. 

\par (ii) Le lecteur pourra essayer de montrer que $G_{\QO_{p}}=U_{\QO_{p}}^ {pm+}U_{\QO_{p}}^ {nm-}N_{\QO_{p}}=U_{\QO_{p}}^ {nm-}U_{\QO_{p}}^ {pm+}N_{\QO_{p}}$ avec $N_{\QO_{p}}=\set{\left(\begin{matrix}a&0  \cr 0&d\cr \end{matrix}\right)\in SL_{2}(\sho[t,t^ {-1}])}$ (agissant par transvections lin\'eaires sur $\A$). 
Autrement dit $\QO_{p}$ a un bon fixateur.
Comme indiqu\'e au d\'ebut, ceci est un r\'esultat d'Auguste H\'ebert \cite{Heb25} pour $\QO_{p}$ remplac\'e par n'importe quel filtre $\QO$ dans $\A$.
Mais $N_{\QO_{p}}\not\subset G_{cl^ {\#}(\QO_{p})}$, donc $G_{\QO_{p}} \neq G_{cl^ {\#}(\QO_{p})}$.

\par (iii) Il est clair que $U_{\QO_{p}}^ {ma+}\neq U_{cl^ {\#}(\QO_{p})}^ {ma+}$ et $U_{\QO_{p}}^ {ma-}\neq U_{cl^ {\#}(\QO_{p})}^ {ma-}$.
Par contre, par d\'efinition, $U_{\QO}^ {ma+}= U_{cl(\QO)}^ {ma+}$ et $U_{\QO}^ {ma-}= U_{cl(\QO)}^ {ma-}$ pour n'importe quel filtre $\QO$ dans $\A$.

\end{document}